\documentclass[reqno,english]{amsart}

\usepackage{amsmath,amsfonts,amssymb,graphicx,amsthm,enumerate,url}
\usepackage[noadjust]{cite}
\usepackage{stmaryrd}
\usepackage{mathrsfs,booktabs,upgreek}
\usepackage{xifthen,xcolor,tikz,setspace}
\usetikzlibrary{decorations.pathmorphing,patterns,shapes,calc,decorations}
\usetikzlibrary{decorations.pathreplacing}
\usepackage{mathtools}
\usepackage[final]{showkeys} 

\definecolor{refkey}{gray}{.75}
\definecolor{labelkey}{gray}{.5}

\usepackage[letterpaper]{geometry}
\usepackage[colorinlistoftodos]{todonotes}

\usepackage[colorlinks=true]{hyperref}
\colorlet{DarkGreen}{green!50!black}
\colorlet{DarkGray}{gray!60!black}

\numberwithin{equation}{section}

\renewcommand{\restriction}{\mathord{\upharpoonright}}
\renewcommand{\epsilon}{\varepsilon}
\newcommand{\given}{\;\big|\;}
\newcommand{\one}{\mathbf{1}}

\usepackage{color}
 \definecolor{refkey}{gray}{.5}
 \definecolor{labelkey}{gray}{.5}
\definecolor{light}{gray}{.9}

\newtheorem{maintheorem}{Theorem}

\newtheorem{maincoro}[maintheorem]{Corollary}

\newtheorem{theorem}{Theorem}[section]
\newtheorem*{theorem*}{Theorem}
\newtheorem{lemma}[theorem]{Lemma}
\newtheorem{lem}[theorem]{Lemma}

\newtheorem{claim}[theorem]{Claim}
\newtheorem{proposition}[theorem]{Proposition}
\newtheorem{prop}[theorem]{Proposition}

\newtheorem{observation}[theorem]{Observation}

\newtheorem{corollary}[theorem]{Corollary}

\theoremstyle{definition}{

\newtheorem{definition}[theorem]{Definition}

\newtheorem*{definition*}{Definition}
\newtheorem{problem}[theorem]{Problem}

\newtheorem{remark}[theorem]{Remark}

}

\newcommand{\E}{\mathbb E}

\renewcommand{\P}{\mathbb P}

\newcommand{\R}{\mathbb R}
\newcommand{\Z}{\mathbb Z}
\newcommand{\cA}{\ensuremath{\mathcal A}}
\newcommand{\cB}{\ensuremath{\mathcal B}}
\newcommand{\cC}{\ensuremath{\mathcal C}}

\newcommand{\cE}{\ensuremath{\mathcal E}}
\newcommand{\cF}{\ensuremath{\mathcal F}}
\newcommand{\cG}{\ensuremath{\mathcal G}}
\newcommand{\cH}{\ensuremath{\mathcal H}}
\newcommand{\cI}{\ensuremath{\mathcal I}}
\newcommand{\cJ}{\ensuremath{\mathcal J}}

\newcommand{\cL}{\ensuremath{\mathcal L}}

\newcommand{\cN}{\ensuremath{\mathcal N}}

\newcommand{\cP}{\ensuremath{\mathcal P}}

\newcommand{\cS}{\ensuremath{\mathcal S}}

\newcommand{\cZ}{\ensuremath{\mathcal Z}}
\newcommand{\llb }{\llbracket}
\newcommand{\rrb }{\rrbracket}

\newcommand{\fC}{\mathfrak{C}}
\newcommand{\fm}{\mathfrak{m}}
\newcommand{\fF}{\mathfrak{F}}

\newcommand{\fS}{\mathfrak{S}}
\newcommand{\fW}{\mathfrak{W}}
\newcommand{\fX}{\mathfrak{X}}

\newcommand{\sB}{{\ensuremath{\mathscr B}}}

\newcommand{\sF}{{\ensuremath{\mathscr F}}}

\newcommand{\sP}{{\ensuremath{\mathscr P}}}

\newcommand{\sT}{{\ensuremath{\mathscr T}}}
\newcommand{\sX}{{\ensuremath{\mathscr X}}}

\newcommand{\sZ}{{\ensuremath{\mathscr Z}}}

\newcommand{\g}{{\ensuremath{\mathbf g}}}
\newcommand{\br}{{\ensuremath{\mathbf r}}}

 \renewcommand{\epsilon}{\varepsilon}
\DeclareMathOperator{\var}{Var}
\DeclareMathOperator{\dist}{dist}
\DeclareMathOperator{\diam}{diam}
\DeclareMathOperator{\cov}{Cov}
\DeclareMathOperator{\ber}{Ber}

\DeclareMathOperator{\hgt}{ht}
\newcommand{\tv}{{\textsc{tv}}}
\newcommand{\spine}{{\textsc{sp}}}
\newcommand{\Tsp}{{\tau_\spine}}
\newcommand{\trivincr}{\varnothing}
\makeatletter
\newcommand{\superimpose}[2]{%
  {\ooalign{$#1\@firstoftwo#2$\cr\hfil$#1\@secondoftwo#2$\hfil\cr}}}
\makeatother

\begin{document}

\title{Maximum and shape of Interfaces in 3D Ising Crystals}

\author{Reza Gheissari}
\address{R.\ Gheissari\hfill\break
Department of Statistics\\ UC Berkeley \\ 367 Evans Hall\\ Berkeley, CA 94720, USA.}
\email{gheissari@berkeley.edu}

\author{Eyal Lubetzky}
\address{E.\ Lubetzky\hfill\break
Courant Institute\\ New York University\\
251 Mercer Street\\ New York, NY 10012, USA.}
\email{eyal@courant.nyu.edu}

\vspace{-1cm}

\begin{abstract}
Dobrushin (1972) showed that the interface of a 3D Ising model with minus boundary conditions above the $xy$-plane and plus below is rigid (has $O(1)$-fluctuations) at every sufficiently low temperature. 
Since then, basic features of this interface---such as the asymptotics of its maximum---were only  identified in more tractable random surface models that approximate  the Ising interface at low temperatures, e.g., for the (2+1)D Solid-On-Solid model.
Here we study the large deviations of the interface of the 3D Ising model in a cube of side-length~$n$ with Dobrushin's boundary conditions, and in particular obtain a law of large numbers for $M_n$, its maximum: if the inverse-temperature $\beta$ is large enough, then $M_n / \log n \to 2/\alpha_\beta$ as $n\to\infty$, in probability, where $\alpha_\beta$ is given by a large deviation rate in infinite volume.

We further show that, on the large deviation event that the interface connects the origin to height $h$, it consists of a 1D spine that behaves like a random walk, in that it decomposes into a linear (in $h$) number of asymptotically-stationary weakly-dependent increments that have exponential tails. As the number $T$ of increments diverges, properties of the interface such as its surface area, volume, and the location of its tip, all obey CLTs with variances linear in $T$.
These results generalize to every dimension $d\geq 3$.
\end{abstract}

{\mbox{}
\vspace{-.8cm}
\maketitle
}
\vspace{-0.8cm}

\section{Introduction}

We study the plus-minus Ising interface in $d$-dimensions at sufficiently low temperatures, where for $d\geq 3$ the interface is known to be rigid and yet its large deviations, including the asymptotic behavior of its maximum, were unknown. The Ising model on a finite subgraph $\Lambda\subset\mathbb Z^d$ is an assignment of $\pm 1$ to the $d$-dimensional \emph{cells} of $\mathbb Z^d$ (faces when $d=2$ and cubes of side-length 1 when $d=3$), collected in the set $\mathcal C(\Lambda)$. These cells are identified with their midpoints, corresponding to the vertices of the dual graph $(\mathbb Z+\frac 12)^d$, and $u,v\in \mathcal C(\Lambda)$ are considered adjacent (denoted $u\sim v$) if their midpoints are at Euclidean distance 1. The Ising model on $\Lambda$ is then the Gibbs distribution $\mu_\Lambda=\mu_{\Lambda,\beta}$ over configurations in $ \Omega = \{\pm 1\}^{\mathcal C(\Lambda)}$ given by 
\[ \mu_\Lambda(\sigma) \propto \exp\left[ -\beta \cH(\sigma)\right]\,, \quad\mbox{for}\quad\cH(\sigma)=\sum_{u\sim v} \one\{\sigma_u \neq \sigma_v\}\,,\]
 where $\beta>0$ is the inverse temperature. Placing boundary condition $\eta$ on the model, $\mu_\Lambda^\eta$, refers to the conditional distribution of $\mu_H$, for some larger given graph $H\supset \Lambda$, where the configuration of $\mathcal C(H)\setminus \mathcal C(\Lambda)$ coincides with $\eta$. These definitions extend to infinite graphs via weak limits, and in the low temperature regime studied here, different boundary conditions $\eta$ on boxes in $\Z^d$ lead to distinct limiting Gibbs distributions~\cite[\S6.2]{Georgii11}.
 
Here, we consider $\beta > \beta_0$ for some fixed $\beta_0$ and $\Lambda=\Lambda_n$, the infinite cylinder of side-length $2n$ in $\Z^d$,
\[\Lambda_n = \llb -n,n\rrb^{d-1}\times \llb -\infty,\infty\rrb=\{ -n,\dots,n\}^{d-1} \times \{-\infty,\dots,\infty\}\,,\]
with boundary conditions that are ($+$) in the lower half-space $\mathcal C(\Z^{d-1}\times \{-\infty,\ldots,0\})$ and ($-$) elsewhere, called \emph{Dobrushin's boundary conditions}.
Let $\mu_n=\mu_{\Lambda_n,\beta}^{\mp}$ denote the Ising model with these boundary conditions, and note that every $\sigma\sim \mu_n$ defines a set of $(d-1)$-cells separating disagreeing spins, which in turn give rise to an \emph{interface} $\cI$ separating the minus and plus phases: in 2D, it is a (maximal) connected component of such separating edges connecting $(-n,0)$ and $(n,0)$; in three dimensions, it is the (maximal) connected component of such separating faces containing $\partial\Lambda_n \cap (\Z^{d-1}\times\{0\})$ (we defer more detailed definitions to~\S\ref{sec:notation}).

The classical argument of Peierls, which established the phase transition in the Ising model for $d\geq 2$, shows that in the above described setting, the size of ``bubbles'' (finite connected components of plus or minus spins) has an exponential tail. One thus looks to determine the behavior of the interface $\cI$.

In the 2D Ising model, the properties of this random interface between plus/minus phases in $\mu_n$ is very well-understood: for $\beta>\beta_c$, the critical point of the Ising model, this interface converges to a Brownian bridge as $n\to\infty$, and detailed quantitative estimates are available for its fluctuations and large deviations for large $n$,  mimicking those of a random walk (see, e.g.,~\cite{DH97,DKS,Hryniv98,Ioffe94,Ioffe95,PV97,PV99}). In view of its height fluctuations that diverge with $n$ (in this case, with variance $C_\beta n$ in the bulk), the interface is referred to as \emph{rough}.

For the 3D Ising model (and in fact extending to every dimension $d\geq 3$), Dobrushin~\cite{Dobrushin72a} famously showed that, for large enough  $\beta$, the plus/minus interface $\cI$ is \emph{rigid} (localized) around height $0$: the height fluctuations are $O(1)$ everywhere. Namely, Dobrushin established that the probability that the interface $\cI$ reaches height at least $h$ above any given $xy$-coordinate in $\llb-n,n\rrb^2$ is $O(\exp(-\tfrac13\beta h))$.
An important consequence of rigidity is that the Gibbs distribution $\mu_{\Z^3}^{\mp}$ arising as the weak limit of $\mu_n$
is not translation-invariant in its $z$-coordinate. It is believed that the interface becomes rigid only after a roughening threshold $\beta_{\textsc r} > \beta_c$, with this roughening phase transition being exclusive to dimension 3.  Interfaces of tilted Dobrushin boundary conditions are, unlike the flat ones, believed to always be rough; see \S\ref{sec:related-work} for more details.

Since Dobrushin's work showing that the interface $\cI$ is typically a flat surface at height $0$, basic features of this interface---such as the asymptotics of its maximum, the shape of the surface near the maximum and the effect of entropic repulsion---were only identified in more tractable random surface models that approximate the Ising interface at low temperatures, e.g., the (2+1)D Solid-On-Solid model by Bricmont, El-Mellouki and Fr\"ohlich~\cite{BEF86} and Caputo et al.~\cite{CLMST1,CLMST2}, and the Discrete Gaussian and $|\nabla\phi|^p$-models in~\cite{LMS16} (in these, the surfaces are height functions, with no overhangs or interacting bubbles that do exist in the Ising model).

\begin{figure}
\centering
\includegraphics[width=0.6\textwidth]{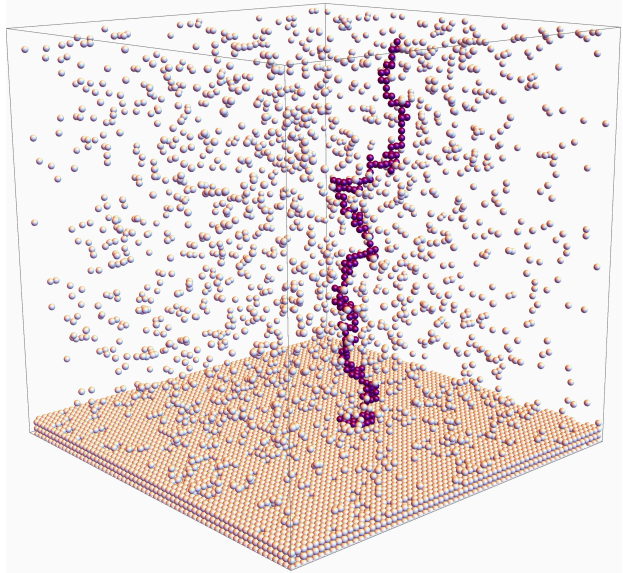}
 \caption{The plus/minus interface $\cI$ in the 3D Ising model $\mu_n$ (side length $n=64$)   with Dobrushin's boundary conditions, when conditioning on $\cI$ reaching height $h=64$.}
  \label{fig:num_pillar}
  \vspace{-0.15cm}
\end{figure}

\medskip
In what follows, for the sake of the exposition, we state our new results on the interface $\cI$ in the context of the 3D Ising model, noting that they extend to any dimension $d\geq 3$ (see Remark~\ref{rem:d>3}).

\subsection{Maximum height} Let $M_n$ be the maximum height ($z$-coordinate) of a face in $\cI$.
Dobrushin's estimate that $\mu_n(\cI\ni(y_1,y_2,h))=O( \exp(-\frac13\beta h))$ shows, by a union bound, that $M_n / \log_n \leq C_\beta$ in probability as $n\to\infty$ for some $C_\beta>0$. As we later explain, a lower bound of matching order, $M_n / \log_n \geq c_\beta$ in probability for some other $c_\beta>0$, can also be deduced from those methods 
via decorrelation estimates.
Our main goals here are obtaining the asymptotics of $M_n$ (law of large numbers (LLN) for the maximum) and characterizing the typical structure of the surface around points conditioned to achieve large deviations.
The first result establishes the LLN and expresses the limit in terms of a large deviation (LD) rate function of having the origin be $*$-connected to height $h$ via ($+$)-spins (denoted $\xleftrightarrow{+}$) within $\mathcal C(\Z^2\times\llb0,h\rrb)$ in the measure $\mu_{\Z^3}^{\mp}$.
\begin{maintheorem}[LLN for the maximum]\label{mainthm:max-lln}
There exists $\beta_0$ such that, for all $\beta>\beta_0$, the maximum $M_n$ of the interface $\cI$ in the 3D Ising model with Dobrushin's boundary conditions $\mu_{\Lambda_n,\beta}^{\mp}$ satisfies
\begin{equation}\label{eq:lln} \lim_{n\to\infty} \frac{M_n}{\log n} = \frac2{\alpha_\beta}\quad\mbox{in probability}\,,\end{equation}
where the constant $\alpha_\beta>0$ is given by
\begin{equation}\label{eq:alpha-beta-def} \alpha_\beta = \lim_{h\to\infty} -\frac1h\log\mu_{\Z^3}^{\mp}\left((\tfrac 12,\tfrac 12,\tfrac 12)\xleftrightarrow[\mathcal C(\Z^2\times\llb0,h\rrb)]{+} ((\Z+\tfrac 12)^2\times\{h-\tfrac 12\})\right)\,,\end{equation}
and satisfies $\alpha_\beta/\beta \to 4$ as $\beta\to\infty$.
\end{maintheorem}

Note that the existence of the limit in~\eqref{eq:alpha-beta-def} is both nontrivial and essential, and its proof (see \S\ref{subsec:limiting-ldp-rate} and in particular Proposition~\ref{prop:limiting-rate}) relies on our results on the structure of the interface $\cI$ conditioned on large deviations in $\mu_n$, which drive an approximate sub-additivity argument.

\begin{figure}
  \centering
  \vspace{-0.25in}
  \begin{tikzpicture}
  
    \node[ellipse,draw,fill=gray!10,inner sep=0mm,minimum width=75pt, minimum height=200pt] (zoom) at (0,0.25) {
    };
    
    \node[left=-26pt] at (zoom) {        \includegraphics[width=0.13\textwidth]{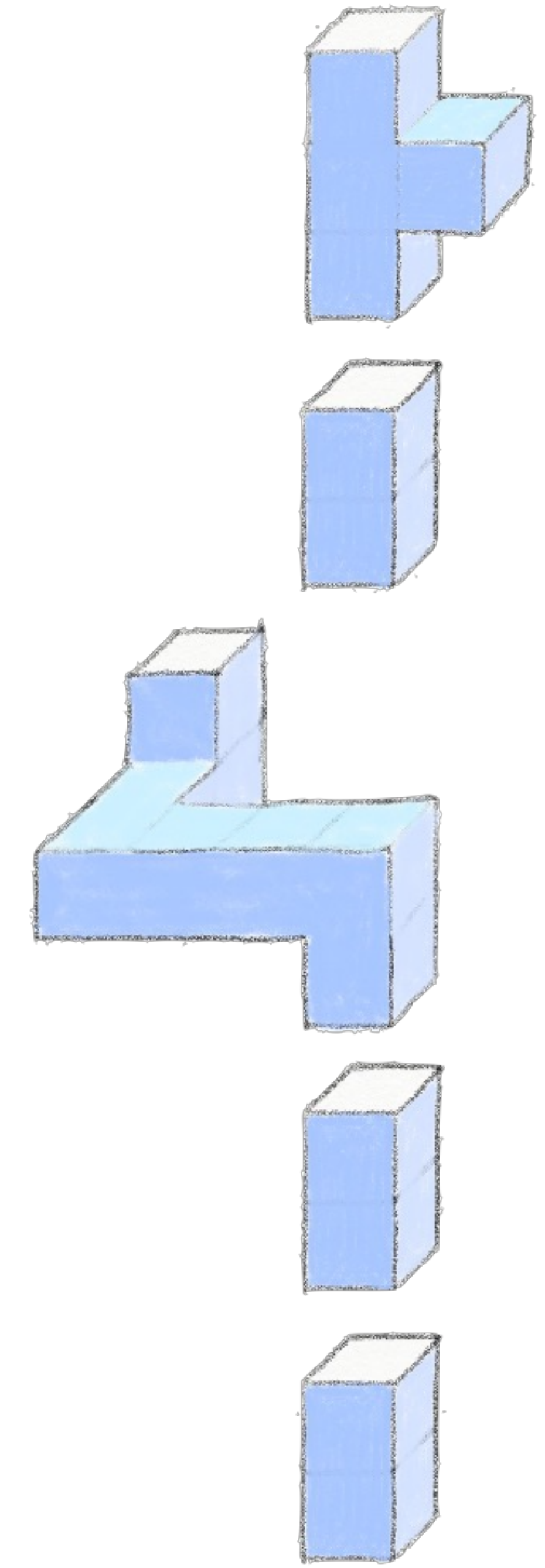}
    };

    \node[circle,draw,fill=gray!10,inner sep=0mm,minimum width=25pt] (org) at (6,1.4) {
    };
    
    \node (pillar) at (6,0)
   {\includegraphics[width=0.65\textwidth]{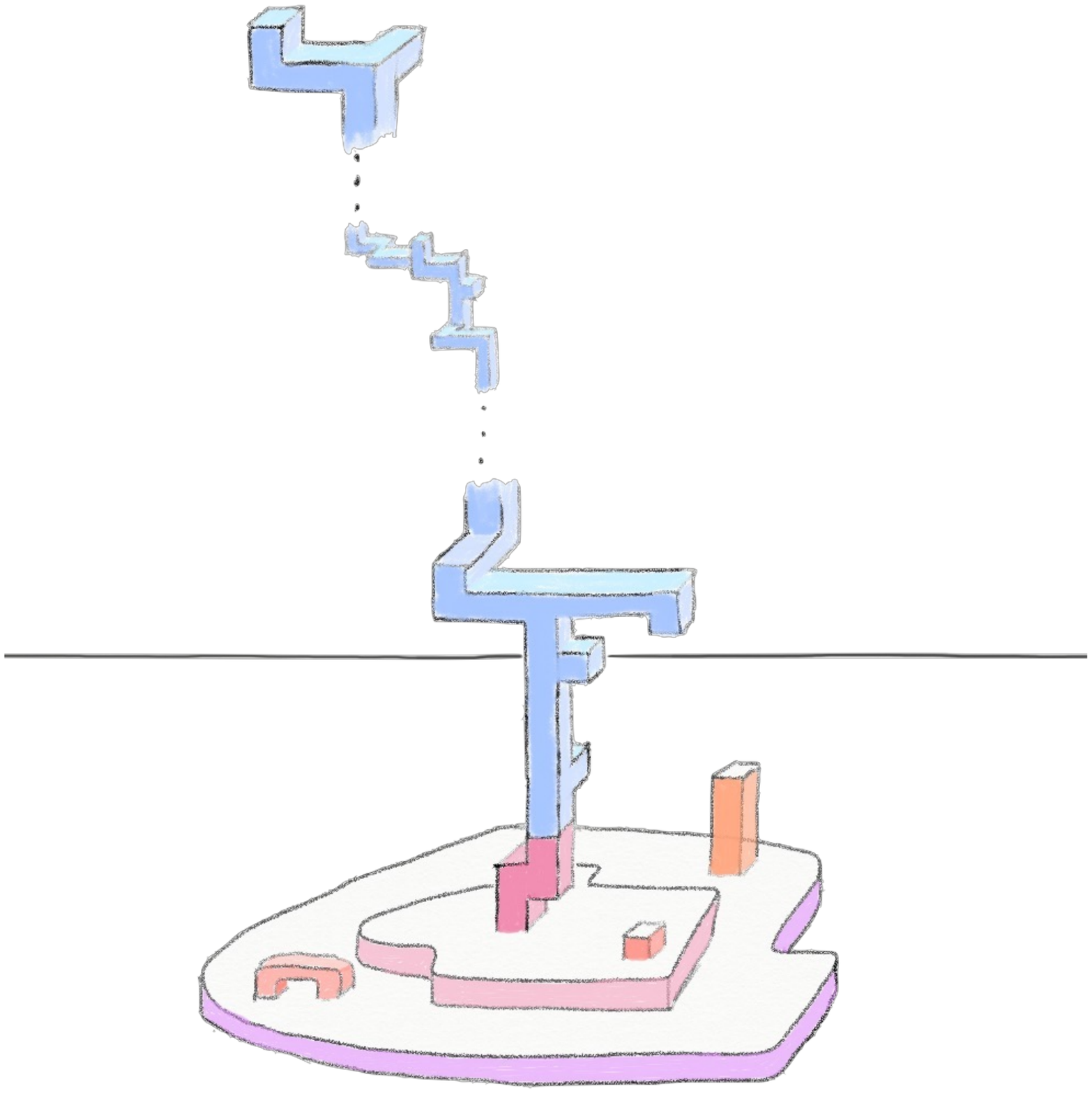}};

	\draw[gray!75] (zoom.70) -- (org.120);
	\draw[gray!75] (zoom.290) -- (org.240);
	
  \end{tikzpicture}
  \vspace{-0.27in}
  \caption{The pillar above a point $x$, denoted $\cP_x$; in blue, the \emph{spine}, partitioned into \emph{increments}.}
  \label{fig:pillar_incr}
  \vspace{-0.05in}
\end{figure}

\subsection{Structure of tall pillars} To formalize the notion of $\cI$ achieving a large deviation above a point $x$, define the \emph{pillar} associated to a point $x\in\llb-n+\frac 12,n-\frac 12 \rrb^2\times\{0\}$ (we defer detailed definitions to~\S\ref{sec:pillar-def}): from a configuration $\sigma\sim\mu_n$, repeatedly delete every finite cluster of ($+$) or ($-$) by flipping its spins (thus eliminating all bubbles), then discard $\mathcal C(\llb-n,n\rrb^2\times\Z_-)$; the pillar of $x$, denoted $\cP_x$, is the resulting (possibly empty) $*$-connected component of ($+$) cells containing $x+(0,0,\frac 12)$, along with all faces of $\cI$ that bound it.

The height of the pillar $\cP_x$, denoted $\hgt(\cP_x)$, is the maximal $y_3$ such that some $(y_1,y_2,y_3)\in\cP_x$. The proof of~\eqref{eq:lln} in Theorem~\ref{mainthm:max-lln} hinges on a large deviation estimate for $\hgt(\cP_x)$ stating (see Proposition~\ref{prop:limiting-ldp-rate-pillar}) that
\[ \lim_{h\to\infty}-\frac1h\log\mu_n(\hgt(\cP_x)\geq h) = \alpha_\beta\,.\]
(Observe that the upper bound on $M_n/\log n$ in~\eqref{eq:lln} readily follows from this by a union bound over $x$.)

A key step in the analysis of the typical structure of $\cP_x$ conditioned on $\{\hgt(\cP_x)\geq h\}$ is to decompose the pillar into increments: define the \emph{cut-points} of $\cP_x$ to be every $y=(y_1,y_2,y_3)\in \cP_x$ such that $y$ is the unique cell in the horizontal slab with height $y_3$ belonging to $\cP_x$. Ordering the cut-points as $v_1,\ldots,v_\sT$ with an increasing third coordinate, their role mimics regeneration points of random walks (though the increment sequence is far from Markovian); thus we refer to the subset of $\cP_x$ delimited by $v_i,v_{i+1}$ (including these two cells) as a \emph{pillar increment} (see Figure~\ref{fig:pillar_incr}). 
Let $\fX$ be the (countable infinite) set of possible increments, and let $A(X)$ be the surface area (number of bounding dual-faces) of an increment $X$. Our next result is a central limit theorem (CLT) for averages of a function along the pillar increment sequence. 

\begin{maintheorem}[CLT for the increments]\label{mainthm:clt}
There exist $\beta_0,\kappa_0>0$ so that the following holds for all $\beta>\beta_0$: for every sequence $T=T_n$ with $1\ll T \ll n$, every non-constant observable on increments $f:\fX\to\R$ such that
\[ f(X) \leq 
e^{\kappa_0 A(X)}\quad\mbox{for every $X\in\fX$}\,,
\]
and every $x=(x_1,x_2,0)$ with $(x_1,x_2)\in\llb-n+\Delta_n+\frac 12,n-\Delta_n-\frac 12\rrb^2$ for some $\Delta_n\gg T$, 
if $(\sX_1,\ldots,\sX_{\sT})$ is the random increment sequence of $\cP_x$, then conditional on the event $\{\sT\geq T\}$, one has that
\[ \frac1{\sqrt{T}} \sum_{t=1}^T \left(f(\sX_t)-\E [f(\sX_t)]\right) \implies \cN(0,\upsigma^2)\quad\mbox{for some $\upsigma(\beta,f)>0$}\,.
\]
\end{maintheorem}
The variance $\upsigma^2$ and asymptotic behavior of  $\frac{1}{\sqrt{T}}\sum_{t=1}^T\E[ f(\sX_t)]$ in Theorem~\ref{mainthm:clt} are expressed in terms of a stationary distribution on increments (see Theorem~\ref{mainthm:structure}\eqref{it:struct-stat}, and Proposition~\ref{prop:clt} for their explicit expressions).
While the above is only conditional on $\{\sT\geq T\}$, we find that $\sT$ and the height of $\cP_x$ are typically comparable (see Lemma~\ref{lem:increment-height-equivalence}):  $\lim_{h\to\infty}\mu_n(\sT\geq (1-\delta_\beta)h\mid \hgt(\cP_x)\geq h)= 1$ (and $\hgt(\cP_x)\geq \sT$ deterministically). In fact, we establish (see Theorem~\ref{mainthm:structure}) that, conditioned on $\{\sT \geq T\}$, the first cut-point typically appears at height $O(\log T)$, and the increment sequence captures all but a negligible portion of the pillar~$\cP_x$.

\begin{figure}
\vspace{-10pt}
    \centering
    \includegraphics[width=0.27\textwidth]{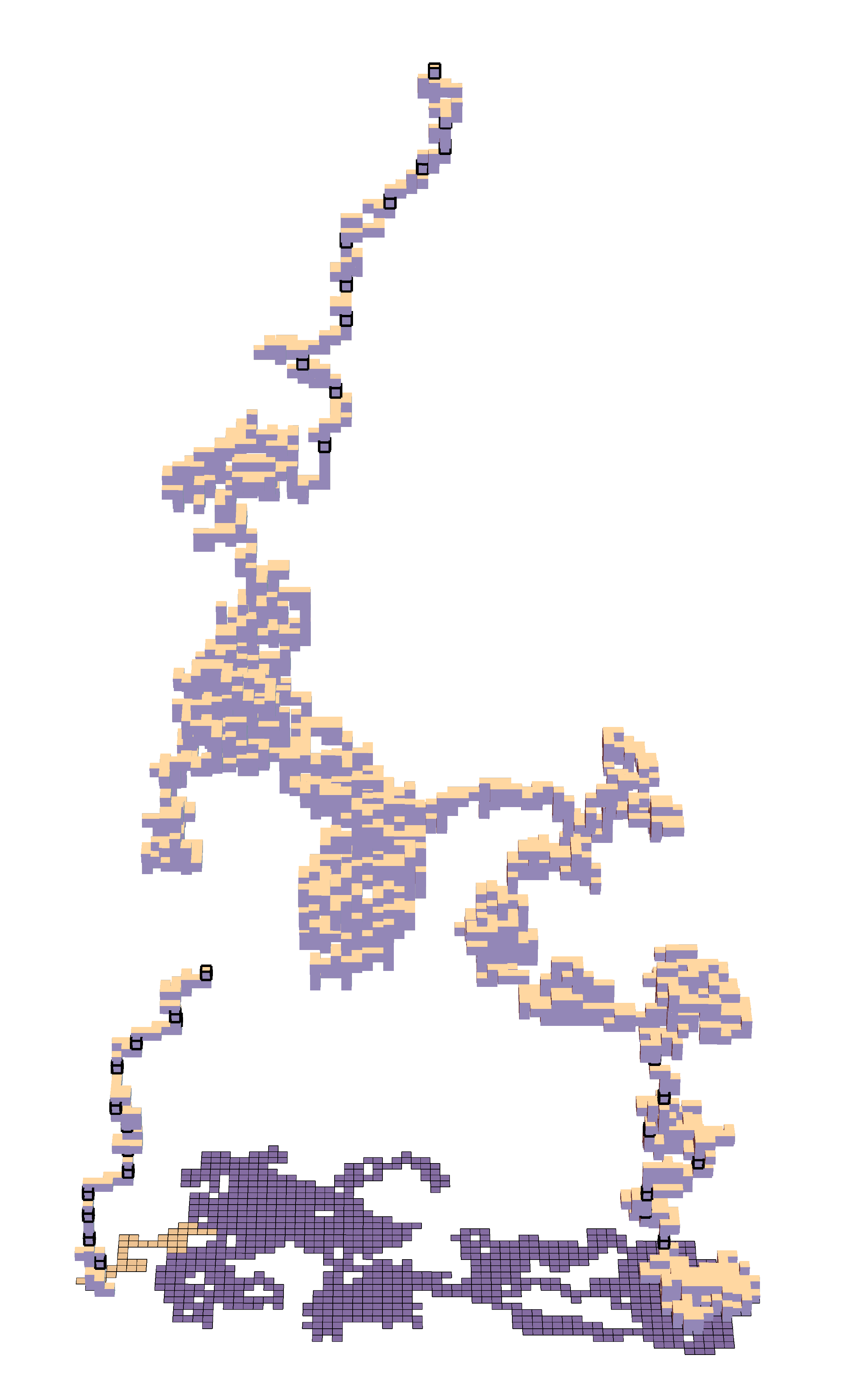}
    \hspace{50pt}
    \raisebox{2pt}{\includegraphics[width=0.24\textwidth]{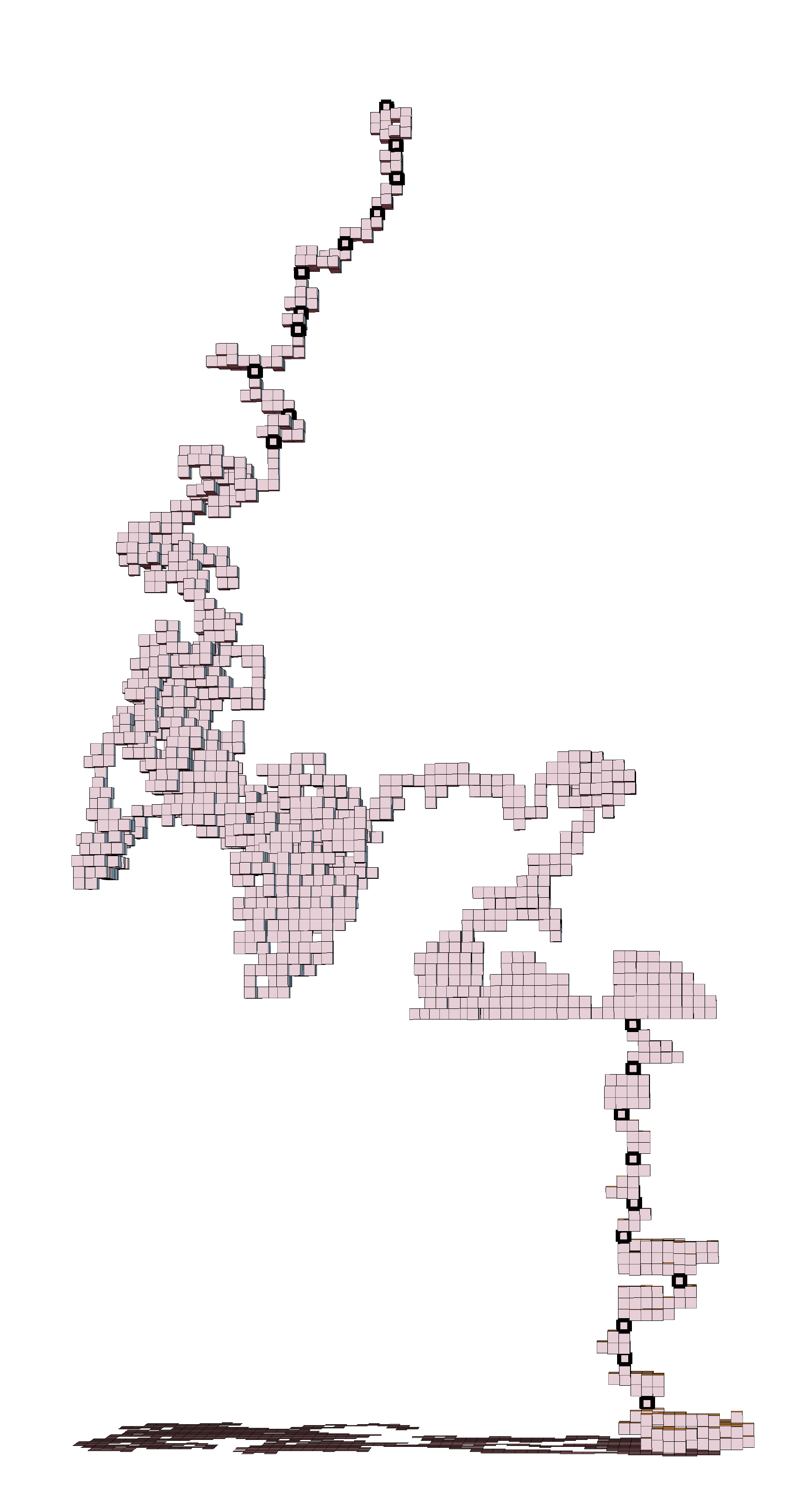}}
    \vspace{-5pt}
    \caption{Two views of a pillar $\cP_x$ with $\sT=20$ increments and its cut-points highlighted. On left: every  pillar $\cP_y$ whose ``shadow'' (its projection on $\R^2\times \{0\}$) intersects that of $\cP_x$ will belong to the \emph{same wall} in Dobrushin's interface decomposition into walls and ceilings.}
    \vspace{-5pt}
    \label{fig:pillar-shadow}
\end{figure}

A special case of the above CLT is that the distribution of the ``tip'' of the pillar conditioned on having at least $T$ increments is asymptotically Gaussian, as are its volume $V(\cP_x)$ and surface area $A(\cP_x)$.

\begin{maincoro}\label{maincor:displacement-surface-area}
There exists $\beta_0$ such that, for every $\beta>\beta_0$ and sequences $T=T_n$ with $1\ll T \ll n$ and  $x=(x_1,x_2,0)$ where $(x_1,x_2)\in\llb -n+\Delta_n+\frac 12,n-\Delta_n-\frac 12\rrb^2$ for some $\Delta_n\gg T$, the pillar at $x$ has that its number of increments $\sT=\sT(\cP_x)$ and height $\hgt(\cP_x)$ satisfy, for some $\lambda(\beta)>1$,
\begin{equation}\label{eq:ratio-T-H} 
\hgt(\cP_x)/\sT \stackrel{\mathrm{p}}\longrightarrow \lambda
\mbox{ conditional on $\{\sT \geq T\} $}\,. 
\end{equation}
Furthermore, conditional on $\{\sT\geq T\}$, the height of $\cP_x$ is asymptotically Gaussian, and moreover:
\begin{enumerate}
\item \emph{distribution of the tip:} the variables $(Y_1,Y_2,\hgt(\cP_x))\in\cP_x$ (arbitrarily chosen if ambiguous) satisfy
\[ \frac{(Y_1,Y_2,\hgt(\cP_x))-(x_1,x_2,\lambda T)}{\sqrt{T}} \implies \cN\bigg(0,
\left(\!\!\begin{smallmatrix} \ \upsigma^2 & 0 & 0\\ 0 &\ \upsigma^2 & 0\\ 0 &0& \ (\upsigma')^2\!\end{smallmatrix}\!\right)
\bigg)\quad\mbox{for some $\lambda(\beta)>1$ and $\upsigma(\beta),\upsigma'(\beta)>0$}\,.\]
\item \emph{volume and surface area:} there exist $\lambda_i(\beta)>1$ and $\upsigma_i(\beta)>0$ ($i=1,2$) such that
\[ \frac{V(\cP_x)-\lambda_1 T}{ \sqrt{T} }\implies\cN(0,\upsigma_1^2)\,,\qquad\mbox{and}\qquad \frac{A({\cP_x})-\lambda_2 T}{\sqrt{T} }\implies\cN(0,\upsigma_2^2)\,.\]
\end{enumerate}
\end{maincoro}

In order to establish the above results, one must control the behavior of the pillar below its first cut-point. But, it is precisely this part of the pillar where the effect of neighboring pillars is the most difficult to control: the abundance of nearby pillars around height $0$ might in principal cause a pillar, conditioned to contain $T$ increments, to have a large (diverging with $T$) segment preceding its first increment. We account for this via a novel decomposition of the pillar into a \emph{base} and a \emph{spine}: the next result shows that the former's total size is typically negligible, while the latter admits a  detailed characterization in terms of its increment sequence.

\begin{maintheorem}[pillar structure]\label{mainthm:structure}
There exists $\beta_0>0$ such that the following holds for all $\beta>\beta_0$: for every sequence $T=T_n$ with $1\ll T \ll n$ and $x=(x_1,x_2,0)$ with $(x_1,x_2)\in\llb-n+\Delta_n+\frac 12,n-\Delta_n-\frac 12\rrb^2$ for some $\Delta_n\gg T$, there exist $c,C>0$ such that, conditional on $\sT\geq T$, the pillar $\cP_x$ has the following structure: 
\begin{enumerate}[(i)]
\item \label{it:struct-base} [Base] There is a cut-point $v_{\Tsp}$ so that the \emph{base} of $\cP_x$, defined as $\sB_x = \{ y\in\cP_x : \hgt(y)\leq \hgt(v_\Tsp)\}$, satisfies $\diam (\sB_x) \leq r$ except with probability $O(\exp(-c \beta r))$ for every $C\log T  \leq r\leq T$.  
\item \label{it:struct-incr} [Spine] The increments $\sX_{\Tsp+1},\ldots\sX_{\sT}$ of the \emph{spine} $\cS_x:=\cP_x\setminus\sB_x$ satisfy, for every $ k,r \leq h$, that the probability that $A({\sX_{\Tsp+k}})\geq r$ is $O(\exp(-c \beta r))$ (letting $A({\sX_t}):=0$ for $t> \sT$).
\item \label{it:struct-mix} [$\alpha$-mixing] For every $k(T)>j(T)$, if $A_1\in \cF_1 := \sigma((\sX_{i})_{i=C\log T}^j)$ and $A_2\in\cF_2:= \sigma((\sX_{i})_{i= k}^{T})$ then the probability of $A_1\cap A_2$ differs from the product of the probabilities of $A_i$ by $O((k-j)^{-10})$.
\item \label{it:struct-stat} [Asymptotic stationarity] There exists a stationary distribution $\nu$ on $\fX^{\Z}$ so that the conditional law of the increments $(\ldots, \sX_{T/2-1}, \sX_{T/2} , \sX_{T/2+1}, \ldots)$ given $\sT\geq T$ converges weakly to $\nu$.
\end{enumerate}
\end{maintheorem}

(These are special cases of stronger statements, which do require additional definitions; for those results implying Items~\eqref{it:struct-base}--\eqref{it:struct-stat}, see Prop.~\ref{prop:base-exp-tail}, Prop.~\ref{prop:exp-tails-increments}, Prop.~\ref{prop:increment-mixing} and Cor.~\ref{cor:limiting-distribution}, respectively.)
As mentioned, each of these require delicately designed maps on interfaces, for which we can control both the change in probability under the map, and its multiplicity; the maps for Items~\eqref{it:struct-base}--\eqref{it:struct-stat} are depicted in Figures~\ref{fig:increment-map}--\ref{fig:stationarity-map} respectively.  

\begin{remark}\label{rem:d>3}
Theorems~\ref{mainthm:max-lln}--\ref{mainthm:structure} generalize naturally to all dimensions $d\geq 3$; the main changes will be that the maximum $M_n$ will have $\frac{M_n}{\log n}\to \frac{(d-1)}{\alpha_\beta}$ in probability, and $\alpha_\beta /\beta \to 2(d-1)$ as $\beta \to\infty$. The results and proofs are otherwise unchanged except that the constants will depend on the dimension $d$, and the lattice notation would be changed, e.g., the interface will be a connected set of $(d-1)$-cells, or plaquettes. For the sake of clarity of exposition and visualization we present all proofs in the most physical $d=3$ setting.
\end{remark}

\begin{remark}
While Theorems~\ref{mainthm:max-lln}--\ref{mainthm:structure} are w.r.t.\ the measure $\mu_n$ (which is the Ising model on the infinite cylinder $\Lambda_n$ with Dobrushin boundary conditions), the fact that the same results hold on the box  $\llb-n,n\rrb^d$ follows from a standard coupling argument. Indeed, by the exponential tails on interface fluctuations and on bubbles,  the interfaces on $\Lambda_n$ and $\llb -n,n\rrb^d$ can be coupled to match  with probability $1-O(e^{-cn})$; likewise their pillars $\cP_x$ conditioned on having at least $T$ increments, agree with probability $1-O(e^{-cn})$ since $T\ll n$.  
\end{remark}

\subsection{Tools and key ideas}\label{subsec:tools}

\subsubsection*{Cluster expansion vs.\ Peierls maps under mixed boundary conditions}

The classical Peierls map---an  
injection from configurations with a specified bubble (a connected set of $(d-1)$-cells homeomorphic to a $(d-1)$-sphere) to ones without it, demonstrating that the energetic cost of such a  bubble outweighs its entropy at large enough $\beta$---is a strikingly effective and robust tool for handling low-temperature behavior under homogeneous boundary conditions. There (within the plus or minus phase) it implies that for any dimension $d\geq 2$, bubbles are microscopic (and their size obeys an exponential tail) at low enough temperature.
However, Peierls maps are insufficient to address the rigidity of the interface in the presence of Dobrushin's boundary conditions: the natural attempt to define a Peierls map on configurations which would ``flatten" the interface is hindered by (a) the interaction of the interface with nearby bubbles, and (b) its self-interactions due to overhangs.

 To overcome this obstacle, Dobrushin used \emph{cluster expansion} (cf.\ also~\cite{Minlos-Sinai}), a robust machinery that, in this case, allows one to disregard the floating bubbles and move to a distribution over interfaces $\cI$ given by 
\begin{equation}\label{eq:cluster-exp-dist} \mu_n(\cI) \propto \exp\bigg[-\beta|\cI| +\sum_{f\in\cI}\g(f,\cI)\bigg]\,,\end{equation}
where $\g$ is a function (over interfaces $\cI$ with a marked face $f$) which is uniformly bounded and local in the sense that $|\g(f,\cI)-\g(f',\cI')|$ decays exponentially in the radius $\br$ about which the balls $B_{\br}(f)$ in $\cI$ and the local neighborhoods of $f$ in $\cI$ and $f'$ in $\cI'$ are isomorphic (see Theorem~\ref{thm:cluster-expansion} in \S\ref{sec:cluster-expansion} for the full statement).
 N.b.\ that by moving to distributions on random interfaces, hiding the interacting bubbles in the Ising model, one loses several useful features of the Ising model: 
 the law of $\cI$ does not have the domain Markov property, and there are long range interactions between faces in $\cI$. 

With this representation, properties of the Ising interface can be deduced from Peierls-like maps. The general strategy for utilizing such maps is as follows. Suppose we wish to show that some set of interfaces $\cA_r$ (e.g., those with height oscillations of at least $r$ above the origin) is exponentially in $r$ rare at $\beta$ large. Then we construct a map $\Psi$ sending $\cA_r$ to a subset $\Psi(\cA_r)$ of interfaces for which we have the following control:
\begin{enumerate}
    \item\label{it:maps-energy} energy gain: for every $\cI\in \cA_r$, the map $\Psi$ induces an energy gain $|\cI| - |\Psi(\cI)|\geq r$. 
    \item\label{it:maps-weight} weight modification: for every $\cI\in \cA_r$, we obtain  $\frac{\mu_n(\cI)}{\mu_n(\Psi(\cI))}\leq e^{ - c \beta (|\cI|- |\Psi(\cI)|)}$ from~\eqref{eq:cluster-exp-dist}.
    \item\label{it:maps-mult} multiplicity: for all $\ell\geq r$, every $\cJ$ in the image of $\Psi$ has at most $C^{\ell}$ pre-images $\cI$ with $|\cI|-|\cJ|=\ell$.
\end{enumerate}
(If we wish to show $\cA_r$ has small probability conditionally on some set $\cB$, we further require  $\Psi(\cA_r)\subset \cB$.) 
%
%
The complication in carrying this out is, of course, the function $\g$, which captures the very same obstacles that hindered the basic Peierls approach---the (hidden in the cluster expansion framework) bubbles in the Ising model and
self-interactions of the interface.
Ideally, one would be able to bound the effect of $\g$ by comparing the faces $f\in\cI$ which were modified under $\Psi$ to faces $f'\in\cJ$ with isomorphic local neighborhoods.

\subsubsection*{Dobrushin's walls and ceilings decomposition and why it fails for LLN} Dobrushin was able to carry out the above approach via a clever combinatorial decomposition of the interface, which reduced the analysis of the maps on the 3D interface to two-dimensional interactions. This decomposition is based on the following partition of $\cI$ tailored to view it as a perturbation of the flat interface $\cL_0:= \cF(\mathbb R^2  \times \{0\})$:  
\begin{itemize}
    \item A \emph{ceiling face} $f\in\cI$ is a horizontal face whose projection on the $xy$-plane is unique among all faces of the interface $\cI$. A \emph{ceiling} of $\cI$ is a maximal connected component of ceiling faces.
    \item 
A \emph{wall face} $f\in\cI$ is a non-ceiling face.
    A \emph{wall} of $\cI$ is a maximal connected component of wall faces.
\end{itemize}
Consequently, one can ``disregard" the ceilings as well as the vertical positions of every wall, and ``standardize" each wall by moving it down to height zero, obtaining a \emph{standard wall representation} of the 3D Ising interface. Importantly, this yields a bijection between collections of standard walls, and interfaces (see Lemma~\ref{lem:interface-reconstruction}), akin to the contour representation of the 2D Ising configuration. 

The natural attempt at a map $\Psi$ is then to have it delete a specific wall $W$ rooted at a face $x\in\cL_0$, from the  standard wall representation of the interface $\cI$, then recover from the resulting standard wall collection, the interface $\Psi(\cI)$. The difficulty is, as usual, due to the function $\g$, and specifically due to \emph{non-deleted} faces whose local neighborhoods would be vertically shifted by $\Psi$. To circumvent this, one may further delete any wall that is ``too close'' to $W$; formally, one defines a \emph{group of walls} according to some criterion of proximity, while relying on the fact that when walls are sufficiently far apart, the exponential decay of $\g$ will negate their interaction. However, deleting too many additional walls can forfeit the second requirement from the map---control over its multiplicity. Dobrushin's criterion was a carefully chosen middle-ground, importantly based solely on two-dimensional distances in the $xy$ directions (see also Definition~\ref{def:group-of-walls}):
\begin{itemize}
     \item Two walls $W$ and $W'$ are said to be ``close" if the interface $\cI$ contains at least
    $\dist(x,x')^2$ faces above $x$ or above $x'$ for some  $x,x'\in\cL_0$ in the projections of $W$ and $W'$ onto $\cL_0$ respectively.
    \item A \emph{group of walls} if a maximal component of pairwise close walls.
\end{itemize}
(Note that ``tall'' walls are easier to group with, and the seemingly arbitrary threshold $\dist(x,x')^2$ plays a special role, via an isoperimetric inequality, in the analysis of faces deleted vs.\ ones that are only shifted.)
The advantage in Dobrushin's combinatorial decomposition is then that under the map $\Psi$, faces only undergo \emph{vertical} shifts, and $xy$-distances between faces are preserved: as such the radius $\br$ coming from $\g$ can be expressed in terms of an $xy$-distance to the nearest deleted wall, so that the above definition of closeness enables the desired control on the contribution from the $\g$ terms in~\eqref{eq:cluster-exp-dist} in terms of $\beta (|\cI|-|\cJ|)$. 

This argument showed that the group of walls adjacent to a fixed face $x
\in\cL_0$ in $\cI$ has an exponential tail, implying the rigidity of $\cI$ and that its maximum height is $O(\log n)$ with probability tending to $1$. However, it is far too crude to handle subtle quantities of interest such as the asymptotics of the maximum (LLN) and the structure of the interface in a local neighborhood surrounding it (e.g., results \`a la Corollary~\ref{maincor:displacement-surface-area}):
\begin{enumerate}[1.]
    \item The classification of faces into walls and ceilings does not relate well to the local spin configuration---as it depends on the behavior of the interface far above/below a face. But, the LLN (Theorem~\ref{mainthm:max-lln}) does embed local spin-spin correlation: the leading order term of the maximum of $\cI$ is given in terms of a connective constant of spin agreement in infinite volume, which operations of walls are too coarse to reflect.
    \item Recall that treating connected sets of wall faces as a single wall means any two connected wall-sets with intersecting shadows on the $xy$-plane are one and the same. While crucial to Dobrushin's reduction of the problem to 2D, this comes in the way of analyzing the connected component of plus spins emanating from a fixed face $x\in\cL_0$; the taller this component is, the more pronounced this issue is  (see Fig.~\ref{fig:pillar-shadow}, left).
    \item Further bundling of walls into groups of walls attaches an extra layer of walls to a connected component of plus spins; moreover, the criterion for this bundling says that if the wall $W_x$ of some face $x\in\cL_0$ has $h$ faces above $x$, then it will collect every distinct wall $W_y$ for $y$ within a circle of area $h$ centered about~$x$ (and so on, in a cascading manner). This would make it impossible to use this framework for more delicate questions such as tightness for the centered maximum (Problem~\ref{prob:tightness}).
    \item Analyzing the effect of operations on walls (beyond simply \emph{deleting the entire} group of walls of~$x\in\cL_0$) is problematic: the collection of walls does not enjoy monotonicity / FKG inequalities, nor a domain Markov property (these properties are critical in the proof of sub-multiplicativity, as explained below). 
\end{enumerate}

\subsubsection*{Maps on the increment sequence and base}

Unlike Dobrushin's proof of the rigidity of $\cI$ which used maps to compare $\cI$ to flatten interfaces,  
in this work we construct Peierls-type arguments with
reference interfaces that, rather than flat, have a three-dimensional large deviation above a point $x\in \cL_0$:
\begin{enumerate}[1.]
    \item At a high level, we would like our maps to ``straighten'' the pillar in the input interface $\cI$, namely we would like to replace an increment in the pillar by a straight column of singleton boxes. 
    The potential interactions of the pillar with its base, whose size and shape are much more difficult to control, necessitates that every map should first ``flatten'' the base as well.
    Consequently, we wish to use a map with a reference interface consisting of a flat plane appended to a modification of the random pillar $\cP_x$ (altered at its base and the designated increment we wish to control). This is achieved in two steps:
    \begin{enumerate}[(i)]
        \item A map $\Psi_i$ to straighten the increment $\mathscr X_{i}$ (see~\S\ref{subsec:def-incr-map} for its definition, and \S\ref{subsec:strategy-increment-map} for its proof strategy).
        \item A map $\Phi_{\mathscr B}$ to flatten the base (see~\S\ref{subsec:def-base-map} for its definition, and \S\ref{subsec:strategy-base-map} for its proof strategy).
    \end{enumerate}
    
    \item Whereas Dobrushin proofs only had vertical shifts, and thus interaction distances were controlled by 2D distances, in the above maps we must account for both horizontal and vertical shifts and their interplay. The subtle choice of $v_\Tsp$, the ``source point'' for the spine as given in Theorem~\ref{mainthm:structure}, serves as a key ingredient: in a sense it protects the pillar from interaction with neighboring ones (whose analysis is essential in the LLN for the maximum---see below) and isolates the effects of horizontal and vertical shifts: below $v_\Tsp$ faces will only be shifted vertically by our maps, and above it they will only undergo horizontal shifts. 
\end{enumerate}



\begin{figure}
    \centering
    \vspace{-0.4in}
    \includegraphics[width=0.32\textwidth]{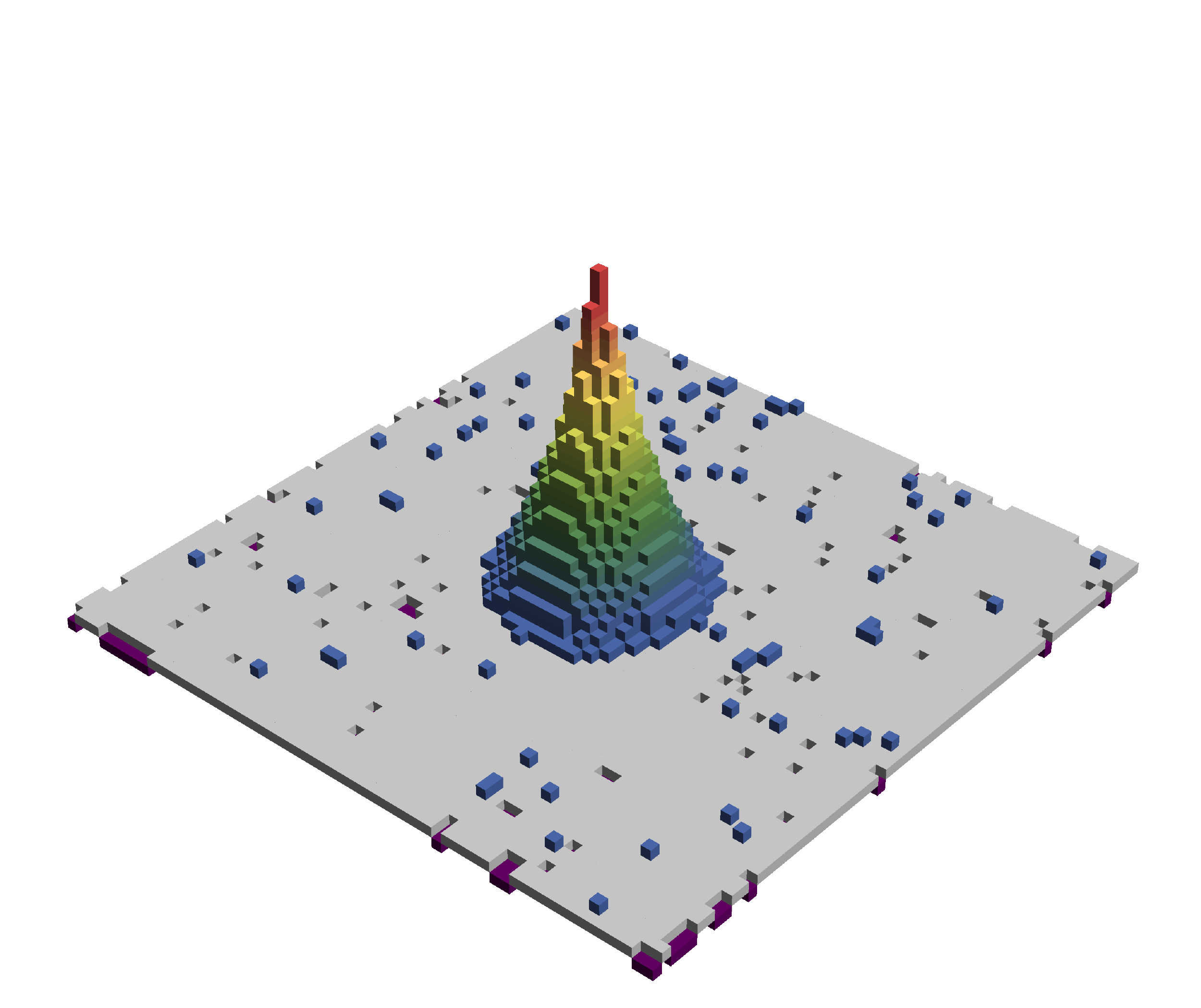}
    \includegraphics[width=0.32\textwidth]{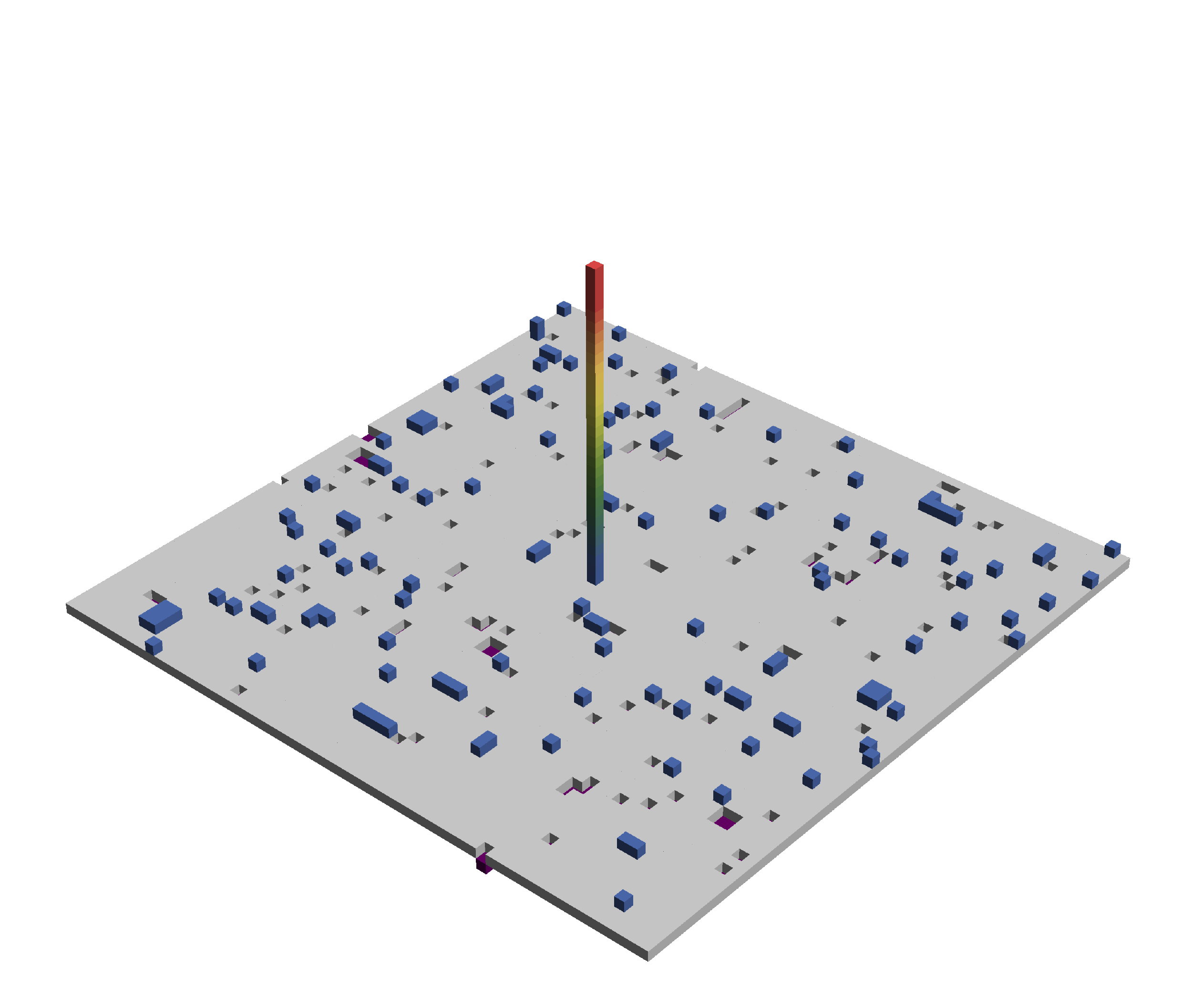}
    \includegraphics[width=0.32\textwidth]{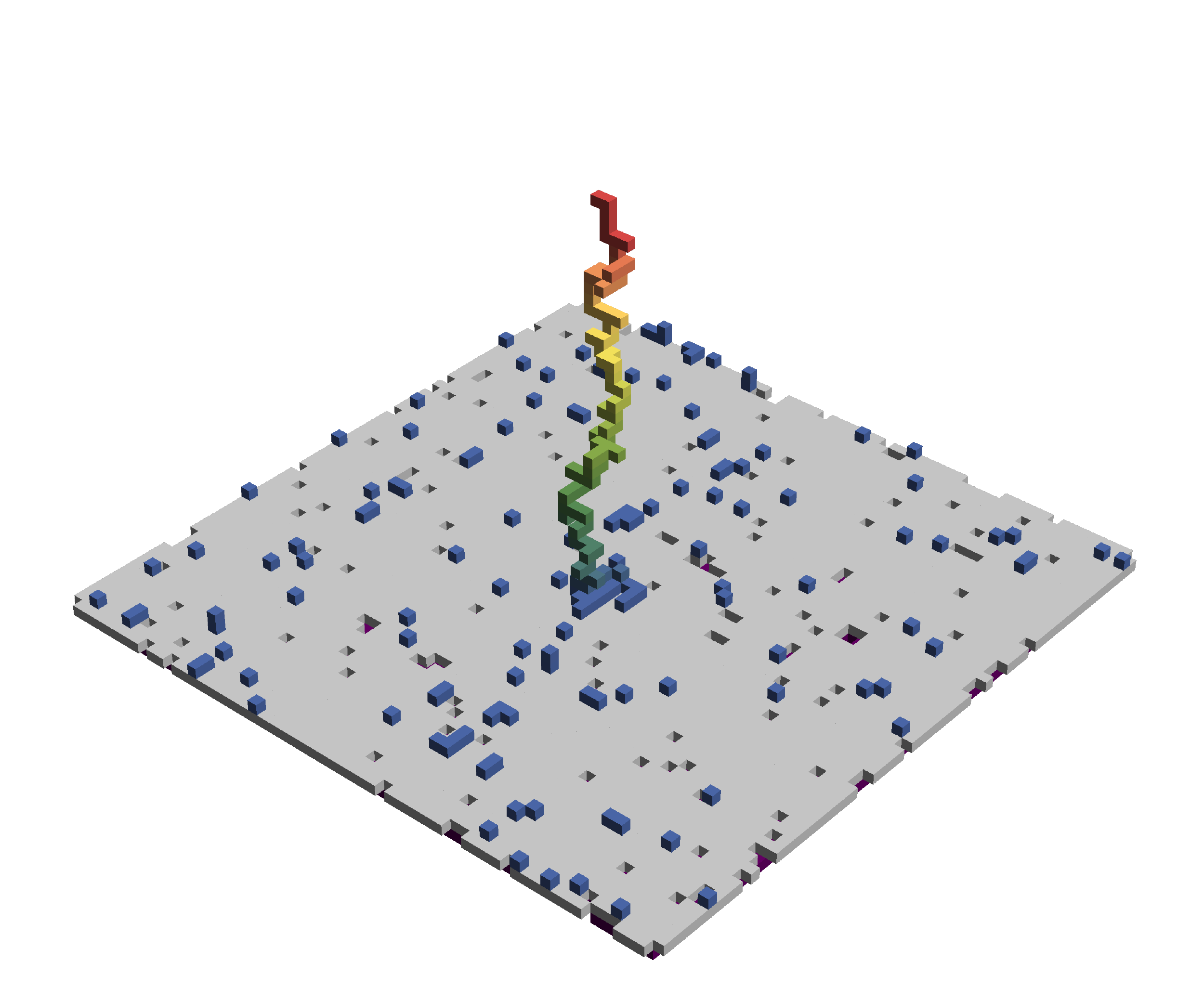}
    \vspace{-0.1in}
    \caption{Typical pillars of the Discrete Gaussian model (left), the SOS model (middle), and the 3D Ising model (right) conditioned on the large deviation event $\hgt(\cP_x)\geq h$.}
    \label{fig:pillar-lds}
    \vspace{-0.1cm}
\end{figure}
\begin{figure}
  \begin{tikzpicture} 
  
  \pgfmathsetmacro{\circmult}{0.35}
  \pgfmathsetmacro{\circscale}{0.6}
  
  \begin{scope}
  \foreach \a / \b in {1/6, 2/5, 2/6, 3/4, 3/5, 3/6, 4/6, 4/7, 4/8, 5/8, 6/4, 6/5, 
6/6, 6/8, 7/4, 7/6, 7/8, 8/4, 8/5, 8/6, 8/7, 8/8, 9/4, 9/6, 
9/7, 9/8, 9/9, 10/4, 10/9, 11/4, 11/5, 11/6, 11/7, 11/8, 11/9,
12/4, 12/5, 13/5, 14/5, 15/5, 16/5, 17/5, 17/6, 17/7, 18/7, 18/8, 
19/8, 20/6, 20/7, 20/8}
	\node [circle,fill=red!50,scale=\circscale,outer sep=0.1cm] (V\a\b) at (\a*\circmult,\b*\circmult) {};
  \foreach \a / \b in {1/4, 1/5,2/3, 2/4,3/3,4/3,4/4,4/5,5/3, 5/4,5/5,5/6,5/7,6/3,6/7,7/3,7/7,8/3,9/3,9/5,10/3,10/5,10/6,10/7,10/8,11/3,12/3,13/3, 13/4,14/4,15/4,16/4,17/4,18/4, 18/5,18/6,19/5, 19/6,19/7,20/5}
	\node [circle,fill=blue!50,scale=\circscale,outer sep=0.1cm] (V\a\b) at (\a*\circmult,\b*\circmult) {};
    \draw[dashed, gray, thin] (V15.north)--(V205.north);    
	\newcommand{\interA}{($(V15.north) + (-0.5*\circmult,0)$) -- (V15.north)  to [bend left=45] (V15.east) to [bend right=45] (V24.north) to [bend left=45] (V24.east) to [bend right=45] (V33.north) to [bend right=45] (V44.west) to (V45.west) to [bend left=45] (V45.north) to [bend right=45] (V56.west) to (V57.west) to [bend left=45] (V57.north) to (V77.north) to [bend left=45] (V77.east) to [bend left=45] (V77.south) to (V67.south) to [bend right=45] (V56.east) -- (V54.east) to [bend right=45] (V63.north) -- (V123.north) to [bend right=45] (V134.west) to [bend left=45] (V134.north) -- (V174.north) to [bend right=45] (V185.west) -- (V186.west) to [bend left=45] (V186.north) to [bend right=45] (V197.west) to [bend left=45] (V197.north) to [bend left=45] (V197.east) -- (V196.east) to [bend right=45] (V205.north) to ++ (0.5*\circmult,0)
	};
	\newcommand{\closeA}{($(V205.north) + (0.5*\circmult,0)$) -- ++ (0,-3*\circmult) -- ++(-20*\circmult,0) -- ++(0,3*\circmult) }
	\draw[ultra thick] \interA ;
	\path[fill=blue!50, fill opacity=0.2] \interA -- \closeA;
	\draw[ultra thick, cyan, fill=blue!30, fill opacity=0.2] (V95.west) to [bend right=45] (V95.south) to (V105.south) to [bend right=45] (V105.east)
	 -- (V108.east) to [bend right=45] (V108.north)	
	  to [bend right=45] (V108.west) to (V106.west) to [bend left=45] (V95.north) to [bend right=45] (V95.west)	 
	 ;
	 \end{scope}
  \begin{scope}[shift={(8,0)}]
  
  \foreach \a / \b in {1/6, 2/6, 2/5, 3/4, 3/5, 3/6, 4/6, 4/7, 4/8, 5/8, 6/4, 6/5, 
6/6, 6/8, 7/4, 7/5, 7/6, 7/7, 7/8, 8/7, 8/8, 8/9, 8/10, 9/5, 9/10, 10/5,10/6,10/7,10/8, 10/10,  11/9,11/10,
12/4, 12/5,12/6,12/7,12/8, 12/9, 13/5, 14/5, 15/5, 16/5, 17/5, 17/6, 17/7, 18/7, 18/8,
19/8, 20/6, 20/7, 20/8}
	\node [circle,fill=red!50,scale=\circscale,outer sep=0.1cm] (V\a\b) at (\a*\circmult,\b*\circmult) {};
  \foreach \a / \b in {1/4,1/5,2/3,2/4,3/3,4/3,4/4,4/5,5/3, 5/4,5/5,5/6,5/7,6/3,6/7,7/3,8/3,8/4,8/5,8/6,9/4, 9/6, 
9/7, 9/8, 9/9,10/4, 10/9,11/3, 11/4,11/5, 11/6, 11/7, 11/8,12/3,13/3, 13/4,14/4,15/4,16/4,17/4,18/4, 18/5,18/6, 19/5, 19/6,19/7,20/5}
	\node [circle,fill=blue!50,scale=\circscale,outer sep=0.1cm] (V\a\b) at (\a*\circmult,\b*\circmult) {};
\draw[dashed, gray, thin] (V15.north)--(V205.north);
	\newcommand{\interA}{($(V15.north) + (-0.5*\circmult,0)$) -- (V15.north)  to [bend left=45] (V15.east) to [bend right=45] (V24.north) to [bend left=45] (V24.east) to [bend right=45] (V33.north) to [bend right=45] (V44.west) to (V45.west) to [bend left=45] (V45.north) to [bend right=45] (V56.west) to (V57.west) to [bend left=45] (V57.north) to (V67.north) to [bend left=45] (V67.east) to [bend left=45] (V67.south) to [bend right=45] (V56.east) -- (V54.east) to [bend right=45] (V63.north) 
-- (V73.north) to [bend right=45] (V84.west) -- (V86.west) to [bend left=45] (V86.north)  to [bend right=45] (V97.west) -- (V99.west) to [bend left=45] (V99.north) -- (V109.north) 
to [bend left=45] (V109.east) to [bend right=45] (V118.north) to [bend left=45] (V118.east) -- (V114.east) to [bend right=45] (V123.north) to [bend right=45] (V134.west) to [bend left=45] (V134.north) -- (V174.north) to [bend right=45] (V185.west) -- (V186.west) to [bend left=45] (V186.north) to [bend right=45] (V197.west) to [bend left=45] (V197.north) to [bend left=45] (V197.east) -- (V196.east) to [bend right=45] (V205.north) to ++ (0.5*\circmult,0)
	};
	\newcommand{\closeA}{($(V205.north) + (0.5*\circmult,0)$) -- ++ (0,-3*\circmult) -- ++(-20*\circmult,0) -- ++(0,3*\circmult) }
	\draw[ultra thick] \interA ;
	\path[fill=blue!50, fill opacity=0.2] \interA -- \closeA;
	\draw[ultra thick, purple, fill=purple!30, fill opacity=0.2] (V95.west) to [bend right=45] (V95.south) to (V105.south) to [bend right=45] (V105.east)
	 -- (V108.east) to [bend right=45] (V108.north)	
	  to [bend right=45] (V108.west) to (V106.west) to [bend left=45] (V95.north) to [bend right=45] (V95.west)	 
	 ;\end{scope}
\end{tikzpicture}
  \caption{The pillar $\cP_x$ vs.\ the $(+)$-component $\sP$ above $x$: on left, $\cP_x=\emptyset$ whereas $\sP\neq\emptyset$ (the interface tunnels underneath $\mathscr P$); on right, $\cP_x\neq\emptyset$ whereas  $\sP=\emptyset$ (a minus bubble).}
  \label{fig:interface-nonmon}
	\end{figure}
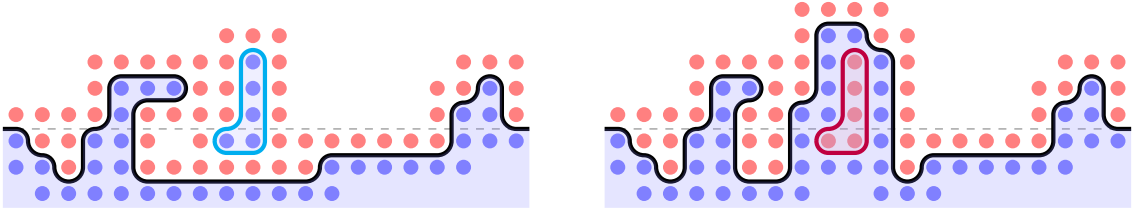

\subsubsection*{Establishing the limiting LD rate function}
As the leading order constant of the maximum of the interface is given by a solution to the LD problem of plus connectivity in infinite volume (much like the maximum of the surface in approximating models for the 3D Ising model such as the ($2+1$)D SOS and DG models were governed by LD problems; see Figure~\ref{fig:pillar-lds}), a prerequisite to the proof of Theorem~\ref{mainthm:max-lln} is to establish existence of the limit given in~\eqref{eq:alpha-beta-def}.
A standard approach to accomplish this would be to establish sub-multiplicativity or super-multiplicativity for $a_h := \mu_{\Z^3}^\mp(A_h)$, where $A_h$ is the event in the right-hand of~\eqref{eq:alpha-beta-def}:
\begin{itemize}
    \item One may expect $(a_h)$ to be super-multiplicative, just like other increasing connection events in the Ising model and other monotone spin systems. However, if we reveal the $+$ connection up to height $h_1$ due to $A_{h_1}$ in hope that only positive information is given on $A_{h_1+h_2}$ (whereby FKG would provide the sought estimate), we find that at height $h_1$ the measure is more negative than at height $0$---the non-translation invariance of the boundary conditions makes a connection from $h_1$ to $h_2$ exponentially less likely than one from height $0$ to $h_1$. 
    \item Instead, we prove approximate sub-multiplicativity via a crucial application of Theorem~\ref{mainthm:structure}\eqref{it:struct-base}. 
    The notion of a pillar is well-suited to describe the ($+$)-component of $x$ above height $0$---which we may reveal up to height $h_1$. The ($-$) spins on its boundary yield negative information, which we may discard via monotonicity and domain Markov; however, this reveals additional ($+$)-spins at height $0$, which encompass positive information. Yet these are part of the \emph{base} $\sB_x$, which Theorem~\ref{mainthm:structure} shows has size at most $C \log^2 h_1$ with probability $1-o(1)$. Tilting the measure by these ($+$)-spins thus costs a factor of $e^{O(\log^2 h_1)}=e^{o(h_1)}$, which does not affect the sought sub-multiplicativity bound.
\end{itemize}
A subtle point worthwhile stressing is that, despite the close connection between the pillar $\cP_x$ and the ($+$)-component above $x$ in $\mathbb R^2 \times [0,\infty)$, neither one necessarily contains the other (see Figure~\ref{fig:interface-nonmon}).


\subsubsection*{Maps on pairs of interfaces for mixing and stationarity}
In order to prove the more refined $\alpha$-mixing and stationarity properties of the increment sequence, we introduce \emph{$2$-to-$2$ maps} that act not on a single interface, but on a pair of interfaces. Importantly, with mixing and stationarity, our aim is not to show some set of interfaces is unlikely, but rather that some set of interfaces have roughly equal probability to some other set of interfaces: e.g., the pair $(\mathscr X_j, \mathscr X_k) = (X_j, X_k)$ is roughly equally likely as $(\mathscr X_j, \mathscr X_k) = (X_j, X_k')$ in the case of mixing, and $\mathscr X_j = X$ is roughly equally likely as $\mathscr X_{k} = X$ in the case of stationarity. There is no relative energy gain here, so we need the cost in the exponent coming from the function $\g$ in~\eqref{eq:cluster-exp-dist} to be~$o(1)$. To resolve this, we instead pair up interfaces, and apply the map to pairs of interfaces, performing a swapping operation to be able to identify each face in the original pair of interfaces, with some face in the image pair of interfaces. 
We explain the subtleties in carrying this through in more detail in~\S\ref{subsec:strategy-mix} and~\S\ref{subsec:strategy-stat}. 

\subsubsection*{Stein's method argument for the CLT}
The proof of the CLT in Proposition~\ref{prop:clt} (which implies Corollary~\ref{maincor:displacement-surface-area}) uses a Stein's method type argument which was used by Bolthausen~\cite{Bolthausen82} to handle stationary, mixing sequences of random variables (appealing to the new results on $\alpha$-mixing and stationarity obtained via the $2$-to-$2$ maps). We explain the complications in our setting compared to that of~\cite{Bolthausen82} in~\S\ref{subsec:strategy-clt}.

\subsubsection*{Comparison to Ornstein--Zernike theory.}
We pause to compare our proof approach above to the well-known Ornstein--Zernike (OZ) theory of which the results of Theorem~\ref{mainthm:structure} may be reminiscent. Since the pioneering works~\cite{CI02,CIV03}, there has been a remarkable line of work analyzing the structure of ``long connections" in the high-temperature Ising model (all $\beta<\beta_c$) in all dimensions $d\geq 2$ using what is known as modernized OZ theory; the analysis was extended to the FK and Potts models (see, e.g.,~\cite{CIV08,Io98}).

Namely, these works have analyzed, in the setting of the Ising model, the shape of a plus cluster connecting the origin to a site $\vec x$ at distance $\|x\|$. Via a decomposition into cut-points or cone-points, and increments between these, these works have identified a renewal structure in the long finite clusters of the high-temperature Ising model, with diffusive random-walk behavior at cut-points, and microscopic excursions in between. 

In $d=2$, by the duality between $\beta<\beta_c$ and $\beta>\beta_c$, OZ theory directly translates to the low-temperature interface under Dobrushin boundary conditions. As such, for all $\beta>\beta_c$, the 2D Ising interfaces have been decomposed into cut-points with a renewal structure, and small increments in between with rapid decay of correlations; this was instrumental in pushing convergence of the interface to a Brownian bridge all the way to $\beta_c$~\cite{GrIo05}. In $d\geq 3$, there is no correspondence between high-temperature connections and low-temperature interfaces; rather, the more naturally analogous low-temperature event is a truncated connection event of the origin being connected by pluses to some $x$ under the infinite-volume minus measure---in percolation language, a connection from $0$ to $x$ not connected to the \emph{unique} infinite component. 

By contrast, in our setting, the pillars of the plus phase are part of the infinite plus component, and are thinned by the \emph{distinct} infinite minus component whose coexistence is forced by the boundary conditions. By Theorem~\ref{mainthm:structure}, these pillars appear to have similar behavior beyond their first cut-point to long finite plus clusters in the minus phase. But, below that first cut-point there is a strong influence from the connection to the infinite plus component. The cut-point, increment decomposition is not helpful for dealing with these interactions with other branches of the infinite component (at the base); thus, controlling the base of the pillar is the most delicate part of our analysis. 
%

It is therefore important to stress that, while appearing similar to our cut-point decomposition of pillars, \emph{one cannot hope to characterize the pillars of the low temperature 3D Ising interface via
the OZ theory}. Indeed, the OZ behavior is valid for all $\beta>\beta_c$ in any dimension, whereas rigidity, let alone the results we prove, is conjectured to be false near $\beta_c$ in dimension $d=3$, as well as under any tilt in dimension $d=3$.

\subsection{Related work and open problems}\label{sec:related-work}

In this section, we give a (by no means complete) overview of literature related to the analysis of random interfaces/surfaces describing separation of phases, and highlight some unresolved problems. 
As discussed, the pioneering work of Dobrushin rigorously established results on such  interfaces of the Ising model via cluster expansion, including in particular rigidity at low temperatures in three (and higher) dimensions, and thus the existence of (infinite-volume) Gibbs measures describing the coexistence of phases. The approach of~\cite{Dobrushin72a}, outlined in \S\ref{sec:dobrushin-definitions}--\ref{sec:dobrushin-proof}, has been used to show rigidity for various other statistical physics models in $d\geq 3$, e.g., for the Widom--Rowlinson model~\cite{BLOP79a,BLP79b}, the Falicov--Kimball models~\cite{DMN00} and percolation and random-cluster/Potts models~\cite{GielisGrimmett,CeKo03}. We also mention that Van-Beijeren gave an elegant and simplified proof of the rigidity of the Ising interface using correlation inequalities in~\cite{vanBeijeren75}. 

Subsequently, cluster expansion was instrumental in analyzing the analogous interface in two dimensions. This line of work culminated in the seminal monograph~\cite{DKS}, showing that the shape of a macroscopic minus droplet in the plus phase takes after the \emph{Wulff shape}, the convex body minimizing the surface energy to volume ratio (where the former is in terms of some explicit, analytic, surface tension $\tau_{\beta}>0$). Microscopic properties of an interface of angle $\theta$ in an $n\times n$ box are by now also very well-understood, with fluctuations on $O(\sqrt n)$ scales, and a scaling limit to a Brownian bridge~\cite{DH97,DKS,GrIo05,Hryniv98}; these hold up to the critical $\beta_c$~\cite{Ioffe94,Ioffe95}. 

 In dimensions three and higher, the microscopic features of the interface are only well-understood for approximations to the random surface separating the plus and minus phases, given by integer valued height functions $\phi:\llb -n,n\rrb^2 \to \Z$ on an $n\times n$ box. Perhaps the most well-studied of these approximations is the Solid-On-Solid (SOS) model, going back to the 1950's (see~\cite{Temperley52} and~\cite{Abraham}); the $(2+1)$-dimensional SOS model (approximating 3D Ising) is a special case of $|\nabla\phi|^p$ models: a class of gradient models with Hamiltonians  $\cH(\phi) = \sum_{x} \sum_i |\nabla_i \phi(x)|^p$ ($p=1$ is the SOS model, and $p=2$ is the discrete Gaussian model (DG)). 
 In particular, the SOS Hamiltonian matches that of Ising with Dobrushin boundary conditions restricted to  configurations where the intersection of the plus spins with each column $\{(x_1, x_2, h):h\in \Z\}$ is connected (i.e., SOS configurations have no overhangs or bubbles, which are microscopic in Ising in the $\beta\to\infty$ limit). 
 
 In the setting of the SOS model at low temperatures, the maximum of the surface is typically of order $\log n$ (see~\cite{BEF86}). In~\cite{CLMST1,CLMST2}, its maximum was found to be tight around $\frac{1}{2\beta} \log n$, by showing that the probability of a ``pillar above a face $x$ reaching height $h$" is $\exp[-4\beta h+O(1)]$; on this large deviation event the interface looks like a vertical column of height $h+O(1)$ with an $O(1)$ ``base" (c.f., Corollary~\ref{maincor:displacement-surface-area}, where for instance, the tip is delocalized, and see the depiction in Figure~\ref{fig:pillar-lds}). Related properties in the presence of a floor inducing entropic repulsion were studied in~\cite{CLMST2}, and extended to the discrete Gaussian and other $|\nabla\phi|^p$-models in~\cite{LMS16}.

\begin{problem}
For $\alpha_\beta$ defined in~\eqref{eq:alpha-beta-def}, what are the asymptotics of $\alpha_\beta - 4\beta$ (next order asymptotics of $\alpha_\beta$) as $\beta\to\infty$? in particular, is it the case that $\alpha_\beta <4\beta$, so that 3D Ising  is ``rougher" than  ($2+1$)D SOS?
\end{problem}

While cluster expansion only converges at sufficiently large $\beta$, it is natural to ask if the rigidity of the interface, and our new results, hold for all $\beta> \beta_c$. This is not believed to be the case, as the Ising model is widely believed to undergo a \emph{roughening transition} for $d=3$ (and no other dimension): much like the SOS and DG approximations, which exhibit phase transitions in $\beta$---whereby they roughen and resemble the discrete Gaussian free field~\cite{BrWa82,FrSp81} for small $\beta$---it is conjectured that for the 3D Ising model there exists a point $\beta_{\textsc r}>\beta_c$ such that, for $\beta \in (\beta_c,\beta_{\textsc r})$, the model has long-range order, yet the typical fluctuations of its horizontal interface diverge with $n$; proving this transition is a longstanding open problem (see, e.g,~\cite{Abraham,BFL82}). 

Much progress has been made in recent years on understanding the distribution of the maximum of the 2D discrete Gaussian free field and its local geometry. It is known for instance (\cite{BDG01,BrZe12,BDZ16}; see also, e.g.,~\cite{Zeitouni16}) that this maximum is tight around an expected maximum that is asymptotically $2\sqrt{2/\pi}( \log n - \frac38\log\log n)$, and that the centered maximum has the law of a randomly shifted Gumbel random variable. 
\begin{problem}\label{prob:tightness}
What are the asymptotics of $\E[M_n] - \frac{2}{\alpha_\beta}\log n$ (next order asymptotics of $\E [M_n]$) as $n\to\infty$?
Are the fluctuations of the centered maximum $O(1)$, i.e., is the sequence $\{\mu_n (M_n - \E [M_n]\in \cdot)\}$ tight?  
\end{problem}

We end this section with other well-studied perspectives on the 3D Ising model at low temperatures. While the interface-based approach of Dobrushin~\cite{Dobrushin72a} 
proved to be extremely fruitful in 2D (where the results hold for interfaces in any angle), in 3D the combinatorics of that argument break down as soon as the ground state is not flat. It remains a well-known open problem to show that there do not exist non-translation invariant Gibbs measures corresponding to interfaces other than those parallel to the coordinate axes. The progress to date on roughness and fluctuations of ``tilted interfaces" has been limited either to 1-step perturbations of a flat interface~\cite{Miracle-Sole1995}, or to results at zero temperature using rich connections to exactly solvable models~\cite{CerfKenyon}. 

In lieu of these approaches, a coarse-graining technique of Pisztora~\cite{Pisztora1996} enabled the establishment of surface tension and a Wulff shape scaling limit for the 3D Ising model at low-temperature: Cerf and Pisztora~\cite{CerfPisztora} considered an Ising model on an $n\times n \times n$ box with all-plus boundary conditions, and showed that conditional on having $ (1+\epsilon) \mu^+(\sigma_0 = -1)n^3$ minus spins (atypically many), the largest minus cluster macroscopically takes on the corresponding Wulff shape.  Results of this sort are focused on the macroscopic behavior of the model (as opposed to the interface-based approach) and do not describe the fluctuations around the limiting shape. In particular, the convergence to the Wulff shape holds all the way up to $\beta_c$ (when combined with~\cite{Bodineau1999,Bodineau2005}), even though near $\beta_c$ (above the roughening transition) it is expected that the interface is not only delocalized, but that the minus cluster actually percolates all the way to the boundary of the box~\cite{BFL82}. 

\subsection{Outline of Paper} In \S\ref{sec:preliminaries}, we first overview the notation of the paper and introduce Dobrushin's decomposition of the interface $\cI$ into \emph{walls and ceilings}; then, in \S\ref{sec:dobrushin-proof}, we recap the proof of rigidity from~\cite{Dobrushin72a} and the bounds this implies on $\mu_n (\hgt(\cP_x)\geq h)$. In \S\ref{sec:increment-prelim}, we define increments of $\cP_x$, and use them to split $\cP_x$ into its base and spine; in \S\ref{sec:increment-exp-tail}, we show that spine increments have an exponential tail on their size. In \S\ref{sec:base}, we prove that the base of a pillar consisting of $T$ increments has an exponential tail on its diameter beyond $C\log T$. Then in \S\ref{sec:ldp-lln}, we use the structural results of \S\ref{sec:increment-prelim}--\ref{sec:base} to prove the existence of the large deviations rate~\eqref{eq:alpha-beta-def}; with this we prove the law of large numbers for the maximum, Theorem~\ref{mainthm:max-lln}.  In \S\ref{sec:mixing-stationarity}, we analyze finer properties of the increment sequence of $\cP_x$, showing in \S\ref{subsec:increment-mixing} that correlations between increments decay polynomially in their distance, and in \S\ref{subsec:increment-stationarity} that the increment sequences are asymptotically stationary. With these in hand, in \S\ref{sec:mean-variance}, we prove a priori estimates on the mean and variance of observables of the increment sequence of $\cP_x$, and in \S\ref{sec:clt} combine the above to prove the CLT of Theorem~\ref{mainthm:clt} and deduce Corollary~\ref{maincor:displacement-surface-area}.

\section{Preliminaries: interfaces, cluster expansion and rigidity}\label{sec:preliminaries}
In this section, we introduce key definitions from Dobrushin's decomposition of 3D Ising interfaces into \emph{walls} and \emph{ceilings} and recap his proof of rigidity of the Ising model interface. We modify the presentation of~\cite{Dobrushin72a} slightly to track certain constants, and this will serve as a useful indication of the difficulties we will encounter when our reference interface is no longer a flat plane. 

\subsection{Notation}\label{sec:notation} In this section we compile much of the notation used globally throughout the paper. 

\subsubsection{Lattice notation}
Since the object of study in the present paper is the interface separating the plus and minus phases, we consider the Ising model as an assignment of spins to the vertices of the dual graph $(\Z^3)^*= (\Z+\frac 12)^3$ so that spins are assigned to the cells of $\Z^3$ and interfaces are subsets of the faces of~$\Z^3$.  

Namely, let $\Z^3$ be the integer lattice graph with vertices at $(x_1,x_2,x_3)\in \Z^3$ and edges between nearest neighbor vertices (at Euclidean distance one). A \emph{face} of $\Z^3$ is the open set of points bounded by four edges (or four vertices) forming a square of side-length one, lying parallel to one of the coordinate axes. A face is \emph{horizontal} if its normal vector is $\pm e_3$, and is \emph{vertical} if its normal vector is one of $\pm e_1$ or $\pm e_2$. 

A \emph{cell} or \emph{site} of $\Z^3$ is the open set of points bounded by six faces (or eight vertices) forming a cube of side-length one. We will frequently identify edges, faces, and cells with their midpoints, so that points with two integer and one half-integer coordinate are midpoints of edges,  points with one integer and two half-integer coordinates are midpoints of faces, and points with three half-integer coordinates are midpoints of cells. A subset $\Lambda \subset \Z^3$ identifies an edge, face, and cell collection via the edges, faces, and cells whose bounding vertices are all in $\Lambda$; denote this edge set $\cE(\Lambda)$, its face set $\cF(\Lambda)$ and its cell set $\cC(\Lambda)$.  

Two edges are adjacent if they share a vertex; two faces are adjacent if they share a bounding edge; two cells are adjacent if they share a bounding face. A set of faces (resp., edges, cells) is connected if for any pair of faces (edges, cells), there is a sequence of adjacent faces (edges, cells) starting at one and ending at the other. We will denote adjacency by the notation $\sim$.

It will also be useful to have a notion of connectivity in $\R^3$ (as opposed to $\Z^3$); we say that an edge/face/cell is $*$-adjacent to another edge/face/cell if and only if they share a bounding vertex.

Throughout the paper, we will use the notation $d(x,y)=|x-y|$ to denote the Euclidean distance in $\mathbb R^3$ between two points $x,y$ (or if they are edges/faces/cells their respective midpoints). Similarly, we will use the notation $B_r(x)$ to denote the (closed) Euclidean ball of radius $r$ about the point $x$. When these balls are viewed as subsets of edges/faces/cells, we include all those whose midpoint is in $B_r(x)$. 
We further denote by $A\oplus B$ the symmetric difference of the face sets $A$ and $B$.

\subsubsection*{Subsets of $\Z^3$}The main subsets of $\Z^3$ with which we will be concerned are of the form of cubes and cylinders. In view of that, define the centered $2n\times 2m \times 2h$ box,  
\begin{align*}
\Lambda_{n,m,h} :=\llb -n,n\rrb \times \llb -m,m\rrb \times \llb -h,h\rrb \subset \Z^3\,,
\end{align*}
where $\llb a,b\rrb := \{a,a+1,\ldots,b-1,b\}$. We can then let $\Lambda_n$ denote the special case of the cylinder $\Lambda_{n,n,\infty}$. The (outer) boundary $\partial \Lambda$ of the cell set $\cC(\Lambda)$ is the set of cells in $\cC(\Z^3)\setminus \cC(\Lambda)$ adjacent to a cell in $\cC(\Lambda)$. 

Additionally, for any $h \in \Z$ let $\cL_h$ be the subgraph of $\Z^3$ having vertex set $\Z^2\times \{h\}$  and correspondingly defined edge and face sets $\cE(\cL_h)$ and $\cF(\cL_h)$. For a half-integer $h\in \Z + \frac 12$, let $\cL_h$ collect the faces and cells in $\cF(\Z^3) \cup \cC(\Z^3)$ whose midpoints have half-integer $e_3$ coordinate $h$. Finally we occasionally use $\cL_{>0} = \bigcup_{h>0} \cL_h$ for the upper half-space and $\cL_{<0} = \bigcup_{h<0} \cL_h$ for the lower half-space.

\subsubsection{Ising model}

An Ising configuration $\sigma$ on $\Lambda \subset \Z^3$ is an assignment of $\pm1$-valued spins to the cells of $\Lambda$, i.e., $\sigma \in \{\pm 1\}^{\cC(\Lambda)}$.  For a finite connected subset $\Lambda\subset \Z^3$, the Ising model on $\Lambda$ with boundary conditions $\sigma(\partial \Lambda) = \eta$ is the probability distribution over $\sigma \in \{\pm 1\}^{\cC(\Lambda)}$ given by 
\begin{align*}
\mu_{\Lambda}^{\eta} (\sigma) \propto \exp \left[ - \beta  \cH(\sigma)\right]\,, \qquad \mbox{where}\qquad \cH(\sigma) =  \sum_{\substack{ v,w\in \cC(\Lambda) \\  v\sim w }} \one\{\sigma_v\neq \sigma_w\} +  \sum_{\substack{ v\in \cC(\Lambda), w\in \partial \Lambda \\  v\sim w}} \one\{\sigma_v\neq \eta_w\}\,. 
\end{align*}
Throughout this paper, we will be considering the boundary conditions $\eta_w = -1$ if $w$ is in the upper half-space ($w_3 > 0$) and $\eta_w = +1$ if $w$ is in the lower half-space ($w_3 <0$). We refer to these boundary conditions as \emph{Dobrushin boundary conditions}, and denote them by $\eta = \mp$; for ease of notation, let $\mu_{n,m,h} = \mu_{\Lambda_{n,m,h}}^{\mp}$. 

\subsubsection*{Domain Markov and FKG properties} 
The Ising model is said to satisfy the \emph{domain Markov property}, meaning that for any two finite subsets $A \subset B \subset \cC(\Z^3)$, and every configuration $\eta$ on $B\setminus A$,  
$$\mu_{B}(\sigma_A \in \cdot \mid \sigma_{B\setminus A} = \eta_{B\setminus A}) = \mu_A^{\eta_{\partial A}}(\sigma_{A}\in \cdot )\,,$$
where we use $\sigma_A$ to denote the restriction of the configuration to the set $A$.
It also satisfies an important consequence of its monotonicity, known as the \emph{FKG inequality}. That is, for any two increasing (in the natural partial order on configurations) functions $f,g:\{\pm 1\}^{\cC(\Lambda)}$, we have 
\begin{align*}
\E_{\mu_\Lambda} \big[f(\sigma)g(\sigma)\big]  \geq \E_{\mu_\Lambda} \big[ f(\sigma)\big] \E_{\mu_{\Lambda}} \big[g(\sigma)\big]\,.
\end{align*} 
A special case of this inequality, is when $f$ and $g$ are indicator functions of increasing events $A$ and $B$ (meaning that if $\sigma \leq \sigma'$ and $\sigma \in A$, then $\sigma'\in A$, and similarly for $B$), yielding $\mu_\Lambda (A,B)\geq \mu_\Lambda (A)\mu_\Lambda (B)$. 

An increasing event that will appear in the proof of Theorem~\ref{mainthm:max-lln}, is a connection event. Namely, we call a cell set a connected set of plus sites in $\sigma$, if it is a connected set of cells such that all the cells are assigned $+1$ under $\sigma$. A \emph{plus cluster} in $\sigma$ is a maximal connected set of plus sites.  If we denote by $\{v\xleftrightarrow{+\,} w\}$ the event that $v,w\in \cC(\Lambda)$ are in the same plus cluster, we see that this is an increasing event. Finally, for a subset $\Lambda'\subset \Lambda$, denote by $\{v \xleftrightarrow[\Lambda']{\;+\,\;}w\}$ the event that $v,w$ are part of the same plus cluster using only cells of $\Lambda'$. 

\subsubsection*{Infinite-volume measures} Care is needed to define the Ising model on infinite graphs, as the partition function becomes infinite; infinite-volume Gibbs measures are therefore defined via what is known as the \emph{DLR conditions}; namely, for an infinite graph $G$, a measure $\mu_G$ on $\{\pm 1\}^G$, defined in terms of its finite dimensional distributions, satisfies the DLR conditions if for every finite subset $\Lambda \subset G$, 
\begin{align*}
\E_{\mu_G(\sigma_{G\setminus \Lambda}\in \cdot)}\big[\mu_G(\sigma_\Lambda \in \cdot\mid \sigma_{G\setminus \Lambda} )\big ] = \mu_{G}(\sigma_\Lambda\in \cdot)\,.
\end{align*}
On $\Z^d$, infinite-volume Gibbs measures arise as weak limits of finite-volume measures, say $n\to\infty$ limits of the Ising model on boxes of side-length $n$ with certain prescribed boundary conditions. At low temperatures $\beta>\beta_c(d)$, the Ising model on $\Z^d$ admits multiple infinite-volume Gibbs measures; taking plus and minus boundary conditions on boxes of side-length $n$ yield the distinct infinite-volume measures $\mu^+_{\Z^3}$ and $\mu^-_{\Z^3}$~\cite{Minlos-Sinai}.

\subsection{Interfaces under Dobrushin boundary conditions}\label{sec:dobrushin-definitions}
We begin with the key combinatorial decomposition from~\cite{Dobrushin72a} describing the interface separating the minus and plus phases under the Dobrushin boundary conditions. We refer the reader to~\cite{Dobrushin72a} for more details. 

\begin{definition}[Interfaces]\label{def:interface} For a domain $\Lambda_{n,m,h}$ with Dobrushin boundary conditions, and an Ising configuration
$\sigma$ on $\mathcal C(\Lambda_{n,m,h})$, the \emph{interface} $\cI=\cI(\sigma)$
is defined as follows:  
\begin{enumerate}
\item Extend $\sigma$ to a configuration on $\mathcal C(\mathbb Z^3)$ by taking $\sigma_v=+1$ (resp., $\sigma_v=-1$) if $v\in \cL_{<0}\setminus \mathcal C(\Lambda_{n,m,h})$ (resp., $v\in \cL_{>0}\setminus \cC(\Lambda_{n,m,h})$). 
\item Let $F(\sigma)$ be the set of faces in $\cF(\mathbb Z^3)$ separating cells with differing spins under $\sigma$.
\item  Call the (maximal) $*$-connected component of $\cL_0 \setminus \cF(\Lambda)$ in $F(\sigma)$, the \emph{extended interface}. (This is also the unique infinite $*$-connected component in $F(\sigma)$.)
\item The interface $\mathcal I$  is the restriction of the extended interface to $\cF(\Lambda_{n,m,h})$. 
\end{enumerate}
\end{definition}


It is easily seen (by Borel--Cantelli) that taking the $h\to\infty$ limit $\mu_{n,m,h}$ to obtain the infinite-volume measure $\mu_{n,m,\infty}^\mp$, the interface defined above stays finite almost surely. Thus, $\mu_{n,m,\infty}^\mp$-almost surely, the above process also defines the interface for configurations on all of $\mathcal C(\Lambda_{n,m,\infty})$. 


\begin{remark}\label{rem:different-interface-conventions}
In lieu of the above definition of the interface due to~\cite{Dobrushin72a}, one could consider other flavors, e.g., letting $\cI$ be a \emph{minimal} connected set of faces separating differing spins (or following some splitting rule which singles out a unique connected set of such faces, e.g., along the northeast diagonal in 2D). A simple Peierls argument implies that the set difference between that definition and Dobrushin's definition consists of finite connected sets of faces with exponential tails on their size.
\end{remark}

\begin{remark}\label{rem:interface-spin-config}
Just as Ising configurations with Dobrushin boundary conditions define an interface, every interface uniquely defines a configuration with exactly one $*$-connected plus component and exactly one $*$-connected minus component. For every $\cI$, we can obtain this configuration $\sigma(\cI)$ by iteratively assigning spins to $\cC(\Lambda_{n,m,h})$, starting from the boundary and proceeding inwards, in such a way that adjacent sites have differing spins if and only if they are separated by a face in $\cI$. 
 Informally, $\sigma(\cI)$ is distinguishing the sites that are in the ``plus phase" and ``minus phase" given the interface $\cI$.

(Note that the extended interface also splits $\cC(\Z^3)$ into precisely two infinite connected (as opposed to $*$-connected) components, along with possibly additional finite connected components.)
\end{remark}

Following~\cite{Dobrushin72a}, we can decompose the faces in $\cI$ and define certain useful  subsets of $\cI$. 
For a face $f\in \cF(\Z^3)$, its \emph{projection} $\rho(f)$ is the edge or face of $\cL_0$ given by $\{(x_1,x_2,0):(x_1,x_2,s)\in f \mbox{ for some $s\in \R$}\}\subset \R^2\times \{0\}$. Specifically, the projection of a \emph{horizontal} face (a face that is parallel to the plane $\cL_0$) is a face in $\cF(\cL_0)$, while the projection of a \emph{vertical} face (one that is not parallel to $\cL_0$) is an edge in $\cE(\cL_0)$. 
The projection of a collection of faces $F$ is $\rho(F) := \bigcup_{f\in F} \rho(f)$, which may consist both of edges and faces of $\cL_0$. 

\begin{figure}
    \centering
    \includegraphics[width=0.21\textwidth]{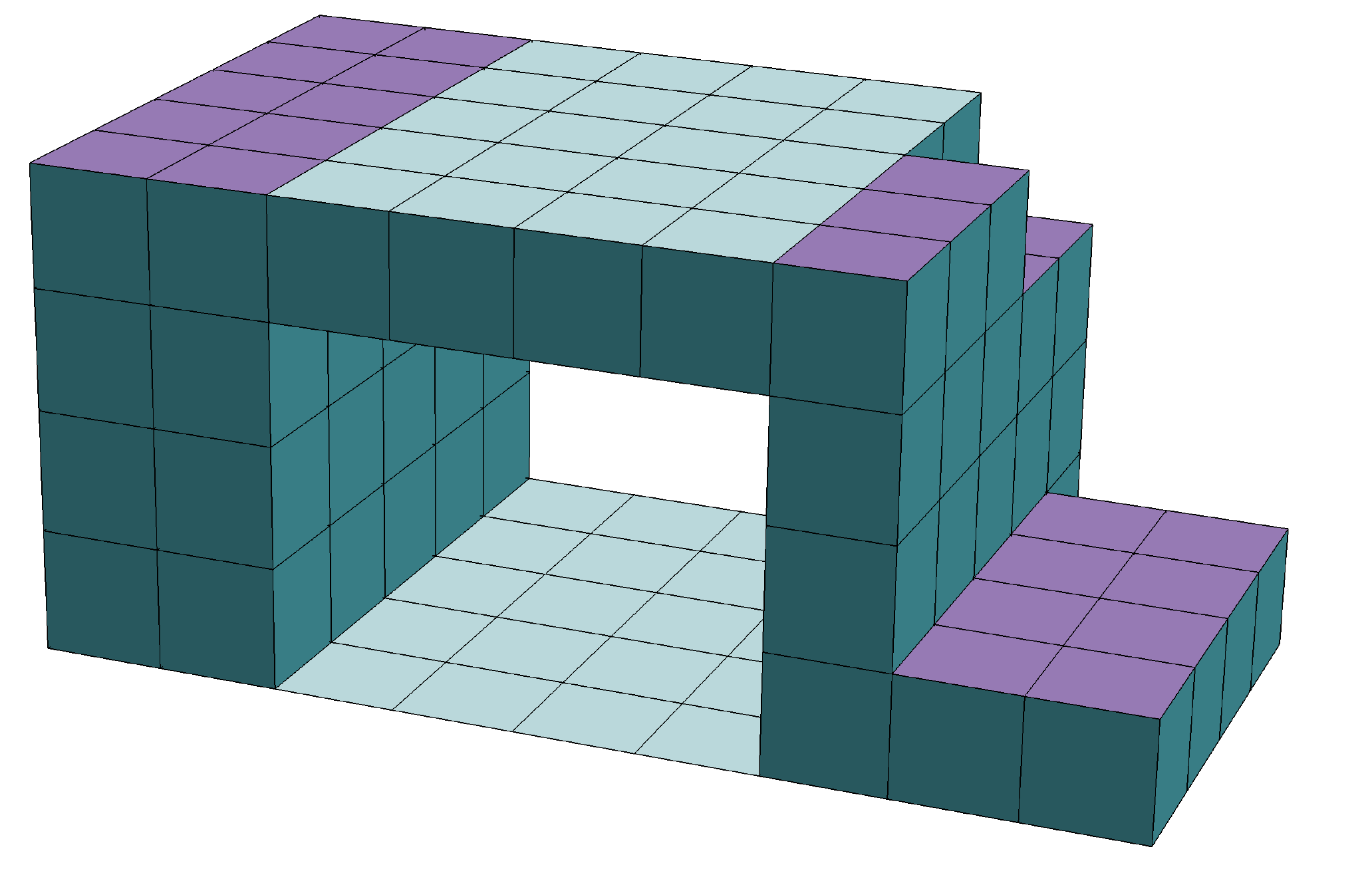}
    \hspace{20pt}
    \raisebox{-8pt}{\includegraphics[width=0.3\textwidth]{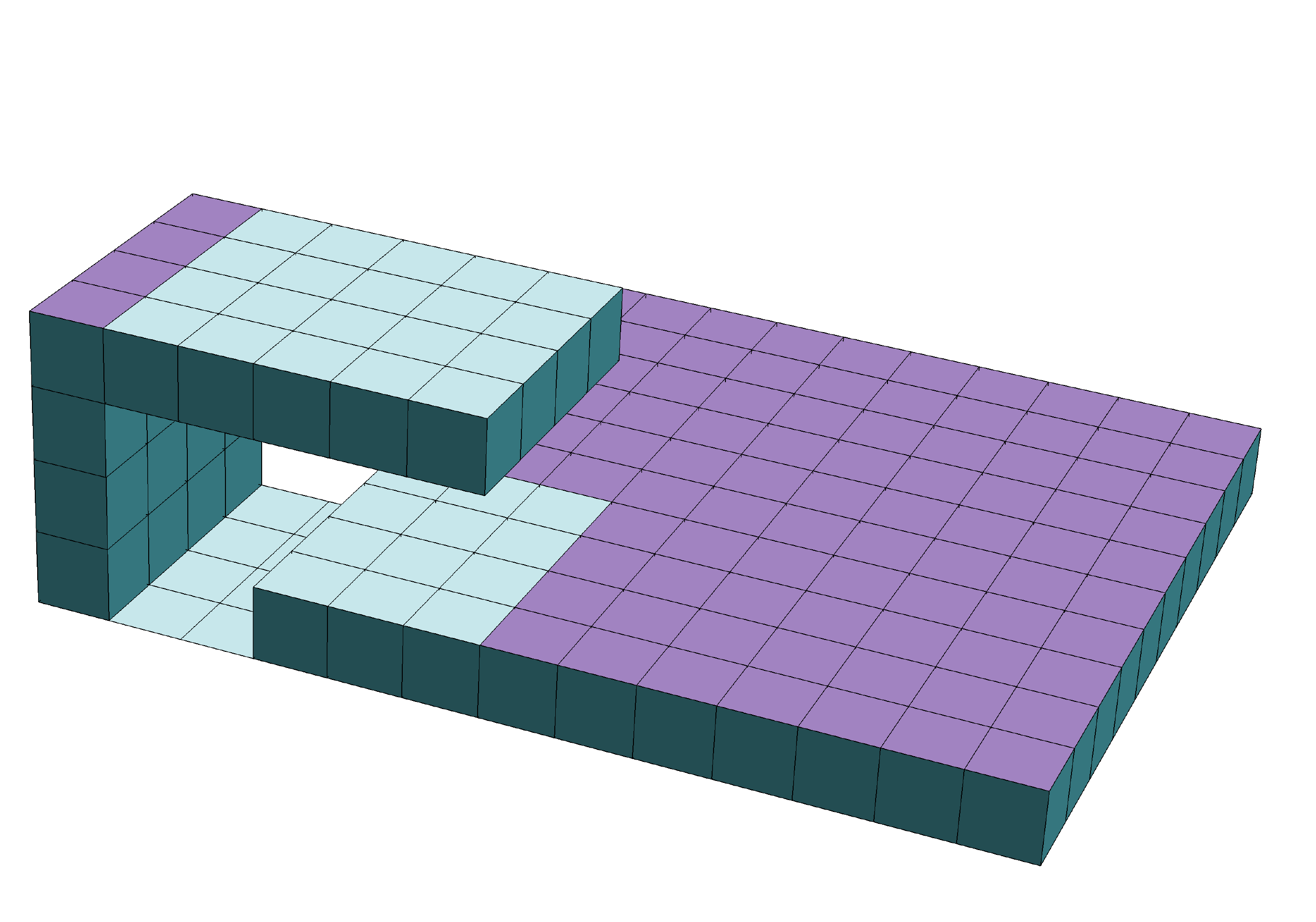}}
    \hspace{20pt}
    \raisebox{-2pt}{\includegraphics[width=0.19\textwidth]{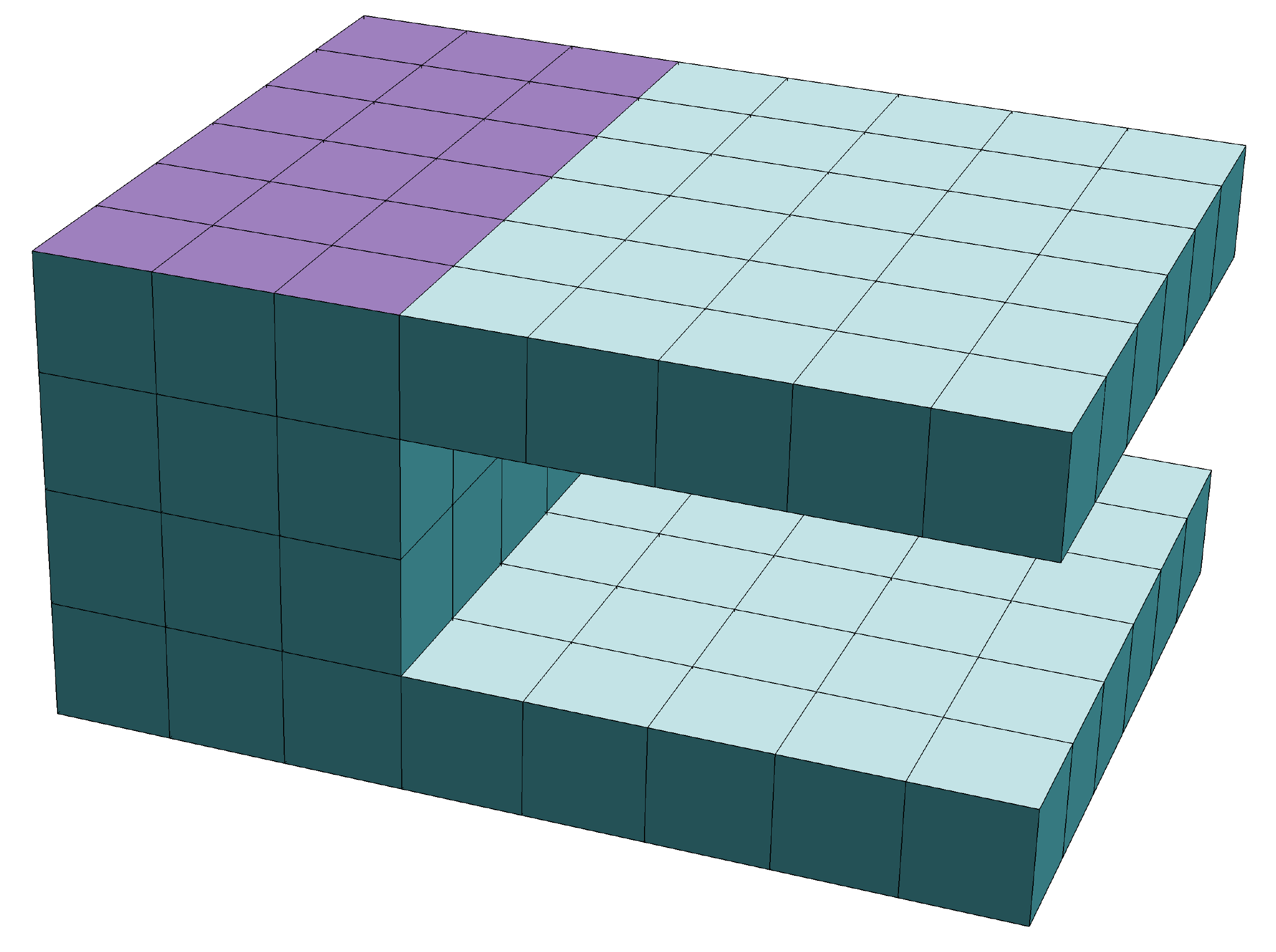}}
    \caption{Three distinct standard walls, with their interior ceiling faces (purple) and wall faces (vertical in teal, horizontal in light blue) as per Definition~\ref{def:ceilings-walls}. Middle example features two distinct ($+$) components in $\R^2\times \R_+$ which correspond to two distinct pillars but form a single wall (consistent with the fact that projections of distinct walls on $\cL_0$ are disjoint).}
    \label{fig:walls-ceil-ex}
\end{figure}

\begin{definition}[Ceilings and walls]\label{def:ceilings-walls}
A face $f\in \cI$ is a \emph{ceiling face} if it is horizontal and there is no $f'\in \cI \setminus  \{f\}$ such that $\rho(f)=\rho(f')$. A face $f\in \cI$ is a \emph{wall face} if it is not a ceiling face. 
A \emph{wall} is a (maximal) $*$-connected set  of wall faces.  A \emph{ceiling} of $\cI$ is a (maximal) $*$-connected set of ceiling faces. 
%
\end{definition}

\begin{definition}[Floors of walls]\label{def:floors}
For a wall $W$, the complement of its projection (a subset of $\R^2$) $$\rho(W)^c:= (\cE(\cL_0)\cup \cF(\cL_0)) \setminus \rho(W)$$ splits into one infinite component, and some finite ones. Any ceiling adjacent to the wall $W$ projects into one of these components; the one that projects into the infinite component is called the \emph{floor} of $W$. 
\end{definition}

This can be reinterpreted with the following notion of nesting of walls and ceilings.

\begin{definition}\label{def:nesting-of-walls}
We say an edge or face $u\in \cE(\cL_0) \cup \cF(\cL_0)$ is interior to a wall $W$ if $u$ is not in the infinite component of $\rho(W)^c$.   

A wall $W$ is \emph{interior} to (or nested in) a wall $W'$ if every element of $\rho (W)$ is interior to $W'$. Similarly, a ceiling $C$ is \emph{interior} to a wall $W$ if every element of $\rho(C)$ is interior to $W$. 
\end{definition}

Observe that of the ceilings $C_0, C_1,..., C_l$ adjacent to a wall $W$, one of them is the floor of $W$---say $C_0$---and the rest are interior to $W$. For any admissible pair of standard walls, as their projections are disjoint,  either one wall is nested in the other, or $\rho(W_x)$ is contained in the infinite component of $\rho(W_y)^c$ and vice versa.

\begin{definition}[Standard walls]\label{def:standard-walls}
A wall $W$ is a \emph{standard wall} if there exists an interface $\cI_W$ such that $\cI_W$ has exactly one wall, $W$---as such it must have as its unique floor a subset of $\cL_0$. 
A collection of standard walls is \emph{admissible} if they are all disjoint and have pairwise disjoint projections (see Figure~\ref{fig:walls-ceil-ex}). 
\end{definition}

\begin{lemma}[{\cite{Dobrushin72a}}]\label{lem:wall-ceiling-bijection}
For a projection of the walls of an interface, each connected component of that projection (as a subset of edges and faces) corresponds to a single wall. Moreover, there is a 1-1 correspondence between the ceilings adjacent to a standard wall $W$ and the connected components of $\rho(W)^c$. Similarly, for a wall $W$, all other walls $W'\neq W$ can be identified to the connected component of $\rho(W)^c$ they project into, and in that manner they can be identified to the ceiling of $W$ to which they are interior.
\end{lemma}

\begin{definition}[Standardization of walls]\label{def:std-of-wall}
To each ceiling $C$, we can identify a unique \emph{height} $\hgt(C)$ since all faces in the ceiling have the same $x_3$ coordinate. For every wall $W$, we can define its \emph{standardization} $\theta_{\textsc {st}}(W)$ which is the translate of the wall by $(0,0,-s)$ where $s$ is the height of its floor.   
\end{definition}

\begin{remark}\label{rem:indexing-walls}
We can index walls as follows: assign an ordering of the faces of $\cL_0$, and index $W$ by the minimal face in $\cL_0$ that shares an edge with $\rho(W)$, and lies either in $\cF(\rho(W))$ or in one of the finite connected components of $\rho(W)^c$.
For any admissible collection of standard walls, the indices of the walls are distinct. 
\end{remark}

\begin{figure}
    \begin{tikzpicture}
    \node (fig1) at (0,0) {
    	\includegraphics[width=0.48\textwidth]{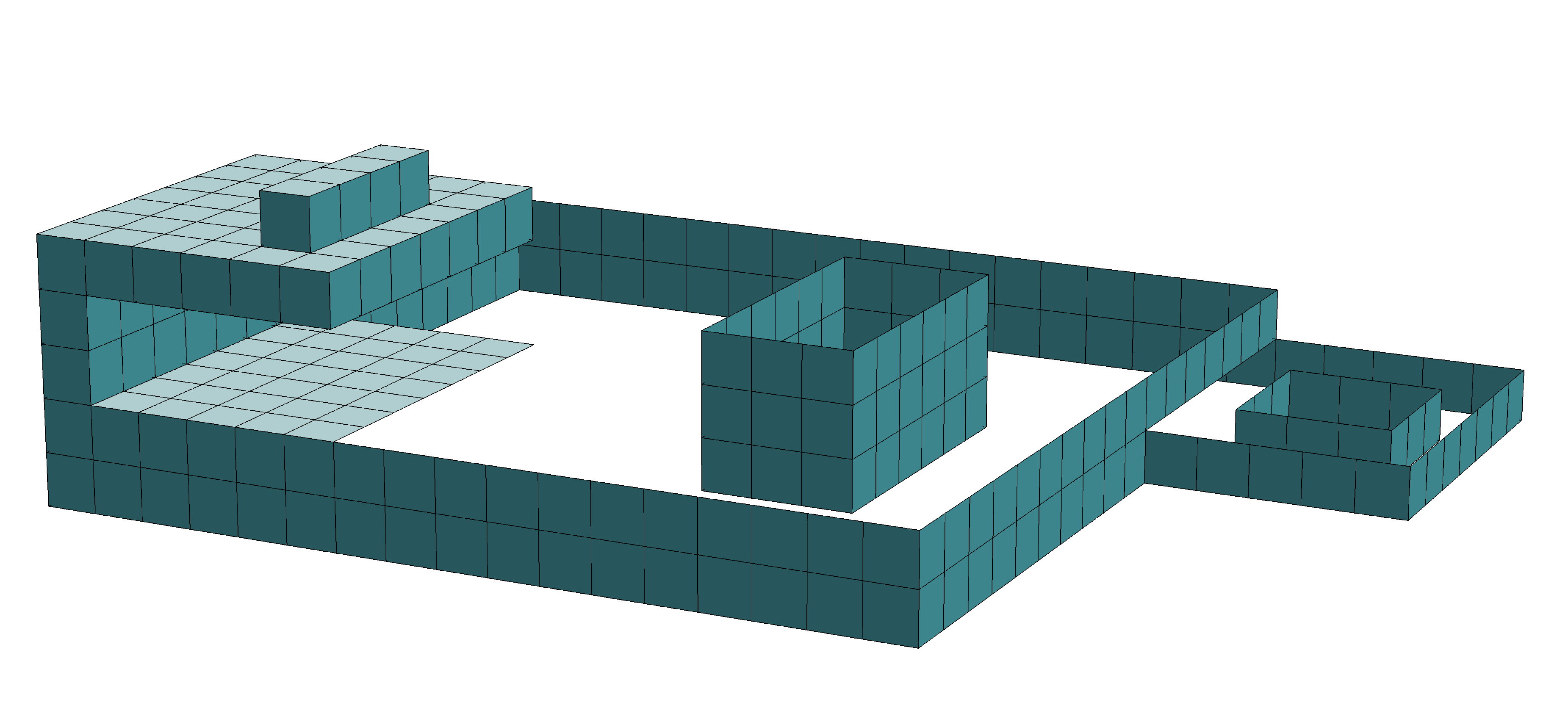}};
    \node (fig2) at (8,0) {
    	\includegraphics[width=0.48\textwidth]{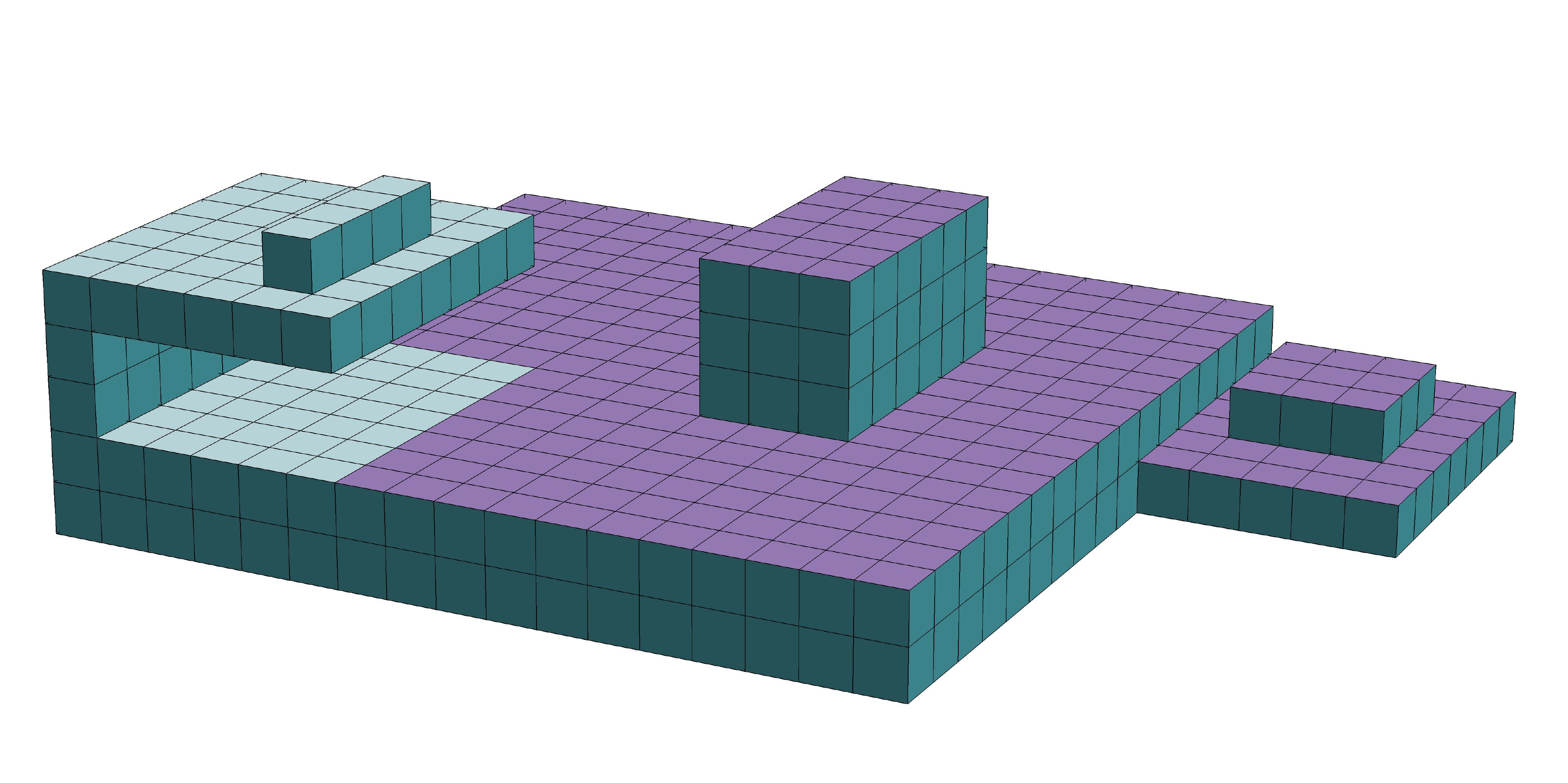}};
    \end{tikzpicture}
    \caption{Correspondence between an interface and its standard wall representation (Lemma~\ref{lem:interface-reconstruction}): three distinct standard walls (left) and their corresponding interface (right).}
    \label{fig:1-1-correspond-std}
\end{figure}

We then have the following important bijection between interfaces and their standard wall representation. 

\begin{definition}\label{def:standard-wall-representation}
Let the \emph{standard wall representation} of an interface $\cI$ be the collection of standard walls given by standardizing all walls of $\cI$. 
\end{definition}

\begin{lemma}[{\cite{Dobrushin72a}}]\label{lem:interface-reconstruction}
There is a 1-1 correspondence between the set of interfaces and the set of admissible collections of standard walls. In particular, the standardization $\theta_{\textsc {st}}(W)$ of a wall $W$ is a standard wall. 
\end{lemma}
\begin{proof}
From an interface, the standard wall representation is an admissible collection of standard walls as projections of distinct walls are disjoint. 
To obtain an interface from an admissible collection of standard walls, it suffices to take the standard wall representation of an interface $\cI$ and describe how the addition of one standard wall $\theta_{\textsc{st}}(W_{t_0})$, compatible with the standard walls of $\cI$ and not interior to any walls in $\cI$, changes $\cI$ to $\cI'$. (One could then construct an interface $\cI$ from its standard wall representation by beginning with the interface $\cL_0$ with empty standard wall representation, and iterating the above procedure, adding the standard walls from innermost outward). 

Consider an interface $\cI$ with standard wall collection $(\theta_{\textsc{st}} W_{t})_{t\neq t_0}$ such that $((\theta_{\textsc{st}} W_{t})_{t\neq t_0})\cup \theta_{\textsc {st}}(W_{t_0})$ is admissible; suppose further that $\theta_{\textsc{st}}$ is not interior to any wall of $\cI$. 
Let $\cJ_{W_{t_0}}$ be the interface whose only wall is the standard wall $\theta_{\textsc{st}} W_{t_0}$, and denote its floor by $C_0$ and non-floor ceilings by $C_1,...,C_l$. 

Construct a face set from $\cI$ and $\theta_{\textsc{st}}W_{t_0}$ as follows: 
\begin{enumerate}
\item Remove all horizontal faces of $\cI$ in $\rho(W_{t_0})$.
\item Vertically shift every face of $\cI$ projecting into one of $\rho(C_i)_{1\leq i \leq l}$ by $\hgt(C_i)$. 
\item Add all faces of $\theta_{\textsc{st}} W_{t_0}$. 
\end{enumerate}
The resulting face set is evidently a valid interface $\cI'$ and one can check that it has standard wall representation $(\theta_{\textsc{st}} W_{t})_{t\neq t_0})\cup \theta_{\textsc{st}} W_{t_0}$. 
\end{proof}

We note the following important observation based on the above bijection.  

\begin{observation}\label{obs:1-to-1-map-faces}
Consider interfaces $\cI$ and $\cJ$, such that the standard wall representation of $\cI$ contains that of $\cJ$ (and additionally has the standardizations $\mathbf W=W_1,...,W_r$). By the construction  in Lemma~\ref{lem:interface-reconstruction}, there is a 1-1 map between the faces of $\cI \setminus \mathbf W$ and the faces of $\cJ \setminus \mathbf H$ where $\mathbf H$ is the set of faces in $\cJ$ projecting into $\rho (\mathbf W)$. Moreover, this bijection can be encoded into a map $f\mapsto \tilde f$ that only consists of vertical shifts, and such that all faces projecting into the same component of $\rho(\mathbf W)^c$ undergo the same vertical shift. 
\end{observation}

 Finally, we introduce a notion of nested walls which will prove useful to bounding the base of tall pillars. 
 
 \begin{definition}\label{def:nested-sequence-of-walls}
To any edge/face/cell $x$, we can assign a \emph{nested sequence of walls} $\fW_x = \bigcup_{s} W_{u_s}$ that is composed of all walls that $\rho(x)$ is interior to (by Definition~\ref{def:nesting-of-walls}, this forms a nested sequence of walls). 
\end{definition}

\begin{observation}
For $u\in \cL_0$, for a nested sequence of walls $\fW_u$, one can read off the height of the face(s) of $\cI$ projecting onto $u$. In particular, if a face $f\in \cI$ has height $h$, its nested sequence of walls must be such that the sum of the heights of the walls in $\fW_{\rho(f)}$  exceeds~$h$.
\end{observation}

\subsection{Interface pillars}\label{sec:pillar-def}
The above definitions were all  from~\cite{Dobrushin72a} and, informally, they reduce the analysis of 3D Ising interfaces to that of a low-temperature 2D polymer model given by the projections of walls. For us, this is insufficient as we aim to study the structure of tall walls, wherein the projection does not carry much information about the shape and height. As such, we define the notion of a pillar above $x\in \cL_0$.  

\begin{definition}[Pillars]\label{def:pillar}
For every interface $\cI$, consider the restriction $\sigma(\cI)\restriction_{\cL_{>0}}$ of the Ising configuration $\sigma(\cI)$ to the upper half-space. For any face $x\in\cL_{0}$, the cell-set $\sigma(\cP_x)$ of the \emph{pillar} $\cP_x=\cP_{x}(\cI)$
above $x$ will be the (possibly empty) $*$-connected plus component in $\sigma(\cI)\restriction_{\cL_{>0}}$ containing $x+(0,0,\frac 12)$. The pillar $\cP_x$ will have face-set consisting of the bounding faces of $\sigma(\cP_x)$ in the upper half-space, so that it is a subset of $\cI$.
\end{definition}

Pillars can be viewed as a subset of some collection of nested walls along with their ceilings as follows: 

\begin{observation}\label{obs:pillar-wall}
The pillar $\cP_x$ is described by $\fW_x$ together with all walls that are nested in some wall of $\fW_x$; namely, if we index walls by enumerating faces of $\cL_0$ in terms of distance to $x$, then the set of walls $\bigcup_{y:\,d(y,x)\leq \diam(\sB_x)} W_{y}$ contain all the information about the pillar~$\cP_x$. 
Moreover, $\cP_x \cap (\mathbb R^2 \times (\lfloor\hgt(v_1)\rfloor,\infty))$ (possibly with the exception of one upper delimiting face) is all a subset of a single wall. 
\end{observation}

Much of this paper is interested in the large deviations regime for the height of such pillars, so we formally define heights of interface subsets. 
\begin{definition}\label{def:heights}
For a point $(x_1,x_2, x_3)\in \R^3$, we say its height is $\hgt(x) = x_3$. The height of a cell is the height of its midpoint. For a pillar  $\cP_x \subset \cI$, its height is given by $$\hgt(\cP_x) = \sup\{x_3: (x_1,x_2,x_3)\in f, f\in \cP_x\}\,.$$ 
It is important to distinguish between situations where $\cP_x$ is empty because the interface lies exactly at face $x$, and when it goes below face $x$; in view of this, if $\cP_x= \emptyset$, we say that $\hgt(\cP_x) = 0$ if $x- (0,0, \frac 12)$ is in the plus phase (i.e., is plus in $\sigma(\cI)$), and  $\hgt(\cP_x) <0$ if $x-(0,0, \frac 12)$ is in the minus phase. 
\end{definition}

\subsection{Excess area} For a pair of interfaces, we need to quantify the energy cost/gain of having one interface over the other. The competition of this energy cost  with respect to the interface $\cL_0$ with the entropy gain from additional fluctuations governs the behavior of the Dobrushin interface.    
\begin{definition}[Excess area]
\label{def:excess-area}
For two interfaces $\cI,\cI'$,
the \emph{excess area} of $\cI$ with respect to $\cI'$,
denoted $\fm(\cI;\cI')$, is given by $$\fm(\cI;\cI'):= |\cI|-|\cI'|\,,$$ where these are the cardinalities of the face-sets of $\cI$ and $\cI'$ respectively. Evidently, for any Dobrushin interface $\cI$, we have that $\fm(\cI; \cL_0\cap \Lambda)\geq 0$.   

We can also define excess areas for subsets of interfaces, and interpret these as the ``excess area of the interface that contains the subset with respect to a reference one that does not." For instance, for a standard wall $W$, if we denote by $\cI_W$ the interface whose only wall is $W$, then  $\fm(W)=\fm(\cI_{W}; \cL_0\cap \Lambda)$. For a wall $W$, its excess area is given by the excess area of the standard wall $\theta_{\textsc {st}}(W)$. The excess area of a collection of walls $F$ is analogously defined, and one can easily see that $\fm(F) = \sum_{W\in F} \fm(W)$.

Finally, define the excess area of a pillar $\fm(\cP_x)$, and of one pillar with respect to another, $\fm(\cP_x;\cP_x')$, via the excess areas of the unique interfaces consisting only of the faces in $\cP_x$ (resp., $\cP_x'$) along with faces of $\cL_0$. 
\end{definition}

\begin{remark}\label{rem:excess-area-properties}
Notice that for a wall $W_x$, its excess area is exactly given by 
\begin{align*}
\fm(W_x) = \fm(\theta_{\textsc {st}}(W_x))= |W_x| - |\cF(\rho(W_x))|
\end{align*}
where $\cF(\rho(W_x))$ is the face set of the projection $\rho(W_x)$. Moreover, for an interface $\cI$ having standard wall representation $(W_t)_{t\in \cL_0}$ per Lemma~\ref{lem:interface-reconstruction}, we have that 
\begin{align*}
\fm(\cI) = \fm(\cI; \cL_0\cap \Lambda) = \sum_{W\in (W_x)_{x\in \cL_0}} \fm(W)\,. 
\end{align*}
\noindent As observed in~\cite{Dobrushin72a}, this form of the excess area makes a few key properties clear: 
\begin{align}\label{eq:excess-area-wall-relations}
\fm(W_x) \geq \frac{1}{2} |W_x| \qquad \mbox{and} \qquad \fm(W_x) \geq |\cE(\rho(W_x)|+ |\cF(\rho(W_x))|\,.
\end{align} 
Moreover, any two faces $x,y\in \cL_0$ interior to the projection $\rho(W_x)$ satisfy $|x-y| \leq \fm(W_x)$. 
\end{remark}

\subsection{Cluster expansion for interfaces describing phase coexistence}\label{sec:cluster-expansion}
Cluster expansion is a classical tool for expressing the partition function of a spin system on a domain as a product of polymer weights (in an appropriate polymer representation of the model) rather than as a sum over weights of configurations. Crucially, this product is an infinite product that only converges in perturbative regimes (e.g., for us $\beta \gg 1$). 

In our setting of the Ising model, these polymers are minimal connected sets of faces which separate differing spins, and are the bounding face-set of a connected set of cells. The associated weight of such a face-set $\gamma$ is given by $e^{-\beta |\gamma|}$. The polymers are then endowed with hard-core interaction rules encoding the admissibility of a collection of polymers, so that it in fact encodes uniquely, an Ising spin configuration. For a full derivation of the validity of cluster expansion, we refer the reader to the book~\cite[Chapter 5]{velenikbook}. 
In our setting, the hard-core polymer interactions 
preclude distinct polymers from sharing any edges or vertices.  

Using this cluster expansion,~\cite{Minlos-Sinai} proved properties of the single-phase Ising measures $\mu^-_{\Z^3}$ and $\mu^+_{\Z^3}$ at low temperatures. An easy implication of this cluster expansion is that one can take a limit of $\mu^{\mp}_{\Lambda_{n,n,h}}$ as $h\to\infty$ and obtain an infinite-volume Gibbs measure on the cylinder $\Lambda_n = \Lambda_{n,n,\infty}$ whose interface is finite almost surely (for each fixed $n$), and this limit does not depend on the boundary conditions taken at the top and bottom of $\Lambda_{n,n,h}$: see~\cite[(2.7) as well as Lemma 3]{Dobrushin72a}. Denote this limiting measure $\mu_n = \mu^\mp_{\Lambda_{n,n,\infty}}$.

Applying the cluster expansion one can compute probabilities of interfaces under this $\mu_n$ measure.

\begin{theorem}[{\cite[Lemma 1]{Dobrushin72a}}]
\label{thm:cluster-expansion} Consider the Ising measure $\mu_{n}= \mu^{\mp}_{n}$ on the cylinder $\Lambda_{n,n,\infty}$. There exists $\beta_0>0$ and a function $\g$ such that for every $\beta>\beta_{0}$ and any
two interfaces $\cI$ and $\cI'$, 
\begin{align*}
\frac{\mu_n(\cI)}{\mu_n(\cI')}= & \exp\bigg(-\beta \fm(\cI; \cI') +\Big(\sum_{f\in\cI}\g(f,\cI)-\sum_{f'\in\cI'}\g(f',\cI')\Big)\bigg)
\end{align*}
and $\g$ satisfies the following for some $\bar{c},\bar{K}>0$
independent of $\beta$: for all $\cI, \cI'$ and $f\in \cI$ and $f'\in\cI'$,
\begin{align}
|\g(f,\cI)| & \leq\bar{K} \label{eq:g-uniform-bound} \\
|\g(f,\cI)-\g(f',\cI')| & \leq \bar K e^{-\bar{c}\br(f,\cI;f',\cI')} \label{eq:g-exponential-decay}
\end{align}
where $\br(f,\cI;f',\cI')$ is the largest radius around
the origin on which $\cI-f$ ($\cI$ shifted by the midpoint of the face $f$) is \emph{congruent} to $\cI'-f'$. That is to say,
\begin{align*}
\br(f,\cI; f', \cI') := \sup \{r: (\cI - f) \cap B_{r}(0) \equiv (\cI' - f') \cap B_r(0)\}
\end{align*}
where the congruence relation $\equiv$ is equality as subsets of $\R^3$, up to, possibly, reflections and $\pm \frac \pi2$ rotations in the horizontal plane.  
\end{theorem}

Throughout the rest of the paper, the constants $\bar c$ and $\bar K$ will be reserved for those of~\eqref{eq:g-uniform-bound}--\eqref{eq:g-exponential-decay}.

\begin{remark}
In~\cite{Dobrushin72a} and other works, the congruence above is written only as a congruence up to translation. However, one can see by following the derivation of Theorem~\ref{thm:cluster-expansion}, that this congruence can also be up to reflections and $\pm \frac \pi 2$ rotations in the $xy$-plane (under which the Ising Hamiltonian is invariant). More precisely, for polymer weights $w(\gamma)$ and interactions $\delta(\gamma, \gamma')$, we can define the \emph{Ursell functions} as
\begin{align*}
\varphi(\gamma_1,\ldots, \gamma_m) = \frac 1{m!}\sum_{G\subset K_m} \prod_{(i,j)\in G}[\delta(\gamma_i,\gamma_j)-1]\,,
\end{align*}
where the sum is over connected subgraphs of the complete graph on $m$ vertices. The cluster expansion formally expresses the partition function of the Ising model on a graph as 
\begin{align*}
\cZ = \exp\Big[\sum_{m\geq 1}\sum_{\gamma_1,\ldots, \gamma_m}\varphi(\gamma_1,\ldots,\gamma_m)\prod_{i\leq m} w(\gamma_i)\Big]\,.
\end{align*}
Theorem~\ref{thm:cluster-expansion} arises from viewing $\mu_{n}^\mp$ with interface $\cI$ as a cost from the disagreements along $\cI$, along with one Ising model above $\cI$ with minus boundary conditions, and one below $\cI$ with plus boundary conditions. The function $\g$ is therefore given by simple algebraic manipulations from the Ursell functions and polymer weights, all of which are invariant under reflections and rotations in the $xy$-plane.
\end{remark}

We end this section with a piece of terminology that we will use frequently. We will say that the radius $\br(f,\cI; f',\cI')$ is \emph{attained} by a face $ g \in \cI$ (resp., $g' \in \cI'$) of minimal distance to $f$ (resp., $f'$) whose presence prevents $\br(f, \cI; f',\cI')$ from being any larger.

\subsection{Rigidity of Dobrushin interfaces}\label{sec:dobrushin-proof}
For the benefit of the reader, we include Dobrushin's proof of rigidity for 3D interfaces from~\cite{Dobrushin72a}, namely that the walls corresponding to horizontal interfaces have exponential tails on their excess areas. This will straightforwardly imply that the probability that the pillar above a face $x\in \cL_0$ reaches a height $h$ has an exponentially decaying tail. We will need the following definition of~\cite{Dobrushin72a} that collects walls that are close, and therefore excessively interact with one another, together. 

\begin{definition}\label{def:group-of-walls}
For a wall $W$, for every edge or face $u\in \rho(W)$, let $N_\rho(u)= \#\{f\in W: \rho(f)= u\}$. We say that two walls $W_1$ and $W_2$ are \emph{close} if there exist $u_1\in \rho(W_1)$ and $u_2 \in \rho(W_2)$ such that 
\begin{align*}
|u_1 - u_2| \leq \sqrt{N_\rho(u_1)} + \sqrt {N_\rho(u_2)}\,.
\end{align*}
Then an admissible set of standard walls $F = \bigcup_i W_i$ is a \emph{group of walls} if it is a maximal connected  component (via the adjacency relation induced by closeness) of walls i.e., every wall in $F$ is close to some other wall in $F$ and no wall not in $F$ is close to a wall of $F$. Index a group of walls by the minimal index of its walls, and let $(\cF_x)_{x\in \cL_0}$ be the admissible group of wall collection of $\cI$. 
\end{definition}

Following the definition of admissible sets of standard walls and Lemma~\ref{lem:interface-reconstruction}, it should be clear how admissible collections of groups of walls would be defined, and that the set of all admissible collections of groups of walls are in 1-1 correspondence with the set of all possible Dobrushin interfaces (see \S5 of~\cite{Dobrushin72a}). 

\begin{remark}\label{rem:strongest-enumerations}
The procedure for sorting the faces of $\cL_0$ and using this ordering to identify each group of walls by the appropriate minimal face in $\cL_0$ that can be used to identify the group of walls, will be  called an \emph{indexing} of $\cI$. Our results will easily be seen to hold uniformly over this indexing (i.e., uniformly over all orderings of the faces in $\cL_0$). 
\end{remark}

\begin{lemma}[{\cite[Lemma 8]{Dobrushin72a}}]\label{lem:dobrushin-wall-ratio}
There exists $\beta_0$ and a universal $C$ such that for $\beta>\beta_0$, for any admissible collection of groups of walls $(F_y)_{y\neq x},F_{x}$,  we have 
\begin{align*}
\frac{\mu_n(\sF_x = F_{x} ,  (\sF_y)_{y\neq x} = (F_y)_{y\neq x})}{\mu_n(\sF_x = \emptyset, (\sF_y)_{y\neq x} = (F_y)_{y\neq x})} \leq \exp [- (\beta-C) \fm(F_{x})\big]\,.
\end{align*}
\end{lemma}

The above readily implies an exponential tail on the size of the group of walls indexed by face $x\in \cL_0$. In fact, it can easily be used to show that the probability that the interface intersects the column $\{(x_1,x_2,s): s\in \R\}$ above a height $H$ decays exponentially in $H$, and with our definition of pillars, we can also use it to show that it implies an exponential tail on $\hgt(\cP_x)$. 

\begin{theorem}[{\cite{Dobrushin72a,Dobrushin72b,Dobrushin73}, see also \cite{BLP79b}}]\label{thm:dobrushin-rigidity}
There exists $C>0$ such that for every $\beta>\beta_0$, for every $x\in \cL_0 \cap \Lambda$, and every $r\geq 1$, 
\begin{align*}
\mu_n(\fm(\sF_x) \geq r)\leq \exp\big[-(\beta-C)r\big]\,.
\end{align*}
Furthermore, we have that for every $h\geq 1$, 
\begin{align*}
\mu_n(\hgt(\cP_x) \geq h)\leq \exp\big[-4(\beta-C) h \big]\,.
\end{align*}
\end{theorem}

\begin{proof}[\textbf{\emph{Proof of Lemma~\ref{lem:dobrushin-wall-ratio}}}]
Let $\Phi_{x}$ be the map that takes an interface $\cI$ and eliminates its group of walls $\sF_{x}$ (if such a group of walls is nonempty), generating the new interface as per Lemma~\ref{lem:interface-reconstruction}. Now for ease of notation, let $\cI$ be the interface with the collection of groups of walls $(\sF_y)_y = (F_y)_y$ and let $\cI'$ be the one with the collection of groups of walls $(\sF'_y)_y$ where $\sF'_y=F_y$ for $y\neq x$ whereas $F'_{x}=\emptyset$, so that $\cI'= \Phi_{x}(\cI)$ and $\fm(\cI;\cI') = |\cI|-|\cI'|  = \fm(F_{x})$. By Theorem~\ref{thm:cluster-expansion}, we have 
\begin{align*}
\frac{\mu_n(\sF_x = F_{x} ,  (\sF_y)_{y\neq x} = (F_y)_{y\neq x})}{\mu_n(\sF_x = \emptyset, (\sF_y)_{y\neq x} = (F_y)_{y\neq x})} = \frac{\mu_n (\cI)}{\mu_n(\cI')} = \exp \Big( - \beta \fm(F_{x}) + \big(\sum_{f\in \cI} \g(f, \cI)- \sum_{f'\in \cI'}\g(f',\cI')\big)\Big)\,.
\end{align*}
We wish to bound the absolute value of the difference of the sums in the right-hand side. Denote the walls constituting $F_{x}$ by $W_{x_1}, W_{x_2}, \ldots, W_{x_l}$ for some $l$. Recall from Observation~\ref{obs:1-to-1-map-faces}, the 1-1 correspondence between  $\cI\setminus F_x$, and the faces of $\cI'$ that do not project in to $\cF(\rho(F_x))$ and encode it with the notation $f \mapsto \tilde f$. Then, we have 
\begin{align*}
\Big| \sum_{f\in \cI} \g(f, \cI) - \sum_{f'\in \cI'} \g(f',\cI')\Big| & \leq \sum_{f\in F_{x}} |\g(f,\cI)| + \sum_{f'\in \cI': \rho(f')\in \cF(\rho(F_{x}))} |\g(f',\cI')|+ \sum_{f\notin F_{x}} \big|\g(f,\cI)- \g(\tilde f, \cI')\big| \\ 
& \leq 3\bar K \fm(F_{x})+   \sum_{f\notin F_{x}}  \bar K \exp \big[- \bar c \br\big(f,\cI; \tilde f,\cI'\big)\big]\,.
\end{align*}
It is clear by construction, that for every $f, \tilde f$ the distance $\br(f,\cI; \tilde f, \cI')$ is attained by the distance to a wall face.  Since the distance between two faces is at least the distance between their projections, and projections of distinct walls are distinct, 
\begin{align*}
 \sum_{f\notin F_x} \bar K \exp\big[{- c \br (f,\cI;\tilde f, \cI')}\big] \leq   \sum_{f\notin F_x} \bar K \max_{u\in \rho(F_x)} \exp\big[-\bar c d(\rho(f),u)\big]\,.
\end{align*}
Then by the definition of groups of walls and closeness of walls, for a ceiling face $f$, $N_\rho(\rho(f)) =1$, and for a wall face $f\notin F_x$, $N_\rho (\rho(f)) \leq |\rho(f) - \rho(g)|^2$ for all $g\in F_x$. Thus this is at most
\begin{align*}
 \sum_{u\in \rho(F_{x})^c} \bar K N_\rho(u) \max_{u'\in \rho(F_x)} \exp[-\bar c d(u,u')] 
& \leq   \sum_{u'\in \rho(F_{x})} \sum_{u\in \rho(F_{x})^c}  \bar K (|u-u'|^2 +1)\exp [-\bar c |u-u'|]\,.
\end{align*}
which by integrability of exponential tails is easily seen to be at most $\bar C (|\cE(\rho(F_{x}))|+ |\cF(\rho(F_{x})|)$ for some  constant $\bar C$, which is in turn at most $\bar C \fm(F_{x})$ by~\eqref{eq:excess-area-wall-relations}.
\end{proof}

It is also important for us to control the number of interfaces that get mapped to the same interface under application of the map $\Phi_x$. We begin with the following geometric observation. 

\begin{observation}[e.g., Lemma 2 in~\cite{Dobrushin72a}]\label{obs:counting-connected}
The number of $*$-connected collections of $k$ faces in $\Z^d$ containing a specified face $f_\star$ is at most $s^k$ for some universal (only lattice-dependent) $s>0$.  
\end{observation}

The following follows from Observation~\ref{obs:counting-connected} and Definition~\ref{def:group-of-walls}; we do not include the proof here, but it can be found as part of the proof of the more complicated combinatorial estimate in Proposition~\ref{prop:base-shrinking-multiplicity}.

\begin{lem}[{\cite[Lemma 9]{Dobrushin72a}}]\label{lem:counting-groups-of-walls}
There exists $s$ such that for any $x\in \cL_0 \cap \Lambda$, the number of possible groups of walls $F_{x}$ with excess area $\fm(F_{x})=k$ is at most $s^{k}$. Likewise, there exists $s'$ such that the number of possible groups of walls $F$ containing $x$ in their interior, with $\fm(F) = k$ is at most $(s')^k$ .
\end{lem}

Together, Lemmas~\ref{lem:dobrushin-wall-ratio} and \ref{lem:counting-groups-of-walls} imply an exponential tail on groups of walls. In various papers~\cite{Dobrushin72b,BLP79b,GielisGrimmett} proving rigidity for such models, they were used to show that the height of the interface above a face $x=(x_1,x_2,0)\in \cL_0$, defined there as $\max  \{h: (x_1, x_2,h)\in\cI\}$, has an exponential tail. Since that definition of height above $x$ differs from the pillar-based perspective we take in the present paper, we modify the argument therein slightly to prove an exponential tail on the height of the pillar $\hgt(\cP_x)$.

\begin{proof}[\textbf{\emph{Proof of Theorem~\ref{thm:dobrushin-rigidity}}}] We begin with the first estimate. Let $\mathbf I_{F_x= \emptyset} = \mbox{Im}(\Phi_x)$ be the set of interfaces where the group of walls $\sF_x$ is empty. By Lemma~\ref{lem:counting-groups-of-walls} (and the definition of the map $\Phi_{x}$ defined above, relying on Lemma~\ref{lem:interface-reconstruction}), we see that for every $\cI'\in \mathbf I_{F_x = \emptyset}$, the pre-image $$\{\cI \in \Phi_x^{-1}(\cI'): \fm(\cI;\cI')=k\}$$ has cardinality at most $s^k$. Then by Lemma~\ref{lem:dobrushin-wall-ratio}, for every $r\geq 1$, 
\begin{align*}
\mu_n(\fm(\sF_x)\geq r)\leq \sum_{k \geq r} \,\,\sum_{\cI'\in \mathbf I_{F_x = \emptyset}} \sum_{\cI\in \Phi_x^{-1}(\cI'):  \fm(\cI; \Phi_x(\cI))=k} \mu_n(\cI) \leq \sum_{k \geq r} \sum_{\cI'\in \mathbf I_{F_x=\emptyset}} \mu_n(\cI') s^k \exp[- (\beta- C) k]
\end{align*} 
from which we obtain by summability of exponential tails, that for some $C'$, for $\beta>\beta_0$, 
\[\mu_n(\fm(\sF_x)\geq r) \leq C'\exp[-(\beta - C') r ] \mu_n(\mathbf I_{F_x = \emptyset}) \leq C' \exp[- (\beta  - C')r]\,.\]
We now turn to bounding the probability of $\hgt(\cP_x)\geq h$. 

In order for $\hgt(\cP_x)\geq h$, by Observation~\ref{obs:pillar-wall}, there must be one sequence of nested walls $\fW_x = (W_{x_s})_s$ all of which contain $x$ in their interior, with $\sum_s \fm(W_{x_s})= h_1$, along with a sequence of nested walls $\fW_y = ( W_{y_t})_t$ with $y_t \neq x_s$ for any $t,s$, containing some $y$ in the interior ceilings of $\fW_x$, such that $\sum_t \fm(W_{y_t}) \geq 4h - h_1$. 
In order to bound this, we can therefore write 
\begin{align*}
\mu_n( & \hgt(\cP_x)   \geq h) \\
&  \leq \mu_n(\fm(\fW_x) \geq 4 h) + \sum_{h_1\leq 4h} \mu_n(\fm(\fW_x) = h_1) \mu_n (\exists y: |y-x|\leq h_1,  \fm(\fW_y)\geq 4h- h_1\mid \fm(\fW_x) = h_1) \\ 
& \leq \mu_n(\fm(\fW_x)\geq 4h)+ \sum_{h_1 \leq 4h} \mu_n(\fm(\fW_x)\geq h_1) \sum_{y:|y-x|\leq h_1} \sup_{(W_{x_s})_s: \fm(\fW_x) = h_1} \mu_n(\fm(\fW_y)\geq 4h - h_1\mid (W_{x_s})_s)\,.
\end{align*}
To bound the probabilities expressed above, let us turn to groups of walls instead of walls, denoting by $\fF_x$ the group of walls of the nested sequence $\fW_x$ and $\fF_y$ corresponding to the nested sequence of walls of $\fW_y$. 
Following ~\cite{Dobrushin72b,BLP79b}, for a group of walls $F_z$, set 
\begin{align*}
\phi^x_z( F_z) = \fm(F_z) \one_{\{\fm(F_z)\geq |z-x|\}}
\end{align*}
and notice that $\fm(\fW_x)\leq \sum_{z} \phi^x_{{z}}(\sF_{{z}})$. Indeed, every wall $W_y\in \fW_x$ must nest $x$ and therefore must have excess area at least $d(y,x)$, from which it follows that its group of walls $F_y$ in turn has $\fm(F_y)\geq d(y,x)$.    
By the tail estimates of part (1) of Theorem~\ref{thm:dobrushin-rigidity}, we can bound  
\begin{align*}
\E_{\mu_n}[ e^{(\beta-2C) \phi^x_z (\sF_z)}\mid (\sF_{z'})_{z'\neq z}] \leq 1+ Ce^{-C |z-x|}
\end{align*} 
(where $\E_{\mu_n}$ is expectation with respect to $\mu_n$) which implies (by iteratively revealing $(F_{z})$ for all $z$) that
$$\E[e^{(\beta-2C) \sum_z \phi^x_{z} (\sF_{z})}] \leq \prod_z (1+Ce^{-C|{z}-x|}) \leq K<\infty\,.$$
By Markov's inequality, then, 
\begin{align}\label{eq:nested-seq-group-of-wall-exp-tail}
\mu_n(\fm(\fF_x)\geq r) \leq \mu_n\Big(e^{(\beta-2C) \sum \phi^x_{z}(F_{z})} \geq e^{(\beta-2C) r}\Big) \leq Ke^{- (\beta-2C) r}\,.
\end{align}
By the same reasoning, for any collection $\fF_x = (F_{x_s})_s$, since the exponential tail of Theorem~\ref{thm:dobrushin-rigidity} on $F_{y}$ for $y\neq x_s$ for any $s$ holds conditionally on $(F_{x_s})_s$, we see that similarly,
\begin{align*}
\sup_{(F_{x_s})_s: \fm(\fF_x) = h_1} \mu_n (\fm(\fF_y)\geq r \mid (F_{x_s})_s) \leq K e^{-(\beta - 2C)r}\,.
\end{align*}
Using that $\fm(\fW_x)\leq \fm(\fF_x)$ deterministically, we can plug in these estimates to see that   
\begin{align*}
\mu_n(\hgt(\cP_x)\geq h) & \leq  Ke^{- 4(\beta - 2C) h} +  K \sum_{h_1 \leq 4h} e^{ -\beta h_1 + 2C h_1} h_1^2 e^{ - (\beta - 2C) (4h - h_1)}  \\ 
& \leq K \sum_{h_1 \leq 4h} h_1^2 e^{-4\beta h + 8Ch} \leq 16K  h^4 e^{- 4\beta h + 8Ch}
\end{align*}
which for $h$ large, is at most $\exp[- 4(\beta -C')h]$ for some other universal constant $C'$. 
\end{proof}

By pairing the Dobrushin result with a straightforward forcing argument, we see the following. 
\begin{prop}\label{prop:rate-upper-lower-bounds}
There exist $\beta_{0}, C>0$ such that for every $\beta>\beta_{0}$, every $n,h$ large and $x\in \cL_0 \cap \Lambda_n$,
\begin{align*}
-4\beta - e^{-4\beta} \leq \frac 1h \log \mu_n(\hgt (\cP_x) \geq h)\leq & - 4\beta +C\,.
\end{align*}
\end{prop}

\begin{proof}The upper bound here was given by the second part of Theorem~\ref{thm:dobrushin-rigidity}. It remains to prove the lower bound; this proof will follow a more traditional coupling  argument.  First of all, with probability $1-\epsilon_\beta$ for some $\epsilon_\beta$ vanishing as $\beta \to\infty$, we have that $\hgt (\cP_x) \geq 0$ using e.g., the reflected version of Theorem~\ref{thm:dobrushin-rigidity}; (also notice that the event $\{\hgt (\cP_x) \geq 0\}$ is an increasing event).

Let $\cP_{\emptyset}$ be the set of all sites $\{x+(0,0,\ell-\frac 12 ):\ell=1,\ldots,h\}$. On the intersection of $\hgt(\cP_x)\geq 0$ with $\sigma(\cP_{\emptyset}) \equiv +1$, the interface has $\hgt(\cP_x)\geq h$, so that by the FKG inequality, it suffices to show the lower bound 
\[\frac 1h \log \mu_n(\sigma(\cP_{\emptyset}) \equiv +1)>  - 4\beta h - e^{-4\beta} h\,.\]

In order to show this estimate, we can expose the spins of $\cP_{\emptyset}$ from bottom up, starting with the one at $x+(0,0,\frac 12)$. With probability at least $\frac 12$, $\sigma_{x- (0,0, \frac 12)} = +1$, and by monotonicity, at worst, all other spins in $\sigma(\cP_{\emptyset}^c)$  are minus; by the domain Markov property and an elementary calculation, the probability of the spin at $x+(0,0,\frac 12)$ being plus is at least $\frac{\exp(- 4\beta)}{1+\exp(-4\beta) }=\frac12(1-\tanh(2\beta))$. Continuing on to the next site in $\cP_{\emptyset}$, conditional on the first one being plus, the same lower bound applies. As such, we can lower bound 
\begin{align*}
\mu_n\left(\sigma(\cP_{\emptyset}) =+1 \given \sigma_{x-(0,0,\frac 12)}=+1\right) \geq 
\frac{e^{-4\beta h}}{(1+e^{-4\beta})^h}> e^{-4\beta h - e^{-4\beta}h}\,,
\end{align*}
concluding the proof as long as $h$ is sufficiently large. 
\end{proof}

\section{Increments and the shape of tall pillars}\label{sec:increment-prelim}
In this section, we give a structural decomposition of a pillar, in the large deviation regime where it reaches a height of $h$. We prove that it is composed of a base---shown in \S\ref{sec:base} to have an exponential tail beyond height $O(\log h)$)---and a spine protruding from this base up to a height of $h$. This spine is further decomposed into a sequence of \emph{increments} between \emph{cut-points} where the spine is one-dimensional and vertical. In the remainder of this section, we give preliminary bounds regarding this decomposition, showing that the total number of increments is comparable to $h$, and has an exponential tail beyond that. In the following \S\ref{sec:increment-exp-tail}, we analyze individual increments, showing that they each have an exponential tail on their excess area.

\subsection{Increments of the pillar} We begin by defining the building blocks of the pillar where the 3D Ising interface undergoes an atypical fluctuation. 

\begin{definition}[Cut-points]\label{def:cut-point}
Call a height $h\in \Z+ \frac 12$ a \emph{cut-height} of the pillar $\cP_x$ if the intersection of the slab $\cL_h$ with $\sigma(\cP_x)$ consists of exactly one (midpoint of a) cell. We can call that single plus site $v\in\sigma(\cP_x)$ a \emph{cut-point} and identify it with its midpoint.
\end{definition}

\begin{definition}[Increments of the pillar]\label{def:increment}
For a pillar $\cP$, we define its increment collection $(\sX_i)_{i}$ as follows. Enumerate the cut-points of $\cP$ as $v_1, v_2, \ldots, v_{\sT}, v_{\sT+1}$ in order of increasing height, for some $\sT$. The $k$-th increment of the pillar $\cP$ is the set of all plus sites in $\sigma(\cP)$ centered at heights between $\hgt(v_{k})$ and $\hgt(v_{k+1})$, inclusively (this is also identified with the bounding sets of faces in $\cP \cap \R^2 \times (\lfloor \hgt(v_{k})\rfloor, \lceil\hgt(v_{k+1})\rceil)$, as before). 
Denote by $\mathbf I_{x,T}$ the set of interfaces which have $\sT \geq T$. 
\end{definition}

Since the pillar does not necessarily end at a cut-point, there may be a remainder of plus sites in the pillar above the height $\hgt(v_{\sT+1})$. We can call this the \emph{remainder} and denote it by $\sX_{>\sT}$; in fact for any $t\leq \sT$, we could denote the remainder beyond the $t$-th increment $\sX_{>t}$ which consists of $(\sX_{t+1},\ldots,\sX_{\sT}, \sX_{>\sT})$. 

\subsection{Comparability of height and number of increments}\label{sec:increment-height-equivalence} In this section, we show that the number of increments (as defined in the preceding subsection) serves as a good proxy for the height of a pillar.  We remark that the converse part of the next lemma would have readily followed had we had an exponential tail for $\fm(\cP_x)$ (when added to Proposition~\ref{prop:rate-upper-lower-bounds})---however, this is false, since $\cP_x$ may contain a wall with surface area $r^2$ and $\epsilon r^2$ nested thermal fluctuations (resulting in $\fm(\cP_x) \geq c r^2$) at a cost of only $\exp(-c r)$. 
\begin{lemma}\label{lem:increment-height-equivalence}
One always has $\mathbf I_{x,k} \subset \{\hgt(\cP_x)\geq k+1\}$ for every $k$. Conversely, there exist absolute constants $C,c>0$ such that, if $\beta>\beta_0$ and $T = \lfloor (1-C/\beta)h \rfloor$ then
\[
\mu_n(\mathbf I_{x,T} \mid \hgt(\cP_x)\geq h) \geq 1-O(e^{-ch})\,.
\]
\end{lemma}
\begin{proof}
The first assertion follows from the fact that, by definition, each increment increases the height of the pillar by at least $1$ and the extremal increment contributes two to the height. 

The lower bound is substantially more involved, and requires the use of a map that replaces a pillar of height $h$ and fewer than $\lfloor (1-C/\beta)h\rfloor $ increments, by  a straight column of height $h$ (consisting of $h-1$ total increments). This will combine the proof of Lemma~\ref{lem:dobrushin-wall-ratio} with some new ideas that will serve as a warm-up for the more sophisticated maps on pillars used in Section~\ref{sec:increment-exp-tail} and especially Section~\ref{sec:base}. 

Let $\Phi_{x,h}$ be the map that takes an interface $\mathcal I$ and generates an interface $\mathcal J$ as follows: 
\begin{enumerate}
    \item Let $(W_z)_{z\in\cL_0}$ be the standard wall representation of $\cI$ per Lemma~\ref{lem:interface-reconstruction}. 
    \item If $[x]:= \{x\}\cup \bigcup_{f\in \cL_0:f\sim x}\{f\}$ delete from the collection $(W_z)_{z\in\cL_0}$, $\fF_{[x]}:= \bigcup_{f\in [x]} \fF_f$ as well as $\fF_{\rho(v_1)}$. 
    \item If the interface $\cI'$ whose standard wall representation equals $\fF_{[x]}\cup\fF_{\rho(v_1)}$ has a cut-height below~$\hgt(v_1)$, let $h^\dagger$ be the highest such cut-height and let  $y^\dagger$ be the index of a wall of $\cP_x$ that attains height $h^\dagger$ and is not included in $\fF_{[x]}\cup\fF_{\rho(v_1)}$. Delete $\fF_{y^\dagger}$ from the standard wall representation obtained after step (2). (The existence of such a $y^\dagger$ is guaranteed by the definition of $v_1$ as the lowest cut-point.)
    \item Add to this standard wall representation the bounding vertical faces of a straight column of $h$ cells above $x$, centered at $x+(0,0,\ell -\frac 1 2): \ell=1,\ldots,h$.
    \item Let $\cJ$ be the interface with the standard wall representation resulting from step (4) as per Lemma~\ref{lem:interface-reconstruction}. 
\end{enumerate}
The map is well-defined because after step (2), there are no walls incident to $x$ nor its bounding edges and the addition of the standard wall in step (4) maintains the admissibility of the standard wall collection. The resulting $\cJ$ therefore has a pillar $\cP_x^\cJ$ consisting of exactly a column of $h-1$ increments, attaining height $h$. 

We next claim that if $\cI$ is such that $\{\hgt(\cP_x)\geq h\}$ but $\sT < (1-\delta)h$ (for $\delta$ to be chosen later), then 
\begin{align}\label{eq:excess-area-inequality}
    \fm(\cI;\cJ) = |\fF_{[x]}\cup \fF_{\rho(v_1)}\cup \fF_{y^\dagger}| - 4h \geq 2\delta h \,.
\end{align}
By Observation~\ref{obs:pillar-wall}, the entirety of the pillar above $\hgt(v_1)$ is deleted and therefore, for each height between $\hgt(v_1)$ and $h$ that is not a cut-height, there is an excess area contribution of 2 faces (due to 6 faces bounding two cells vs.\ 4 faces bounding one cell), totaling to  $2 (h- \frac 12 -\hgt(v_1) - \sT)$. For heights between $0$ and $\hgt(v_1)$, we claim that the interface having standard wall representation $\fF_{[x]}\cup \fF_{\rho(v_1)} \cup \fF_{y^\dagger}$ has no cut-heights, in which case it would follow that those heights together contribute at least $2(\hgt(v_1)-\frac 12)$ to the excess area and~\eqref{eq:excess-area-inequality} would follow. Indeed, if no $y^\dagger$ is chosen in step (3), then by definition there were no cut-heights of the interface corresponding to $\fF_{[x]}\cup \fF_{\rho(v_1)}$ below $\hgt(v_1)$, so suppose there was a highest such cut-height at $h^\dagger$ and a corresponding $y^\dagger$ was selected (noting that then $W_{y^\dagger}$ must be distinct from $W_{\rho(v_1)}$). Then since the walls $W_{y^\dagger}$ and $W_{\rho(v_1)}$ must each attain the height $h^\dagger$, there can be no cut-heights at or below $h^\dagger$ in the interface corresponding to $\mathfrak W_{\rho(v_1)} \cup \mathfrak W_{y^\dagger}$ and therefore there also cannot be any at or below $h^\dagger$ in the interface corresponding to $\fF_{[x]}\cup \fF_{\rho(v_1)}\cup \fF_{y^\dagger}$. 

Having constructed the map $\Phi_{x,h}$, the proof now proceeds in two parts: (1) we show that the relative weight $\mu_n(\cI)/\mu_n(\cJ)$ is exponentially decaying in $(\beta - C)\fm(\cI;\cJ)$ and (2) we show that the multiplicity of the map $\Phi_{x,h}$ is at most exponentially growing in $\fm(\cI;\cJ)$.

To begin with the first, it suffices for us to show the bound 
\begin{align}\label{eq:part-function-estimate}
    \Big|\sum_{f\in \cI} \g(f,\cI) - \sum_{f'\in \cJ} \g(f',\cJ)\Big| \leq C[\fm(\cI;\cJ)+ h]\leq C[1+\delta^{-1}]h\,.
\end{align}
To establish such a bound, we decompose $\cI$ and $\cJ$ into different subsets of faces as follows: 
\begin{itemize}
    \item Let $\mathbf E$ be the set of all faces in the groups of walls $\fF_{[x]}\cup \fF_{\rho(v_1)}\cup\fF_{y^\dagger}$.
    \item Let $\mathbf F$ be the set of $f\in \cJ$ such that $\rho(f)\in \rho(\cF(\mathbf E))$, added in place of a removed horizontal wall face in $\rho(\mathbf E)$ to ``fill in" the interface. 
    \item Let $\mathbf G$ be the bounding vertical faces of a column of $h$ cells above $x$, added in step (4) of $\Phi_{x,h}$.  
\end{itemize}
Under this decomposition, there is a 1-1 correspondence between $\cI \setminus \mathbf E$ and $\cJ \setminus (\mathbf F \cup \mathbf G)$ via vertical shifts as determined by Observation~\ref{obs:1-to-1-map-faces}; encode this into $f\mapsto \tilde f$. Then, 
\begin{align*}
    \Big|\sum_{f\in \cI} \g(f,\cI) - \sum_{f'\in \cJ} \g(f',\cJ)\Big| & \leq \sum_{f\in \mathbf E} |\g(f,\cI)|+ \sum_{f'\in \mathbf F} |\g(f',\cJ)| + \sum_{f'\in \mathbf G} |\g(f',\cJ)| + \sum_{f\in \cI \setminus \mathbf E} \Big|\g(f,\cI) -\g(\tilde f, \cJ)\Big|\,.
\end{align*}
The first quantity is bounded by $\bar K |\mathbf E|= \bar K |\fF_{[x]}\cup \fF_{\rho(v_1)} \cup \fF_{y^\dagger}|\leq 2\bar K (1+2\delta^{-1}) \fm(\cI;\cJ)$ by~\eqref{eq:excess-area-wall-relations} and \eqref{eq:excess-area-inequality}. Similarly, the second term is at most $\bar K (1+2\delta^{-1}) \fm(\cI;\cJ)$ and the third term is at most $4\bar K h$ which is in turn at most $2\bar K \delta^{-1}\fm(\cI;\cJ)$. The last term is bounded similarly to the proof of Lemma~\ref{lem:dobrushin-wall-ratio}. By construction, for every $f$, the radius $\br(f,\cI;\tilde f, \cJ)$ is attained by the distance to a wall face, and as before, moving to the distance between projections,
\begin{align*}
    \sum_{f\in \cI \setminus \mathbf E} \bar K \exp[-\bar c \br(f,\cI;\tilde f,\cJ)]\leq \sum_{f\in \cI \setminus \mathbf E} \bar K \max_{u\in \rho(\mathbf E \cup \mathbf G)}\exp[-\bar c d(\rho(f),u)]
\end{align*}
and using the definition of closeness of walls, there is a $\bar C>0$ such that this is at most 
\begin{align*}
    \sum_{u'\in \rho(\mathbf E\cup \mathbf G)} \sum_{u\in \rho(\mathbf E)^c} \bar K (|u-u'|^2 +1) \exp[-\bar c |u-u'|] \leq \bar C \bar K |\mathbf E\cup \mathbf G|\,,
\end{align*}
 yielding~\eqref{eq:part-function-estimate} by applying ~\eqref{eq:excess-area-wall-relations} and \eqref{eq:excess-area-inequality} as above. 

We next wish to bound the multiplicity of the map, i.e., we wish to show that for every $M$ and every $\cJ$ in the image of $\Phi_{x,h}$, the size of the set $\{\cI\in \Phi_{x,h}^{-1}(\cJ): \fm(\cI;\cJ)= M\}$. Every such $\cI$ can be identified with the choices of the standard walls in $\fF_{[x]}\cup \fF_{\rho(v_1)}\cup \fF_{y^\dagger}$, so that it suffices to bound the number of possible such choices leading to excess area $M$. By iteratively applying Lemma~\ref{lem:counting-groups-of-walls} (first choosing how many walls constitute $\fF_{[x]}$, and then the size of each such that the total size is at most $2M$), then enumerating over the at most $M^2$ choices of where to place $\rho(v_1)$ and $\rho(y^\dagger)$ (using Observation~\ref{obs:pillar-wall} these must be interior to some wall of $\mathfrak W_{x}$), and then enumerating over the choices for $\fF_{\rho(v_1)}$ and $\fF_{y^\dagger}$, we see that for some universal $\bar s>0$, 
\[
\{\cI\in \Phi_{x,h}^{-1}(\cJ): \fm(\cI;\cJ)= M\}\leq {\bar s}^M\,.
\]
(See e.g., the proof of Proposition~\ref{prop:base-shrinking-multiplicity} for more details on a similar enumeration process.) 

We now combine the two parts above to conclude the desired. Expressing $\mu_n(\mathbf I_{x,(1-\delta)h}^c, \hgt(\cP_x)\geq h)$ as 
\begin{align*}
    \sum_{\cI\in \mathbf I_{x,(1-\delta)h}^c, \hgt(\cP_x)\geq h} \mu_n(\cI)  & \leq \sum_{M\geq 2\delta h} \,\sum_{\cJ: \hgt(\cP_x^\cJ)\geq h} \,\sum_{\cI\in \Phi_{x,h}^{-1}(\cJ): \fm(\cI;\cJ) =M} \mu_n(\cJ) e^{ - \beta M+C(1+\delta^{-1})M} \\ 
    & \leq \sum_{M\geq 2\delta h} \bar s^M e^{ - \beta M+C(1+\delta^{-1})M+ M \log \bar s } \mu_n(\hgt(\cP_x)\geq h)\,,
\end{align*}
from which the lemma follows by dividing through by $\mu_n(\hgt(\cP_x)\geq h)$ and taking $\beta$ large and $\delta = C'/\beta$ for some sufficiently large $C'$.
\end{proof}

\subsection{Spine and base of pillars}
The fundamental difficulty in understanding the structure of pillars conditionally on reaching a height $h$, or on having $T$ increments, is the interactions of the pillar with nearby oscillations of the interface, particularly at low heights, where these are plentiful. Towards this, it will be important to us to isolate the portion of the pillar which interacts most strongly with other pillars near it---called the \emph{base}---and the rest of the pillar, which climbs above all oscillations in some ball of radius $O(h\vee T)$ about $x$, called the \emph{spine}. The ball of proximity grows with the number of increments $T$ we are conditioning on having, as the pillar's $(xy)$-coordinates diffuse as $T$ grows; this creates the complication that the definitions of the spine and base must be $T$ dependent.

For a set $A\subset \cF(\mathbb Z^3)$, let $\cC_{r}(A)$ denote the points in $\R^3$ whose projection is distance at most $r$ from $\rho(A)$:  
\[\cC_{r}(A) = \{y:  \min_{f\in A} |\rho(y)- \rho(f)| \leq r\}\,.
\]Let $R_0$ be some sufficiently large constant, e.g., to be chosen in Lemma~\ref{lem:tame} to be $100$. For two faces $x,y\in \cL_0$, let $\langle \langle x,y \rangle \rangle$ be a minimal connected set of faces of $\cL_0$ connecting $x$ to $y$. For ease of notation, for a cut-point $v$, we'll define 
$$\cC_{v, x, T}: = \cC_{R_0 T}\big(\langle \langle \rho (v),x\rangle \rangle\big)\,.$$

\begin{definition}[Spine]\label{def:spine}
Consider an interface $\cI$ with pillar $\cP_x = \cP_x(\cI)$. For each $T$, let $\Tsp=\Tsp(T)$ be the minimal index $i\geq 1$ such that the cut-point $v_i$ of $\cP_x$ lies above the largest height attained by walls in $\cI \setminus \cP_x$ indexed by faces in $\cC_{v_{i},x,T}$ (in every possible ordering of $\cF(\cL_0)$). We then call $v_{\Tsp}^T$ the $T$-\emph{source-point}. When $T$ is understood from the context (e.g., for $\cI\in \mathbf I_{x,T}$) we drop it from the notation and write $v_\Tsp$. With respect to that $T$ the \emph{spine} $\cS_x$ will then be the $*$-connected component of $\cP_x$ consisting of all sites/faces above $\hgt(v_\Tsp)- \frac 12$, i.e., consisting of the increments $\sX_{\Tsp},\ldots,\sX_{\sT}, \sX_{>\sT}$. 
\end{definition}

\begin{definition}[Base]\label{def:base}
 For an interface $\cI$ in ${\mathbf I}_{x,T}$ with pillar $\cP_x$, let the \emph{base} $\sB_x$ of the pillar be given by the entirety of the pillar below the height $\hgt(v_\Tsp)+ \frac 12$. In general, a base can be identified with the set-difference $\cP_x\setminus \cS_x$, along with the four bounding faces of the $T$-source-point $v_\Tsp$.  
\end{definition}

\noindent (For reference, in Figures~\ref{fig:pillar_incr} and~\ref{fig:base-map}, the spine is shaded in blue, and the base is shaded in pink and orange.) We will show in Section~\ref{sec:base} that, with high probability, most of the increments of a pillar belong to the spine.

\subsection{Properties of increments}
Let $\fX$ be the set of all possible rooted increments, where an increment is identified with a $*$-connected subset of plus sites in the upper half-space of $\Z^3$ consisting of a cut-point plus site at $(\frac 12,\frac 12,\frac 12)$, as well as a cut-point plus site at its largest height, and such that no height in between these is a cut-height. As usual, we also identify such increments with a $*$-connected collection of faces that bound its plus sites; however, in this face set $\cF(X)$ we exclude the bottom-most and top-most delimiting faces (since, viewing this increment as a subset of an interface, those faces would not be present in the interface). 

Let $\fX_{\textsc{rem}}$ be the set of rooted remainders i.e., $*$-connected subsets of plus sites in the upper half-space of $\Z^3$ where we only impose that they have exactly one cell in the slab $\cL_{\frac{1}{2}}$ at $(\frac 12,\frac 12,\frac 12)$. Correspondingly, its face set is the set of faces that bound it, now excluding only the bottom-most delimiting face (at $(\frac 12, \frac 12, 0)$). 
\begin{definition}
For each increment $\sX_i$ in the pillar $\cP_x$, recall that its bottom-most and top-most cells are $v_i$ and $v_{i+1}$ respectively. 
We define the height of an increment $X_i\in \fX$ by $\hgt(X_i) = \hgt(v_{i+1}) - \hgt(v_i)$. 
\end{definition} 

\begin{figure}
\vspace{-0.3cm}
\centering
  \begin{tikzpicture}
    \node (fig1) at (-4.5,0) {
	\includegraphics[width=.30\textwidth]{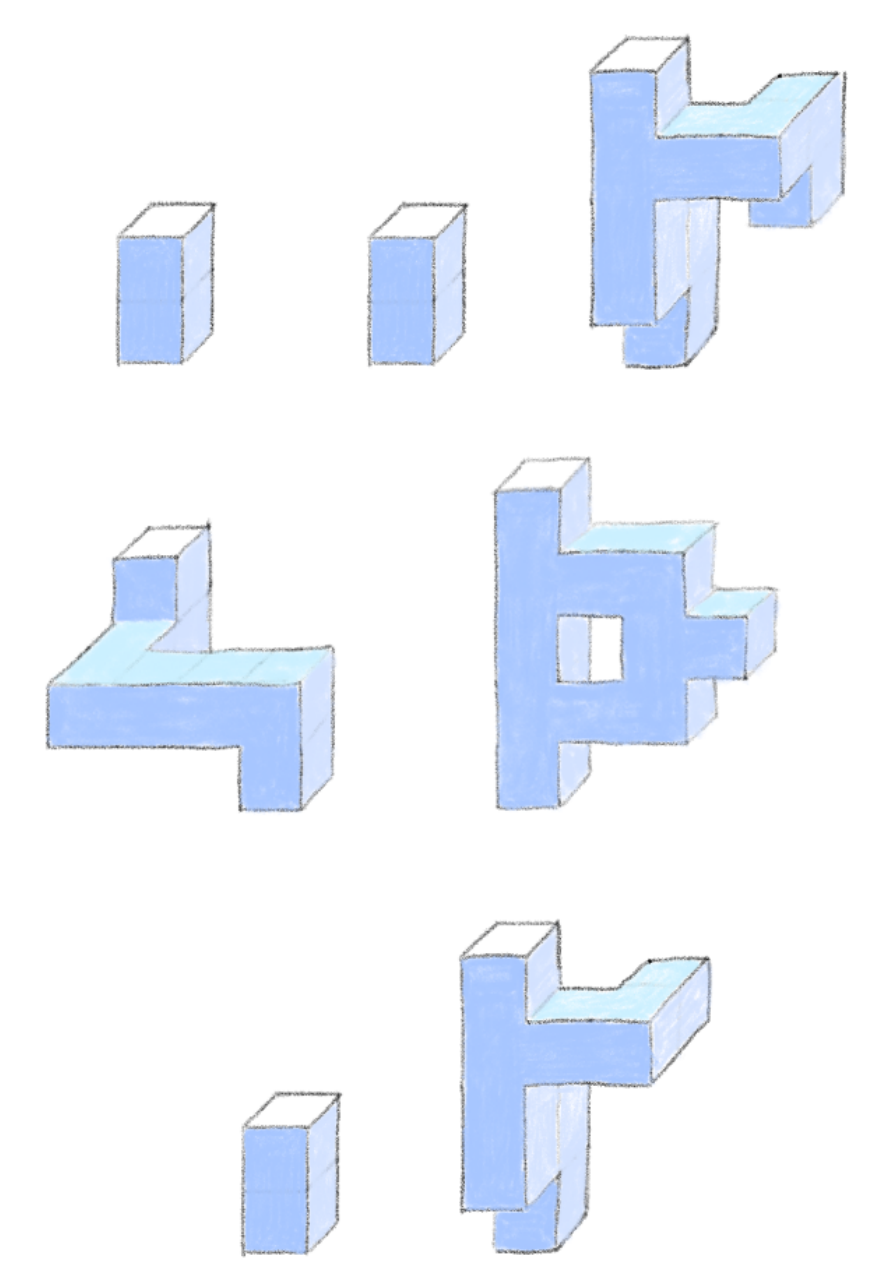}
	};
    \node (fig2) at (4.5,0) {
    \includegraphics[width=.220\textwidth]{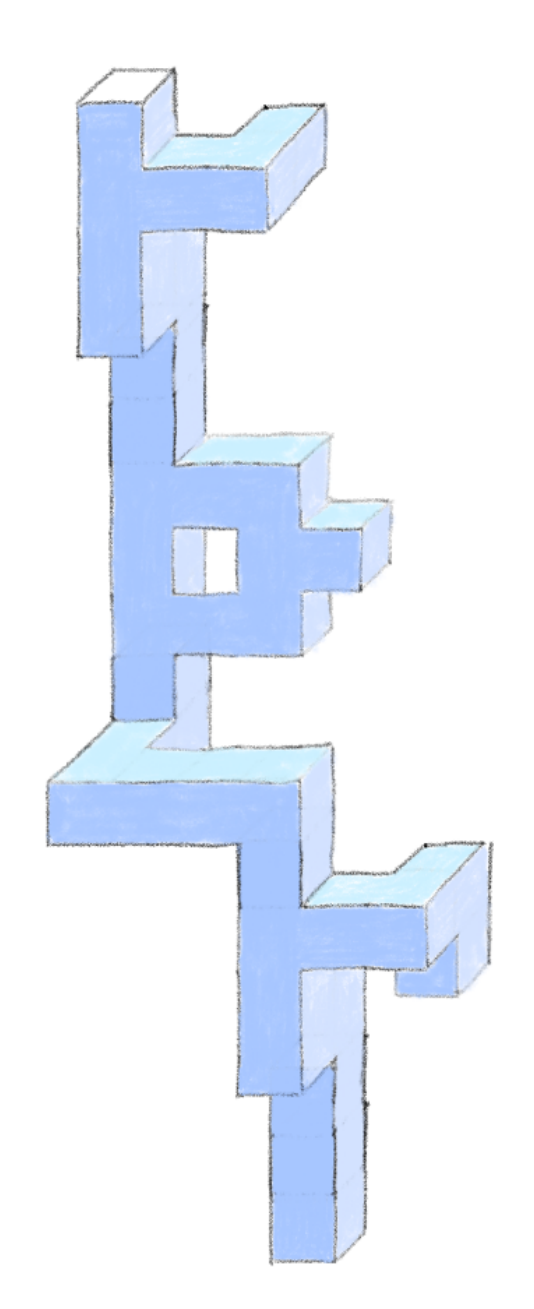}
    };
   \node at (-6.1,1.35) {$X_1$};
   \node at (-4.6,1.35) {$X_2$};
   \node at (-3.2, 1.35) {$X_3$};
   \node at (-5.35, -1.15) {$X_4$};
   \node at (-3.85, -1.15) {$X_5$};
   \node at (-5.35, -3.7) {$X_6$};
   \node at (-3.85, -3.7) {$X_7$};

    \draw [decorate,decoration={brace,amplitude=5pt}]
(4.45,-4.1) -- (4.45,-3.3) node [black,midway,xshift=-0.4cm] 
{\footnotesize $X_1$};
    \draw [decorate,decoration={brace,amplitude=5pt}]
(5.2,-2.7) -- (5.2,-3.5) node [black,midway,xshift=.4cm] 
{\footnotesize $X_2$};

    \draw [decorate,decoration={brace,amplitude=10pt}]
(4.3,-3.25) -- (4.3,-1.35) node [black,midway,xshift=-0.6cm] 
{\footnotesize $X_3$};
    \draw [decorate,decoration={brace,amplitude=5pt}]
(5,-.05) -- (5,-1.45) node [black,midway,xshift=.4cm] 
{\footnotesize $X_4$};

    \draw [decorate,decoration={brace,amplitude=10pt}]
(3.35,-.45) -- (3.35,1.65) node [black,midway,xshift=-0.6cm] 
{\footnotesize $X_{5}$};

    \draw [decorate,decoration={brace,amplitude=5pt}]
(4.15,2.3) -- (4.15,1.5) node [black,midway,xshift=.4cm] 
{\footnotesize $X_6$};

    \draw [decorate,decoration={brace,amplitude=10pt}]
(3.2,1.7) -- (3.2,3.7) node [black,midway,xshift=-0.6cm] 
{\footnotesize $X_{7}$};

\draw[->, thick] (-1.5,0)--(1.5,0);

  \end{tikzpicture}
\vspace{-0.5cm}
 \caption{A stretch of the spine (right) decomposes into a collection of rooted increments which are members of $\fX$ (left). The increments $X_1, X_2$ and $X_6$ are trivial increments $X_\trivincr$.}
  \label{fig:spine-construction}
\vspace{-0.2cm}
\end{figure}
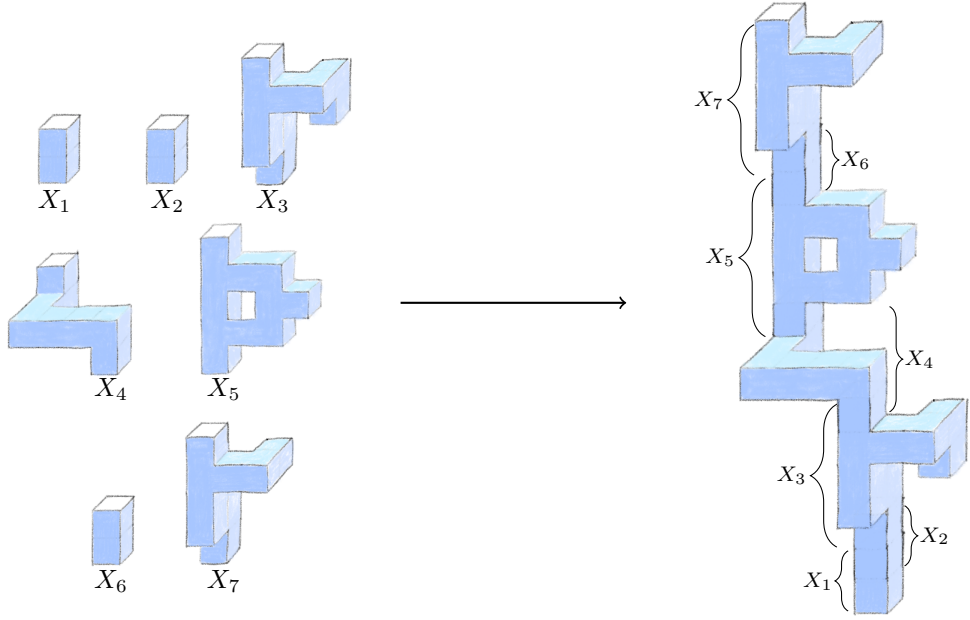

\begin{lem}\label{lem:spine-increment-bijection}
There is a 1-1 correspondence between the triplets of $v_\Tsp$, the collection of $T$ rooted increments $(X_i)_{i\leq T}$ for $X_i \in \fX$, and remainder increment $X_{>T}\in \fX_{\textsc{rem}}$, and the set of spines of at least $T$ increments. 
\end{lem}
\begin{proof}
Identifying the increment sequence given a spine was described by the definition of increments. Obtaining from this increment sequence, the rooted increments, consists only of shifting each by the vector $- v_{i}+(\frac 12, \frac 12, \frac 12)$; the rooted remainder is similarly recovered. 

Given a sequence of $T$ rooted increments, a source point $v_\Tsp$ and a remainder $X_{>T}$, we can reconstruct the cell-set of the spine by taking the union over $i$ of the 
translates of $X_i$ by the vectors $-(\frac 12, \frac 12, \frac 12) +  v_j$ where $v_j$ are defined inductively as increments are stacked. (Naturally, the rooted remainder $X_{>T}$ is shifted by $-(\frac 12, \frac 12, \frac 12) + v_{T+1}$.) As a consequence, we can identify the set of all rooted spines of at least $T$ increments with the set $\fX^{T} \times \fX_{\textsc{rem}}$. See Figure~\ref{fig:spine-construction} for a visualization of this scheme.
\end{proof}

We will always use the notation $X_\trivincr$ to denote the trivial increment that consists of exactly two plus cells, one on top of the other (the rooted one has the plus sites centered at $(\frac 12 ,\frac 12,\frac 12)$ and $(\frac 12,\frac 12,\frac 32)$). The trivial remainder increment consists of exactly one plus cell at $(\frac 12,\frac 12,\frac 12)$. 

\begin{definition}\label{def:increment-excess-area}
The excess area of an increment $X\in \fX$ is given by its excess area as compared to the trivial increment $X_\trivincr$ so that $$\fm(X):= |\cF(X)| - |\cF(X_{\trivincr})| = |\cF(X)|- 8\,.$$ The excess area of a remainder $X_{>T}$ is measured with respect to the trivial remainder increment, so that $\fm(X_{>T}):= |\cF(X_{>T})|-5$ (recall that the remainder increment includes its upper delimiting face(s)). 
\end{definition}
\begin{remark}\label{rem:increment-excess-area}Notice that $\fm(\sX_i)\geq \sqrt 2 |\rho(v_{i+1})- \rho(v_{i})|$ (the nontrivial increment $X$ of height $2$ consisting of two $*$-adjacent cubes has $|\cF(X)|=10$ and $\fm(X)=2$) and for every $X\neq X_\trivincr$, 
$$\fm(X) \geq \frac{1}{5}|\cF(X)| \qquad \mbox{and}\qquad |\cF(X)|\geq 6(\hgt(X)-1)+8\,.$$
\end{remark}


\begin{definition}\label{spine-excess-area}

For a spine $\cS_x$ of $\sT-\Tsp$ increments $(\sX_i)_{\Tsp \leq i\leq \sT}$ and remainder $\sX_{>\sT}$, the excess area of the spine $\fm(\cS_x)$ with respect to a trivial increment sequence of height $t-\Tsp$, 
 $$\fm_t(\cS_x) = \fm(\cI; \cI_{\trivincr, t})= \fm(\sX_{>t}) + \sum_{\Tsp \leq i\leq t} \fm(\sX_i)\,,$$
 and if $t>\sT$, we set $\fm_t(\cS_x) = \fm(\cS_x)$. The excess area of the spine (dropping the index $t$) is $\fm(\cS_x):= \fm_\sT(\cS_x)$. 

 The height of a spine $\cS_x$ is $\hgt(\cS_x) =\hgt(\cP_x) - \frac 12 - \hgt(v_{\Tsp}) = \hgt(\sX_{>\sT}) + \sum_{\Tsp \leq i \leq \sT} \hgt(\sX_i)$.
\end{definition}
\begin{definition}\label{def:truncated-interface}
For any interface $\cI \in \mathbf I_{x,T}$, let $\cI_{\textsc {tr}}$ be its \emph{truncation}, with cell-set $\sigma(\cI_{\textsc {tr}}):=(\sigma(\cI)\setminus \sigma(\cS_x))\cup \{v_\Tsp\}$ where we have removed all plus sites of the spine besides $v_\Tsp$ from $\cI$, and face-set consisting of the faces in $\cI$ that bound cells in $\sigma(\cI_{\textsc {tr}})$. A truncation is $T$-admissible if its pillar $\cP_x(\cI_{\textsc {tr}})$ has a $T$-source point $v_\Tsp$ and nothing above $\hgt(v_\Tsp)+\frac 12$. (Recall that the property of being a $T$-source point is independent of the increment sequence of the spine above it).
\end{definition}

\subsection{Exponential tail on the number of increments}\label{sec:increment-tail}
Here, we show that a spine $\cS_x$ of an interface in $\mathbf I_{x,T}$ has an exponential tail on the surface area (as well as excess area) of its remainder $\sX_{>T}$. This implies an exponential tail on the number of increments beyond $T$ in a spine conditioned on having at least $T$ increments. Since we are only looking at a portion of the increment above a cut-point, it is droplet-like, and the proof does not involve any of the more delicate issues we will encounter in later sections.

\begin{lem}\label{lem:tip-highest-increment}
There exists $C>0$ such that for every $\beta>\beta_0$, every $T$ and every $r>0$, 
\begin{align*}
\mu_n\big( \fm(\sX_{>T}) \geq r \given \mathbf I_{x,T} \big)\leq \exp \big[-(\beta- C)r \big]\,.
\end{align*}
In particular, $\mu_n(\mathbf I_{x,T+k} \mid \, \mathbf I_{x,T})\leq \exp [-4k(\beta - C)]$. Moreover, these estimates also hold conditionally on any $T$-admissible truncation $\cI_{\textsc{tr}}$ and  spine increment sequence  $(\sX_i)_{\Tsp\leq i\leq T}= (X_i)_{\Tsp\leq i\leq T}$. 
\end{lem}
\begin{proof}
Let $\Phi_T: \mathbf I_{x,T} \mapsto \mathbf I_{x,T}$ be the map that, for each $\cI \in \mathbf I_{x,T}$, generates the interface $\Phi_T (\cI)$ by replacing $\sX_{>T}$ with the trivial remainder, and agrees with $\cI$ otherwise.
It should be clear that $\Phi_T(\cI) \in \mathbf I_{x,T}\setminus \mathbf I_{x,T+1}$; moreover, the pillar of $\Phi_T(\cI)$ will have height equal to $\hgt(v_{T+1})+\frac 12$. By Theorem~\ref{thm:cluster-expansion}, for any $\cI \in \mathbf I_{x,T}$, 
\begin{align*}
\frac{\mu_n(\cI)}{\mu_n(\Phi_T(\cI))}  = \exp \Big( -\beta \fm(\cI; \Phi_T (\cI)) + \sum_{f\in \cI} \g(f,\cI) - \sum_{f'\in \Phi_T (\cI)} \g(f',\Phi_T(\cI))\Big)\,.
\end{align*}
By definition of excess areas of remainders, $\fm(\cI; \Phi_{T}(\cI)) = \fm(\sX_{>T})$. Suppose without loss of generality that $\cI$ has remainder $X_{>T}$ such that $\fm(X_{>T})\geq 1$ as the lemma is trivially satisfied for $r=0$.  
For ease of notation, let $\cI' = \Phi_T (\cI)$ and consider the difference of the sums in the exponential. By~\eqref{eq:g-uniform-bound}--\eqref{eq:g-exponential-decay},
\begin{align*}
\Big| \sum_{f\in \cI} \g(f,\cI) - \sum_{f'\in \cI'} \g(f',\cI')\Big| & \leq \sum_{f\in \cI \cap \cI'} |\g(f,\cI)-\g(f,\cI')| + \sum_{f\in \cI\setminus \cI'} |\g(f,\cI)|+ \sum_{f\in \cI'\setminus \cI} |\g(f,\cI')|  \\
& \leq \sum_{f\in \cI \cap \cI'} \sum_{f'\in \cI \oplus \cI'} \bar K \exp[-\bar c d(f,f')] + \sum_{f\in \cI \oplus \cI'} \bar K \leq (\bar C+ \bar K)|\cI \oplus \cI'|\,,
\end{align*}
for some constant $\bar C$. 
But by construction, we have that  $|\cI \oplus \cI'| = \fm(X_{>T})+2$, where the additive 2 comes from the upper-bounding face of the remainder, which is shifted between $\cI$ and $\cI'$. Consequently, we have that for some universal $C$ independent of $\beta$, for every $\cI \in \mathbf I_{x,T}$,
\begin{align*}
\frac{\mu_n(\cI)}{\mu_n(\cI')} \leq \exp\big[-(\beta -C) \fm(\cI;\cI')\big] = \exp\big[- (\beta - C)\fm(X_{>T})]\,.
\end{align*}
At the same time, we claim that for every $\cI'\in \mathbf I_{x,T}\setminus \mathbf I_{x,T+1}$, there are at most $s^{k}$ elements in the pre-image $\Phi_T^{-1}(\cI')$ with excess area $k$, for some universal $s>0$. 
Since $\fm(\cI; \cI')= |\cI \oplus \cI'| - 2$, to every $\cI \in \Phi_T^{-1}(\cI')$ of excess area $\fm(\cI;\cI')=k$, we can uniquely identify the connected set of faces constituting $\cI\oplus \cI'$ of cardinality $k+2$ containing the upper bounding face of $v_{T+1}$. By Observation~\ref{obs:counting-connected}, the number of such sets is at most $s^{k+2}$. 
We now can expand the probability $\mu_n (\fm(X_{>T}) \geq r, \mathbf I_{x,T})$ 
\begin{align*}
\sum_{\cI\in \mathbf I_{x,T}: \fm(X_{>T})\geq r} \mu_n(\cI) = \sum_{k\geq r} \,\,\sum_{\cI\in \mathbf I_{x,T}:\fm(X_{>T}) = k} \mu_n (\cI) \leq \sum_{k\geq r} \,\,\sum_{\cI' \in \Phi_T(\mathbf I_{x,T})}  e^{-(\beta - C)k} s^{k+2} \mu_n(\cI')\,,
\end{align*}
At this point, since $\Phi_T(\mathbf I_{x,T}) \subset \mathbf I_{x,T}$, we see that for $\beta>\beta_0$, this is at most 
\begin{align*}
\sum_{\cI' \in \Phi_T (\mathbf I_{x,T})} C'e^{-(\beta - C')r} \mu_n(\cI') \leq C' e^{-(\beta - C')r} \mu_n(\mathbf I_{x,T})\,,
\end{align*}
for some other constant $C'>0$ independent of $\beta$; dividing both sides by $\mu_n(\mathbf I_{x,T})$ implies the first inequality. The second inequality follows because $\cI \in \mathbf I_{x,T+1}$ implies that $\fm(\cI;\Phi_T(\cI))= \fm(X_{>T}) \geq 4$. 

To see the analogous conditional estimates, fix a $T$-admissible truncation $\cI_{\textsc{tr}}$ and first $T$ increments of the spine $(X_i)_{i\leq T}$, and let $\hat {\mathbf I}_{x,T}$ be the set of interfaces in $\mathbf I_{x,T}$ having $\cI_{\textsc{tr}}$ and $(\mathscr X_i)_{i\leq T} = (X_i)_{i\leq T}$. Repeating the argument above, we see that $\mu_n(\fm(X_{>T})\geq r, \hat{\mathbf I}_{x,T})$ can be expressed as  
\begin{align*}
\sum_{\cI\in \mathbf I_{x,T}: \fm(X_{>T})\geq r} \mu_n(\cI) = \sum_{k\geq r} \,\,\sum_{\cI\in \mathbf I_{x,T}:\fm(X_{>T}) = k} \mu_n (\cI) \leq \sum_{k\geq r} \,\,\sum_{\cI' \in \Phi_T(\mathbf I_{x,T})}  e^{-(\beta - C)k} s^{k+2} \mu_n(\cI')\,.
\end{align*}
Observing that $\Phi_T(\hat{\mathbf I}_{x,T})\subset \hat{\mathbf I}_{x,T}$, we see that the right-hand side is at most $C'e^{-(\beta - C')r}\mu_n(\hat{\mathbf I}_{x,T})$ and dividing through by $\mu_n(\hat{\mathbf I}_{x,T})$ yields the desired conditional estimate. 
\end{proof}

\subsection{Increment sequences are typically tame}

Before turning to the tail estimates on the increments themselves, we  prove an easy preliminary estimate, showing that under the event $\mathbf I_{x,T}$, the probability that $\cS_x$ is not contained in a ball of radius of order $T$ centered at $v_\Tsp$ is exponentially small in $T$.

Let $r_0$ be a large constant, say $20$, and let $R_0:= 5r_0$; we will reserve these letters for these specific constants. We now define a notion of \emph{tameness} for spines, and subsequently in Lemma~\ref{lem:tame} demonstrate that with high probability, a spine is tame. 

\begin{definition}\label{def:tame}
Fix $T$; for every $t$,  a spine in $\cS_x \in \fX^{t} \times \fX_{\textsc{rem}}$ is \emph{tame} with respect to $\mathbf I_{x,T}$ if  
\begin{align*}
\fm(\cS_x) \leq r_0 T\,, \qquad\mbox{and}\qquad \hgt(\cS_x) \leq r_0 T\,.
\end{align*}
Call an interface $\cI\in \mathbf I_{x,T}$ \emph{tame} if its spine $\cS_x$ is tame, and denote by $\bar{\mathbf I}_{x,T}$ the set of tame interfaces in~$\mathbf I_{x,T}$. 
\end{definition}

Before turning to the proof that spines are typically tame, we pause to comment on the usefulness of restricting to tame spines going forward. 

\begin{remark}\label{rem:tame}
First of all, notice that the tameness of a spine is only a property of the increment sequence constituting the spine $\fX^{\sT-\Tsp} \times \fX_{\textsc{rem}}$ and does not depend on the truncation below it. Moreover, note that any spine $\cS_x$ with source point $v_\Tsp$ that is tame is such that the spine $\cS_x$ is contained entirely in a cylinder of radius $r_0 T$ and height $r_0 T$ above (and centered at) $v_\Tsp$.  
This is in turn confined to the cylinder $\cC_{2r_0 T}(v_\Tsp) \subset \cC_{v_\Tsp, x, T}$, so that for any $x$ such that $d(x, \partial \Lambda_n)\geq  100 T$, adjoining to any $T$-admissible interface $\cI_{\textsc{tr}}$ any tame spine (identified with an element of $\fX^{\sT- \Tsp} \times \fX_{\textsc{rem}}$),  yields a valid interface in $\bar{\mathbf I}_{x,T}$. 
\end{remark}

Additionally, notice that, by construction, if $\cI\in \bar{\mathbf I}_{x,T}$, for any face $f\in \cS_x$, the distance $d(f,\cI_{\textsc {tr}})$ is attained by a face in $\cI_{\textsc {tr}}\cap \cC_{R_0 T}(v_\Tsp) \subset \cC_{v_\Tsp, x, T}$, as the distance to $\cI_{\textsc {tr}}\setminus \cC_{R_0 T}(v_\Tsp)$ is at least $3r_0 T$ while the distance to $v_\Tsp$ is at most $2r_0 T$. 

We now prove that spines of interfaces in $\mathbf I_{x,T}$ are exponentially unlikely in $T$ to not be tame.

\begin{lem}\label{lem:tame}
There exists $C, \beta_0>0$ such that for every $\beta>\beta_0$,  such that for every $T$, every $T$-admissible truncated interface $\cI_{\textsc {tr}}$, we have that for every $r\geq 8 T$
\begin{align*}
\mu_n(\fm_T(\cS_x) \geq r \mid \cI_{\textsc {tr}}, \mathbf I_{x,T}) \leq \exp[-(\beta-C) r]\,.
\end{align*}
In particular, $\mu_n(\bar{\mathbf I}_{x,T}\mid \cI_{\textsc {tr}}, \mathbf I_{x,T}) \geq 1- O(e^{- (\beta-C) r_0 T})$, and hence also  $\mu_n(\bar{\mathbf I}_{x,T}\mid \mathbf I_{x,T}) \geq 1- O(e^{- (\beta-C) r_0 T})$.
\end{lem}

\begin{proof}
The second statement follows from the fact that $\fm(\cS_x) \leq \fm_T(\cS_x)$ and $\hgt(\cS_x) \leq T+ \frac 14 \fm_T(\cS_x)$ and $T \leq \frac 34 r $ when $r\geq 8T$. 
Therefore, we focus on proving the bound on $\{\fm_T(\cS_x)\geq r\}$. Let $\cI \in \mathbf I_{x,T}$ be such that it has $T$-admissible truncation $\cI_{\textsc {tr}}$ with source-point index $\Tsp$, and  spine $\cS_x$ with increment collection $(X_i)_{\Tsp\leq i\leq T}$ and $X_{>T}$, such that $\fm_T(\cS_x) = \fm(X_{>T})+\sum_{\Tsp\leq i\leq T} \fm(X_i) \geq r$. Let $\cI_{\trivincr,T}$ be the interface with the same $T$-admissible truncation and spine of exactly $T-\Tsp$ increments that are all $X_\trivincr$. By Theorem~\ref{thm:cluster-expansion},  
\begin{align*}
\frac{\mu_n(\cI)}{\mu_n(\cI_{\trivincr,T})} = \exp \Big(-\beta \fm(\cI;\cI_{\trivincr,T})+ \sum_{f\in \cI} \g(f,\cI) - \sum_{f'\in \cI_{\trivincr,T}} \g(f',\cI_{\trivincr,T})\Big)\,,
\end{align*}
and we recall that $\fm(\cI; \cI_{\trivincr,T}) = \fm_T(\cS_x)$. 
Denote by $\cS_x(\cI_{\trivincr,T})$ the spine of $\cI_{\trivincr,T}$. We can bound the difference,
\begin{align*}
\Big| \sum_{f\in \cI} \g(f,\cI)- \sum_{f'\in \cI_{\trivincr,T}} \g(f', \cI_{\trivincr,T})\Big| \leq \sum_{f\notin \cS_x} |\g(f,\cI )- \g(f,\cI_{\trivincr,T})| + \sum_{f\in \cS_x} |\g(f,\cI)| + \sum_{f'\in \cS_x(\cI_{\trivincr,T})} |\g(f',\cI_{\trivincr,T})|\,.
\end{align*}
By~\eqref{eq:g-uniform-bound}, the latter two terms contribute at most $\bar K (|\cS_x| + 4T+1) = \bar K (\fm(\cI;\cI_{\trivincr,T}) + 8T+2)$. By~\eqref{eq:g-exponential-decay}, the first term is bounded as 
\begin{align*}
 \sum_{f\notin \cS_x} \bar K \exp\big[-\bar c \br(f,\cI; f,\cI_\trivincr)\big] &  \leq  \sum_{f\notin \cS_x} \sum_{f'\in \cI \oplus \cI_\trivincr} \bar K e^{-\bar c d(f,f')}  
\leq  \sum_{f'\in \cI\oplus \cI_\trivincr} \sum_{f\in \mathbb Z^3} \bar K e^{-\bar c d(f,f')}\,,
\end{align*} 
which by integrability of exponential tails is at most $\bar C |\cI \oplus \cI_{\trivincr,T}|\leq \bar C (\fm(\cI ;\cI_{\trivincr,T})+ 8T+2)$ for some universal $\bar C$. As such, once $\fm(\cI; \cI_{\trivincr,T}) \geq 8T$, say, this is comparable up to a universal constant to $\fm_T(\cS_x) = \fm(\cI;\cI_{\trivincr,T})$. Also, notice that the number of possible spines $\cS_x$ of excess area $\fm_T(\cS_x)= k$ is at most the number of connected sets of faces of size $k+1$ incident to the upper-delimiting face of $v_\Tsp$, which is at most $s^{k+1}$, for some universal $s$ by Observation~\ref{obs:counting-connected}. 
Thus, there is a universal $C$ such that for any $r\geq r_0 T$, we have 
\begin{align*}
\mu_n(\fm_T(\cS_x) \geq r\mid \cI_{\textsc {tr}}, \mathbf I_{x,T}) & \leq \sum_{k\geq r} \sum_{\substack{\cS_x : \cS_x \cup \cI_{\textsc {tr}} \in \mathbf I_{x,T} \\ \fm_T(\cS_x) = k}} \mu_n (\cS_x \mid \cI_{\textsc {tr}}) \leq \sum_{k\geq r} \sum_{\substack{\cS_x : \cS_x \cup \cI_{\textsc {tr}} \in \mathbf I_{x,T} \\ \fm_T(\cS_x) = k}} \frac{\mu_n(\cS_x, \cI_{\textsc {tr}})}{\mu_n(\cS_x(\cI_{\trivincr,T}), \cI_{\textsc {tr}})} \\ 
&  \leq  \sum_{k\geq r} s^{k+1} \exp [ - (\beta -C)k]\,,
\end{align*}
at which point, absorbing the $s^{k}$ into the exponential, yields the desired bound for some different $C$. 
\end{proof}

\begin{remark}\label{rem:tame-conditional-on-height}
If $T$ is comparable to $h$, we can attain a version of Lemma~\ref{lem:tame} that also conditions on the event $\{\hgt(\cP_x)\geq h\}$. Namely, for any $\cI_{\textsc{tr}}$, if we set $T' = T \vee (\hgt(\cP_x) -\frac 12- \hgt(v_\Tsp)+\Tsp)$, and apply the proof of Lemma~\ref{lem:tame} with respect to $\cI_{\trivincr, T'}$, we would see see that for every $\cI \in \mathbf I_{x,T}\cap \{\hgt(\cP_x)\geq h\}$, we have 
\begin{align*}\frac{\mu_n(\cI)}{\mu_n(\cI_{\trivincr, T'})} \leq \exp \big[ - (\beta-C) \fm_{T'}(\cS_x) + 8CT'\big]\,.
\end{align*}     
As long as $r\geq 8T'$, this would imply that $\mu_n(\fm_{T'}(\cS_x)\geq r\mid \cI_{\textsc {tr}}, \mathbf I_{x,T}, \hgt(\cP_x)\geq h)\leq \exp(- (\beta - C) r)$; therefore, as long as $8T' \leq 8(T\vee h)$ is less than $20 T$, e.g.,  as long as $\frac h2 \leq T \leq h$, we have for every $\cI_{\textsc{tr}}$,
\begin{align}\label{eq:tame-conditional-on-height}
\mu_n\big(\bar{\mathbf I}^c_{x,T} \mid \mathbf I_{x,T}, \hgt(\cP_x)\geq h, \cI_{\textsc {tr}}\big) \leq \exp\big[- 4(\beta - C)h)\big]\,.
\end{align} 
\end{remark}

\section{Exponential tail on increment excess areas}\label{sec:increment-exp-tail}

In this section, we control the excess areas of the increments that constitute the spine of a tall pillar. Of course it could be that the source point of the spine is itself
an order $T$ distance from $x$ and the base contributes macroscopically to the surface area, but this is ruled out in Section~\ref{sec:base}. Henceforth, take $T$ to be large and take $x$ to be any point in the ``bulk" of $\cL_0 \cap \Lambda_{n,n,\infty}$ relative to $T$, e.g., $d(x, \partial \Lambda_n) \geq 100 T$.

We show an exponential tail on the excess area of the $i$-th increment of the spine of an interface $\cI \in \bar{\mathbf I}_{x,T}$;  the bound will be uniform over both the truncated interface and all the increments below the $i$-th one.

\begin{proposition}\label{prop:exp-tails-increments}
There exists $c_0>0$ such that for every $\beta >\beta_0$, every $T$, and every $i\leq T$, we have that 
\begin{align*}
\mu_n \big(\fm(\sX_{\Tsp+i}) \geq r \mid \bar{\mathbf I}_{x,T}\big)\leq \exp [-c_0 \beta r]\,,
\end{align*}
where if $\Tsp+i >\sT$, we define $\fm(\sX_{\Tsp+i })=0$. 
In fact, for every $T$-admissible truncation $\cI_{\textsc{tr}}$ and every sequence of increments $(X_{\Tsp+j})_{j< i}\in \fX$, we have the same estimate: 
\begin{align*}
\mu_n \big(\fm(\sX_{\Tsp+i}) \geq r \mid  \cI_{\textsc {tr}}, (\sX_{\Tsp+j})_{j< i} = (X_{\Tsp+j})_{j< i}, \bar{\mathbf I}_{x,T}\big) \leq \exp \big[-c_0 \beta r\big]\,.
\end{align*}
\end{proposition}

A useful corollary of the above proposition is the following tail estimate on a quantity measuring the interaction of the spine with the truncated interface $\cI_{\textsc {tr}}$.  

\begin{corollary}\label{cor:increment-interaction-bound}
Let $c_0>0$ be the constant from Proposition~\ref{prop:exp-tails-increments}. There exists some $C>0$ such that for every $\beta>\beta_0$ and every $T$, for each $T$-admissible truncation $\cI_{\textsc{tr}}$ and every $r>0$, 
\begin{align*}
\mu_n \Big ( \sum_{i\geq 1} |\cF(\sX_{\Tsp+i})| e^{- \bar c i} +  |\cF(\sX_{>T})|e^{-\bar c (T+1-\Tsp)} \geq r \mid \cI_{\textsc{tr}}, \bar{\mathbf I}_{x,T}\Big) \leq C \exp \big[-\tfrac12 c_0 \beta(r-C)\big]\,.
\end{align*}
Similarly, for every $T$-admissible truncation $\cI_{\textsc{tr}}$ and increment sequence $(X_{\Tsp +i})_{i\leq i_0}$, we have  
\begin{align*}
\mu_n \Big( |\cF(\sX_{>T})|e^{ - \bar c (T+1 - \Tsp-i_0)} + \sum_{i\geq i_0} |\cF(\sX_{\Tsp+i})|e^{ - \bar c (i - i_0)}\geq r \given  \cI_{\textsc {tr}}, (\sX_{\Tsp+i})_{i<i_0}&  = (X_{\Tsp+i})_{i<i_0}, \bar{\mathbf I}_{x,T}\Big)\\ 
& \leq C \exp \big[-\tfrac12 c_0 \beta (r-C)\big]\,.
\end{align*}  
\end{corollary}
\begin{proof}
By Proposition~\ref{prop:exp-tails-increments}, and in particular its second assertion, for any $T$-admissible truncated interface $\cI_{\textsc {tr}}$ with source-point index $\Tsp$,  the sequence $(\fm(\sX_{\Tsp+i}))_{i\geq 1}$ is dominated by a sequence of i.i.d.\ exponential random variables $\xi_i$ with rate $c_0 \beta$ (as seen by revealing the increments one at a time from bottom to top). Noting that for every $0< \lambda \leq \frac12c_0 \beta$ and every $i\geq 1$,
\[ \E_{\mu_n} [\exp\left(\lambda e^{-\bar c i} \xi_i\right) ] = \left[1-\lambda e^{-\bar c i}/(c_0 \beta) \right]^{-1} \leq 1 + 2 (c_0\beta)^{-1} \lambda e^{-\bar c i} \leq \exp(e^{-\bar c i})\,,\]
 we set $\lambda=\frac12 c_0 \beta$ and obtain that
\begin{align*}
\E_{\mu_n} \Big[ \exp\Big(\lambda \sum_{i\geq 1} |\cF(\sX_{\Tsp+i})|e^{ - \bar c i}\Big) \mid \cI_{\textsc {tr}}, \bar{\mathbf I}_{x,T} \Big] & \leq \prod_{i\geq 1} \E \Big[ \exp\Big( \lambda (\xi_i+4) e^{ - \bar c i} \Big) \Big] \leq \exp\bigg(\sum_{i\geq 1}\left(1+4\lambda\right) e^{-\bar c i}\bigg) \\
& \leq \exp\bigg(\frac{1+2c_0\beta}{1-e^{-\bar c}}\bigg)\,.
\end{align*} 
Letting $\gamma=1/(1-e^{-\bar c})$, this implies by Markov's inequality that
\begin{align*}
\mu_n\Big(\sum_{i\geq 1} |\cF(\sX_{\Tsp+i})| e^{- \bar c i} \geq r \mid \cI_{\textsc {tr}}, \bar{\mathbf I}_{x,T}\Big) \leq e^{(1+2c_0\beta)\gamma-\lambda r} = e^\gamma \exp\left[-\tfrac12 c_0 \beta (r-4\gamma)\right]\,.
\end{align*}
The matching conditional bounds follow from the analogous conditional estimates in Proposition~\ref{prop:exp-tails-increments}. 
\end{proof}

We prove Proposition~\ref{prop:exp-tails-increments} by constructing a map for shrinking increments of the pillar. In order to do so, we define a map between collections of pillars that replaces increments of the pillar with $X_\trivincr$, decreasing the excess area of the increment and, in turn the pillar---the complication is that unlike the map $\Phi_x$ of~\cite{Dobrushin72a}, the effect of this removal is not localized and translates the entirety of the pillar above that increment. 

\begin{figure}
\vspace{-0.3cm}
\centering
  \begin{tikzpicture}
      \draw[->, thick] (-1.5,2) -- (1.5,2);
    \node at (0,2.4) {$\Psi_i$};
    \draw [decorate,decoration={brace,amplitude=5pt}]
(4.3,-2.7) -- (4.3,-1) node [black,midway,xshift=-0.7cm] 
{\footnotesize $\mathbf E^{\emptyset; \cJ}_{0}$};
    \draw [decorate,decoration={brace,amplitude=5pt}]
(3.6,-.2) -- (3.6,1.5) node [black,midway,xshift=-.7cm] 
{\footnotesize $\mathbf E^{\emptyset;\cJ}_{1}$};
    \draw [decorate,decoration={brace,amplitude=5pt}]
(-4.4,-2.7) -- (-4.4,-1.15) node [black,midway,xshift=-0.6cm] 
{\footnotesize $X_{j_0}$};
    \draw [decorate,decoration={brace,amplitude=5pt}]
(-5.15,-.35) -- (-5.15,1.35) node [black,midway,xshift=-.6cm] 
{\footnotesize $X_{j_1}$};
    \node (fig1) at (-4.2,0) {
	\includegraphics[width=.40\textwidth]{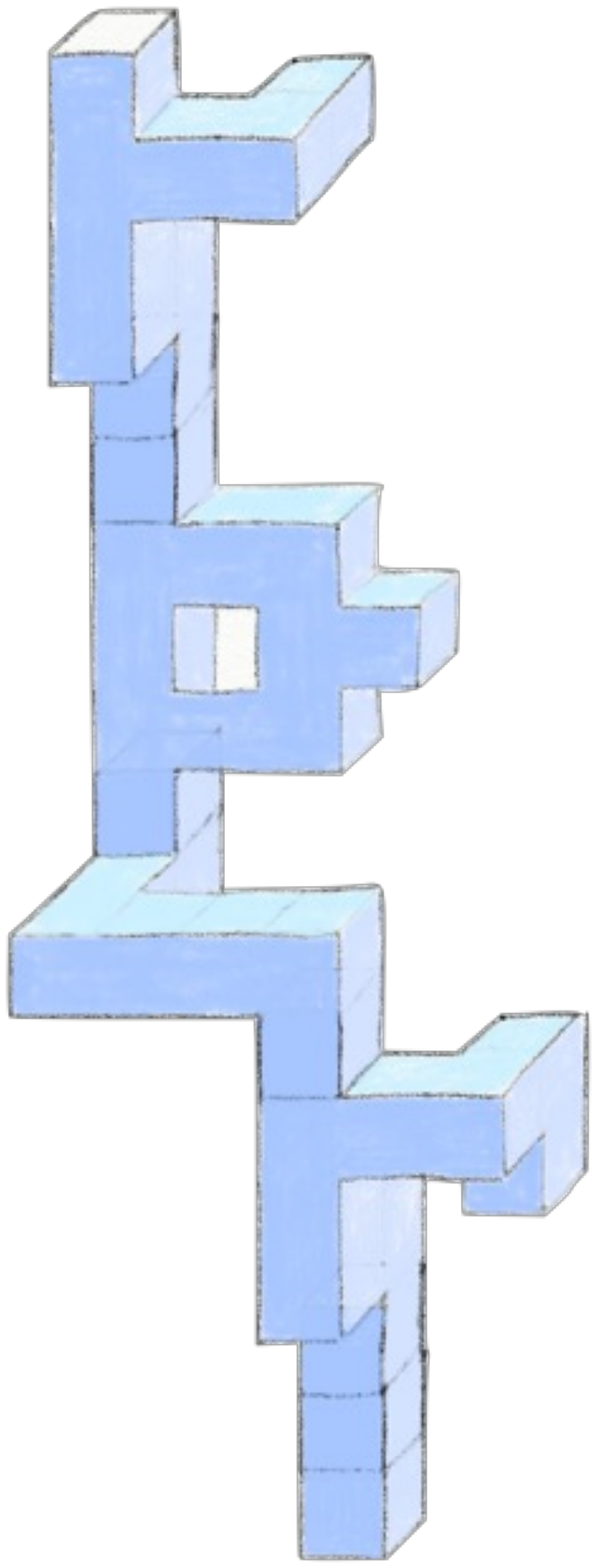}
	};
    \node (fig2) at (4.2,0) {
    \includegraphics[width=.40\textwidth]{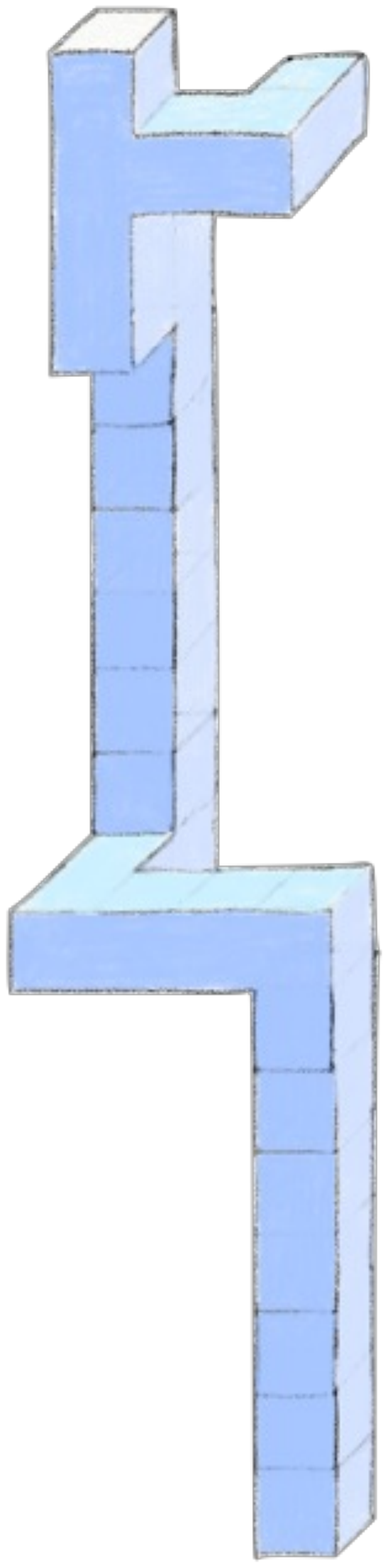}
    };
    \draw (5.0, -2.7)--(5.2, -2.7) -- (5.2, -3.35); 
    \node[font=\small] at (5.45, -3.2) {$\mathbf G$};
    
        \draw [decorate,decoration={brace,amplitude=5pt}]
(5.05,-.25) -- (5.05,-.95) node [black,midway,xshift=.75cm] 
{\footnotesize $\theta^{(0)}\mathbf F_{0}$};

        \draw [decorate,decoration={brace,amplitude=5pt}]
(4.3,3.35) -- (4.3,1.55) node [black,midway,xshift=.75cm] 
{\footnotesize $\theta^{(1)}\mathbf F_{1}$};

        \draw [decorate,decoration={brace,amplitude=5pt}]
(-4.45,3.2) -- (-4.45,1.45) node [black,midway,xshift=.55cm] 
{\footnotesize $\mathbf F_{1}$};

        \draw [decorate,decoration={brace,amplitude=5pt}]
(-3.7,-.25) -- (-3.7,-.95) node [black,midway,xshift=.55cm] 
{\footnotesize $\mathbf F_{0}$};

    \draw (-3.55, -2.7)--(-3.35, -2.7) -- (-3.35, -3.35); 
    \node[font=\small] at (-3.15, -3.2) {$\mathbf G$};
  \end{tikzpicture}
\vspace{-0.75cm}
 \caption{The increment map $\Psi_i$ sends the stretch of increments on the left to the stretch on the right. The increment $X_{j_0}= X_{\Tsp+i}$ is replaced by a stretch of five trivial increments; the increment $X_{j_1}=X_{\Tsp+i+2}$ is also replaced by trivial increments as $\fm(X_{j_1})\geq \fm(X_{j_0})e^{ \bar c }$.} 
  \label{fig:increment-map}
\vspace{-0.2cm}
\end{figure}

\subsection{The increment reduction map \texorpdfstring{$\Psi_i$}{Psi\_i}}\label{subsec:def-incr-map}
For each $T$ and $i\leq T$, we define a map $\Psi_i$ that replaces the $i$-th increment of a spine with a stretch of trivial increments $X_\trivincr$. 

\begin{definition}\label{def:increment-shift-map}
For every $i \leq T$, we will define the map $\Psi_i: \bar {\mathbf I}_{x,T}\to \bar {\mathbf I}_{x,T}$. 
Suppose $\cI \in \bar {\mathbf I}_{x,T}$, consists of a $T$-admissible truncated interface $\cI_{\textsc {tr}}$ with source point index $\Tsp$, an increment sequence $(X_{\Tsp+j})_{j\leq  T-\Tsp}$, and remainder $X_{>T}\in \fX_{\textsc{rem}}$. Then $\Psi_i(\cI)$ will have the same truncated interface $\cI_{\textsc {tr}}$, and its spine will have increment sequence $(X'_{\Tsp+j})_{j\leq T- \Tsp}$ and $X'_{>T}$ constructed as follows. If $X_i =X_{\trivincr}$ or if $\Tsp+i >\sT$, then let $(X'_j)_{j\leq T} = (X_j)_{j\leq T}$ and $X'_{>T} = X_{>T}$;  otherwise, construct the increment sequence of $\Psi_i(\cI)$ by taking the increment sequence $(X_j)_{\Tsp \leq j\leq T}$ and 
\begin{enumerate}
\item Mark the index $\Tsp + i$, as well as every index $j>\Tsp+ i$ having the property that 
\[\fm(X_{j}) \geq \fm(X_{\Tsp+i}) e^{ \frac 12 \bar c (j-\Tsp- i)}\,.
\]
Also mark the remainder if it has $\fm(X_{>T}) \geq \fm(X_{\Tsp+i}) e^{\frac 12 \bar c (T+1 - \Tsp - i)}$. 
\item Label the sequence of marked indices $j_0 = \Tsp+ i$, and $j_0 < j_1 <\ldots <j_\kappa$, where, if the remainder is marked, $j_\kappa$ is ${>T}$. 
\item For each marked index $j_k$, replace $X_{j_k}$ in the increment sequence $(X_{j})_{j}$ by a stretch of $\hgt(X_{j_k})$ consecutive trivial increments $(X_\trivincr, \ldots, X_\trivincr)$, to obtain $(X'_{j})_{j}$. 
\end{enumerate}
\end{definition}

We refer the reader to Figure~\ref{fig:increment-map} for a visualization of the map $\Psi_i$.

\subsection{Strategy of the map \texorpdfstring{$\Psi_i$}{Psi\_i}}\label{subsec:strategy-increment-map}
Let us briefly describe the strategy behind the construction of the map above. Our goal is to show an exponential tail on the excess area of the $(\Tsp+i)$'th increment conditionally either on having at least $T$ increments---Proposition~\ref{prop:exp-tails-increments}---or on having height at least $h$ and $T$ increments---Proposition~\ref{prop:increment-conditional-on-height}. Towards this, we wish to construct a map $\Psi_i$ having that 
\begin{enumerate}
    \item For every $\mathcal I \in \bar {\mathbf I}_{x,T}\cap \{\hgt(\cP_x)\geq h\}$, the interface $\Psi_i(\mathcal I) \in \bar {\mathbf I}_{x,T}\cap \{\hgt(\cP_x)\geq h\}$
   \item  $\mu_n(\bar {\mathbf I}_{x,T}\cap \{\fm(\mathscr X_{\Tsp+i})>r\})\leq e^{-(\beta - C)r} \mu_n(\Psi_i(\bar{\mathbf I}_{x,T}))$ via the steps (1)--(3) of~\eqref{it:maps-energy}--\eqref{it:maps-mult} in \S\ref{subsec:tools} as well as the analogue of this inequality, with both events also intersected with $\{\hgt(\cP_x)\geq h\}$. 
\end{enumerate}

Towards this, our map replaces the $(\Tsp+i)$'th increment by a sequence of $\hgt(\mathscr X_{\Tsp+i})$ trivial increments, yielding an energy gain that is comparable to $\fm(\mathscr X_{\Tsp+i})$. N.b.\ replacing it by just one trivial increment would not ensure that the resulting pillar also attains the same height as the original pillar. 

Unlike changes in the standard wall representation, changes in the increment sequence subsequently induce a \emph{horizontal} shift of all increments above the $(\Tsp+i)$'th one. These horizontally shifted increments $\sX_{j}$ can then interact with increments below $\mathscr X_{\Tsp+i}$ via the term $\g(f,\cI; f',\Psi_i(\cI))$ of~\eqref{eq:g-exponential-decay}.  By~\eqref{eq:g-exponential-decay}, this quantity decays exponentially in the distance to $\sX_{\Tsp+i}$, so that if the excess area $\fm(\sX_j)$ is larger than $e^{ - d(\sX_j,\sX_{\Tsp+i})}$, we cannot compare the contribution of the perturbative $\g$ term to the energy gain of the map. For this reason, we additionally delete all increments whose excess areas are greater than some exponential factor times their distance to $\sX_{\Tsp+i}$. Iterating this procedure up the spine yields the map $\Psi_i$.

The following remark summarizes the properties of the map $\Psi_i$ that we will use in its analysis.   

\begin{remark}
By construction, the excess area of the spine of $\Psi_i(\cI)$ is at most the excess area of the spine of $\cI$, the map $\Psi_i(\cI)$ keeps the height of $\cS_x$, and thus also $\cP_x$, fixed, and the map $\Psi_i$ only increases the number of increments of the spine. Therefore, for every $\cI\in \bar{\mathbf I}_{x,T}$, we have $\Psi_{i}(\cI)\in \bar{\mathbf I}_{x,T}$. Moreover, notice that the truncated interfaces of $\cI$ and $\Psi_{i}(\cI)$, and their first $\Tsp+i-1$ increments, agree. 
\end{remark}

\subsection{Analysis of the map \texorpdfstring{$\Psi_i$}{Psi\_i}}
We will bound the effect of the map on the energy in Proposition~\ref{prop:probability-ratio-increment}, and its multiplicity in Lemma~\ref{lem:multiplicity-increment}. Combining these will imply Lemma~\ref{prop:exp-tails-increments}, the main result in this section.

\begin{proposition}\label{prop:probability-ratio-increment}
There exists $C>0$ such that for every $\beta>\beta_0$ and every $i\leq T$, if $\cI\in \bar {\mathbf I}_{x,T}$, 
\begin{align*}
\Big| \log \frac{\mu_n(\cI)}{\mu_n(\Psi_i (\cI ))}  + \beta \fm(\cI; \Psi_i (\cI))\Big| \leq C \fm(\cI ;\Psi_i (\cI))\,.
\end{align*}
\end{proposition}
\begin{proof}
For ease  of notation, fix any such $i$ and let $\cJ =\Psi_i (\cI)$. Suppose that $\cI$ has $T$-admissible truncation $\cI_{\textsc{tr}}$ with increment sequence $(X_{\Tsp+i})_{i\leq T-\Tsp}$ and remainder $X_{>T}$. If $\cJ = \cI$, the inequality trivially holds, so let us assume that $\cI$ is such that $\fm(X_{\Tsp+i}) >0$.  By Theorem~\ref{thm:cluster-expansion}, we can express
\begin{align*}
\frac{\mu_n(\cI)}{\mu_n(\cJ)} = \exp\Big(- \beta \fm(\cI ; \cJ) +  \sum_{f\in \cI} \g(f,\cI) - \sum_{f'\in \cJ} \g(f',\cJ)\Big)\,.
\end{align*}
Let $\{j_k\}_{k\leq \kappa}$ be as in Definition~\ref{def:increment-shift-map}. To allow us to consider the increments and remainder in a uniform manner, let $j_\kappa:= T+1$ if $j_\kappa$ is ``$>T$", so that $X_{j_{\kappa}}$ refers to $X_{>T}$ if $j_\kappa = T+1$. We have   
\[\fm(\cI; \Psi_i (\cI)) \geq \frac 13 \sum_{k\leq \kappa} \fm(X_{j_k})\,,\]
since every nontrivial increment with height bigger than $1$ must have at least six faces at each height between $\lfloor \hgt(v_{j})\rfloor $ and $\lceil \hgt(v_{j+1})\rceil$, whereas the stretch of trivial increments would have four faces at those heights.
Now let us split the set of faces in $\cI$ into the following sets (refer to Figure~\ref{fig:increment-map}): 
\begin{itemize}
\item For each $k\leq \kappa$, let $\mathbf E_{k}$ be the set of faces $\cF(X_{ j_k})$ in $\cI$.
\item For each $k\leq \kappa$, let $\mathbf F_k$ be the (possibly empty if $j_{k+1}= j_k +1$ or if $j_{k} = T+1$) set of all faces between $v_{j_k+1}$ and $v_{j_{k+1}}$ (not-inclusive), with $\mathbf F_{\kappa}$ defined as the set of all faces above $v_{ j_\kappa+1}$.
\item Let $\mathbf G$ be the set of all remaining faces in $\cI$
\end{itemize}
Also, for notation, let $\mathbf E^\emptyset_k\subset \mathbf E_k$ be the bounding faces of $v_{j_k}$ and $v_{j_{k+1}}$ in $\cI$ (so $|\mathbf{E}^\emptyset_k|=8$) and if $j_\kappa = T+1$, then $\mathbf E_{\kappa}^{\emptyset}$ will only be the four bounding faces of $v_{j_{\kappa}}$ in $\cI$. Let $\mathbf E^{\emptyset; \cJ}_k$ be the corresponding faces in $\cJ$, i.e., the faces of the $\hgt(X_{j_k})$ consecutive trivial increments, so that $|\mathbf E^{\emptyset;\cJ}_k| = 4 \hgt(X_{j_k})$ (if $j_\kappa = T+1$, also include the top-most bounding face of the spine in $\cJ$).  

By definition, the faces in $\mathbf G$ are shared between $\cI$ and $\cJ$, the faces $\bigcup_k \mathbf E_k\setminus \mathbf E^{\emptyset}_k$ are precisely those that are removed by the map $\Psi_i$, and the faces in $\mathbf F_k$ can be translated to correspond in a one-to-one fashion to the faces in $\cJ\setminus (\mathbf G\cup \bigcup_{k} \mathbf E^{\emptyset;\cJ}_k)$.  
Namely, if for every $k\leq \kappa$, we set $\theta^{(k)}$ to be the shift map by the vector 
\begin{align*}
 - \sum_{ 0\leq \ell \leq k} \rho(v_{j_\ell+1}) - \rho(v_{j_\ell})
\end{align*}
then every face $f\in \mathbf F_k$, is identified with the face $\theta^{(k)} f$ in $\cJ$, and for $k<\kappa$, each stretch $\theta^{(k)} \mathbf F_k$ is delimited from below by the upper-bounding face of $\theta^{(k)} \mathbf E^\emptyset_k \subset \mathbf E^{\emptyset; \cJ}_k$ and from above by the lower-bounding face of $\theta^{(k+1)} \mathbf E^{\emptyset}_{k+1}$. By construction, we have  
\begin{align*}
\cI = \mathbf G\cup \bigcup_{k\leq \kappa} \mathbf E_k \cup \mathbf F_k\,,  \qquad \mbox{and}\qquad \cJ = \mathbf G \cup \bigcup_{k\leq \kappa}\mathbf E_k^{\emptyset; \cJ}\cup \theta^{(k)}\mathbf F_k\,.
\end{align*}
We can therefore split up the sum 
\begin{align}\label{eq:increment-map-g-terms}
\Big|\sum_{f\in \cI} \g(f,\cI)- \sum _{f'\in \cJ} \g(f',\cJ) \Big| \leq & \sum_k \sum_{f\in \mathbf E_k}|\g(f,\cI)|   + \sum_k \sum_{f\in \mathbf E_k^{\emptyset; \cJ}} |\g(f,\cJ)| \nonumber \\
& + \sum_k \sum_{f\in \mathbf F_k} \big|\g(f,\cI)- \g(\theta^{(k)} f ,\cJ)\big|   + \sum_{f\in \mathbf G} |\g(f,\cI)-\g(f,\cJ)|\,.
\end{align}
We bound these sums one at a time. By~\eqref{eq:g-uniform-bound}, along with $|\cF(X_{j_k})| \leq 5\fm(X_{j_k})$ (by Remark~\ref{rem:increment-excess-area} and $\fm(X_{j_k})>0$), the first and second sums in~\eqref{eq:increment-map-g-terms} are bounded above as 
\[\sum_k \sum _{f\in \mathbf E_k} |\g(f,\cI)| + \sum_k \sum_{f\in \mathbf E_k^{\emptyset;\cJ}} |\g(f,\cJ)| \leq 2 \bar K \sum_k |\cF(X_{j_k})| \leq  10 \bar K \fm(\cI ;\cJ)\,.\]
Let us now turn to the third term of~\eqref{eq:increment-map-g-terms}, which we can bound as follows: first of all, notice that for any face $f\in \mathbf F_k$ in increment $\cF(X_j)$ for $j_k <j<j_{k+1}$, the radius $\br(f,\cI;\theta^{(k)} f,\cJ)$ is attained either by some face in a spine (belonging to precisely one of $\cI$ or $\cJ$), in which case its value is at least $(j-j_k)\wedge (j_{k+1}-j)$, or by some face in $\mathbf G$ (the faces in $\mathbf G$ are the same in both $\cI$ and $\cJ$, but will be at different relative locations to $f$ vs.\ $\theta^{(k)} f$). Let us take any $k\leq \kappa$, fix a $j_k < j <j_{k+1}$ ($j>j_\kappa$ if $k=\kappa$ and $j_{\kappa}\neq T+1$) and a face $f\in \cF(X_j) \subset \mathbf F_k$, and expand 
\begin{align*}
\big|\g(f,\cI)-\g(\theta^{(k)} f,\cJ)\big| & \leq \bar K \exp[-\bar c \br(f,\cI ;\theta^{(k)} f, \cJ)]  \nonumber \\ 
& \leq  \bar K \exp[{-\bar c (j-j_k)}] + \bar K \exp[{-\bar c (j_{k+1}-j)}] +  \bar K \exp[-\bar c d(\{f,\theta^{(k)}f\},\mathbf G)]\,. 
\end{align*}  
Notice that since both $\cI$ and $\cJ$ are in $\bar{\mathbf I}_{x,T}$, their spines are contained in $\cC_{2r_0 T}(v_\Tsp)$. As a consequence,  $d(f,\mathbf G)$ is attained by a face in $\cC_{R_0 T}(v_\Tsp)$, and is at least $j- \Tsp- i \geq j-j_k$ (n.b.\ there are $j-\Tsp-i$ cut-points separating $\mathbf{G}$ and $f$), and the same holds for $d(\theta^{(k)}f,\mathbf{G})$; therefore, the above becomes 
\begin{align}\label{eq:II_k-splitting}
|\g(f,\cI)-\g(\theta^{(k)} f,\cJ)| \leq 2\bar K \exp\big[- \bar c (j-j_k)\big] + \bar K \exp\big[- \bar c (j_{k+1} - j)\big]\,.
\end{align} 
Now summing the first term in the right-hand side in~\eqref{eq:II_k-splitting} over all $k$ and $f\in \mathbf{F}_k$,
\begin{align*}
\bar K \sum_{k\leq \kappa} \sum_{j_k < j <j_{k+1}} |\cF(X_j)| e^{-\bar c (j- j_k)}  \leq  \bar K \sum_{k\leq \kappa} \sum_{j_k<j< j_{k+1}} \Big[ 8+ \fm(X_{j_k})\Big] e^{\frac 12 \bar c (j-j_k)} e^{-\bar c (j-j_k)}\,,
\end{align*}
using that $|\cF(X_j)|= 8 + \fm(X_j)$ for $j_k <j<j_{k+1}$, and the facts that the reduction map was not applied at index $j$ and was applied at $j_k$, so that 
$$\fm(X_j)\leq \fm(X_{\Tsp+i})e^{\frac 12 \bar c (j-\Tsp-i)}\leq \fm(X_{j_k})e^{\frac 12 \bar c (j-j_k)}\,.$$ 
Then, by integrability of exponential tails, we see that the right-hand side above is in turn bounded by $\bar C\sum_{k\leq \kappa} \fm(X_{j_k}) = \bar C \fm(\cI;\cJ)$ for some universal $\bar C$. The second term in~\eqref{eq:II_k-splitting} can similarly be bounded as 
\begin{align*}
\bar K \sum_{k<\kappa} \sum_{j_k < j < j_{k+1}} \fm(X_j)e^{-\bar c (j_{k+1}-j)}\leq \bar K \sum_{k<\kappa} \sum_{j_k < j < j_{k+1}} \fm(X_j) 
 \leq \bar C\sum_{k<\kappa} \fm(X_{j_{k+1}})\,.
\end{align*}
for some constant $\bar C$, where we used that  
$$\sum_{j_k < j < j_{k+1}}\fm(X_j) \leq \sum_{j_k < j < j_{k+1}} \fm(X_{\Tsp - i}) e^{\frac 12 \bar c (j-\Tsp- i)} \leq \bar C\fm(X_{\Tsp-i}) e^{\frac 12 \bar c (j_{k+1}-{\Tsp -i})} \leq \bar C \fm(X_{j_{k+1}})\,.$$
Again, by integrability of exponential tails, we  see that for some other $\bar C>0$, the contribution of this term is bounded by $\bar C\sum_{k< \kappa} \fm(X_{j_k}) \leq \bar C\fm(\cI;\cJ)$. 
It remains to bound the fourth sum in~\eqref{eq:increment-map-g-terms}: for faces $f\in \mathbf G$, the radius $\br(f,\cI; f, \cJ)$ must be attained by a face in $\cI \oplus \cJ$, so that 
\begin{align*}
\sum_{f\in \mathbf G} |\g(f,\cI)-\g(f,\cJ)|  
& \leq \sum_{f\in \mathbf G\cap \cC_{R_0 T}(v_\Tsp)} \sum_{g\in \cI \oplus \cJ} \bar K e^{-\bar c d(f,g)}+ \sum_{f\in \mathbf G, f\notin \cC_{R_0 T}(v_\Tsp)} \sum_{g\in \cI \oplus \cJ} \bar K e^{-\bar c d(f,g)}\,.
\end{align*}
Since $\cI \oplus \cJ \subset \cC_{2r_0 T}(v_\Tsp)$, integrating the exponential tail, the second sum above is at most $O(T^2 e^{- \bar c r_0 T})$. On the other hand, by definition of the spine and the fact that it is tame, 
\begin{align*}
 \sum_{f\in \mathbf G\cap \cC_{R_0 T}(v_\Tsp)} \sum_{g\in (X_j)_{j\geq j_0}\cup (X_j')_{j\geq j_0}}  e^{-\bar c d(f,g)} & \leq 2 \sum_{j\geq j_0} |\cF(X_j)| e^{-\bar c (j-\Tsp - i)} \\ 
& \leq 8 \bar C + 2 \sum_{k \leq \kappa} \fm(X_{j_k}) + 2 \sum_{j\geq j_0: j \notin \{j_k\}_{k\leq \kappa}} \fm(X_{\Tsp + i})e^{- \frac{1}{2} \bar c (j-\Tsp - i)}\,,
\end{align*} 
which we again find to be bounded by $\bar C \fm(\cI;\cJ)$ for some other universal constant $\bar C$. 
Plugging all the above bounds into~\eqref{eq:increment-map-g-terms}, we see that for some universal $C$ (independent of $\beta$), we have 
\[
\Big|\sum_{f\in \cI} \g(f,\cI)-\sum_{f'\in \cJ} \g(f',\cJ)\Big| \leq C \fm(\cI;\cJ)\,. \qedhere
\]
\end{proof}

We now bound the multiplicity of the map $\Psi_i$ for a fixed excess area $\fm(\cI;\Psi_i (\cI))$. 
\begin{lemma}\label{lem:multiplicity-increment}
For every $i \leq T$ and every $\cJ\in \Psi_i (\bar {\mathbf I}_{x,T})$, there exists an $s>0$ such that for every $K$, 
$$|\{\cI \in \Psi_i^{-1}(\cJ): \fm(\cI;\cJ)=K\}| \leq s^{K}\,.$$
\end{lemma}
\begin{proof}
Fix any $i$ and $\cJ \in \Psi_i (\bar {\mathbf I}_{x,T})$. For a fixed spine $\cS_x(\cJ)$ it suffices to bound the number of spines $\cS_x (\cI)$ for which the map $\Psi_i$ sends $\cS_x (\cI)$ to $\cS_x(\cJ)$ with $\fm(\cS_x(\cI);\cS_x(\cJ))=K$ (as the map $\Psi_i$ fixes all faces of the interface in $\cI_{\textsc {tr}}$). We first observe a few basic facts. 

By definition of $\Psi_i$, any spine $\cS_x(\cI)$ that gets mapped to $\cS_x(\cJ)$ by $\Psi_i$ is such that their increment sequences $(X_j)_{j< i}$ coincide, and therefore the spines agree up to the $i$-th increment of $\cS_x(\cJ)$, which will satisfy $X_{\Tsp + i} ' = X_\trivincr$. 
In particular, for a given $\cJ$, the interface $\cI$ is uniquely identified by the collection of increments $(X_{j_k})_{k\leq \kappa}$ and the indices of those increments $\{j_k\}_{k\leq \kappa}$, since the rest of its spine is given by increments which are the same in both $\cI$ and $\cJ$. 
Therefore, starting from $v_{\Tsp + i}$ for the interface $\cJ$ (which coincides with the same cell for $\cI$), we can build a set of faces that uniquely identify the interface $\cI$ by taking the union of all the increments $(X_j)$ between $X_{j_0} = X_{\Tsp + i}$ and the final $X_{j_\kappa}$ (inclusive). 

The union of these increments, viewed as a subset of the spine $\cS_x$,  clearly forms a $*$-connected set of faces in $\cF(\Z^3)$ that are $*$-adjacent to the upper-bounding face of the marked cell $v_{i}$. We claim that this subset of $\cS_x$ has cardinality bounded above by $CK$ for some universal $C$. This follows from the fact that the cardinality of the face set of an increment is at most $4\fm(X_j)$ so long as $X_j \neq X_{\trivincr}$, so that the total cardinality of the face set of $\bigcup_{j_0 \leq j \leq j_{\kappa}} \cF(X_j)$ is at most four times
\begin{align*}
 \sum_{k\leq \kappa} \fm(X_{j_k})+ \sum_{k<\kappa} \sum_{j_k < j < j_{k+1}} [1+\fm(X_j)] & \leq \fm(\cI;\cJ)+ \sum_{j\notin \{j_k\},j_0<j<j_{\kappa}} 2\fm(X_i) e^{\frac 12 \bar c( j-\Tsp - i )} \\
& \leq \fm(\cI;\cJ) + 2\bar C\sum_{k\leq \kappa} \fm(X_{j_k})\,,
\end{align*}
which is, in turn, bounded above by $C \fm(\cI;\cJ)$ for some large enough, universal $C$.  Since this rooted face-set uniquely identifies $\cI\in \Psi_{i}^{-1}(\cJ)$, the result then follows immediately from Observation~\ref{obs:counting-connected}. 
\end{proof}

\begin{proof}[\textbf{\emph{Proof of Proposition~\ref{prop:exp-tails-increments}}}]
Since $\fm(\cI; \Psi_i(\cI)) \geq \frac 13 \fm(\sX_{\Tsp+i})$, it will suffice for us to show the upper bound on $\mu_n(\fm(\cI;\Psi_i(\cI))\geq r \mid \cI_{\textsc{tr}}, \bar{\mathbf I}_{x,T})$. Fix a $T$-admissible truncation $\cI_{\textsc{tr}}$ and an $i\leq T$, and express $\mu_n(\fm(\cI;\Psi_i(\cI))\geq r , \cI_{\textsc{tr}}, \bar{\mathbf I}_{x,T})$ as
\begin{align*}
\sum_{\cI \in \bar{\mathbf I}_{x,T}\cap \cI_{\textsc{tr}}, \fm(\cI;\Psi_i(\cI))\geq r} \mu_n(\cI) & = \sum_{k\geq r} \sum_{\cJ\in \Psi_i(\bar{\mathbf I}_{x,T}\cap \cI_{\textsc{tr}})} \sum_{\cI\in \Psi_i^{-1}(\cJ): \fm(\cI;\cJ) = k} \frac{\mu_n(\cI)}{\mu_n(\Psi_i(\cI))} \mu_n(\cJ) \\
& \leq \sum_{\cJ \in \Psi_i(\bar{\mathbf I}_{x,T}\cap \cI_{\textsc{tr}})} \mu_n(\cJ) \sum_{k\geq r}   s^{k} e^{- (\beta -C) k}\,,
\end{align*}
where we used the shorthand $\bar{\mathbf I}_{x,T}\cap \cI_{\textsc{tr}}$ to denote the set of interfaces in $\bar{\mathbf I}_{x,T}$ with that truncation, and the inequality followed from Proposition~\ref{prop:probability-ratio-increment} and Lemma~\ref{lem:multiplicity-increment}. Since $\Psi_i(\bar{\mathbf I}_{x,T}\cap \cI_{\textsc {tr}}) \subset \bar{\mathbf I}_{x,T}\cap \cI_{\textsc {tr}}$, by integrability of exponential tails, the sum over $k\geq r$ is at most $C e^{- (\beta - C)r}$ for some universal constant $C$, leaving 
\begin{align*}
\mu_n (\fm(\cI;\Psi_i(\cI))\geq r,\cI_{\textsc {tr}}, \bar{\mathbf I}_{x,T})  \leq C \mu_n (\bar{\mathbf I}_{x,T}, \cI_{\textsc {tr}}) \exp \big[-(\beta - C) r\big]\,,
\end{align*}
for some universal constant $C$. Dividing both sides by $\mu_n(\cI_{\textsc {tr}}, \bar{\mathbf I}_{x,T})$ yields the first estimate. The matching estimate conditional also on $(\sX_{\Tsp+j})_{j<i}$ follows by repeating the argument, additionally restricting our sum to interfaces with that increment sequence, as the map $\Psi_i$ fixes all increments before the $(\Tsp+i)$-th one. 
\end{proof}

\subsection{Exponential tail conditionally on 
\texorpdfstring{$\hgt(\cP_x)\geq h$}{ht(P\_x)>=h}}

In the proofs of the existence of a limiting large deviation rate and the law of large numbers for the maximum of the interface, it will be important to work with the monotone event $\{\hgt(\cP_x)\geq h\}$ rather than $\mathbf I_{x,T}$. The fact that the map $\Psi_i$ keeps the height of a spine fixed allows us to also deduce the analogous exponential tails conditional on $\{\hgt(\cP_x)\geq h\}$. In fact, if one were only interested in estimates conditional on $\bar{\mathbf I}_{x,T}$ (as are relevant to the shape theorem and central limit theorem), the map $\Psi_i$ could be simplified to replace each increment in $(X_{j_k})_k$ by $X_\trivincr$, keeping the number of increments fixed, but shrinking the height.  

\begin{proposition}\label{prop:increment-conditional-on-height}
There exists $C>0$ such that for every $\beta>\beta_0$, every $T\leq h$, every half-integer $h_1 \leq h$, and every $T$-admissible truncation $\cI_{\textsc{tr}}$ having $\hgt(v_\Tsp)\leq  h_1$, we have 
\begin{align*}
\mu_n \big(|\cF(\cS_x \cap \cL_{h_1})| \geq 4+ r \mid \hgt(\cP_x)\geq h, \cI_{\textsc{tr}},  \bar{\mathbf I}_{x,T}\big) \leq \exp[- \beta r/C]\,.
\end{align*}
Similarly, for every $i\leq T$, and every $T$-admissible truncation $\cI_{\textsc{tr}}$ with $\Tsp<i$, and sequence $(X_j)_{j<i}$,
\begin{align}\label{eq:increment-conditional-on-height}
\mu_n(\fm(\sX_i) \geq r \mid  (X_j)_{j<i}, \cI_{\textsc {tr}} , \hgt(\cP_x)\geq h, \bar{\mathbf I}_{x,T}) \leq \exp[- \beta r/C]\,.
\end{align} 
This latter estimate also implies the analogue of Corollary~\ref{cor:increment-interaction-bound}, also conditioned on $\{\hgt(\cP_x)\geq h\}$ for $h\geq T$.
\end{proposition}

\begin{proof}
The proof of~\eqref{eq:increment-conditional-on-height} goes  similarly to the proof of Proposition~\ref{prop:exp-tails-increments}. Namely, if we restrict the proof therein to interfaces additionally having $\hgt(\cP_x)\geq h$, and notice that for all such interfaces, their image under $\Psi_i$ is also a subset of $\{\hgt(\cP_x)\geq h\}$. With this observation, the natural modifications yield the desired. 

The proof of the first inequality in Proposition~\ref{prop:increment-conditional-on-height} is more subtle as the increment intersecting $\mathcal L_{h_1}$ is random. For each interface $\cI$ having $|\cF(\cS_x\cap \cL_{h_1})|\geq 4+r$, let $\tau_1(\cI)$ denote its increment index such that $\mathscr X_{\Tsp+\tau_1}$ intersects $\cL_{h_1}$ non-trivially. Then let $\Psi_{h_1}$ denote the map that for each $\cI$ having $|\cF(\cS_x\cap \cL_{h_1})|\geq 4+r$, is defined by $\Psi_{h_1}(\cI) = \Psi_{\tau_1(\cI)} (\cI)$.  
Fix a $T$-admissible truncation $\cI_{\textsc{tr}}$ and express $\mu_n(|\cF(\cS_x)\cap \cL_{h_1})|\geq 4+r, \hgt(\cP_x)\geq h, \cI_{\textsc{tr}}, \bar{\mathbf I}_{x,T})$ as 
\begin{align*}
     \sum_{\substack{\cI \in \bar{\mathbf I}_{x,T}\cap\hgt(\cP_x)\geq h\cap \cI_{\textsc{tr}}  \\ |\cF(\cS_x \cap \cL_{h_1})|\geq 4+r}} \mu_n(\cI) =  \sum_{k\geq r}  \sum_{\cJ \in \Psi_{h_1}(\bar{\mathbf I}_{x,T}\cap \cI_{\textsc{tr}}\cap \hgt(\cP_x)\geq h)}  \mu_n(\cJ) \sum_{\tau_1}
    \sum_{\substack{\cI \in \Psi_{\tau_1}^{-1}(\cJ)  \\ \mathscr X_{\Tsp+\tau_1} \cap \cL_{h_1}\neq \emptyset \\ \fm(\cI;\cJ)= k}}\frac{\mu_n(\cI)}{\mu_n(\Psi_{\tau_1}(\cI))} 
\end{align*}
%
At this point, we notice that for each $\cJ$, and $k\geq r$, there are at most $k$ possible choices of $\tau_1$ such that $\{\cI: \cI \in \Psi_{\tau_1}^{-1}(\cJ), \fm(\cI;\cJ)=k\}$ is non-empty. This is because, if $X_{\Tsp+\tau_1}\cap \cL_{h_1}\neq \emptyset$ it must be that  $h_1 - v_{\Tsp+\tau_1} \leq k$ (the excess area of the map is at least $\hgt(v_{\Tsp+\tau_1 + 1}) - \hgt(v_{\Tsp + \tau_1})$). Reading off from $\mathcal J$, the increment index intersecting $\cL_{h_1 - k-1}$, one of the next $k$ increment indices must be $\tau_1$. Combining this with Propositions~\ref{prop:probability-ratio-increment} and Lemma~\ref{lem:multiplicity-increment}, we see that this is at most 
\begin{align*}
    \sum_{\cJ\in \Psi_{h_1}(\bar{\mathbf I}_{x,T}\cap \cI_{\textsc{tr}}\cap \hgt(\cP_x)\geq h)} \mu_n(\cJ) \sum_{k\geq r} ks^k e^{-(\beta - C)k}\leq C \mu_n (\bar{\mathbf I}_{x,T}\cap \cI_{\textsc{tr}}\cap \hgt(\cP_x)\geq h) e^{ - (\beta - C)r}\,.
\end{align*}
since $\Psi_{h_1}(\bar{\mathbf I}_{x,T}\cap \cI_{\textsc{tr}}\cap \{\hgt(\cP_x)\geq h\})\subset (\bar{\mathbf I}_{x,T}\cap \cI_{\textsc{tr}}\cap \{\hgt(\cP_x)\geq h\})$. Dividing both sides out by $\mu_n(\bar{\mathbf I}_{x,T}\cap \cI_{\textsc{tr}}\cap \hgt(\cP_x)\geq h)$ then yields the desired.
\end{proof}

\section{Exponential tails on the base of a pillar}\label{sec:base}
In Section~\ref{sec:increment-exp-tail}, we showed that the increments of the spine each have exponential tails on their excess areas. Here, we show that the groups of walls that constitute the interface apart from the spine, but are ``near" the spine have excess area at most order $O(\log T)$. 
As a consequence, we see that $|v_\Tsp- x|$ is at most $O(\log T)$ with high probability; this difference, and the base more generally, will be negligible as far as any $T\to\infty$ limit theorems are concerned.  

As before, take $T$ to be large and take $x$ to be a point in the ``bulk" of $\cL_0 \cap \Lambda_{n,n,\infty}$, e.g., $d(x, \partial \Lambda_n) \geq 100 T$.

\begin{proposition}\label{prop:base-exp-tail}
There exists $K,c>0$ such that for every $\beta>\beta_0$, we have for every $r\geq K\log T$, 
\begin{align}
\mu_n(\hgt(v_\Tsp) \geq r \mid \bar{\mathbf I}_{x,T}) & \leq \exp[-  c\beta r ]\,, \label{eq:hgt-vsp}\\ 
\mu_n(\diam (\rho(\sB_x)) \geq r \mid \bar{\mathbf I}_{x,T}) & \leq \exp[- c\beta r]\,. \label{eq:base-diameter}
\end{align}
where the diameter $\diam (\rho(\sB_x)) := \max_{x,y\in \rho(\sB_x)} |x-y|$. Finally, we can also deduce that for $r\geq K\log T$,
\begin{align}\label{eq:base-total-increment-decay}
\mu_n\Big(\sum_{i\leq \Tsp} \fm(\sX_i) \geq r \mid \bar{\mathbf I}_{x,T}\Big) \leq \exp [-c\beta r]\,.
\end{align} 
\end{proposition}

\begin{remark}
In particular,~\eqref{eq:hgt-vsp} implies that for every $r\geq K\log T$, with probability $1- e^{ - c \beta r}$, we have $\Tsp\leq r$; the bound~\eqref{eq:base-diameter} immediately implies the analogous bound on $|\rho(v_\Tsp)-x|$. 
\end{remark}

\begin{remark}
Since we prove that the projection of the spine attains an order  $\sqrt T$ distance from $x$, there should be groups of walls of $\cI\setminus \cP_x$ onto which $\cP_x$ projects, that attain an excess area $c\log T$; thus the order of the bounds on $\hgt(v_{\Tsp})$ and $\diam (\rho(\sB_x))$ is correct. On the other hand, we expect that $\hgt(v_1)$ is order one, and already the increments starting from $\sX_1$ have exponential tails; the difficulty in proving this is in controlling the interactions of $\sX_1,\ldots, \sX_{\Tsp}$ with nearby groups of walls which attain a higher height.  
\end{remark}

\begin{figure}
\vspace{-0.5cm}
\centering
  \begin{tikzpicture}
    \draw[->, thick] (-2,2) -- (2,2);
    \node at (0,2.4) {$\Phi_{\sB}$};
    \node (fig1) at (-4,0) {
	\includegraphics[width=.40\textwidth]{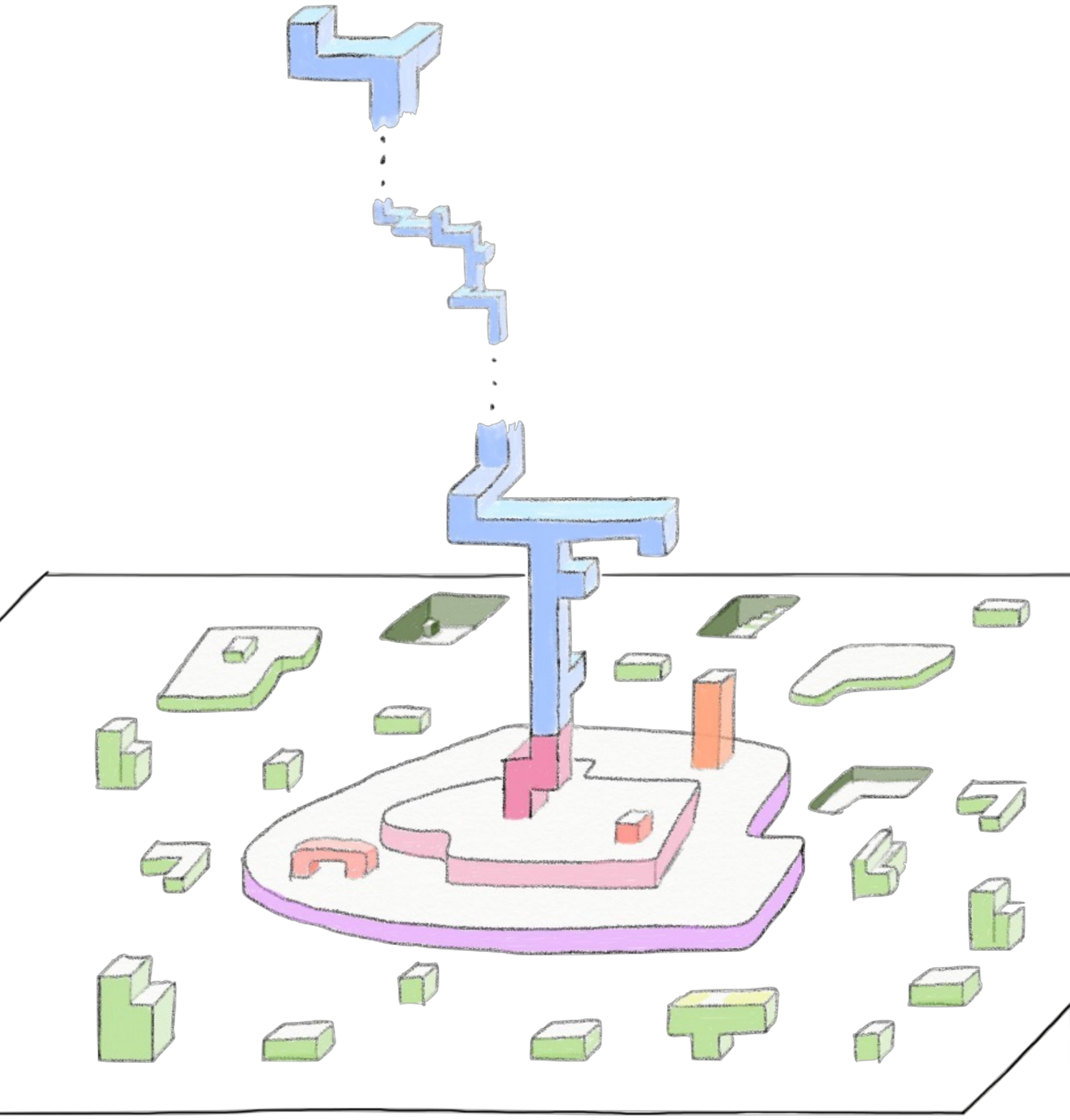}
	};
    \node (fig2) at (4,.16) {
    \includegraphics[width=.40\textwidth]{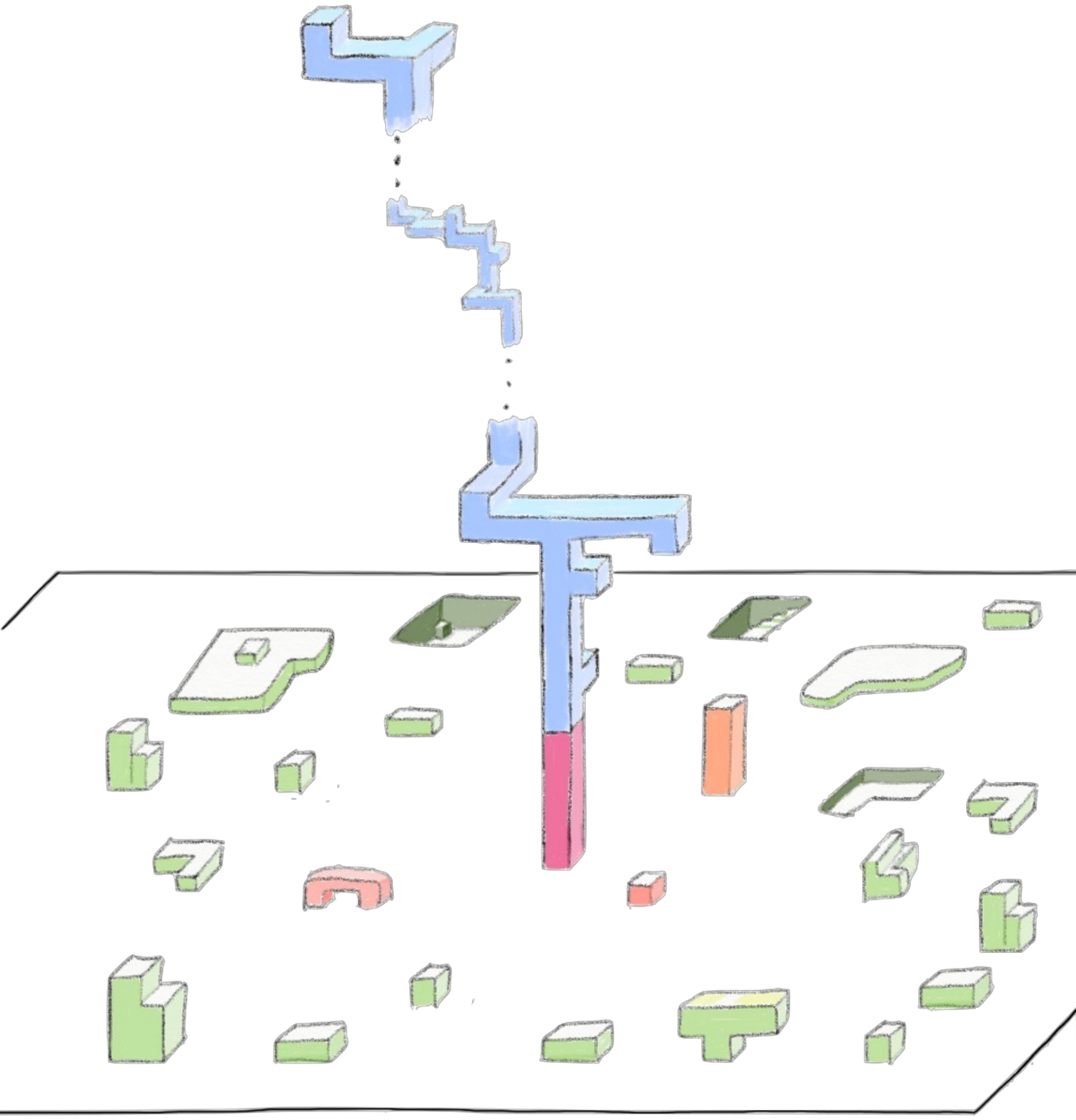}
    };
        \filldraw[draw=white, fill=white] (4.95,-1.45) rectangle (5.35,-.5);
  \end{tikzpicture}
\vspace{-1cm}
 \caption{The base reduction map $\Phi_{\sB}$ sends the interface $\cI$ (left) to $\Phi_{\sB}(\cI)$ (right) by eliminating the nested sequence of walls $\fW_{v_\Tsp}$ (pink, left), along with the vertical column which attains the height $v_\Tsp$ (orange, left), replacing it by a straight vertical column above $x$ (pink, right) and shifting the spine $\cS_x$ (blue, left) to lie above that column.}
  \label{fig:base-map}
\vspace{-0.2cm}
\end{figure}

\subsection{The base reduction map \texorpdfstring{$\Phi_{\mathscr B}$}{Phi}}\label{subsec:def-base-map}
We first define a map that eliminates at least one group of walls of excess area larger than $K\log T$ in $\cC_{v_\Tsp, x, T}$, and in doing so, allows one to lower the height of the source point for the spine. The map also shifts $v_\Tsp$ in the $xy$-plane, to lie above $x$ if  $|\rho(v_\Tsp)- x|\geq K\log T$.  

In order to study the impact of this map on $\mu$, it will help to formulate the base $\sB_x$ in terms of the definitions outlined in Section~\ref{sec:dobrushin-definitions}. In view of this, for any interface $\cI \in \mathbf I_{x,T}$, recall from Definition~\ref{def:truncated-interface} that it has a truncated interface $\cI_{\textsc {tr}}$. We can define the groups of walls corresponding to $\cI_{\textsc {tr}}$ via Lemma~\ref{lem:interface-reconstruction}. 
Fix some $K$ sufficiently large with respect to all other constants that are independent of $\beta$. We refer to Figure~\ref{fig:base-map} for a visualization of the map $\Phi_{\sB}$. 
\begin{definition}
\label{def:base-map}
Let $\Phi_{\sB}: \bar {\mathbf{I}}_{x,T}\to {\mathbf{I}}_{x,T}$
generate an interface $\Phi_{\sB} (\cI)$ from $\cI$ as
follows. Suppose $\cI_{\textsc {tr}}$ has standard wall representation $(W_z)_{z\in \cL_0}$, base $\sB_x$ and source point $v_\Tsp$, and further suppose that its spine $\cS_x$ has increment collection $(X_i)_{\Tsp \leq i\leq T}$ and remainder $X_{>T}$. 
If $\hgt(v_\Tsp)\leq K\log T$ and $\diam (\rho(\sB_x))\leq K\log T$,
then the map is set to be trivial, $\Phi_{\sB} (\cI ) = \cI$. Otherwise, if $\hgt(v_\Tsp)> K\log T$ and/or $\diam (\rho(\sB))>  K\log T$,
then construct $\cJ = \Phi_{\sB}(\cI)$, by 
\begin{enumerate}
\item \label{step:mark-walls}Mark in the standard wall representation $(W_z)_{z\in \cL_0}$ the groups of walls $\fF_{v_\Tsp}$, as well as $\fF_{[x]}=\bigcup_{f\in [x]} \fW_{f}$, where $[x]=\{x\} \cup \bigcup_{f\in \cL_0: f\sim x} \{f\}$.
\item \label{step:mark-extra-walls} If there exists  some $h< \hgt(v_\Tsp)$ such that the interface with standard wall representation $\fW_{v_\Tsp} \cup \fW_{[x]}$ intersects $\cL_h$ in exactly one plus cell, take the largest such height $h^\dagger$ and let $y^\dagger$ be an index in $\cL_0 \cap \cC_{v_\Tsp,x,T}$ of a wall that attains height $h^\dagger$ and is not included in $\fW_{[x]} \cup \fW_{v_\Tsp}$, and then mark~$\fF_{y^\dagger}$. (If $h<\hgt (v_\Tsp)$, such a wall must exist by the definition of the source point $v_\Tsp$.)
\item \label{step:rem-groups-of-walls} Remove all the marked walls, i.e., $\fF_{v_\Tsp}\cup \fF_{[x]}\cup \fF_{y^\dagger}$ from the standard wall representation of $\cI_{\textsc {tr}}$. 
\item Let $\cI'$ be the interface with the resulting standard wall representation, let $\mathbf h-1$ be the height of a highest wall indexed by $\cI'\cap \cC_{v_\Tsp,x,T}$ and let $h_\star := \mathbf h\vee (\hgt(v_{\Tsp})+\frac 12)$.
\item \label{step:add-column}To that new interface, add the new standard wall consisting of the vertical bounding faces of a column of height $\hgt(h_\star)$ above $x$: i.e., the cells $x+(0,0,\ell -\frac 12): \ell=1,\ldots,h_\star$. The resulting interface has a $T$-source point, which we will denote $v_{\Tsp({\cJ})}$ . 
\item \label{step:shift-spine} Shift the spine $\cS_x$ by the vector $x+(0,0,{h_\star} -\frac 12)-v_\Tsp$; i.e., the increment sequence of the new spine $\cS_x(\cJ)$ sourced at $v_{\Tsp({\cJ})}$ will be $h_\star-\frac 12 - \hgt(v_{\Tsp({\cJ})})$ trivial increments, followed by the increment sequence of $\cS_x$. 
\end{enumerate}
\end{definition}

\subsection{Strategy of the map \texorpdfstring{$\Phi_\mathscr B$}{Phi\_B}}\label{subsec:strategy-base-map}
As in~\S\ref{subsec:strategy-increment-map}, to obtain exponential tail bounds conditionally on $\cP_x$ having $T$ increments and/or attaining height $h$, we require that $\Phi_{\mathscr B}$ send $\bar{\mathbf I}_{x,T}\cap \{\hgt(\cP_x)\geq h\}$ to $\mathbf I_{x,T}\cap \{\hgt(\cP)\geq h\}$.

Ideally, the map would replace the base $\mathscr B_x$ with a single column of height $\hgt(v_\Tsp)$ above $x$, and have an energy gain proportional to  $\diam(\rho(\mathscr B_x))$ and $\hgt(v_\Tsp)$. Since the spine of $\cI$ starts at $v_\Tsp$ and not above $x$, we must additionally shift the spine to lie above $x$, so that together with the added column, it forms the new pillar $\cP_x^{\cJ}$. We summarize the role the different steps above play in constructing such a map. 
\begin{enumerate}
    \item Step~\eqref{step:mark-walls} above marks the nested sequence of walls supporting $v_\Tsp$, and therefore attaining $\hgt(v_\Tsp)$, for deletion. It additionally marks the nested sequence of walls of $x$, to clear out space for the spine to be shifted horizontally and reattached above $x$.
    
    \item If we only performed this first step of deletions, and then added back a column of height $\hgt(v_\Tsp)$ above $x$, the map would have an excess area proportional to $\diam (\mathscr B_x)$ but not necessarily to $\hgt(v_\Tsp)$. In particular, if the groups of walls $\fF_{v_{\Tsp}}$ were mostly composed of trivial increments, so that its excess area is not much more than $4\hgt(v_\Tsp)$, the energy gain would not be proportional to $\hgt(v_\Tsp)$. 
    
    In order to obtain an energy gain proportional to $\hgt(v_\Tsp)$, we note that by definition of $v_{\Tsp}$, every height below $\hgt(v_\Tsp)$ must have been intersected by at least two plus sites in $\sigma(\cI)\cap (\cL_0 \cap \cC_{v_\Tsp,x,T})$. Therefore, if $\fF_{v_\Tsp}\cup \fF_{[x]}$ did not contain excess energy larger than, say,  $6\hgt(v_\Tsp)$, there must exist another group of walls indexed by $\cL_0 \cap \cC_{v_\Tsp,x,T}$ whose excess energy is also comparable to $\hgt(v_\Tsp)$. Step~\ref{step:mark-extra-walls} finds one such groups of walls and additionally marks it for deletion.    
    \item The remaining steps~\eqref{step:rem-groups-of-walls}--\eqref{step:shift-spine} then shift the spine to lie above $x$, and reconnect the shifted spine to $x$ via a column of $h_\star$ plus sites. 
    The reason $\hgt(v_\Tsp)$ is possibly increased to height $h_\star$, is the following: even though the map $\Phi_{\mathscr B}$ only deletes walls of $\cI_{\textsc{tr}}$, it could be that the deletion of such a wall \emph{increases}  the height of some $W$ nested in a wall of $\fF_{[x]} \cup \fF_{\rho(v_{\Tsp}}\cup \fF_{y^\dagger}$. The effect of this would be that placing the spine at $x+(0,0,\hgt(v_\Tsp))$ could take the spine very close to the vertical shift of $W$, and their interactions could be large: in order to maintain the separation between the spine and the new truncated interface, we therefore place the spine at the higher height $\mathbf h$. 
\end{enumerate}


\subsection{Properties of the map \texorpdfstring{$\Phi_{\mathscr B}$}{Phi\_B}}\label{subsubsec:properties-base-map} 
With the above remarks in place, we now establish, formally, the fact that $\Phi_{\mathscr B}$ is well-defined, and show that it has an energy gain that is comparable to each of the quantities we prove exponential tails on, in~\eqref{eq:hgt-vsp}--\eqref{eq:base-total-increment-decay}. 

\begin{lem}
\label{lem:base-map-properties}The map $\Phi_{\sB}:\bar{\mathbf I}_{x,T} \to{\mathbf I}_{x,T}$
is well-defined, keeps the height of $\cP_x$ fixed, and,  
$$\fm(\cI; \Phi_{\sB}(\cI)) = \fm(\fF_{[x]}\cup \fF_{v_\Tsp}\cup \fF_{y^\dagger})-4h_\star \geq 2\big((h_\star-1)\vee \diam  (\rho(\sB_x))\big)\vee \frac 13 \sum_{i\leq \Tsp} \fm(X_i)\,,$$
as long as the map $\Phi_{\sB}$ was nontrivial ($\Phi_{\sB}(\cI)\neq \cI$). 
\end{lem}

\begin{proof}
For the fact that $\Phi_\sB$ is well-defined, we observe that by definition of the source-point, for every height  $h<h(v_\Tsp)$, the set of walls indexed by faces in $\cC_{v_\Tsp,x,T}$ intersect that height in at least two plus sites; by Definition~\ref{def:nested-sequence-of-walls}, any face $f\in \cI_{\textsc {tr}}$ is also attained by the interface corresponding to the nested sequence of walls $\fW_{\rho(f)}$. Thus, the sequence(s) described by steps~\eqref{step:mark-walls}--\eqref{step:mark-extra-walls} of Definition~\ref{def:base-map} exist. At step~\eqref{step:add-column}, since all walls projecting onto $[x]$ have been removed, the standard wall being added by the vertical column of plus sites above $x$, maintains the admissibility of the remaining standard wall collection. 

Since $\cI\in \bar{\mathbf I}_{x,T}$, and $\cC_{R_0 T} (x) \subset \cC_{v_{\Tsp(\cI)},x,T}$, 
the resulting $T$-source point $v_{\Tsp(\cJ)}$ has height at most $h_\star$; moreover, when we shift the spine of $\cI$ in~\eqref{step:shift-spine}, the resulting spine will be confined to 
$\cC_{R_0 T} (x)= \cC_{v_{\Tsp({\cJ})}, x,T}$; finally, the new pillar $\cP_x(\Phi_{\sB}(\cI))$ has first $h_\star-1\geq \Tsp-1$ increments that are trivial, followed by $T- \Tsp+1$ increments from $\cS_x$, so it has at least $T$ increments total. Thus, the map yields a valid interface $\Phi_{\sB}(\cI) \in {\mathbf I}_{x,T}$.  

The lower bounds on the excess area follow from the following considerations. Since the marked sequences of walls in items~\eqref{step:mark-walls}--\eqref{step:mark-extra-walls} of Definition~\ref{def:base-map} intersected each height below $\hgt(v_\Tsp)$ in at least two sites, we removed an excess of $6(\hgt(v_\Tsp))-\frac 12)$ vertical faces from $\cI$, and added back at most $4(\hgt(v_\Tsp)+\frac 12)$ faces in step~\eqref{step:add-column}; the excess area $\fm(\cI; \Phi_{\sB}(\cI))$ is at least the difference between these. If $h_\star=\hgt(v_{\Tsp})+\frac 12$, then this implies $\fm(\cI;\Phi_{\mathscr B}(\cI)) \geq 2(h_\star-1)$, as desired. Now suppose otherwise that $h_\star = \mathbf h>\hgt(v_\Tsp)+\frac 12$. This could only have happened if there had been a wall nested inside one of $\fF_{[x]}\cup \fF_{v_\Tsp}\cup \fF_{y^\dagger}$ that was shifted vertically \emph{upward} by at least $\mathbf h -\frac 12 - \hgt(v_{\Tsp})$ upon deletion of $\fF_{[x]}\cup \fF_{v_\Tsp}\cup \fF_{y^\dagger}$ (the maximum height of that wall must have been below $\hgt(v_{\Tsp})$ in $\cI$). For such a vertical shift to be possible, its nesting sequence of walls must have had height at least 
$(\mathbf h-\hgt(v_\Tsp)-\frac 12)$ and therefore excess area at least $9 (\mathbf h-\hgt(v_\Tsp)-\frac 12)$ faces (with the extreme case being a $3\times 3$ column with the nested wall in its center). On the other hand, $4(\mathbf h-\hgt(v_\Tsp)-\frac 12)$ faces were added, so
$\fm(\cI;\Phi_{\mathscr B}(\cI)) \geq 2(h_\star-1)$ still holds. 

Let us turn to the bound w.r.t.\ $\diam (\mathscr B)$. In order for $v_\Tsp$ to be in the pillar of $x$, there must be a wall in $\cI_{\textsc{tr}}$ containing both $x$ and $\rho(v_\Tsp)$ in its interior, which will be marked and removed in item~\eqref{step:mark-walls}  of Definition~\ref{def:base-map}; in fact the maximal nested wall containing both $v_\Tsp$ and $x$ bounding the entirety of the base is marked by item~\eqref{step:mark-walls}, resulting in an excess area of at least $2\diam (\rho(\sB_x)) \geq 2|\rho(v_\Tsp)- x|$. 

Lastly, the fact that $\fm(\cI;\Phi_\sB(\cI))\geq \frac 13 \sum_{i\leq \Tsp} \fm(X_i)$ follows from the facts that all the increments $(\sX_1,\ldots,\sX_{\Tsp})$ are part of the same wall, which contains $v_\Tsp$ in its interior, so it is removed, and replaced by a straight vertical column of the same height. 
\end{proof}

\subsection{Estimating the effect of \texorpdfstring{$\Phi_{\sB}$}{Phi\_B}}
We bound the change in probability under application of the map $\Phi_{\sB}$. For ease of notation, locally in these sections we will simply denote this map by $\Phi= \Phi_{\sB}$. 

\begin{prop}
\label{prop:base-shrinking-1}There
exists $C>0$ such that the following holds for all $\beta>\beta_0$. 
For every $\cI\in\mathbf{\bar{I}}_{x,T}$ with spine increment sequence $(X_i)_{\Tsp\leq i\leq T}$ and remainder $X_{>T}$, 
\begin{align*}
 \Big|\log \frac{\mu_n(\cI)}{\mu_n(\Phi(\cI))}+ \beta \fm\big(\cI; \Phi(\cI)\big)\Big| \leq  C \Big[ \fm(\cI;\Phi(\cI))+ |\cF(X_{>T})| e^{ - \bar c (T+1 - \Tsp)} + \sum_{\Tsp\leq i\leq T} |\cF(X_i)| e^{-\bar c (i- \Tsp)}\Big] \,.
\end{align*}
\end{prop}

\begin{proof}
Suppose that $\cI\in \bar{\mathbf I}_{x,T}$ with truncation $\cI_{\textsc{tr}}$, increment sequence $(X_i)_{i\leq T}$ and remainder $X_{>T}$ and suppose that one of $\hgt(v_\Tsp)$ or $\diam (\rho(\sB_x))$ are at least $K\log T$, as otherwise the inequality is trivially satisfied. Set $\cJ=\Phi(\cI)$
for ease of notation; by Theorem \ref{thm:cluster-expansion},
\begin{align*}
\frac{\mu_n(\cI)}{\mu_n(\cJ)} = \exp\Big(-\beta \fm(\cI;\cJ)+\sum_{f\in\cI}\g(f,\cI)-\sum_{f'\in\cJ}\g(f',\cJ)\Big)\,.
\end{align*}
We wish to bound the absolute difference between the sums above by the right-hand side of Proposition~\ref{prop:base-shrinking-1}. 
We will decompose $\cI$ and $\cJ$ into different subsets
of faces, in order to pair up faces of $\cI$ with faces of $\cJ$ that locally do not feel the effect of $\Phi_{\sB}$. Let 
\[\fF= \fF_{[x]}\cup \fF_{v_\Tsp}\cup \fF_{y^\dagger} = (F_{x_s})_s \cup (F_{\rho(v_{\Tsp})_s})_s \cup (F_{y^\dagger_s})_s\,,
\] be the nested sequences of groups of walls marked in steps~\eqref{step:mark-walls}--\eqref{step:mark-extra-walls} that were eliminated in step~\eqref{step:rem-groups-of-walls} of Definition~\ref{def:base-map} (indexed by $[x],y^\dagger, v_\Tsp \in \cL_0\cap \cC_{v_\Tsp,x,T}$).  Sets $\mathbf E,\mathbf F,\mathbf G \subset \cI \cup \cJ$ will consist of all those faces that were
removed from $\cI$ or added to $\cJ$:
\begin{itemize}
\item Let $\mathbf E$ be the set of all faces in the groups of walls $\fF$; these were removed in step~\eqref{step:rem-groups-of-walls} of $\Phi_{\mathscr B}$. 
\item Let $\mathbf F$ be the set of $f$ in $\cJ$ such that $\rho(f)\in \rho(\cF(\mathbf E))$, added in place of a removed face in $\rho(\fF)$ to ``fill in" the interface.  
\item Let $\mathbf G$ be the set of all other faces added to form $\cJ$, namely the single wall consisting of the bounding faces of a vertical column above $x$ added in step~\eqref{step:add-column} of $\Phi_{\mathscr B}$. 
\end{itemize}
Also, for any $f\in \cI_{\textsc {tr}} \setminus \mathbf E$, we set $\tilde f$ to be the vertical translation of $f$ as governed by the interface corresponding to the remaining walls after $\mathbf E$ have been removed: see Observation~\ref{obs:1-to-1-map-faces}.
Finally, for every $f\in \cS_x= \cS_x (\cI)$, let $\theta_{v_\Tsp,x} f$ be its translation by $x+(0,0,h_\star -\frac 12)-v_\Tsp$. This decomposition allows us to expand, 
\begin{align}
\bigg|\sum_{f\in\cI}\g(f,\cI)-\sum_{f'\in\cJ}\g(f',\cJ)\bigg|\leq & 
\sum_{f\in \mathbf E}|\g(f,\cI)|+
\sum_{f\in \mathbf F}|\g(f,\cJ)|+\sum_{f\in \mathbf G}|\g(f,\cJ)| \nonumber \\
 & + \sum_{f\in \cS_x(\cI)} \Big|\g(f,\cI) - \g(\theta_{v_\Tsp,x} f,\cJ)\Big| +  \sum_{f\in\cI_{\textsc {tr}}\setminus \mathbf E} \Big| \g(f,\cI)- \g(\tilde f,\cJ)\Big|  \,. \label{eq:part-function-splitting-base}
\end{align}
Let us begin with the first three terms, for which crude bounds suffice.
By~\eqref{eq:g-uniform-bound} and Lemma~\ref{lem:base-map-properties}, there is a universal constant $C>0$ such that they are at most  
\begin{align*}
 \bar K \big( |\mathbf E| + |\mathbf F| + |\mathbf G|\big) \leq \bar K \big( 2\fm(\fF_{[x]}\cup \fF_{v_\Tsp}\cup \fF_{y^\dagger})+ \fm(\fF_{[x]}\cup \fF_{v_\Tsp}\cup \fF_{y^\dagger}) + 4h_\star \big) \leq C\fm(\cI;\cJ)\,.
 \end{align*}

Now, let us turn to the fourth term in~\eqref{eq:part-function-splitting-base}, which encodes the contributions from the spine. Since the entire spine $\cS_x (\cI)$ is translated by the same vector $x+(0,0,h_\star -\frac 12)-v_\Tsp$, for every $f\in \cS_x(\cI)$,  the radius $\br(f,\cI; \theta_{v_\Tsp,x}f,\cJ)$ is attained either by a face at height at most $\hgt(v_{\Tsp})-\frac 12$ in $\cI$ or at most $h_\star-1$ in $\cJ$, or by a face outside of $\cC_{v_\Tsp,x,T}$. However, since the increment sequence $(X_i)_{\Tsp\leq i\leq T}, X_{>T}$ is tame, and the height of the pillar is fixed by the map $\Phi_{\mathscr B}$, it must in fact be attained by a face in $\cC_{v_\Tsp,x,T}\supset \cC_{v_{\Tsp({\cJ})},x,T}$ of height at most $\hgt(v_{\Tsp})-\frac 12$ in $\cI$ or $h_\star-1$ in $\cJ$. The contribution from the fourth term in~\eqref{eq:part-function-splitting-base} is at most  
 \begin{align*}
\sum_{\Tsp\leq i \leq T+1} \sum_{f\in \cF(X_i)}   \bar K \exp[-\bar c \br(f,\cI; \theta_{v_\Tsp,x}f,\cJ)] & \leq \sum_{\Tsp\leq i\leq T+1} \sum_{f\in \cF(X_i)} \bar K e^{-\bar c (\hgt(f) - \hgt(v_\Tsp)-\frac 12)} \\
& \leq  \sum_{\Tsp\leq i\leq T+1}  \bar K |\cF(X_i) | e^{-\bar c (i- \Tsp)}\,,
 \end{align*} 
where we again used subscript ``$T+1$" to indicate $>T$ here.
(Notice that the radius $\br (f,\cI; \theta_{v_\Tsp,x}f,\cJ)$ is attained by a face whose height is at most $\hgt(h_\star)$ and if $h_\star \neq\hgt(v_\Tsp)+\frac 12$, then $\theta_{v_\Tsp,x}$ shifts the spine vertically accordingly, so that $\hgt(\theta_{v_\Tsp,x} f)-h_\star=\hgt(\theta_{v_\Tsp,x} f)-\hgt(v_{\Tsp})-\frac 12$.)
 
 It remains to control the contribution from the interactions of the truncated pillar $\cI_{\textsc{tr}}$ with the application of the map $\Phi_{\mathscr B}$. The key idea here is that either they interact through the spine, in which case the contribution is bounded as the above term, or they interact through the groups of walls in $\cI_{\textsc {tr}}$, in which case they are controlled as in the proof of Lemma~\ref{lem:dobrushin-wall-ratio}. To this end, let $\bar{\cJ}_{\textsc{tr}}$ be the image of the truncated interface under steps~\eqref{step:mark-walls}--\eqref{step:rem-groups-of-walls} of $\Phi_{\mathscr B}$, prior to the addition of the faces in $\mathbf G$.
 Then, we can bound the difference 
 $$\sum_{f\in \cI_{\textsc {tr}}, f\notin \mathbf E} \Big| \g(f,\cI)- \g( \tilde f, \cJ)\Big| \leq  \sum_{f\in \cI_{\textsc {tr}}, f\notin \mathbf E} \bar K \exp \big[ - \bar c \br(f,\cI; \tilde f,\cJ)\big]$$
 by noticing that the distance $\br(f,\cI; \tilde f,\cJ)$ is either attained by a face in $\cS_x$, a face in $\theta_{v_\Tsp,x} \cS_x$, a face in the set $\mathbf G$, or is equal to $\br(f,\cI_{\textsc {tr}}; \tilde f, \bar{\cJ}^{tr})$. This lets us bound 
 \begin{align}\label{eq:truncation-interactions}
 \sum_{f\in \cI_{\textsc {tr}} \setminus \mathbf E} \bar K e^{-\bar c \br(f,\cI; \tilde f,\cJ)} \leq & \,\,
 \sum_{f\in \cI_{\textsc {tr}}\setminus \mathbf E} \, \sum_{g\in \cS_x}\bar K  [e^{-\bar c d(f,g)}+ e^{- \bar c d(\tilde f, \theta_{v_\Tsp,x} g)}] + \sum_{f\in \cI_{\textsc {tr}}\setminus \mathbf E} \, \sum_{g\in \mathbf G}\bar K  [e^{- \bar c d(f,g)}+ e^{- \bar c d(\tilde f,g)}] \nonumber \\
 &  + \sum_{f\in \cI_{\textsc {tr}}\setminus \mathbf E} \bar K e^{- \bar c \br(f,\cI_{\textsc {tr}}; \tilde f, \bar{\cJ}_{\textsc{tr}})}\,.
\end{align} 
As argued for $f\in \cS_x$, the first term in the right-hand side of~\eqref{eq:truncation-interactions} can be bounded from above by 
\begin{align*}
 2\bar K \Big[   \sum_{\Tsp \leq i\leq T} \sum_{g\in \cF(X_{i})} \sum_{f:\, d(f,g) \geq i- \Tsp} e^{-\bar c d(f,g)} + & \, \sum_{g\in \cF(X_{>T})} \sum_{f:\,d(f,g)>T+1 - \Tsp}e^{ - \bar c d(f,g)}\Big] \\ 
 & \leq 2\bar C \Big[|\cF(X_{>T})|e^{- \bar c (T+1- \Tsp)} + \sum_{\Tsp \leq i\leq T} |\cF(X_i)| e^{-\bar c (i- \Tsp)}\Big]\,,
\end{align*}
for some universal $\bar C$, where we used that $\cS_x$ is tame and the definition of the source point. The second term in~\eqref{eq:truncation-interactions} is trivially bounded above by 
\begin{align*}
\sum_{f\in\cI_{\textsc {tr}} \setminus \mathbf E}\, \sum_{g\in \mathbf G} \bar K [e^{- \bar c d(f,g)}+ e^{- \bar c d(\tilde f,g)}] \leq 2 \sum_{g\in \mathbf G} \sum_{f\in \cF(\Z^3)} \bar K e^{- \bar c d(g,f)} \leq 2 \bar K  \bar C |\mathbf G|\,. 
\end{align*}
By Lemma~\ref{lem:base-map-properties}, we have that $|\mathbf G|= 4h_\star \leq  2\fm(\cI; \cJ)$. 
Finally, we bound the last term of~\eqref{eq:truncation-interactions} as in the proof of Lemma~\ref{lem:dobrushin-wall-ratio} and~\ref{lem:increment-height-equivalence}: by construction, for every $f, \tilde f$ pair, the distance $\br(f,\cI_{\textsc {tr}}; \tilde f, \bar {\cJ}^{tr})$ is attained by the distance to a wall face, and therefore, moving to the distance between projections,
\begin{align*}
\sum_{f\in \cI_{\textsc {tr}}\setminus \mathbf E} \bar K \max_{u\in \rho(\fF)} \exp [- \bar c d(\rho(f),u)] & \leq \sum_{u'\in \rho(\fF)^c} \bar K N_\rho(u) \max_{u\in \rho(\fF)} \exp[-\bar c d(u,u')] \\ 
& \leq \sum_{u' \in \rho(\fF)^c} \sum_{u\in \rho(\fF)} \bar K (|u-u'|^2 + 1) \exp [ - \bar c |u-u'|]\,.
\end{align*}
By integrability of exponential tails, this is at most $\bar C|\cF(\rho(\fF))|+\bar C|\cE(\rho(\fF))| \leq 2\bar C\fm(\cI; \cJ)$ for some universal constant $C$. Combining all the above estimates concludes the proof. 
\end{proof}

\subsection{Bounding the multiplicity of \texorpdfstring{$\Phi_{\sB}$}{Phi\_B}}

Here, we bound the multiplicity of the map $\Phi_{\mathscr B}$. This
is where we use the fact that the nested sequences of groups of walls we eliminated had excess
area at least $ K \log T$. 
\begin{prop}
\label{prop:base-shrinking-multiplicity} There exists $s$  independent
of $\beta$ such that for every $T$, every $\cJ\in\Phi_{\sB}(\bar{\mathbf I}_{x,T})$
and every $k$, 
$$\big|\{\cI \in \Phi_{\sB}^{-1}(\cJ): \fm(\cI ; \cJ)= k\}\big| \leq s^k\,.$$
\end{prop}
\begin{proof}
If $k\leq K\log T$, the map $\Phi_{\mathscr B}$ must be the identity map and we must have $k=0$, so the bound is trivially satisfied; therefore, suppose $k\geq K\log T$.
In order to bound the multiplicity of the map, we will uniquely identify any pre-image $\cI$ with several collections of admissible groups of walls, indicating the nested sequence(s) of walls that are marked in steps~\eqref{step:mark-walls}--\eqref{step:mark-extra-walls} of Definition~\ref{def:base-map}, along with their groups of walls. The requirement of $k\geq K\log T$ will allow us to pick the centers of the nested sequence of walls from step~\eqref{step:mark-walls}, amongst the faces in $\cL_0$ that were in the cylinder $\cC_{v_\Tsp,x,T}$.  

\begin{claim}\label{claim:identifying-I-from-J}
 Given $\cJ\in \Phi_{\mathscr B}(\bar{\mathbf I}_{x,T})$, one can uniquely identify $\cI\in \Phi_{\mathscr B}^{-1}(\cJ)$ from
\begin{enumerate} 
\item the site $v_\Tsp$,
\item  the groups of nested walls $\fF^1 = \fF_{[x]}$, $\fF^2 = \fF_{v_\Tsp}$,
\item a groups of nested walls $\fF^3$ which is either empty if $\fF^1\cup \fF^2$ intersect every height below $h(v_\Tsp)$ in more than one cell, or $\fF_{y^\dagger}$ for some $y^\dagger\in \cC_{v_\Tsp,x,T}$.
\end{enumerate}
\end{claim}
\begin{proof}[\textbf{\emph{Proof of Claim~\ref{claim:identifying-I-from-J}}}]To prove the claim, we reconstruct $\cI$ given this collection. First, in order to read-off $h_\star$, we need to read $\mathbf h$ by removing the pillar $\cP_x(\cJ)$ from $\cJ$ and finding the height of a highest wall in $\mathcal C_{v_\Tsp,x,T}$. With $h_\star$ in hand, take the interface $\cJ$, and remove the set of faces $\mathbf G$ from it by eliminating the column wall above $x$ up to height $h_\star$. This leaves a truncated interface $\bar{\cJ}_{\textsc{tr}}$ along with a spine $\cS$. By construction, this spine is exactly the spine $\cS_x$ of $\cI$ up to the translation $v_\Tsp - x-(0,0,h_\star -\frac{1}{2})$. The truncated interface $\bar{\cJ}_{\textsc{tr}}$ has a standard wall representation, to which we can add all standard walls in $\fF^1 \cup \fF^2\cup \fF^3$. The resulting collection of standard walls is admissible and can then be mapped back to an interface by Lemma~\ref{lem:interface-reconstruction}, which is exactly the truncated interface $\cI_{\textsc {tr}}$. Appending the spine $\cS_x$ at $v_\Tsp$ yields $\cI$. 
\end{proof}  

With the claim in hand, we begin by enumerating the number of choices we have for $v_\Tsp$; since $\fm(\cI;\cJ) \geq 2\hgt(v_\Tsp)$ and $2|\rho(v_\Tsp)- x|$ by Lemma~\ref{lem:base-map-properties}, there are at most $k^3$ possible choices of $v_\Tsp$.

We now wish to bound the number of possible pairs of (a) collections of groups of nested walls $\fF_{[x]}$ and $\fF_{v_\Tsp}$, and (b) collections of groups of walls corresponding to the nested walls $\fF_{y^\dagger}$, indexed by some face $y^\dagger$ in $\cL_0 \cap \cC_{v_\Tsp,x,T}$. 
Take the at most three sequences of nested walls identified by steps~\eqref{step:mark-walls}--\eqref{step:mark-extra-walls} in Definition~\ref{def:base-map} of $\cI$, denote them by $(W^{a}_{i})_i$ with groups of walls $\fF^a = (F^{a}_{j})_j$ for $a\in \{1,2,3\}$ (so that $W^{a}_{i}$ is nested in $W^{a}_{i+1}$). One can generate a $*$-connected set of faces out of each such sequence as follows: 
\begin{enumerate}
\item Assign to each point in $u\in \rho(\bigcup_j F_j^{a})$ the set $R_u$ of faces in $\cL_0$ a distance at most $\sqrt {N_\rho(u)}$ from $u$, 
\item For every wall $W_{i}^{a}$ nested in $W_{i+1}^{a}$, assign to it the set $R_{i+1}^{a}$ of minimal collection of faces in $\cL_0$ connecting $W_i^a$ to $W_{i+1}^{a}$. In the case of $\fF_{[x]}$ (resp., $\fF_{v_{\Tsp}}$, or $\fF_{y^\dagger}$) include also the faces of $[x]$ resp., ($\rho(v_{\Tsp})$ and $y^\dagger$) and connect them via a shortest path of faces $R_1^a$ to $W_1^a$.
\end{enumerate}
The union of the groups of walls $\fF^{a}$, along with the face sets $\bigcup_{u\in \rho(\fF^a)} R_u$ and $\bigcup_i R_i^{a}$ is a $*$-connected set of faces by the definition of groups of walls, and by construction. Moreover, given this union, one can recover the set $\fF^a$ because any face $f\in \cL_0$ in this union is in $\fF^a$ if and only if another face in the union projects onto it (otherwise it couldn't be a wall face). The cardinality of this union of faces is bounded above by  
\begin{align*}
\sum_{a\in \{1,2,3\}} \Big[|\fF^a| + \sum_{u\in \rho(\fF^a)} |R^{a}_u|+ \sum_i |R^{a}_i|\Big]\,.  
\end{align*}
By Remark~\ref{rem:excess-area-properties}, $|\fF^a|\leq 2 \fm(\fF^a)$ for each $i,a$. By construction, and the nesting of walls, $\sum_i |R_i^a| \leq  \sum_i \fm(W_{i}^a)\leq \fm(\fF^a)$. Finally, by definition of groups of walls, 
\begin{align*}
\sum_i \sum_{u\in \rho(F_i^a)} |R_u^a| \leq \sum_i \sum_{u\in \rho(F_i^a)} N_\rho(u)\leq 2 \fm(\fF^a)\,.
\end{align*}
Because $\fm(\fF^a) \leq 2\fm(\cI;\cJ)$, 
we see that for each $a\in \{1,2,3\}$, the union described above is a connected collection of at most $12\fm(\cI; \cJ)$ faces rooted at some specific face ($x$ or $\rho(v_\Tsp)$ in the cases $a=1,2$). Therefore, the number of possible such collections of groups of walls of nested walls associated to $\cI \in \Phi^{-1}(\cJ): \fm(\cI; \cJ) = k$ is bounded as follows: pick an origin $y^{\dagger}\in \cL_0 \cap \cC_{v_\Tsp,x,T}$, pick $0\leq k_1, k_2,k_3 \leq k$ and then finally, to each of $[x], v_\Tsp, y^\dagger$, associate a connected group of faces of size at mo
st $k_a$. The number of choices of origin $y^\dagger$ is at most the size of $\cL_0 \cap \cC_{v_\Tsp,x,T}$, which is at most $(R_0 T+ k)^2$. The number of total such choices is then easily seen to be at most $(R_0 T+k)^2 k^3 s^{k}$, which is at most exponential in $k$ as long as $k\geq K\log T$ for some large $K$ to make the $O(T^2)$ term negligible.   
\end{proof}

\begin{proof}[\textbf{\emph{Proof of Proposition~\ref{prop:base-exp-tail}}}]
By Lemma~\ref{lem:base-map-properties}, the event that $\cI$ has $\hgt(v_\Tsp) \geq r$ implies that $\fm(\cI;\Phi_{\sB}(\cI))\geq r$, and similarly, the event that $\diam (\sB_x)\geq r$ implies that $\fm(\cI; \Phi_{\sB}(\cI))\geq r$. As such, let us fix an $r\geq K\log T$; for ease of notation, let us denote, for the rest of this section, 
\begin{align*}
\Gamma_r = \Big\{\cI\in \bar{\mathbf I}_{x,T}: |\cF(\sX_{>T})|e^{- \bar c (T+1 - \Tsp)} + \sum_{\Tsp \leq i\leq T} |\cF(\sX_i)|e^{-\bar c (i- \Tsp)} < r\Big\}
\end{align*}
Then we can write    
\begin{align*}
\mu_n \big( \fm(\cI; \Phi(\cI))\geq r  \mid \bar{\mathbf I}_{x,T}\big) \leq  \mu_n \Big( \fm(\cI; \Phi(\cI))\geq r, \Gamma_r^c \mid \bar{\mathbf I}_{x,T}\Big) + \mu_n \Big( \fm(\cI; \Phi(\cI))\geq r, \Gamma_r \mid \bar{\mathbf I}_{x,T}\Big)\,.
\end{align*}
By Corollary~\ref{cor:increment-interaction-bound}, the first quantity on the right-hand side is at most $C \exp(- \beta r/C)$ for some universal $C$. The latter quantity can be bounded as follows by Propositions~\ref{prop:base-shrinking-1}--\ref{prop:base-shrinking-multiplicity}:
\begin{align*}
\sum_{\substack {\cI\in \bar{\mathbf I}_{x,T}\cap \Gamma_r \\ \fm(\cI; \Phi(\cI)) \geq r}} \mu_n (\cI) & =  \sum_{\cJ\in \Phi(\bar{\mathbf I}_{x,T}\cap \Gamma_r)} \sum_{k\geq r} \sum_{\substack{\cI\in \Phi^{-1}(\cJ) \\ \fm(\cI;\Phi(\cI)) = k}}  \frac{\mu_n (\cI)}{\mu_n (\Phi(\cI))} \mu_n (\cJ) \\ 
& \leq \sum_{\cJ\in \Phi(\bar{\mathbf I}_{x,T}\cap \Gamma_r)} \mu_n (\cJ) \sum_{ k \geq r}  s^k \exp\big[- (\beta - C) k\big]\,.
\end{align*}
for some other universal constant $C$ (where we absorbed the contribution from the increments to the right-hand side of Proposition~\ref{prop:base-shrinking-1}, which was at most an extra $r$, into the $C$). 
By integrability of exponential tails, and the fact that $\mu_n(\Phi(\bar{\mathbf I}_{x,T}\cap \Gamma_r)) \leq \mu_n(\mathbf I_{x,T}) \leq 2\mu_n(\bar{\mathbf I}_{x,T})$ by Lemma~\ref{lem:tame}, we have 
\begin{align*}
\mu_n \big( \fm(\cI; \Phi(\cI))\geq r \mid \bar{\mathbf I}_{x,T}\big) \leq 2e^{ - (\beta - C)r} \mu_n(\bar{\mathbf I}_{x,T})\,.
\end{align*}
Dividing both sides by $\mu_n(\bar{\mathbf I}_{x,T})$ concludes the proof.
\end{proof}

As we did in Proposition~\ref{prop:increment-conditional-on-height}, since the map $\Phi_{\sB}$ keeps the height of the pillar $\cP_x$ fixed, we can also prove the estimates of Proposition~\ref{prop:base-exp-tail}.   

\begin{proposition}\label{prop:base-conditional-on-height}
There exists $K,c>0$ such that for every $\beta>\beta_0$, we have for every $r\geq K\log h$, every~$T\in \llb \frac h2, h\rrb$,
\begin{align*}
\mu_n (\hgt(v_\Tsp) & \geq r \mid \hgt(\cP_x) \geq r, \bar{\mathbf I}_{x,T}) \leq \exp(- c\beta r) \,, \\ 
\mu_n (\diam (\rho(\sB_x)) & \geq r \mid \hgt(\cP_x) \geq r, \bar{\mathbf I}_{x,T}) \leq \exp(- c\beta r)\,.
\end{align*}
\end{proposition}

\begin{proof}
The proof is again analogous to the proof of Proposition~\ref{prop:increment-conditional-on-height} and we therefore do not include all the details. For any $h$ and any $T\leq h$, we can expand as above,
\begin{align*}
\mu_n \big( \fm(\cI; \Phi(\cI))\geq r \mid \hgt(\cP_x)\geq h, \bar{\mathbf I}_{x,T}\big) &   \leq  \mu_n \Big( \fm(\cI; \Phi(\cI))\geq r, \Gamma_r^c \mid \hgt(\cP_x)\geq h, \bar{\mathbf I}_{x,T}\Big) \\ 
& \quad + \mu_n \Big(\fm(\cI; \Phi(\cI))\geq r, \Gamma_r \mid \hgt(\cP_x)\geq h, \bar{\mathbf I}_{x,T}\Big)\,.
\end{align*}  
The same estimate on the first term on the right-hand side holds from the conditional estimate~\eqref{eq:increment-conditional-on-height} of Proposition~\ref{prop:increment-conditional-on-height}, and the analogue of Corollary~\ref{cor:increment-interaction-bound} under the measure that also conditions on $\{\hgt(\cP_x)\geq h\}$, by taking a supremum over all truncated interfaces $\cI_{\textsc {tr}}$, and noting that the exponential tails on spine increments are uniform in $\cI_{\textsc{tr}}$. The second term on the right-hand side, we also bound as in the proof of Proposition~\ref{prop:base-shrinking-1}, summing only over interfaces that also have the property that $\{\hgt(\cP_x)\geq h\}$, and using that $\mu_n(\Phi(\bar{\mathbf I}_{x,T}))\leq \mu_n(\hgt(\cP_x)\geq h, \mathbf I_{x,T}) \leq 2\mu_n(\hgt(\cP_x)\geq h, \bar{\mathbf I}_{x,T})$ as long as $T\geq \frac {h}2$ by~\eqref{eq:tame-conditional-on-height} of Remark~\ref{rem:tame-conditional-on-height}.  Following the rest of the proof with these modifications yields the desired estimates.  
\end{proof}

\section{Large deviation rate and law of large numbers for the maximum}\label{sec:ldp-lln}
In this section, we use the results of Sections~\ref{sec:increment-prelim}--\ref{sec:base}, to prove Theorem~\ref{mainthm:max-lln}. We begin, in Section~\ref{subsec:decorrelation}, with a rough equivalence between pillars and groups of walls, and recall early decorrelation estimates of Dobrushin~\cite{Dobrushin73} for groups of walls in the bulk of $\Lambda_n$. In Section~\ref{subsec:limiting-ldp-rate}, we show the existence of the limiting large deviation rate for the event $\{\hgt(\cP_x)\geq h\}$ and relate it to an infinite-volume large deviation rate under the measure $\mu_{\Z^3}^{\mp}$. The key estimate there will be the following: 

\begin{proposition}\label{prop:limiting-ldp-rate-pillar}
The limit $\alpha_\beta$ given by~\eqref{eq:alpha-beta-def} exists and moreover, there exists $\beta_0$ such that for all $\beta>\beta_0$, for every sequence $h=h_n$ such that $1\ll h \ll n$ and every $x= x_{n} \in \cL_0\cap \Lambda_{n_h}$ such that $d(x_{n}, \partial \Lambda_{n})\gg h_n$, 
\begin{align}\label{eq:ldp-rate}
\lim_{n\to\infty} -\frac{1}{h_n} \log \mu_{n}\big( \hgt(\cP_{x_{n}})\geq h_n \big)& =  \alpha_\beta\,.
\end{align}
As a consequence, the quantity $\alpha_\beta \in [4\beta - C, 4\beta + e^{ - 4\beta}]$ for a universal constant $C$. 
\end{proposition}

\noindent In Section~\ref{subsec:max-lln}, we use the decorrelation estimates for pillars and the existence of this large deviation rate to show that the maximum height of an interface satisfies a law of large numbers. 

\subsection{Decorrelation estimates for groups of walls and pillars}\label{subsec:decorrelation}
 In this section, we use the decomposition of pillars into a base and a spine, and in particular, the exponential tail on the size of the base proved in Section~\ref{sec:base}, to show that the structure of a  pillar is, with high probability, captured by the groups of walls indexed by faces within a $o(T)$ neighborhood of $x$. We use this to translate decorrelation estimates for groups of walls into decorrelation estimates for pillars.

The following is then an immediate corollary of Eq.~\eqref{eq:nested-seq-group-of-wall-exp-tail}. 
\begin{proposition}\label{prop:pillar-groups-of-walls}
With $\mu_n$-probability $1-O(e^{-c\beta r})$, the nested sequence of walls $\fW_x$ is indexed by faces a distance at most $r$ from $x$, (and therefore so are all walls nested in a wall of $\fW_x$). 
\end{proposition}

With these equivalences in mind, we recall some decorrelation estimates for groups of walls proved by Dobrushin in~\cite{Dobrushin72b,Dobrushin73}. The first of these says that the dependence of the law of a group of walls $F_x$ on the containing box size $n$ decays exponentially fast in the distance between $x$ and $\partial \Lambda_n$. When combined with Proposition~\ref{prop:pillar-groups-of-walls}, this will imply that the law of the pillar above a face in $\cL_0$ (approximately) does not depend on the side-length $n$ or on the position of $x$, as long as $x$ is sufficiently far from $\partial \Lambda_n$.

\begin{proposition}[{\cite{Dobrushin72b}, \cite[Lemma 5]{Dobrushin73}, as well as \cite[Prop.~2.3]{BLP79b}}]\label{prop:dependence-on-volume}
There is a $C>0$ such that for every $\beta>\beta_0$, every $n\leq m$, for a sequence of $x=x_n \in \cL_0 \cap \Lambda_n$, 
\begin{align*}
\big\|\mu_n\big((\sF_y)_{y:|y-x|\leq r} \in \cdot \big) - \mu_m\big((\sF_y)_{y:|y-x|\leq r} \in \cdot \big)\big\|_{\tv} \leq C \exp \big( - (d(x,\partial \Lambda_n)-r)/C\big)\,.
\end{align*}
In particular, sending $m$ to $\infty$, and using tightness of $(\sF_y)_{y}$, this estimate holds if we replace $\mu_m$ by $\mu_{\Z^3}^{\mp}$. 
\end{proposition}

\begin{corollary}\label{cor:pillar-dependence-on-volume}
There is a $C>0$ such that for every $\beta>\beta_0$, every $n, m$ and two sequences $x = x_n$ and $y=y_m$ such that $d(x,\partial \Lambda_n) \wedge d(y,\partial \Lambda_m) \geq r$,  
\begin{align*}
\|\mu_n \big(\cP_x \in \cdot\big) - \mu_m \big(\cP_y \in \cdot \big) \|_{\tv} \leq C \exp [-r/C]\,.
\end{align*}
\end{corollary}
\begin{proof}
For any interface, with a standard wall representation
$(F_y)_{y\in \cL_0}$, we can set $\cI^{(R),x}$ which is the interface having only groups of walls indexed by $y:|y-x|<R$ and let $\cP_x^{(R)}$ be the pillar of $x$ in the interface $\cI^{(R),x}$.  By Observation~\ref{obs:pillar-wall} and Proposition~\ref{prop:pillar-groups-of-walls}, with probability $1-O(e^{-c\beta R})$, the pillars $\cP_x$ and $\cP_x^{(R)}$ are equal. 
Take an $N$ large which we will send to infinity, and expand the difference 
\begin{align*}
\|\mu_n(\cP_x\in \cdot) - \mu_m (\cP_y\in \cdot)\|_{\tv} & \leq \|\mu_n(\cP_x\in \cdot) - \mu_{N} (\cP_x\in \cdot)\|_{\tv} + \|\mu_{N}(\cP_x \in \cdot) - \mu_N(\cP_y\in \cdot)\|_\tv \\ 
& \qquad + \|\mu_m (\cP_y\in \cdot ) - \mu_N(\cP_y\in \cdot)\|_{\tv}\,.
\end{align*}
The first term above is bounded as follows: there exists $C>0$ such that for all $\beta>\beta_0$, 
\begin{align*}
\big\|\mu_n \big((\sF_y)_{y:|y-x|\leq r} \in \cdot\big)   - \mu_N \big((\sF_y)_{y:|y-x|\leq r}\in \cdot\big) \big\|_{\tv} + \mu_n\big( \cP_x \neq \cP_x^{(r)}\big) &  + \mu_N \big(\cP_x\neq \cP_x^{(r)} \big) \\
& \leq   Ce^{- (d(x,\partial \Lambda_n) - r)/C}+ e^{-\beta r/C}\,,
\end{align*}
as if $\cP_x$ is contained in the ball of radius $r$ around $x$, then the pillar $\cP_x$ is a marginal of the collection of groups of walls $(\sF_y)_{y:|y-x|\leq r}$. The third term is bounded analogously.  
In order to bound the second term, 
\begin{align*}
\|\mu_N(\cP_x\in \cdot) - \mu_N(\cP_y\in \cdot)\|_{\tv} \leq \|\mu_N(\cP^{(R)}_x\in \cdot) - \mu_N(\cP^{(R)}_y \in \cdot) \|_\tv + 2e^{ - R\beta /C}\,.
\end{align*}  
Taking $N \to \infty$, the first term on the right-hand side here vanishes as the infinite-volume measure $\mu^{\mp}_{\Z^3}$ is invariant under translations in the $xy$-directions~\cite{Dobrushin73}. Sending $N\to\infty$ first, then $R\to\infty$, and replacing $r$ by say $2r$, we obtain the desired inequality. 
\end{proof}

We also mention a result of Dobrushin showing that groups of walls decorrelate exponentially fast in their distance. That they decorrelate exponentially fast conditionally on the other groups of walls of the interface follows relatively straightforwardly from the cluster expansion and definition of groups of walls---however, a powerful bound of Dobrushin from~\cite{Dobrushin68,Dobrushin73} allows one to translate conditional decorrelation estimates for random fields to unconditional ones. This estimate, together with the equivalence of groups of walls and pillars, greatly simplifies the second moment estimate in Section~\ref{subsec:max-lln}. 

\begin{proposition}[{\cite{Dobrushin73}, see also~\cite[Proposition 2.1]{BLP79b}}]\label{prop:group-of-wall-correlations}
There is a $C>0$ such that for every $\beta>\beta_0$, every $n$ and two sequences $x= x_n$ and $y=y_n$, 
\begin{align*}
&\left\|\mu_n \left((\sF_s)_{|s-x|<r}\in\cdot, (\sF_t)_{|t-y|<r} \in\cdot\right )  - \mu_n \big((\sF_s)_{|s-x|<r}\in \cdot\big)  \mu_n \big((\sF_t)_{|t-y|<r} \in \cdot \big) \right\|_\tv  \leq C e^{-(|x-y|-2r)/C}\,.
\end{align*}
\end{proposition}  

\begin{corollary}\label{cor:pillar-correlations}
There is a $C>0$ such that for every $\beta>\beta_0$, every $n$ and every two sequences $x= x_n$ and $y=y_n$ such that $d(x,y) \geq r$, we have 
\begin{align*}
\|\mu_n (\cP_x\in \cdot, \cP_y \in \cdot) - \mu_n (\cP_x \in \cdot)\mu_n (\cP_y\in \cdot) \|_\tv \leq  C\exp [- r/C]\,.
\end{align*}
\end{corollary}
\begin{proof}
Fix an $r$, recall the definition of $\cP_x^{(r)}$ and $\cP_y^{(r)}$, and use the shorthand $(\sF_s)$ and $(\sF_t)$ for  $(\sF_s)_{|s-x|<r}$  and $(\sF_t)_{|t-y|<r}$. 
Then, 
\begin{align*}
\|\mu_n(\cP_x \in \cdot, \cP_y\in \cdot) - \mu_n(\cP_x \in \cdot) \mu_n(\cP_y\in \cdot)\|_{\tv} & \leq \|\mu_n(\cP_x \in \cdot, \cP_y\in \cdot) - \mu_n ( \cP_x^{(r)}\in \cdot, \cP_y^{(r)}\in \cdot)\|_\tv \\ 
& \quad + \|\mu_{n}((\sF_s)\in \cdot, (\sF_t)\in \cdot) -\mu_n((\sF_s)\in \cdot)\mu_n((\sF_t)\in \cdot)\|_\tv\\ 
& \quad + \|\mu_n(\cP^{(r)}_x \in \cdot)\mu_n(\cP^{(r)}_y\in \cdot) - \mu_n(\cP_x\in \cdot)\mu_n(\cP_y\in \cdot)\|_{\tv}
\end{align*}
where the second term is as it is because $\cP_x^{(r)}$ is a marginal of $(\sF_s)$ and $\cP_y^{(r)}$ is a marginal of $(\sF_t)$. The second term above, then, is exactly the quantity bounded by Proposition~\ref{prop:group-of-wall-correlations}. The first and third terms are bounded by $\exp(- r/C)$ by Observation~\ref{obs:pillar-wall} and Proposition~\ref{prop:pillar-groups-of-walls}, yielding the desired. 
\end{proof}

\subsection{Limiting large deviation rate}\label{subsec:limiting-ldp-rate}
In this section, we use an approximate sub-additivity argument to demonstrate the existence of a limiting large deviation rate for the probability that $\hgt(\cP_x)$ exceeds $h$ as $h\to\infty$. We will first show how Proposition~\ref{prop:limiting-ldp-rate-pillar} follows from the existence of the limit in~\eqref{eq:alpha-beta-def}; we then prove the existence of the limit in~\eqref{eq:alpha-beta-def} leveraging the fact that connection events are increasing, to use the monotonicity and FKG property of the Ising model. 

Let us begin by proving Proposition~\ref{prop:limiting-ldp-rate-pillar} given the existence of the limit in~\eqref{eq:alpha-beta-def}. Without loss, we will change from sequences indexed by $n$ to sequences indexed by $h$, so that $n_h$ is any sequence having $n_h\gg h$ and $x_{h}$ is such that $d(x_h, \partial \Lambda_{n_h})\gg h$. Recall that $x\xleftrightarrow[A]{+} y$ denotes that there is a $*$-connected path of $+$ sites in $\mathcal C(A)$ between $x$ and $y$. 
Let us denote by $A_h$ the event, measurable with respect to the configuration on $\cC(\Z^2 \times \llb 0,h\rrb)$, defined by 
\begin{align*}
A_h = A_h^x = \left\{\sigma: x+(0,0, \tfrac 12) \xleftrightarrow[\Z^2\times\llb 0,h\rrb]{+} \cL_{h-\frac 12} \right\}\,.
\end{align*}
We will show that the limit in~\eqref{eq:ldp-rate} is equal to the following limit 
\begin{align}\label{eq:A_h-limit}
\lim_{h\to\infty} - \frac 1h \log \mu_{n_h} \big(A_{h}^{x_{h}}\big)\,,
\end{align}
which we will show exists and equals the infinite-volume limit $\alpha_\beta$ defined in~\eqref{eq:alpha-beta-def}.

\begin{proof}[\textbf{\emph{Proof of Proposition~\ref{prop:limiting-ldp-rate-pillar}, given existence of~\eqref{eq:alpha-beta-def}}}]
For every $n$ large, every $x\in \cL_0 \cap \Lambda_n$, we claim that we have the comparability of events: there exists $\epsilon_\beta$ vanishing as $\beta\to\infty$ such that
\begin{align}\label{eq:equality-of-events}
(1-\epsilon_\beta)\mu_n(A_h^x) \leq \mu_n(\hgt(\cP_x)\geq h) \leq (1+\epsilon_\beta) \mu_n (A_h^x)\,.
\end{align}
(This indicates that the connectivity event $A_h^x$ serves as a good proxy for the relevant event $\{\hgt(\cP_x)\geq h\}$: refer to Figure~\ref{fig:interface-nonmon} for examples of configurations in $\{\hgt(\cP_x)\geq h\}^c\cap A_h^x$ (left) and $\{\hgt(\cP_x)\geq h\}\cap (A_h^x)^c$.)

Letting $A_h = A_h^x$, on the one hand, by Definitions~\ref{def:pillar}--\ref{def:heights}, we have  
\begin{align*}
A_h \cap \{ \hgt(\cP_x)\geq 0\} \subset \{\hgt(\cP_x)\geq h\}\,;
\end{align*}
since $\fF_x =\emptyset$ implies $\hgt(\cP_x)\geq 0$, by~\eqref{eq:nested-seq-group-of-wall-exp-tail},  $\mu_n(\hgt(\cP_x)\geq 0)\geq 1-\epsilon_\beta$, and the FKG inequality implies the left-hand side of~\eqref{eq:equality-of-events}. 
On the other hand, given $\{\hgt(\cP_x)\geq h\}$, the event $(A_h^x)^c$ implies that there is a (nearest-neighbor) connected component of minuses separating $x+(0,0,\frac 12)$ from the inner boundary of the pillar, and in particular, from height $h$, in the slab $\cC(\Z^2 \times \llb 0,h\rrb)$. If $\hgt(\cP_x)\geq h \geq 1$, this is in the plus phase of the Ising model with interface $\cI$, and thus the probability of such a half-bubble of minuses is at most the probability that $x+(0,0,\frac 12)$ is not $*$-connected by plus sites to $\infty$ in $\mathcal C(\mathbb Z^2 \times \llb 0,\infty\rrb)$ under $\mu^+_{\Z^3}$; this probability is in turn at most $\epsilon_\beta$ by a classical Peierls argument. Thus, we can express
\begin{align*}
\mu_{n}(\hgt(\cP_x)\geq h) & = \mu_n(\hgt(\cP_x)\geq h, A_h^x) + \mu_n(\hgt (\cP_x)\geq h, (A_h^x)^c) \\ 
& \leq \mu_n (A_h^x) + \epsilon_\beta \mu_n(\hgt(\cP_x)\geq h)\,,
\end{align*}
from which the right-hand side of~\eqref{eq:equality-of-events} follows. It remains to show that the limit~\eqref{eq:A_h-limit} is given by $\alpha_\beta$. 

By Corollary~\ref{cor:pillar-dependence-on-volume} and the fact that the distance from $x_{h}$ to the boundary grows faster than $h$, if we show~\eqref{eq:ldp-rate} for one such sequence of $x_{h}$, it implies it for every such sequence (the error $e^{- c d(x_h, \partial \Lambda_{n_h})}$ vanishes after taking a logarithm, dividing by $h$ and sending $h\to\infty$). Now take a fixed $x$, say $(\frac 12, \frac 12, 0)$ and any two sequences $n_h$ and $m_h$ such that $\frac {n_h}h$ and $\frac{m_h}{h}$ go to infinity. By~\eqref{eq:equality-of-events}, and Corollary~\ref{cor:pillar-dependence-on-volume}, the following limits are equal (if they exist), 
\begin{align*}
\lim_{h\to\infty} -\frac 1h \log \mu_{n_h} (A_h) = \lim_{h\to\infty} - \frac 1h \log \mu_{m_h}(A_h)\,,
\end{align*}
and since this holds for every sequence $m_h\gg h$, both are equal to the limit in~\eqref{eq:alpha-beta-def}. Finally, since the upper and lower bounds of Proposition~\ref{prop:rate-upper-lower-bounds} on $\mu_n(\hgt(\cP_x)\geq h)$ hold for all sufficiently large $h$ and are both independent of $n$, it is clear that for every $\beta>\beta_0$, we have $\alpha_\beta \in [4\beta -C, 4\beta + e^{-4\beta}]$. 
\end{proof}

Both~\eqref{eq:alpha-beta-def} and Proposition~\ref{prop:limiting-ldp-rate-pillar} would follow if we show for a fixed $x$, say $(\frac 12, \frac 12, 0)$, and some sequence $n_h\gg h$, that the limit~\eqref{eq:A_h-limit} exists, and call it $\alpha_\beta$. To see the existence of~\eqref{eq:A_h-limit}, we rely on the fact that $A_h$ is an increasing event; we would like to leverage the monotonicity and FKG property of the Ising model to show sub/super multiplicativity of $\mu_{n_h}(A_h)$. The problem with this is that on the one hand, the event of reaching a height $h_1$ gives positive information towards the event of going from height $h_1$ to $h_2$, while on the other hand, the Ising measure $\mu^{\mp}_{n_h}$ near height $h_1$ is much more negative than it is near height $0$. We overcome this by a careful revealing procedure, that exposes the plus connected component of $x+(0,0,\frac 12)$ and controls the amount of positive information obtained by this revealing via the estimates of Sections~\ref{sec:increment-exp-tail}--\ref{sec:base}.

\begin{prop}\label{prop:limiting-rate} For every $\beta>\beta_0$, for every $h_1$ large, and  $h_2 \in \llb \frac12 h_1 , 2 h_1\rrb$ if $x$ is such that $\frac{d(x,\partial \Lambda_{n_h})}h \to\infty$, then
\[
\log \mu_{n_{h_1 + h_2 }}(A_{h_1+h_2}) \leq \log \mu_{n_{h_1}}(A_{h_1}) + \log \mu_{n_{h_2}}(A_{h_2}) + O(\log^2{[h_1 + h_2]})\,.
\]
\end{prop}

Let us first conclude the proof of~\eqref{eq:alpha-beta-def} and in turn, Proposition~\ref{prop:limiting-ldp-rate-pillar}, by applying an approximate version of Fekete's sub-additivity lemma.  

\begin{proof}[\textbf{\emph{Proof of~\eqref{eq:alpha-beta-def} in Theorem~\ref{mainthm:max-lln}}}]
By an approximate version of Fekete's Lemma (\cite[Theorem 23]{deBruijnErdos}, also, see~\cite[Theorem 1.9.2]{Steele97}), since $\int t^{-2} (\log t)^2 dt <\infty$, Proposition~\ref{prop:limiting-rate} implies that 
\begin{align*}
\lim_{h\to\infty} \frac 1h \mu_{n_h} (A_h) = \alpha_\beta
\end{align*}
for some $\alpha_\beta\in [-\infty,\infty]$, and by the above proof of Proposition~\ref{prop:limiting-ldp-rate-pillar}, this $\alpha_\beta$ is also the same limit as in~\eqref{eq:ldp-rate}. As argued above, this implies that $\alpha_\beta \in [4\beta -C, 4\beta + e^{-4\beta}]$ for some universal constant $C$ given in Proposition~\ref{prop:rate-upper-lower-bounds}, so $\alpha_\beta/\beta\to4$ as $\beta\to\infty$. 
\end{proof}

We now turn to proving the approximate sub-additivity of the sequence $(\log \mu_{n_{h}}(A_h))_h$.
 
\begin{proof}[\textbf{\emph{Proof of Proposition~\ref{prop:limiting-rate}}}]
Recall that we may fix $x = (\frac 12, \frac 12, 0)$ and set $A_h = A_h^x$. We will also be interested in the vertical shift of $A_h$, defined by 
\begin{align*}
\theta_{h_1} A_{h_2}= \theta_{h_1} A_{h_2}^x = \left\{ \sigma: x+(0,0,h_1+\tfrac 12)
\xleftrightarrow[\Z^2 \times \llb h_1, h_1+h_2\rrb]{+} \cL_{h_1 + h_2- \tfrac 12} \right\}\,.
\end{align*}
By translation, it is evident that $\mu_{n}^{\mp} (\theta_{h_1} A_{h_2})= \mu_{n}^{\mp, -h_1}(A_{h_2})$, where $(\mp,-m)$ boundary conditions are those that are plus on $\partial \Lambda \cap \cL_{<-m}$ and minus on $\partial \Lambda\cap \cL_{>m}$. For every $n$, by monotonicity in boundary conditions, 
\begin{align*}
\mu_{n}^{\mp, -h_1}(A_{h_2}) \leq \mu_{n} (A_{h_2})\,.
\end{align*}
Finally, 
denote by $\sP$ the $*$-connected plus component of $x+(0,0,\frac 12)$ in $\llb -n_{h_1+h_2},n_{h_1+h_2} \rrb^2 \times \llb 0,h_1\rrb$, and notice, crucially, 
that on the event that $\hgt(\cP_x)\geq 0$, this plus component satisfies $\sP \subset \sigma(\cP_x)$. 

Our goal is to say that conditionally on a connected plus component $\mathscr P$ reaching height $h_1$, the probability of reaching a further height $h_1+h_2$ is at most $\mu_n^{\mp,-h_1}(A_{h_2})\leq \mu_n(A_{h_2})$. This does not hold true, as the fact that $\mathscr P$ reached height $h_1$ contains positive information. We define a set $\Gamma_{x,h_1}$ of plus components $\mathscr P$  which, due to our structural results on tall pillars, has positive probability on the event $A_{h_1+h_2}$, such that for every $\mathscr P \in \Gamma_{x,h_1}$, the positive information obtained from revealing $\mathscr P$ is not too large. 

More precisely, let $\Gamma_{x,h_1}$ be the event that $\sP$ satisfies (for $C, K$ to be chosen sufficiently large later)
\begin{enumerate}
\item its intersection with $\cL_{h_1- \frac 12}$ is at most a single cell,
\item its bounding face-set has size at most $Ch_1$,  
\item its intersection with $\cL_{\frac 12}$ has diameter at most $K\log h_1$. 
\end{enumerate}
(Notice that $\Gamma_{x,h_1}$ is a decreasing event.) 
The proposition will follow from the following two claims. 

\begin{claim}\label{claim:limiting-rate-2}
For every $\beta>\beta_0$, there exist choices of $C, K$ above, such that for every $h_1$ and $h_2 \in \llb \frac 12 h_1 , 2h_1\rrb$ sufficiently large, as long as $\lim_{h\to\infty} \frac{n_h}{h}=\infty$, we have
\begin{align*}
\mu_{n_{h_1+h_2}}(A_{h_1 + h_2}, \Gamma_{x,h_1}) \leq  e^{C\beta (K\log h_1)^2} \mu_{n_{h_1}} (A_{h_1}) \mu_{n_{h_2}} (A_{h_2})\,.
\end{align*}
\end{claim}

\begin{claim}\label{claim:limiting-rate-1}
For every $\beta>\beta_0$, there exists a constant $c>0$ and choices of $C, K$ above, such that for every $h_1$ and $h_2 \in \llb \frac 12 h_1 , 2h_1\rrb$ sufficiently large, as long as $\lim_{h\to\infty} \frac{n_h}{h}=\infty$, we have   
\begin{align*}
\mu_{n_{h_1 + h_2}}(A_{h_1 + h_2}, \Gamma_{x,h_1}) \geq c \mu_{n_{h_1+h_2}} (A_{h_1 + h_2})\,.
\end{align*}
\end{claim}

Clearly, combining Claims~\ref{claim:limiting-rate-2}--\ref{claim:limiting-rate-1} and taking logarithms on both sides concludes the proof. For ease of notation, set $h= h_1 + h_2$.

\medskip
\noindent \textbf{Proof of Claim~\ref{claim:limiting-rate-2}:} Since $h_1$ and $h_2$ are comparable, and $n_h$ is such that it diverges faster than $h$, by Corollary~\ref{cor:pillar-dependence-on-volume} and the equivalence~\eqref{eq:equality-of-events}, as argued before, incurring  errors that are decaying faster than any exponential in $h$, we can switch from $\mu_{n_h}$ to  $\mu_{n_{h_1}}$ and $\mu_{n_{h_2}}$; thus it will suffice for us to show the inequality 
\begin{align*}
\mu_{n_{h}} (A_{h},\Gamma_{x,h_1})\leq e^{C\beta (K\log h_1)^2} \mu_{n_{h}}(A_{h_1})\mu_{n_{h}}(A_{h_2})\,.
\end{align*}
 We begin by using the domain Markov property, the containment  $A_{h_1} \supset A_h$, and the measurability of $A_{h_1}\cap \Gamma_{x,h_1}$, and in particular $\sP$, with respect to $\sigma_{\llb -n_{h}, n_{h}\rrb ^2 \times \llb 0,h_1\rrb}$ to express 
\begin{align}\label{eq:subadditivity-condition-h1}
\mu_{n_h}(A_{h}, \Gamma_{x,h_1}) & = \mu_{n_{h}} (A_{h_1}, \Gamma_{x,h_1}) \E_{\sP} \big[ \mu_{n_h}\big(A_{h} \mid \sP)\mid  A_{h_1} \cap \Gamma_{x,h_1}\big] \nonumber \\ 
& \leq \mu_{n_h}(A_{h_1})  \E_{\sP} \big[ \mu_{n_h}\big(\theta_{h_1} A_{h_2}^Y\mid \sP)\mid  A_{h_1} \cap \Gamma_{x,h_1}\big]\,,
\end{align}
where $Y\in \cL_0$ is the projection of the singleton dictated by item (1) of $\Gamma_{x,h_1}$ when $A_{h_1}\cap \Gamma_{x,h_1}$ occurs. The expectations are with respect to the law of $\sP$ under $\mu_{n_h}$. 

We need to bound the latter term on the right-hand side of~\eqref{eq:subadditivity-condition-h1} by the quantity $e^{C\beta (K\log h_1)^2} \mu_{n_h}(\theta_{h_1} A_{h_2})$ to obtain the claim. We investigate this latter term as follows: notice that since $A_{h_1} \cap \Gamma_{x,h_1}$ are measurable with respect to the plus $*$-connected component $\sP$, we can condition on $\sP \in A_{h_1}\cap \Gamma_{x,h_1}$ by starting from the site at $x+(0,0,\frac 12)$ and only revealing its $*$-connected plus-component in $\Lambda_{n_h,n_h,h_1}\cap \cL_{>0}$. 

This revealing process exposes $\sP$, along with minus vertices along its entire boundary (sites  in $\cC(\mathbb Z^3 \setminus \sP)$ that are $*$-adjacent to $\sP$) inside $\Lambda_{n_h,n_h,h_1}\cap \cL_{>0}$. Let $\sigma(\sP)\subset \Lambda_{n_h,n_h,h_1}\cap \cL_{>0}$ be the set of sites ``interior to" $\sP$, so that if the revealing procedure revealed a finite (nearest-neighbor) connected component of minus spins, corresponding to a minus bubble in $\Lambda_{n_h,n_h, h_1}\cap \cL_{>0}$, set them to plus and continue revealing their interior; in this manner, $\sigma(\sP)$ are the sites which we know to be in the plus phase given $\sP$. 

Let $\sP^{\mp}$ boundary conditions on $\Lambda_{n_h,n_h, \infty}\setminus \sigma(\sP)$ be the $\mp$ boundary conditions that additionally have plus spins in all of $\sigma(\sP)$, and minus spins along the boundary of $\sigma(\sP)$ in $\Lambda_{n_h,n_h,h_1}\cap \cL_{>0}$. By domain Markov, these boundary conditions are equivalent to those that have the same minus spins, but only set $\sigma(\sP) \cap \cL_{\frac 12}$ and $Y+(0,0,h_1 - \frac 12)$ to plus. Then by monotonicity and the FKG inequality, we have that
\begin{align}\label{eq:subadditivity-condition-pillar}
\E_{\sP} [\mu_{n_h}(\theta_{h_1} A_{h_2}^Y \mid \sP) \mid  A_{h_1} \cap \Gamma_{x,h_1}]  \leq  \E_{\sP} \big[\mu_{n_h}^{\sP^{\mp}}(\theta_{h_1} A_{h_2}^{Y})\mid A_{h_1}\cap \Gamma_{x,h_1}\big]\,.
\end{align} 
But then, we are able to express for any such $\sP$ in $A_{h_1}\cap \Gamma_{x,h_1}$, 
\begin{align*}
\mu_{n_h}^{\sP^{\mp}} (\theta_{h_1} A_{h_2}^Y) \leq \mu_{n_h}^{\sP^+}(\theta_{h_1} A_{h_2}^Y)\leq e^{C\beta (|\sigma(\sP)\cap \cL_{\frac 12}|+1)} \mu_{n_h}(\theta_{h_1} A_{h_2}^{Y})
\end{align*}
where $\sP^+$ boundary conditions are $\mp$ boundary conditions that additionally have plus spins in $\sigma(\sP)\cap \cL_{\frac 12}$ and $Y+(0,0, h_1 - \frac 12)$; here, the first inequality was by monotonicity, and the second inequality holds for some universal constant $C$, by application of the \emph{finite energy property} of the Ising model to set all spins at height $\frac 12$ in $\sigma(\sP)$ and the spin at $Y+(0,0,h_1 - \frac 12)$ to  be plus. As noted earlier, for every $Y\in \cL_0$, $$\mu_{n_h}(\theta_{h_1}A^{Y}_{h_2})\leq \mu_{n_h}(A_{h_2}^Y)\,.$$ 

Since $\sP\in \Gamma_{x,h_1}$, the distance $|Y-x|\leq Ch_1$; then deterministically $d(Y,\partial \Lambda_{n_{h}})$ is proportional to $d(x,\partial \Lambda_{n_h})$ so that by the coupling of Corollary~\ref{cor:pillar-dependence-on-volume} and the comparison~\eqref{eq:equality-of-events}, up to an additive error of $\exp[- c d(Y,\partial \Lambda_{n_h})]$, which goes to zero faster than any exponential decay in $h$, we can replace $\mu_{n_h}(A_{h_2}^Y)$ by $\mu_{n_h}(A_{h_2}^x)= \mu_{n_h}(A_{h_2})$. 
Plugging this into~\eqref{eq:subadditivity-condition-h1}, and using the inequality~\eqref{eq:subadditivity-condition-pillar}, we see that 
\begin{align*}
\mu_{n_h}(A_{h}, \Gamma_{x,h_1}) &  \leq  \mu_{n_h}(A_{h_1})\cdot  \mu_{n_h}(A_{h_2})\cdot \E_{\sP} \big[ e^{C\beta (|\sigma(\sP)\cap \cL_{\frac 12}|+1)} \mid A_{h_1} \cap \Gamma_{x,h_1}\big] \\ 
& \leq \mu_{n_h}(A_{h_1}) \mu_{n_h}(A_{h_2}) \sup_{\sP\in A_{h_1}\cap \Gamma_{x,h_1}} \exp\big[{C\beta (|\sigma(\sP)\cap \cL_{\frac 12}|+1)}\big]\,.
\end{align*} 
By definition of $\Gamma_{x,h_1}$, any plus component $\sP$ in $\Gamma_{x,h_1}$ has $\diam (\sP\cap\cL_{\frac 12})\leq K\log h_1$ for some sufficiently large (but independent of other parameters) $K$, so that $|\sigma(\sP)\cap \cL_{\frac 12}|\leq (K\log h_1)^2$, concluding the proof. 

\medskip
\noindent \textbf{Proof of Claim~\ref{claim:limiting-rate-1}:} We wish to lower bound the probability $\mu_{n_{h}}(\Gamma_{x,h_1} \mid A_{h})$. We will use the equivalence~\eqref{eq:equality-of-events} to translate the conditioning on $A_h$ to conditioning on $\{\hgt(\cP_x) \geq h\}$. Let us express, 
\begin{align*}
\mu_{n_h} (\Gamma_{x,h_1}, \hgt(\cP_x)\geq h) & = \mu_{n_h}(\Gamma_{x,h_1} ,  A_{h}, \hgt(\cP_x)\geq h) + \mu_{n_h} (\Gamma_{x,h_1}, A_h^c, \hgt(\cP_x)\geq h)  \\
& \leq \mu_{n_h}(\Gamma_{x,h_1}, A_h) + \mu_{n_h} (\Gamma_{x,h_1} , \hgt(\cP_x)\geq h) \mu_{n_h} (A_h^c\mid \Gamma_{x,h_1}, \hgt(\cP_x) \geq h)\\ 
& \leq \mu_{n_h}(\Gamma_{x,h_1}, A_h) + \mu_{n_h} (\Gamma_{x,h_1} , \hgt(\cP_x)\geq h)\frac{\mu_{n_h} (A_h^c, \hgt(\cP_x)\geq h)}{\mu_{n_h}(\hgt(\cP_x) \geq h, \Gamma_{x,h_1})} \,.
\end{align*}
Assume for the moment that we also have that for every $\delta>0$ there is $h$ large enough such that for the appropriate choice of sufficiently large $C$ and $K$, 
\begin{align}\label{eq:gamma-given-height}
\mu_{n_h} \big( \Gamma_{x,h_1} \mid \hgt (\cP_x) \geq h \big) \geq 1-\delta\,.
\end{align}
Then, using also that $\mu_{n_h} (A_h^c \mid \hgt(\cP_x)\geq h)\leq 1/2$ for $\beta$ large enough (as mentioned above, this is at most $\epsilon_\beta$ by the classical Peierls argument), we would obtain 
\begin{align*}
\mu_{n_h}(\Gamma_{x,h_1}, A_h) & \geq [1- \tfrac{1}{2(1-\delta)}]\mu_{n_h}(\Gamma_{x,h_1},\hgt(\cP_x)\geq h) \geq  (1-\delta)[1-\tfrac1{2(1-\delta)}] \mu_{n_h} (\hgt(\cP_x)\geq h)\\ 
&  \geq \tfrac12(1-\delta-\tfrac12) \mu_{n_h} (A_h)
\end{align*}  
by~\eqref{eq:equality-of-events}. Therefore, it suffices for us to show~\eqref{eq:gamma-given-height}. For the choice of $T= \frac 12 h$, we can express,
\begin{align*}
\mu_{n_h} ( \Gamma_{x,h_1}^c \mid \hgt(\cP_x) \geq h) & \leq \mu_{n_h}(\mathbf I_{x,T}^c  \mid \hgt(\cP_x) \geq h)+ \mu_{n_h}(\bar{\mathbf I}_{x,T}^c  \mid  \mathbf I_{x,T}, \hgt(\cP_x) \geq h) \\\
& \quad + \mu_{n_h}( \Gamma_{x,h_1}^c \mid \bar{\mathbf I}_{x,T}, \hgt(\cP_x) \geq h)\,.
\end{align*}
The first term on the right-hand side is $o(1)$ as $h\to\infty$ as long as $\beta$ is sufficiently large by Lemma~\ref{lem:increment-height-equivalence}. 
The second term is bounded from above by $O(e^{ - c\beta h})$ for some universal $c>0$ by Remark~\ref{rem:tame-conditional-on-height}. 
For the third term, we can union bound by the conditional probabilities of violating each of the three events constituting $\Gamma_{x,h_1}$; moreover, it suffices to bound the corresponding probabilities for $\cP_x$ since $\sigma(\cP_x)\supset \sP$. 

The conditional probability of violating item (2) of the definition of $\Gamma_{x,h_1}$ is simply the probability of the base having surface area at least $(C-2R_0)h_1$, since we are conditioning on the spine being tame; by Proposition~\ref{prop:base-conditional-on-height} and the fact that $h_1$ and $h$ are comparable, this is $O(\exp({ - c\beta h_1^{1/3}}))$ as long as $C$ sufficiently large. The conditional probability of item (3) of $\Gamma_{x,h_1}$ is bounded by the conditional probability of the base having diameter at least $K\log h$ which is also $o(1)$ in $h$ as long as $K$ is large enough, by Proposition~\ref{prop:base-conditional-on-height}. Finally, the conditional probability of item (1) is bounded by the probability of the base having height at least $h_1$, which is again $o(1)$ since $h$ is comparable to $h_1$, or the increment intersecting height $h_1$ in the spine being non-trivial, which is at most $\delta/2$ for $\beta>\beta_0$ by Proposition~\ref{prop:exp-tails-increments}. 

Combining these estimates, one obtains the desired for $\beta>\beta_0$ once $h$ is large enough. 
\end{proof}

\subsection{Law of large numbers for the maximum}\label{subsec:max-lln}
In this section we use Proposition~\ref{prop:limiting-ldp-rate-pillar} to obtain a law of large numbers for the maximum of the 3D interface on $\Lambda_{n,n,\infty}$. The proof follows from a simple second moment method; the fact that  the correlations between large deviations of the pillar above $x$ and $y$ decays exponentially in $|x-y|$, follows from the equivalence between groups of walls and pillars, and the decay of correlations between groups of walls shown in Section~\ref{subsec:decorrelation}.

\begin{proof}[\textbf{\emph{Proof of~\eqref{eq:lln} in Theorem~\ref{mainthm:max-lln}}}]
Fix $\alpha_\beta$ to be that given by~\eqref{eq:ldp-rate} of Theorem~\ref{prop:limiting-ldp-rate-pillar}, equal to ~\eqref{eq:alpha-beta-def}. 
We need to show that for every $\epsilon>0$, 
\begin{align*}
\lim_{n\to\infty} \mu_n \Big(\frac{1}{\log n}\max_{x\in \cL_0} \hgt(\cP_x) \leq \frac{2}{\alpha_\beta} + \epsilon\Big) = 1\,, \quad \mbox{and}\quad 
\lim_{n\to\infty} \mu_n \Big(\frac{1}{\log n} \max_{x\in \cL_0} \hgt(\cP_x) \geq \frac{2}{\alpha_\beta} - \epsilon\Big) = 1\,.
\end{align*}
\noindent \textbf{Upper bound:}
To see an upper bound on the maximum of the interface, we use a union bound as follows: for any two sequences $a_n$ and $K_n$ going to $\infty$ as $n\to\infty$, we can write 
\begin{align*}
\mu_n\big(\max_{x\in \cL_0} \hgt(\cP_x) \geq K_n\big) \leq \sum_{x:\, d(x,\partial \Lambda_n)\leq a_n K_n}\mu_n(\hgt(\cP_x)\geq K_n) + \sum_{x:\, d(x,\partial \Lambda_n)\geq a_n K_n} \mu_n(\hgt(\cP_x)\geq K_n)\,.
\end{align*}
For all of the first summands, we use the estimate of Theorem~\ref{thm:dobrushin-rigidity}, that for every $x\in \cL_0 \cap \Lambda_n$ (including those close to the boundary $\partial \Lambda_n$), the first sum is bounded above by 
\begin{align*}
\sum_{d(x,\partial\Lambda_n)\leq a_n K_n} \mu_{n}(\hgt(\cP_x)\geq K_n) \leq 4a_n  nK_n \exp (-(4\beta - C)  K_n)\,,
\end{align*} 
for some universal constant $C$. 
For the second summands, since $x$ is such that $\frac{d(x,\Lambda_{n})}{K_n}\geq a_n$ which goes to infinity as $n\to\infty$, the conditions of~\eqref{eq:ldp-rate} of Proposition~\ref{prop:limiting-ldp-rate-pillar} are met, so as long as $a_n K_n \ll n$, 
\begin{align*}
\sum_{x:\, d(x,\partial \Lambda_n)\geq a_n K_n} \mu_n(\hgt(\cP_x)\geq K_n) \leq (n-a_n K_n)^2 \exp(- \alpha_\beta K_n + T_n)\,,
\end{align*}  
for some sequence $T_n = o(K_n)$. 
Taking $$K_n = \frac{2}{\alpha_\beta}\log n+\kappa_n\,,$$ for some sequence $\kappa_n = o(K_n)$ to be chosen subsequently in terms of $T_n$, we see that 
\begin{align*}
\mu_n\Big(\max_{x\in \cL_0} \hgt(\cP_x) \geq K_n\Big) \leq C a_n n \log n e^{-\frac{2(4\beta -  C)}{\alpha_\beta} \log n} + n^2 e^{- 2\log n - \alpha_\beta \kappa_n+ T_n}\,.
\end{align*}
Since $\alpha_\beta  - 4\beta \in [-  C, e^{-4\beta}]$, as long as $\beta$ is sufficiently large, we have that $\frac{2(4\beta -  C)}{\alpha_\beta} >1$, in which case the first term is $o(1)$ so long as, say, $a_n< n^{\frac {4\beta -  C}{\alpha_\beta}- \frac 12}$. At the same time, if we take $\kappa_n$ proportional to $T_n$ such that $\frac{\kappa_n}{T_n} >\frac{1}{\alpha_\beta}$ uniformly in $n$, the latter term is also $o(1)$. Since $T_n=o(K_n)$, also $\kappa_n = o(K_n)$, so that $K_n \leq (\frac{2}{\alpha_\beta} + \epsilon)\log n$ for every $\epsilon>0$ for large enough $n$.

\medskip
\noindent \textbf{Lower bound:} In order to obtain the matching lower bound, we use an easy second moment argument.   
Fix any small $\epsilon>0$ and take 
$$ K_n = \Big(\frac{2}{\alpha_\beta} - \epsilon \Big) \log n\,.$$
Now, begin by defining the subset of faces in $\cL_0\cap \Lambda_n$, 
\begin{align*}
\overline {\cL}_0 = \big\{(x_1,x_2,0)\in \cL_0: x_1= \tfrac 12 + i \lfloor K_n^3\rfloor , x_2= \tfrac 12 + j \lfloor K_n^3\rfloor \mbox{ where } (i,j)\in \llb 1- \tfrac{n}{2K_n^3}, \tfrac n{2K_n^3} -1\rrb^2 \big\}\,.
\end{align*} 
Then, we can define the random variable, 
\begin{align*}
Z= Z_{K_n} = \sum_{x\in \overline{\cL}_0} \one_{\{E_x\}}\,, \qquad \mbox{where} \qquad E_x = \{\hgt(\cP_x) \geq K_n\}\,.
\end{align*} 
First of all, notice that for the above choice of $K_n$, we have that for $n$ sufficiently large, 
\begin{align*}
\E[ Z ] \geq \Big(\frac{n}{K_n^3}\Big)^2 e^{- \alpha_\beta K_n -T_n} \geq \Big(\frac{n}{K_n^3}\Big)^2 e^{-2\log n + \epsilon  \alpha_\beta \log n -T_n} \geq n^{\delta}\,,
\end{align*}
for some $\delta>0$ (small depending on $\epsilon$), since $T_n =o(K_n) = o(\log n)$. 

We now wish to do a second moment estimate for $Z$ and use the fact that the events therein are weakly correlated (exponentially decaying in their distance), to show that for $K_n$ as above with any $\epsilon>0$, 
\begin{align}\label{eq:Paley-Zygmund}
\P (Z >0]) \geq \frac{(\E[Z])^2}{\E [ (Z)^2]} \geq 1-o(1)\,.
\end{align}
Expanding out $\E[ Z^2] = \E[Z]+ \sum_{x\neq y\in \bar{\cL}_0} \mu_n(E_x,E_y)$, by Corollary~\ref{cor:pillar-correlations}, we have  
\begin{align*}
|\mu_n(E_x)\mu_n(E_y)- \mu_n(E_x,E_y)| & \leq \|\mu_n(\cP_x\in \cdot)\mu_n(\cP_y\in \cdot)- \mu_n(\cP_x \in\cdot, \cP_y \in \cdot)\|_\tv \leq e^{ - c K_n^2}\,,
\end{align*} 
which is smaller than any polynomially decaying function of $n$ by the choice of $K_n$ and the fact that $|x-y|\geq \frac 14 K_n^3$. Therefore, we can bound 
\begin{align*}
\sum_{x,y\in \bar{\cL}_0: x\neq y} \mu_n (E_x, E_y)\leq \sum_{x,y\in \bar{\cL_0}, x\neq y} [ \mu_n(E_x)\mu_n(E_y)+e^{- cK_n^2} ] \leq (\E[Z])^2+ O(n^2 e^{-cK_n^2})\,.
\end{align*}
Plugging this bound in, we see that as $n\to\infty$,
\begin{align*}
\frac{(\E[Z])^2}{\E[Z^2]}\geq \frac{(\E[Z])^2}{\E[Z]+ (\E[Z])^2 + o(1)} \to 1\,,
\end{align*}
as long as $\E[Z]$ is diverging, which as noted earlier, is indeed the case for our choice of $K_n$. 
\end{proof}

\section{Finer properties of the increment sequence of the spine}\label{sec:mixing-stationarity}
In this section, we begin to analyze the shape of the pillars of the interface that attain the maximum of Section~\ref{subsec:max-lln}. We show that for tall pillars consisting of $T$ increments, their spine can be decomposed into an asymptotically (as you get further
from the base or tip) stationary sequence of weakly mixing increments. In particular, the increment sequence, viewed from the $(T/2)$-th increment converges weakly to a stationary bi-infinite sequence of increments, with polynomially decaying bounds on its mixing rate. 

Since this section (and most of the remainder of the paper) is concerned with the properties of pillars under the event $\bar{\mathbf I}_{x,T}$, let us henceforth take any sequence $n  = n_T$ and $x= x_T$ satisfying $n_T\gg T$ and $d(x_T,\partial \Lambda_{n_T})\gg T$ and denote the Ising measure on $\Lambda_{n_T}$ conditional on $\bar{\mathbf I}_{x_T,T}$ by
\[
\pi_T (\cdot ) : = \mu_n(\cdot \mid \bar{\mathbf I}_{x,T})\,.
\]

\noindent In \S\ref{subsec:increment-mixing}, we prove a spatial mixing estimate for the increment sequence $(\sX_1,.., \sX_T)$:

\begin{proposition}\label{prop:increment-mixing}
 For every $\gamma$, there exist $\beta_0,K$, and $C$ such that for every $\beta>\beta_0$ and $K\log T \leq j< k \leq T$, 
\begin{align*}
\sup_{E_j\subset \fX^{j-K\log T}, E_k\subset \fX^{T-k}}\big|\pi_T((\sX_{K\log T},&\ldots, \sX_j)\in E_j,   (\sX_k, \ldots, \sX_{T})\in E_k) \\ 
 & -\pi_T((\sX_{K\log T},\ldots,\sX_j)\in E_j)\pi_T((\sX_k, \ldots, \sX_{T})\in E_k)\big|\leq C  |k-j|^{-\gamma}\,.
\end{align*}
\end{proposition}

\noindent In \S\ref{subsec:increment-stationarity}, we prove that the increment sequence is asymptotically stationary away from the base and the tip. 
\begin{proposition}\label{prop:increment-stationarity}
For every $\gamma$, there exist $\beta_0, K$, and $C$ such that for every $\beta >\beta_0$, every $K\log T \leq j\leq T$ and $K\log T' \leq j'\leq T'$, and every $s\leq (T- j) \wedge (T' - j')$, 
\begin{align*}
\sup_{E\in \fX^s} |\pi_T((\sX_j, \ldots, \sX_{j+s}) \in E) &  - \pi_{T'}((\sX_{j'}, \ldots, \sX_{j'+s}) \in E)|\\
&  \leq C \big[(j-K\log T)\wedge (j'-K\log T')\big]^{ - \gamma} \vee \big[(T-j-s) \wedge (T' - j'-s)\big]^{ - \gamma}\,. 
\end{align*}
\end{proposition}
\noindent In \S\ref{subsec:limiting-distribution} we combine Propositions~\ref{prop:increment-mixing}--\ref{prop:increment-stationarity}, to define a limiting distribution on increment sequences.

\begin{corollary}\label{cor:limiting-distribution}
For every $\gamma$ large, let $\beta>\beta_0$ where $\beta_0$ is the one given by Propositions~\ref{prop:increment-mixing}--\ref{prop:increment-stationarity} for that $\gamma$. There exists a stationary distribution $\nu= \nu_\beta$ on $\fX^{\Z}$ so that, if  $a_T$ has $({a_T\vee (T- a_T)})/{\log T}\to\infty$ as $T\to\infty$, then the law of $(\ldots, \sX_{a_T-1}, \sX_{a_T} , \sX_{a_T + 1}, \ldots)$  under $\pi_T$ converges weakly to $\nu((\ldots,\sX_{-1}, \sX_{0}, \sX_1,\ldots)\in \cdot)$. 
In particular, the distribution $\nu$ satisfies 
\begin{enumerate}
    \item There exists $c>0$ (independent of $\beta$) such that $\nu(\fm(\sX_0)\geq r)\leq \exp[- c\beta r]$ for every $r$.
    \item There exists $C= C_\beta$ such that $\|\nu(\sX_{0} \in \cdot, \sX_k\in \cdot) - \nu(\sX_0 \in \cdot)\nu(\sX_k \in \cdot)\|_\tv \leq Ck^{-\gamma}$ for every $k$.
\end{enumerate}
\end{corollary}

The key step in the proofs of Propositions~\ref{prop:increment-mixing}--\ref{prop:increment-stationarity} is the use of what we call ``two-to-two" maps, which are bijections on the set of \emph{pairs} of interfaces $\bar{\mathbf I}_{x,T} \times \bar{\mathbf I}_{x,T}$, in contrast to all the maps we have applied up to this point. The reason for this is that any ``one-to-one" map $\Phi$ that changes an increment $\sX_i$ sustains a multiplicative cost of $e^{\pm C\fm(\sX_i)}$ in the ratio $\mu_n(\cI)/\mu_n(\Phi(\cI))$, which would overwhelm the upper bounds we wish to attain. ``Two-to-two" maps give us a mechanism of avoiding any such costs, and ensuring all faces in the pair $(\cI,\cI')$ are identified with faces in $\Phi(\cI, \cI')$ with which they have congruent local neighborhoods. We explain this in more detail in Sections~\ref{subsec:strategy-mix}--\ref{subsec:strategy-stat}.

\subsection{Proof of Proposition~\ref{prop:increment-mixing}: mixing properties of the increment sequence}\label{subsec:increment-mixing}

We wish to show that the correlations between the $j$-th and $k$-th increments decay polynomially fast in their distance, with the exponent of the polynomial increasing with $\beta$.

Fix any $\gamma$ and let $K$ be such that if $c_\sB$ is the constant from Proposition~\ref{prop:base-exp-tail}, $c_{\sB} \beta K>\gamma$. Next, fix $K\log T < j< k<T$, and let $L = \lceil \frac {4\gamma}{\bar c} \log |k-j|\rceil$; due to our freedom to take $C$ as desired, we may assume without loss that $|k-j|$ is sufficiently large. Fix any $E_j\in \fX^{j-K\log T}$, $E_k\in \fX^{T-k}$, and, in order to simplify notation, let us denote the tuples $\sZ_j = (\sX_{K\log T}, \ldots, \sX_j)$ and $\sZ_k = (\sX_k, \ldots, \sX_{T})$, with fixed instantiations $Z_j = (X_{K\log T}, \ldots, X_j)$ and $Z_k = (X_k, \ldots, X_T)$.  

Let $\cA_{jk}$ denote the set of all $T$-admissible truncated interfaces, increment sequences $(X_i)_{i \in \llb j+1, k-1\rrb}$, and remainder increment $X_{>T}$. 
For any triplet $(A, Z_j, Z_k)$ where $A= A_{jk}\in \cA_{jk}$, we write $\pi_T(Z_j,Z_k,A)$ to denote the probability that the random interface under $\pi_T$ has $\sZ_j = Z_j, \sZ_k = Z_k$ and has $\cI_{\textsc {tr}}$, $(\sX_i)_{i\in \llb j+1, k-1 \rrb}$ and $\sX_{>T}$ agreeing with $A_{jk}$.

We begin by expressing the left-hand side in the proposition as  
\begin{align*}
\sum_{\substack{Z_j\in E_j, Z_k \in E_k \\ Z_j' \in \fX^{j-K\log T}, Z_k'\in \fX^{T-k}}} \Big[\sum_{A,A'\in \cA_{jk}} \pi ( Z_j, Z_k, A) \pi_T(Z_j', Z_k', A') - \sum_{\tilde A, \tilde A' \in \cA_{jk}} \pi_T(Z_j, Z_k', \tilde A) \pi_T(Z_j', Z_k, \tilde A')\Big]
\end{align*}
Define a set of nice interfaces $\Gamma_L$ for which we have good decorrelation between $\sZ_j$ and $\sZ_k$. 
First, let $\Gamma_{\trivincr,L} = \Gamma_{\trivincr, L}(j,k)$ be the set of pairs of increment sequences $\fX^T \times \fX^T$ for which there is a stretch of $2L$ consecutive indices between $j$ and $k$ on which both $A$ and $A'$ have trivial increments. That is, 
\begin{align*}
\Gamma_{\trivincr,L} = \{ ((X_i)_{i\leq T}, (X_i')_{i\leq T}): \exists \tau_L \in \llb j, k\rrb, (X_{\tau_L-L}, \ldots, X_{\tau_L+ L}) = (X'_{\tau_L - L},\ldots, X'_{\tau_L+L}) = (X_\trivincr, \ldots, X_{\trivincr})\}.
\end{align*}
Abusing notation, under the event $\Gamma_{\varnothing,L}$, let $\tau_L$ be the smallest index greater than $j$ such that the stretch of length $2L$ centered at $\tau$ satisfies the condition of $\Gamma_{\varnothing,L}$.
We can now define a map $\Phi_{\textsc{mix}}$ on pairs of increment sequences, that swaps the increment stretches above $\tau_L$: see Figure~\ref{fig:mixing-map} for a visualization.

\begin{figure}
\vspace{-0.3cm}
\centering
  \begin{tikzpicture}
    \node (fig1) at (-4.75,0) {
	\includegraphics[width=.350\textwidth]{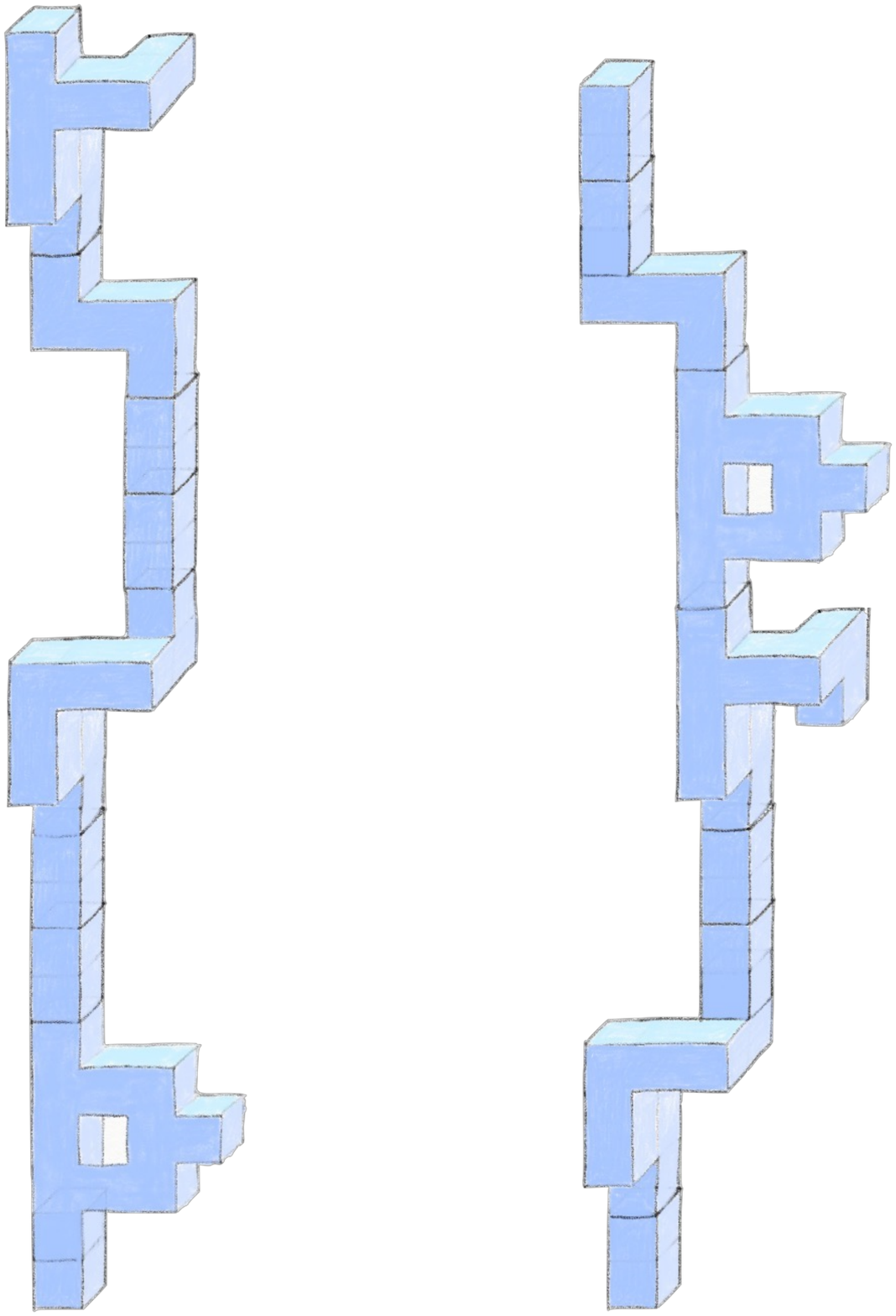}
	};
    \node (fig2) at (4.75,.1) {
    \includegraphics[width=.350\textwidth]{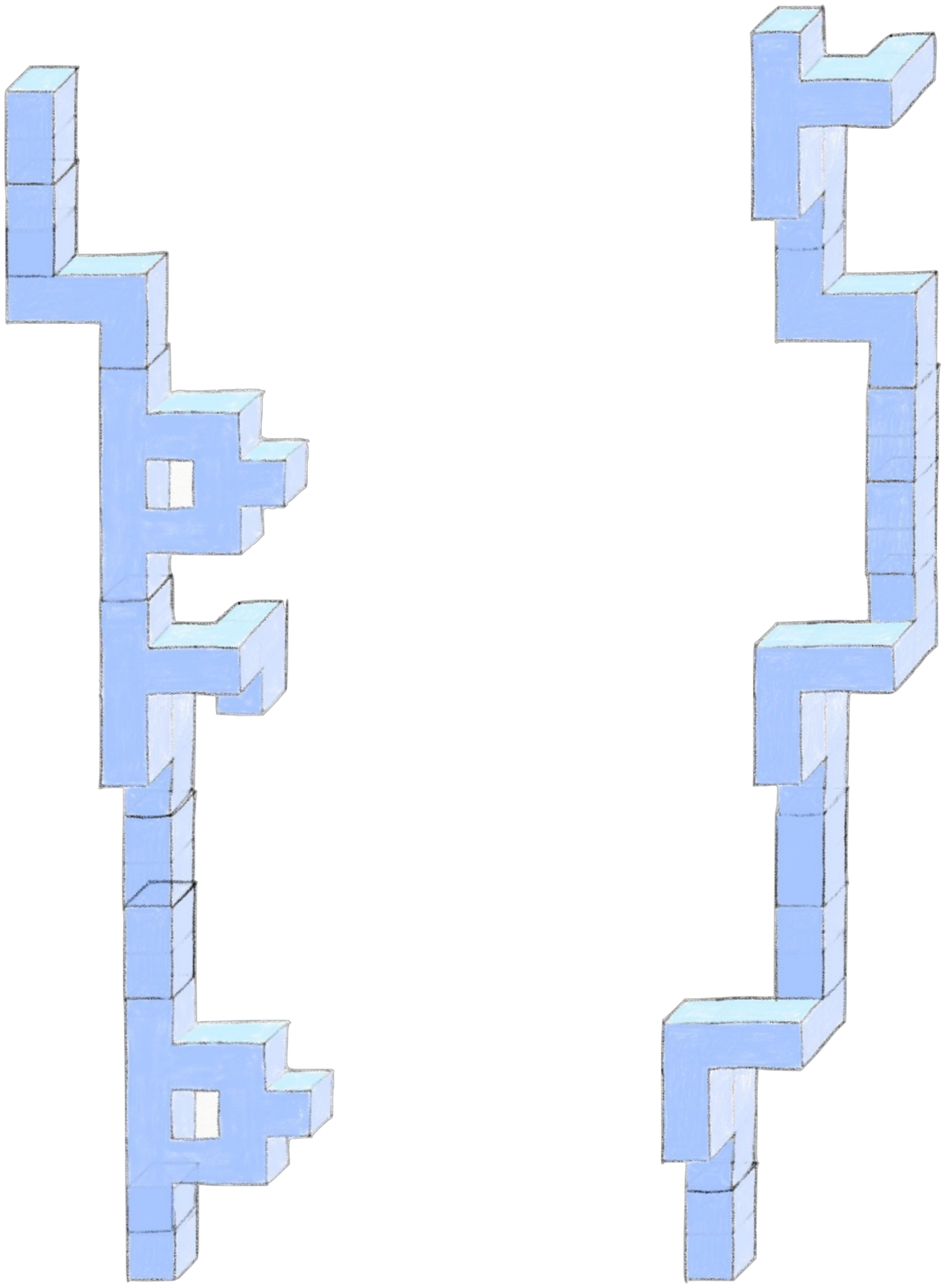}
    };
    \draw [decorate,decoration={brace,amplitude=7.5pt}, thick]
(-6,3.65) -- (-6,-1.2);
    \draw [decorate,decoration={brace,amplitude=7.5pt}, thick]
(-4.1,-1.4) -- (-4.1,3.3); 
    \draw[<->, thick] (-4.5, 1.5) to [out = 120, in = 60] (-5.6,1.5);
     \node at (2.4, -1.3) {$2L$};
     \draw[->] (2.4, -1.45)--(2.4, -1.95);
     \draw[->] (2.4, -1.15)--(2.4, -.6);
     
     \draw[<-] (3.15, -1.3)--(3.3, -1.3);
     \node at (3.5, -1.3) {$\tau_L$};
     
        \draw [decorate,decoration={brace,amplitude=5pt}]
(-3.5,-.93) -- (-3.5,.35) node [black,midway,xshift=-0.4cm] 
{\footnotesize $X'_{k}$};

        \draw [decorate,decoration={brace,amplitude=5pt}]
(-7.2,-.9) -- (-7.2,.4) node [black,midway,xshift=-0.5cm] 
{\footnotesize $X_{k}$};

        \draw [decorate,decoration={brace,amplitude=3pt}]
(-3.9,-3.36) -- (-3.9,-2.84) node [black,midway,xshift=-0.4cm] 
{\footnotesize $X'_{j}$};

        \draw [decorate,decoration={brace,amplitude=5pt}]
(-7.15,-3.35) -- (-7.15,-2) node [black,midway,xshift=-0.5cm] 
{\footnotesize $X_{j}$};

          \node at (6.95, -1.2) {$2L$};
     \draw[->] (6.95, -1.35)--(6.95, -1.85);
     \draw[->] (6.95, -1.05)--(6.95, -.5);

      \draw[->, thick] (-1.5,2) -- (1.5,2);
    \node at (0,2.4) {$\Phi_{\textsc {mix}} = \Phi_1^2 \times \Phi_2^1$};

  \end{tikzpicture}
\vspace{-0.5cm}
 \caption{The map $\Phi_{\textsc{mix}}= \Phi_1^2 \times \Phi_2^1$ acting on a pair of increment sequences in $\Gamma_{\trivincr,L}$}
  \label{fig:mixing-map}
\vspace{-0.2cm}
\end{figure}

\begin{definition}
For each $j,k$, let $\Phi_{\textsc{mix}} = \Phi_{\textsc{mix}}(j,k):(\fX^T \times \fX_{\textsc{rem}})^2 \to (\fX^T \times \fX_{\textsc{rem}})^2$ be given as follows. For any pair of increment sequences $(X^{(1)},X^{(2)}) = ((X^{(1)}_i)_{i\leq T}, X^{(1)}_{>T},(X^{(2)}_i)_{i\leq T},X^{(2)}_{>T})$ let $\Phi_{\textsc{mix}}(X^{(1)}, X^{(2)}) = (\Phi_1^2(X^{(1)},X^{(2)}), \Phi_2^1(X^{(1)}, X^{(2)}))$ be the pair of increment sequences attained as follows: if $(X^{(1)}, X^{(2)})\notin \Gamma_{\varnothing,L}$, let $\Phi_{\textsc{mix}}(X^{(1)}, X^{(2)})=(X^{(1)}, X^{(2)})$; otherwise
\begin{enumerate}
\item Find the first run of $2L$ consecutive indices between $j$ and $k$ on which both $X^{(1)}$ and $X^{(2)}$ are trivial increments, and call the middle index of this run $\tau_L\in \llb j,k\rrb$. 
\item Let $\Phi_1^2 ( X^{(1)}, X^{(2)})$ have increment sequence given by 
\begin{align*}
\Phi_{1}^{2}(X^{(1)}, X^{(2)}) = (X_1^{(1)},\ldots, X^{(1)}_{j},\ldots, X^{(1)}_{\tau_L}, X^{(2)}_{\tau_L + 1}, \ldots, X^{(2)}_{k},\ldots, X^{(2)}_{>T})\,.
\end{align*}
\item Let $\Phi^1_2(X^{(1)}, X^{(2)})$ have increment sequence given by 
\begin{align*}
\Phi^1_2(X^{(1)}, X^{(2)})= (X_1^{(2)},\ldots,X^{(2)}_{j},\ldots,X^{(2)}_{\tau_L},X^{(1)}_{\tau_L + 1},\ldots, X^{(1)}_k,\ldots,X^{(1)}_{>T})\,.
\end{align*}
\end{enumerate}
Abusing notation, we define $\Phi_{\textsc{mix}}$ on $\bar{\mathbf I}_{x,T} \times \bar{\mathbf I}_{x,T}$ that uses the same truncations of the pair $(\cI, \cI')$ and applies the map $\Phi$ to their respective pairs of increment sequences in their pillars $\cP_x$ and $\cP_x'$. If the two interfaces are both tame and also satisfy $\Tsp \vee \Tsp'\leq K\log T$, the resulting pair of interfaces would be in $\mathbf I_{x,T} \times \mathbf I_{x,T}$.
\end{definition}

\subsubsection{Strategy of the map $\Phi_{\textsc{mix}}$}\label{subsec:strategy-mix}
We briefly motivate the construction of the map $\Phi_{\textsc{mix}}$. We first describe the complications that would arise if we used a map that sent one interface to another interface, instead of acting on pairs of interfaces.  In order to prove a mixing property on the increment sequence, one would want to construct a map which maps an interface with an increment $X_j$ and an increment $X_k$, to an interface with some other increment $X_{j}'$ and the same $X_k$, say having $\fm(X_j) = \fm(X_j')$. If the weight distortion of such a map is $o(|k-j|^{-c})$ for some $c>0$, we will have shown that  that conditioning on the presence of the increment $X_j$ vs.\ $X_j'$ does not influence the conditional probability of $X_k$. Unlike the maps in Sections~\ref{sec:increment-exp-tail}--\ref{sec:base} there is no energy gain in such a map; however, the replacement of $X_j$ by $X_j'$ inevitably costs an $e^{C(|X_j|\vee |X_j'|)}$ in the weight ratio, coming from the uniform bound on $\g$~\eqref{eq:g-uniform-bound}. 

In order to obtain ratios of weights that are $o(1)$ in $|k-j|$, we use a second interface, whose increment sequence has $X_j'$ and $X_k'$, and we demonstrate that the sequence is mixing by showing that the probabilities of a pair of interfaces having increment pairs $\{(X_j,X_k),(X_j', X_k')\}$ is close to the probability of the pair having $\{(X_j, X_k'), (X_j',X_k)\}$. Then, in the control of the $\g$ term, we could identify faces from $X_j,X_k$ with one another and $X_j',X_k'$ with one another across the pairs of interfaces. However, a naive application of this kind of map would lead to a $1\pm \epsilon_\beta$ weight distortion, rather than one that is $1+o(1)$ in $|k-j|$. More precisely, every increment would feel the change in the $\g$ term in terms of its distance to the increment where we spliced the interface to perform the swap---in particular, the increments near the splicing location, if they disagree between the pair of interfaces, will contribute a constant, but not $o(1)$ to the weight distortion.  

To improve this to something decaying polynomially in $|k-j|$, the map $\Phi_{\textsc{mix}}$ relies on the existence of a sequence of consecutive  increments of logarithmic length in $|k-j|$, that are trivial in both interfaces. In that case, after the splicing, for every face in either of the interfaces, the radius of congruence is bounded by half the length of the consecutive sequence of interfaces, and by~\eqref{eq:g-exponential-decay}, the weight distortion is at most polynomially decaying in $|k-j|$ for large enough $\beta$, as desired. Refer to Figure~\ref{fig:mixing-map} for a visualization.

\subsubsection{Analysis of the map $\Phi_{\textsc{mix}}$}
We now use the map $\Phi_{\textsc{mix}}$ to define a good set of pairs of increment sequences, refining the set $\Gamma_{\varnothing,L}$, on which we will have good control on the ratio of probabilities under $\Phi_{\textsc{mix}}$. Let $\Gamma_L$ be the set of  $(\cI, \cI')\in \bar{\mathbf I}_{x,T}\times \bar{\mathbf I}_{x,T}$  such that its pair of increment sequences are in $\Gamma_{\trivincr,L}$, and additionally having 
\begin{enumerate}
\item Their source point indices $\Tsp, \Tsp'$ are both less than $K\log T$; denote this event $\Gamma_\Tsp$.
\item The pair of interfaces $(\cI, \cI')$ are such that $\Phi_{\textsc {mix}}(\cI,\cI')$ are both tame; denote this event $\bar \Gamma$. 
\item Their increment sequences $(X_i)$ and $(X_i')$ satisfy 
\begin{align*}
|\cF(X_{>T})| e^{ - \bar c (T+1-\tau_L+L)} + \sum_{i} |\cF(X_{\tau_L+i})| e^{ - \bar c (L+ i)} \leq e^{ - \bar c L /2}\,.
\end{align*}
and analogously for $(X_i')$; denote this event $\Gamma_{i>\tau_L}$. 
\end{enumerate}

We will separately consider the cases where $(A,Z_j, Z_k)\times( A', Z_j', Z_k')$ and $(\tilde A, Z_j', Z_k)\times(\tilde A', Z_k', Z_j)$ are in $\Gamma_L$ and the cases when they aren't: without loss of generality, let us consider the former pairs (the latter estimate would hold after swapping $Z_j'$ with $Z_j$). The contribution from pairs of interfaces where one is not in $\Gamma_L$ are bounded above by the sum of 
\begin{align*}
\sum_{\substack{ Z_j\in E_j, Z_k \in E_k,  Z_j', Z_k', A,A': \\ (A, Z_j, Z_k), (A', Z_j', Z_k')\in \Gamma_L^c}} \pi_T( Z_j, Z_k, A) \pi_T ( Z_j', Z_k', A') 
& \leq \pi_T^{\otimes 2} (\Gamma_L^c)\,.
\end{align*}
\begin{lemma}\label{lem:good-pairs-of-interfaces}
For the choices of $\gamma, K, L$ above, for $\beta>\beta_0$, we have $\pi_T^{\otimes 2}(\Gamma_L^c)\leq |k-j|^{-\gamma}$. 
\end{lemma}
\begin{proof}[\emph{\textbf{Proof of Lemma~\ref{lem:good-pairs-of-interfaces}}}] By a union bound,  we can express 
\[ 
\pi_T^{\otimes 2}(\Gamma_L^c)\leq  2 \pi_T(\Gamma_{\Tsp}^c) + \pi_T^{\otimes 2}( \bar \Gamma^c \mid \Gamma_\Tsp) + \pi_T^{\otimes 2} (\Gamma_{\varnothing,L}^c \mid \Gamma_\Tsp)+\pi_{T}^{\otimes 2} (\Gamma_{i>\tau_L} \mid \Gamma_{\varnothing,L},  \Gamma_{\Tsp})\,.
\]
By Proposition~\ref{prop:base-exp-tail} and the fact that $\Tsp \leq \hgt(v_{\Tsp})+\frac 12$, we have that $\pi_{T}(\Gamma^c_{\Tsp})\leq e^{ - c_\sB \beta K \log T} \leq T^{-\gamma}$. 
In order for $\Phi_{\textsc{mix}} (\cI, \cI')$ to not be tame, one of $\fm_T(\cS_x)$ or $\fm_T(\cS'_x)$ must be at least $r_0/2$; by Lemma~\ref{lem:tame} and the fact that $r_0/2>8T$, then,  $\pi_{T}^{\otimes 2}(\bar{\Gamma} \mid \Gamma_\Tsp)$ is at most $2e^{-c\beta T}$ for some $c>0$. We now turn to the latter two terms above. 
By Proposition~\ref{prop:exp-tails-increments}, in particular the conditional version of it, given any $\cI_{\textsc{tr}}$ (in particular any $\Tsp$) the sequence $(\mathbf{1}\{\sX_i = X_\trivincr\})_{i\geq \Tsp}$ (which includes the increments between indices $j$ and $k$ by $\Gamma_{\Tsp}$) stochastically dominates a sequence of independent $\ber(1-\epsilon_\beta)$ coin tosses, for some $\epsilon_\beta$ going to zero as $\beta\to\infty$. 
As a consequence, $\pi_T^{\otimes 2}(\Gamma_{\varnothing,L}^c)$ is at most the probability that a set of $|k-j|$ i.i.d.\ $\ber((1-\epsilon_\beta)^2)$ coin flips has no sequence of $2L$ consecutive ones. 
Thus, for large enough $\beta$, (depending on $\gamma,\bar c$)
\[ \pi^{\otimes 2}_T (\Gamma_{\trivincr,L}^c\mid \Gamma_{\Tsp}) 
\leq (1-(1-\epsilon_\beta)^{4L})^{|k-j|/(2L)} \leq \exp\left[-(1-\epsilon_\beta)^{4L} |k-j|/(2L)\right] \leq \exp\left[ -|k-j|^{3/4}\right]\,.
\]
By Corollary~\ref{cor:increment-interaction-bound}, conditional on the entire increment sequence up to $\tau_L+L$ (which contains the information of $\Tsp, \Gamma_\Tsp, \Gamma_{\trivincr,L}$  the concentration estimate on the excess areas of subsequent increments holds (uniformly in the choice of $\tau_L$). Combining these, one sees the bound (where the conditioning on $\Gamma_{\varnothing,L}(\tau_L)$ is to say that $\Gamma_{\varnothing,L}$ happens for that specific $\tau_L$),
\begin{align*}
\pi_T^{\otimes 2}( \Gamma_{i>\tau_L} & \mid \Gamma_{\varnothing,L}, \Gamma_\Tsp) \\
& \leq 2 \sup_{\cI_{\textsc{tr}}\in \Gamma_\Tsp, \tau_L} \pi_{T}\Big(|\cF(\sX_{>T})|e^{-\bar c (T+1 - \tau_L +L)}+ \sum_{i} |\cF(X_{\tau_L+i})| e^{ - \bar c (L + i)} \geq e^{ -\bar c L /2}\mid \cI_{\textsc{tr}}, \Gamma_{\varnothing,L}(\tau_L)\Big)\,,
\end{align*}
which is at most $2\exp ( - c \beta \bar c L/2)$; therefore, $\pi_T^{\otimes 2} (\Gamma_L^c) \leq \exp[ -|k-j|^{3/4}] + 2 \exp[- c\beta \bar c L/2]$.

Our choice of $L$ was precisely such that as long as $\beta c>1$, the latter quantity is at most $|k-j|^{-2\gamma}$, which dominates the first term. 
\end{proof}

\medskip
On the other hand, when both pairs of triplets $((A,Z_j, Z_k),( A', Z_j', Z_k'))$ and $((\tilde A, Z_j', Z_k),(\tilde A', Z_k', Z_j))$ are in $\Gamma_L$, we are left to control 
\begin{align*}
\sum_{\substack{ Z_j\in E_j, Z_k \in E_k \\ Z_j', Z_k', A,A' \\ (A, Z_j, Z_k), (A', Z_j', Z_k')\in \Gamma_L}} \pi_T(Z_k, Z_j, A) \pi_T(Z_k', Z_j', A') - \sum_{\substack{ Z_j\in E_j, Z_k \in E_k \\ Z_j', Z_k', \tilde A,\tilde A' \\ (\tilde A, Z_j, Z_k'), (\tilde A', Z_j', Z_k)\in \Gamma_L}} \pi (Z_k, Z_j', \tilde A') \pi (Z_k', Z_j, \tilde A)\,.
\end{align*}

 Now that we have restricted to tame interfaces, with well-behaved increment sequences, we can naturally view $\Phi_{\textsc{mix}}$ as a map on $(\bar{\mathbf I}_{x,T} \times \bar{\mathbf I}_{x,T}) \cap \Gamma_L$. This restriction yields the following correspondence. 
 
 \begin{claim}\label{claim:mixing-bijection}
 The restriction of $\Phi_{\textsc{mix}}$ to $(\bar{\mathbf I}_{x,T} \times \bar{\mathbf I}_{x,T}) \cap \Gamma_L$ is a bijection from $(\bar{\mathbf I}_{x,T} \times \bar{\mathbf I}_{x,T}) \cap \Gamma_L$ to itself. 
 \end{claim}
 
\begin{proof}[\textbf{\emph{Proof of Claim~\ref{claim:mixing-bijection}}}] Since $\Phi_{\textsc{mix}}= \Phi_{\textsc{mix}}^{-1}$, it suffices to show that for every pair $(\cI,\cI')\in (\bar{\mathbf I}_{x,T} \times \bar{\mathbf I}_{x,T}) \cap \Gamma_L$, we have $\Phi_{\textsc{mix}}(\cI,\cI')\in (\bar{\mathbf I}_{x,T} \times \bar{\mathbf I}_{x,T}) \cap \Gamma_L$. Indeed, as mentioned, the fact that $(\cI,\cI')\in \Gamma_{\Tsp}$ ensures that $\Phi_{\textsc{mix}}(\cI,\cI')\in \mathbf I_{x,T}\cap \mathbf I_{x,T}$; the fact that $(\cI,\cI')\in \bar \Gamma$, by definition, guarantees that $\Phi_{\textsc{mix}}(\cI,\cI')$ are both tame. Finally, the fact that $\Phi_{\textsc{mix}}(\cI,\cI')$ remains in $\Gamma_L$ holds for the following reasons: (1) holds as the source points are unchanged by the map; (2) holds as $\Phi_{\textsc{mix}}= \Phi_{\textsc{mix}}^{-1}$; (3) holds for $\Phi_{\textsc{mix}}(\cI,\cI')$ since the pair of increment sequences above $\tau_L$ in $\Phi_{\textsc{mix}}(\cI,\cI')$ are exactly the pair of increment sequences above $\tau_L$ of $(\cI,\cI')$. 
\end{proof}
 
With the claim in hand, notice that $\Phi_{\textsc{mix}}$ preserves the $\tau_L$ at which $\Gamma_{\trivincr,L}$ is attained and, we have 
\begin{align*}
\Phi_{\textsc{mix}}\Big( (Z_j, Z_k, A), (Z_j', Z_k', A')\Big) 
= \Big(\big(Z_j, Z_k', \Phi_1^2 ( A, A')\big) , \big(Z_j', Z_k, \Phi_2^1 (A, A')\big)\Big)
\end{align*}
in the sense that the $Z_k$ and $Z_k'$ get swapped by application of the map in the manner desired, as does everything else in the spine above index $\tau$. 
Using this bijection, we rewrite the difference above as
\begin{align*}
& \sum_{\substack{ Z_j\in E_j, Z_k \in E_k \\ Z_j', Z_k', A,A' \\ (A, Z_j, Z_k), (A', Z_j', Z_k')\in \Gamma_L}} \pi_T(Z_j, Z_k, A) \pi ( Z_j', Z_k', A') - \pi_T(Z_j, Z_k', \Phi_1^2(A, A')) \pi_T(Z_j', Z_k, \Phi_2^1(A,A'))\,.
\end{align*}
Now fix $Z_j\in E_j, Z_k \in E_k$, $Z_k', Z_j'$ and $(A,A')$ such that the above triplets are in $\Gamma_L$; for ease of notation, let $\tilde A = \Phi_1^2(A,A')$ and $\tilde A' = \Phi_2^1 (A,A')$. Consider the quantity 
\begin{align}\label{eq:mixing-difference-to-fraction}
|\pi_T(Z_k, Z_j, A) & \pi_T(Z_k', Z_j', A') -  \pi_T(Z_j, Z_k' ,  \tilde A) \pi_T(Z_j', Z_k, \tilde A')| \nonumber \\
& \qquad \qquad =   \pi_T(Z_j, Z_k' ,  \tilde A) \pi_T(Z_j', Z_k, \tilde A') \bigg|\frac{\mu_n(Z_j, Z_k, A) \mu_n(Z_j', Z_k', A')}{\mu_n(Z_j, Z_k', \tilde A) \mu_n(Z_j', Z_k, \tilde A')} - 1\bigg|\,;
\end{align}
since each of the triplets are in $\bar{\mathbf I}_{x,T}$,  expressing e.g., $\pi_T(Z_j, Z_k, A) = \frac{\mu_n(Z_j,Z_k,A)}{\mu_n(\bar{\mathbf I}_{x,T})}$, the contributions from $\mu_n(\bar{\mathbf I}_{x,T})$  cancel out. 
Let us now focus on the difference in the absolute value in~\eqref{eq:mixing-difference-to-fraction}, and in particular the ratio of the probabilities of the two pairs of interfaces. 
This formulation allows us to apply the machinery of Theorem~\ref{thm:cluster-expansion} to the pair of interfaces: for ease of notation, let us denote the interface given by $(Z_j,Z_k,A)$ by $\cI^{jk}$, and denote $\cI^{jk'}= (Z_j, Z_k', \tilde A)$, $\cI^{j'k}= (Z_j', Z_k, \tilde A')$ and $\cI^{j'k'}= (Z_j', Z_k', A')$ analogously. Express
\begin{align*}
\frac{\mu_n (Z_j, Z_k, A) \mu_n (Z_j', Z_k', A')}{\mu_n (Z_j, Z_k', \tilde A) \mu_n (Z_j', Z_k, \tilde A')} & = \frac{e^{\beta(|\cI^{jk}|+|\cI^{j'k'}|)}}{e^{\beta(|\cI^{jk'}|+|\cI^{j'k}|)}}\cdot \frac{e^{\sum_{f\in \cI^{j'k'}}\g(f,\cI^{j'k'}) +\sum_{f\in \cI^{jk}}\g(f,\cI^{jk})}}{e^{\sum_{f\in \cI^{j'k}}\g(f,\cI^{j'k}) +\sum_{f\in \cI^{jk'}}\g(f,\cI^{jk'})}} \\ 
& = \exp \bigg[\sum_{f\in \cI^{jk}}\g(f,\cI^{jk}) +\sum_{f\in \cI^{j'k'}}\g(f,\cI^{j'k'}) -\sum_{f\in \cI^{jk'}}\g(f,\cI^{jk'})  - \sum_{f\in \cI^{j'k}}\g(f,\cI^{j'k})  \bigg]\,.
\end{align*}
We now turn to bounding the absolute value of the exponent. Recall that $\tau_L$ denotes the first index of the increment run of $2L$ trivial increments in both $(A,A')$ (and consequently also $(\tilde A, \tilde A')$ for the same $\tau_L$). 

\begin{lem}\label{lem:mixing-part-functions}
There is a universal $\bar C$ such that for any pair $(\cI^{jk},\cI^{j'k'})\in (\bar{\mathbf I}_{x,T}\times \bar{\mathbf I}_{x,T}) \cap \Gamma_L$, we have
\begin{align*}
\Big|\sum_{f\in \cI^{jk}}\g(f,\cI^{jk})  +\sum_{f\in \cI^{j'k'}}\g(f,\cI^{j'k'}) -\sum_{f\in \cI^{jk'}}\g(f,\cI^{jk'})  - \sum_{f\in \cI^{j'k}}\g(f,\cI^{j'k})\Big| \leq \bar C \exp\big[-  \bar c L/2\big]\,.
\end{align*}
\end{lem} 

Let us first conclude the proof of Proposition~\ref{prop:increment-mixing} given Lemma~\ref{lem:mixing-part-functions}. By our choice $L$, the right-hand side above is at most $\bar C | k-j|^{-2\gamma}$, from which we would deduce that 
\begin{align*}
\Big|\frac{\mu_n(Z_j, Z_k, A) \mu_n(Z_j', Z_k', A')}{\mu_n(Z_j, Z_k', \tilde A) \mu_n(Z_j', Z_k, \tilde A')} - 1\Big| \leq 2\bar C|k-j|^{-2\gamma}\,.
\end{align*}
Since this upper bound is independent of $Z_j, Z_k, Z_j', Z_k', A, A'$, when we sum~\eqref{eq:mixing-difference-to-fraction}, it factors out, and the sum of the probabilities over some subset of interfaces in $\bar{\mathbf I}_{x,T}$ is of course at most one. Combining this with the contribution from terms not in $\Gamma_L$ yields an additional $|k-j|^{-\gamma}$, implying the desired estimate.  

\medskip
\noindent 
\textbf{Proof of Lemma~\ref{lem:mixing-part-functions}}.
It will be important to use the structure of the map $\Phi_{\textsc {mix}}$ to choose the right pairing of summands in the different sums on the left-hand side above. To that end, let us define the following subsets of faces of the interfaces we consider: let $\cI_{\textsc {tr}}$ and $\cI'_{\textsc {tr}}$ be the respective truncations of $A$ and $A'$. Let
\begin{align*}
F^{jk}_- & = \Big\{f\in \cI_{\textsc{tr}} \cup \mbox{$\bigcup_{i\leq \tau_L}$} \cF(X_i)\Big\}\,, & \qquad \mbox{and}  \qquad & F^{jk}_+  = \Big\{f\in \mbox{$\bigcup_{i>\tau_L}$} \cF (X_i) \cup \cF(X_{>T})\Big\}\,.
\end{align*}
be the sets of all faces ``below" $X_{\tau_L}$, and all faces ``above" $X_{\tau_L}$ respectively. 
In this manner, $\cI^{jk} =  F_-^{jk} \cup F_+^{jk}$, and we can define $F^{j'k'}_{\pm},F^{j'k}_{\pm},F^{jk'}_{\pm}$ analogously (where whether or not $j$ is primed indicates whether $\cI_{\textsc{tr}}$ or $\cI_{\textsc{tr}}'$ is used in $F_-$). By definition, we have the equalities 
\begin{align*}
F_{-}^{jk} = F_-^{jk'}\,, \qquad \mbox{and}\qquad F_-^{j'k'}=F_-^{j'k}\,.
\end{align*}
Let $\theta_{A,A'}$ be the shift map by the vector $-v_{\tau_L+1} + v'_{\tau_L+1}$  (where $v_{\tau_L+1}$ is that cut-point in $(Z_j,A)$ and $v'_{\tau_L+1}$ is that in $(Z_j',A')$) and let $\theta_{A',A}$ be the shift by $-v'_{\tau_L+1} + v_{\tau_L+1}$. Then observe that 
\[ \theta_{A',A} F_+^{j'k'} =  F_+^{jk'}\,,   \, \qquad \mbox{and}\qquad  \theta_{A,A'} F_+^{jk} = F_+^{j'k}\,.
\]
Using this decomposition of the faces in the four interfaces, we can express
\begin{align}\label{eq:part-function-splitting-mixing}
\Big|\sum_{f\in \cI^{j'k'}}\g(f,\cI^{j'k'})  +\sum_{f\in \cI^{jk}} & \g(f,\cI^{jk})  -\sum_{f\in \cI^{j'k}}\g(f,\cI^{j'k})  - \sum_{f\in \cI^{jk'}}\g(f,\cI^{jk'})\Big| \nonumber \\ 
& \leq \sum_{f\in F_{+}^{jk}} |\g(f,\cI^{jk})- \g(\theta_{A,A'}f,\cI^{j'k})| + \sum_{f\in F_{+}^{j'k'}} |\g(\theta_{A',A}f,\cI^{j'k'})- \g(f,\cI^{jk'})| \nonumber \\
& \quad + \sum_{f\in F_{-}^{jk}} |\g(f,\cI^{jk})- \g(f,\cI^{jk'})|+ \sum_{f\in F_{-}^{j'k'}} |\g(f,\cI^{j'k'})- \g(f,\cI^{j'k})|\,.
\end{align}

We begin by bounding the contributions of faces above $X_{\tau_L}$ , i.e., the first line of~\eqref{eq:part-function-splitting-mixing}; we write the bound for one of the sums as the other will  evidently be analogous: 
\begin{align*}
\sum_{f\in F_+^{jk}}\big|\g(f, \cI^{jk}) - \g(\theta_{A,A'}f,\cI^{j'k})\big| \leq \sum_{f\in F_+^{jk}} \bar K e^{ - \bar c \br(f, \cI^{jk}; \theta_{A,A'} f, \cI^{j'k})}\,.
\end{align*}
By tameness of all of $\cI^{jk}, \cI^{jk'}, \cI^{j'k}$, and $\cI^{j'k'}$, the radius $\br$ is either attained by a face below $v_\Tsp$, in which case for a face $f\in \cF(X_{\tau_L}+i)$, the radius is at least $\tau_L-\Tsp +i\geq L+i$ or, it is attained in the differences between the spines $(X_\Tsp, \ldots, X_{\tau_L})$ and $(X_\Tsp',\ldots, X_{\tau_L}')$---but since all the increments between $\tau_L-L$ and $\tau_L$ are trivial both in $\cI^{jk}$ and $\cI^{j'k}$, this distance would be at least $L + i$. The above is at most $\bar K$ times 
\begin{align*}
 \sum_{i\geq 1} \sum_{f\in \cF(X_{\tau_L+i})} e^{ - \bar c (L+i)} + \sum_{f\in \cF(X_{>T})} e^{-\bar c (T+1 - \tau_L +L )} \leq \sum_{i\geq 1} |\cF(X_{\tau_L+i})| e^{ - \bar c ( L+i )} + |\cF(X_{>T})| e^{-\bar c (T+1 - \tau_L +L )} \,,
\end{align*}
and the fact that the interfaces are both in $\Gamma_L \subset \Gamma_{i>\tau_L}$ implies this is at most $\bar K e^{ - \bar c L/2}$. The sum over $f\in F_+^{j'k'}$ in the first line of~\eqref{eq:part-function-splitting-mixing} is handled identically. 
Next, we consider the contributions from the increments below~$\tau_L$ as well as in the truncated interface, say the faces in $F_-^{jk}$ (the sum over $f\in F_-^{j'k'}$ is again identical). Notice that for these faces, the radius $\br$ is attained by a face in $F_+^{jk} \oplus F_+^{jk'}$ with increment index at least $\tau_L+L$, so that  
\begin{align*}
\sum_{f\in F_-^{jk}} |\g(f,\cI^{jk}) - \g(f, \cI^{jk'})| &  \leq \sum_{i\leq L} 4 \bar K  e^{- \bar c (L+i)}+ \sum_{\Tsp \leq i\leq \tau_L-L} \sum_{f\in X_i} \sum_{g\in F_+^{jk} \cup F_+^{jk'}} \bar K e^{ - \bar c d(f, g)} \\ 
& \quad \, + \sum_{f\in \cI_{\textsc{tr}}} \sum_{g\in F_+^{jk}\cup F_+^{jk'}} \bar K e^{ - \bar c d(f,g)}\,.
\end{align*}
By tameness, the distance between any face in $\cI_{\textsc{tr}}$ to a face $g\in F_+^{jk}\cup F_+^{jk'}$ that is in the $(\tau_L+i')$-th increment, is at least, $\tau_L + i'- \Tsp \geq i'+L$, so that
\begin{align*}
\sum_{f\in F_-^{jk}} |\g(f,\cI^{jk}) - \g(f, \cI^{jk'})|  &  \leq 4\bar C e^{- \bar c L} + \sum_{i'\geq 1}  \sum_{g\in X_{\tau_L+i'}\cup X_{\tau_L+i'}'} \sum_{\substack { f\in \cF(\Z^3) \\ \hgt( f) \leq \hgt(v_{\tau_L})} } \bar K e^{- \bar c d(g,f)} \\ 
& \leq 4\bar Ce^{- \bar c l} + \sum_{i' \geq 1} |\cF(X_{\tau+i'})|\vee |\cF(X_{\tau+i'}')| \bar C e^{ - \bar c (l + i')}\,. 
\end{align*}
which is at most $5 \bar C e^{ - \bar c L/2}$ by our assumption that the interfaces are in $\Gamma_L$.  Combining all of these in to~\eqref{eq:part-function-splitting-mixing} and using our choice of $L$ yields  Lemma~\ref{lem:mixing-part-functions}.
This completes the proof of Proposition~\ref{prop:increment-mixing}.

\subsection{Proof of Proposition~\ref{prop:increment-stationarity}: spine increments are asymptotically stationary}\label{subsec:increment-stationarity}
In this section, we prove Proposition~\ref{prop:increment-stationarity}, showing that spine increments are asymptotically stationary in the sense that changing the conditioning from $T$ to $T'$ and the location of an increment stretch from $j = j_T$ to $j' = j'_{T'}$ does not change the law much as long as $j$ and $j'$ are in the bulks of their respective spines. Up to the choice of the two-to-two map, which is tailored to proving stationarity estimates here, much of the proof will match that of the mixing and we therefore omit some repeated details. 

Fix any $\gamma$, and let $K$ be such that if $c_\sB$ is the constant from Proposition~\ref{prop:base-exp-tail}, $c_\sB \beta K >2\gamma$. Next fix $j, j'$ and $s$ satisfying the required conditions, and let 
\[
L = \lceil \tfrac{4\gamma}{\bar c} \log D\rceil\,,  \qquad \mbox{where}\qquad D = (j-K\log T) \wedge (j'- K\log T') \wedge (T' - (j'+s))\wedge (T- (j+s))\,.
\]  
Due to our freedom to take $C$ as desired, we may assume without loss that $D$ is sufficiently large. 
Let us denote the tuples $\sZ_j = (\sX_{j}, \ldots, \sX_{j+s})$ and $\sZ_{j'} = (\sX_{j'}, \ldots, \sX_{j'+s})$, with fixed instantiations $Z_j$ and $Z_{j'}$ in $\fX^{s}$. Let $\cA_{T}$ denote the set of all $T$-admissible truncated interfaces along with increment sequences $(X_i)_{i\notin \llb j, j+s \rrb}$ and remainder increment $X_{>T}$. Begin by expressing the left-hand side in Proposition~\ref{prop:increment-stationarity} as
\begin{align*}
\bigg|\sum_{Z_j \in E}\sum_{A_{T}\in \cA_T} \pi (Z_j, A_T) \sum_{\tilde Z_{j'}\in \fX^s, A_{T'}\in \cA_{T'}}\pi_{T'}(\tilde Z_{j'}, A_{T'}) -   \sum_{\tilde A_{T}\in \cA_T, \tilde Z_j\in \fX^s} \pi_{T'} (\tilde Z_{j}, \tilde A_{T}) \sum_{Z_j'\in E} \sum_{\tilde A_{T'} \in \cA_{T'}} \pi_{T} (Z_{j'}, \tilde A_{T'})\bigg|
\end{align*}
We follow the same strategy of the proof of Proposition~\ref{prop:increment-mixing}. Namely, define the events  $\Gamma_{L}^-$ and $\Gamma_{L}^+$ as the following subsets of pairs of increment sequences $(\fX^{T}\times \fX_{\textsc{rem}}) \times (\fX^{T'}\times \fX_{\textsc{rem}})$,
\begin{align*}
\Gamma_{\trivincr, L}^- & = \{ ((X_i)_{i\leq T}, (X_i')_{i\leq T'}): \exists \tau_L^-   \in \llb L, D-L\rrb \quad\;\;\mbox{ s.t.\ } X_{j-\tau_L^- +i} = X'_{j'-\tau_L^- +i} = X_\trivincr\mbox{ for all $i=-L,\ldots,L$}
\}\,, \\
\Gamma_{\trivincr, L}^+ & = \{ ((X_i)_{i\leq T}, (X_i')_{i\leq T'}): \exists \tau_L^+ \in \llb L+s, D-L\rrb \mbox{ s.t.\ } X_{j+\tau_L^+ +i} = X'_{j'+\tau_L^+ +i} = X_\trivincr\mbox{ for all $i=-L,\ldots,L$}
\}\,.
\end{align*}
We can now define a map $\Phi$ that takes a pair of interfaces and swaps the increment stretch between $j-\tau_L^-$ and $j+\tau_L^+$ in $\cI$ with the stretch between $j' - \tau_L^-$ and $j' + \tau_L^+$ in $\cI'$: refer to Figure~\ref{fig:stationarity-map} for a visualization. 

\begin{figure}
\vspace{-0.3cm}
\centering
  \begin{tikzpicture}
    \node (fig1) at (-4.5,0) {
	\includegraphics[width=.350\textwidth]{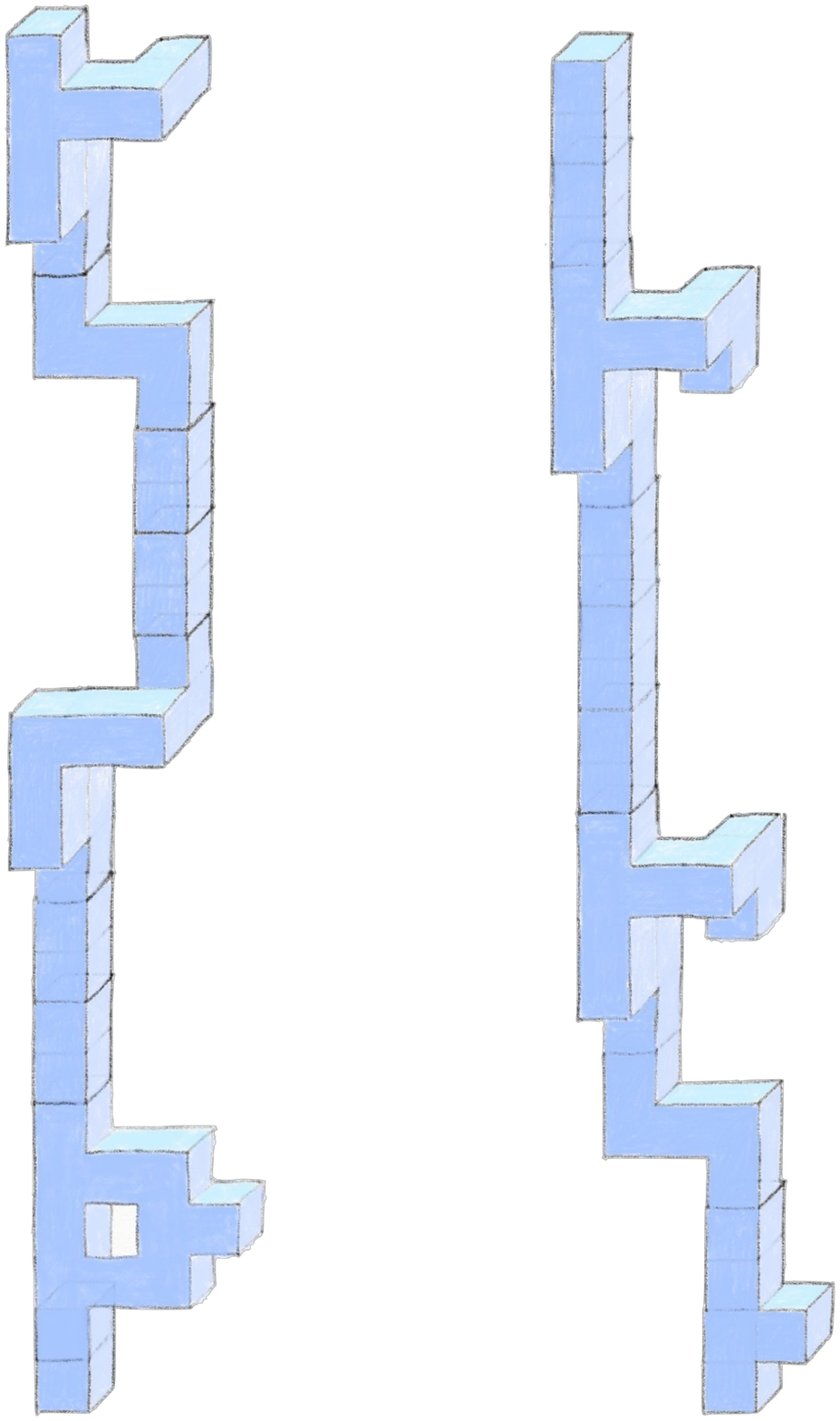}
	};
    \node (fig2) at (4.95,0) {
    \includegraphics[width=.350\textwidth]{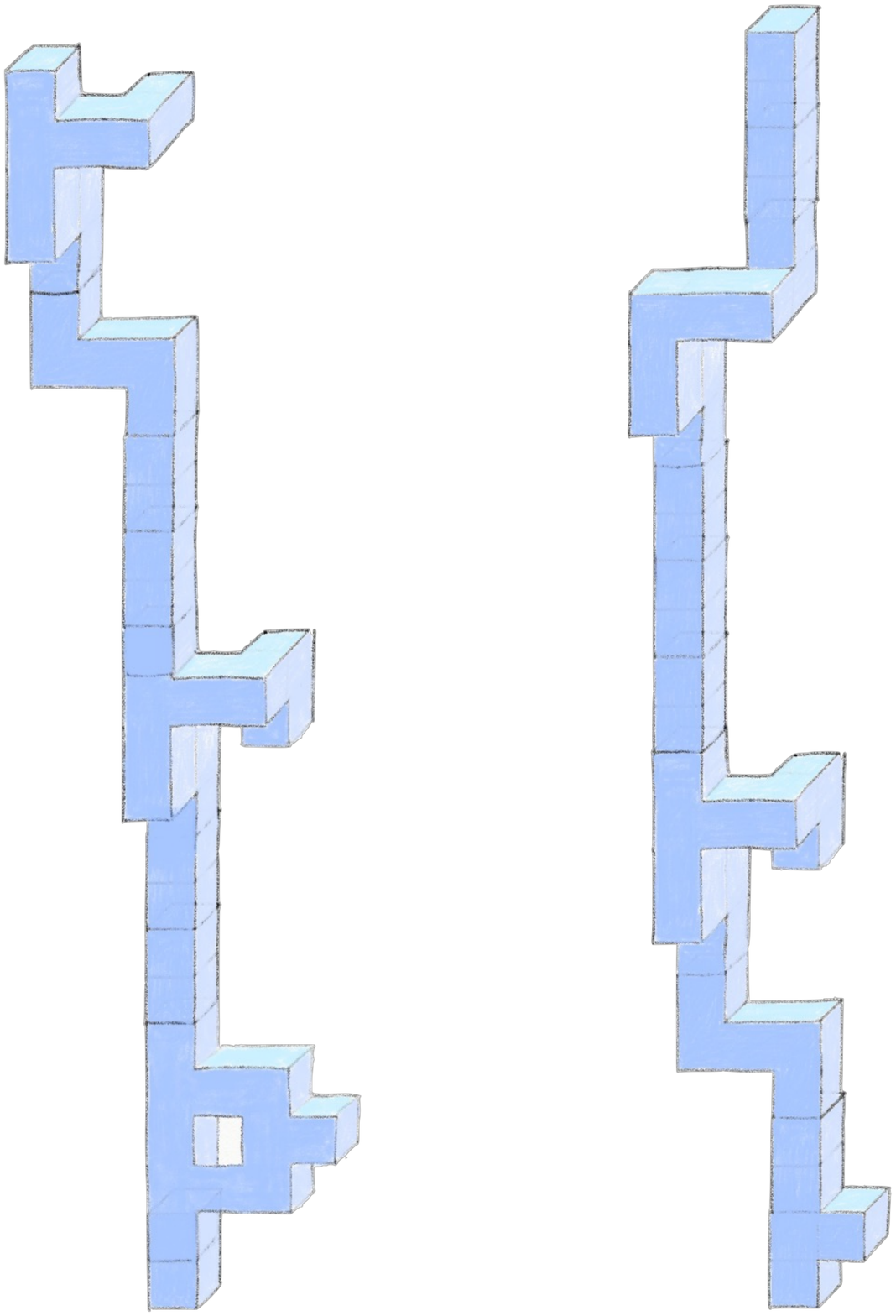}
    };
    
          \draw[->, thick] (-1.5,2) -- (1.5,2);
    \node at (0,2.4) {$\Phi_{\textsc {stat}} = \Phi^{121} \times \Phi^{212}$};
    
        \draw [decorate,decoration={brace,amplitude=5pt}, thick]
(-5.5,1.05) -- (-5.5,-1.3);
    \draw [decorate,decoration={brace,amplitude=5pt}, thick]
(-4,.55) -- (-4,2.85); 
    \draw[<->, thick] (-5.1, -.1) -- (-4.4,1.7);
    
            \draw [decorate,decoration={brace,amplitude=5pt}]
(-6.6,-.9) -- (-6.6,.4) node [black,midway,xshift=-0.5cm] 
{\footnotesize $X_{j}$};

     \draw[->] (2.92, -1.56)--(3.07, -1.56);
     \node at (2.4, -1.5) {$j- \tau_L^-$};
     
          \draw[->] (2.79, .57)--(2.94, .57);
     \node at (2.26, .56) {$j+ \tau_L^+$};
     
               \draw[->] (5.66, .45)--(5.81, .45);
     \node at (5.13, .44) {$j'- \tau_L^-$};
     
                    \draw[->] (5.66, .45)--(5.81, .45);
     \node at (5.13, .44) {$j'- \tau_L^-$};
     
                         \draw[->] (6.14, 2.83)--(6.29, 2.83);
     \node at (5.61, 2.83) {$j'+ \tau_L^+$};

        \draw [decorate,decoration={brace,amplitude=5pt}]
(-3.3,2.35) -- (-3.3,1.15) node [black,midway,xshift=0.4cm] 
{\footnotesize $X_{j'}$};

  \end{tikzpicture}
\vspace{-0.5cm}
 \caption{The map $\Phi_{\textsc{stat}}= \Phi^{121} \times \Phi^{212}$ acting on a pair of increment sequences in $\Gamma_{\trivincr, L}$. }
  \label{fig:stationarity-map}
\vspace{-0.2cm}
\end{figure}

\begin{definition}
For any $j,j'$, let  $\Phi_{\textsc{stat}}= \Phi_{\textsc{stat}}(j,j'): (\fX^T \times \fX_{\textsc{rem}})^2 \to (\fX^{T'} \times \fX_{\textsc{rem}})^2$ be given as follows. For any pair of increment sequences $(X^{(1)}, X^{(2)})= ((X^{(1)}_i)_i, X^{(1)}_{>T}, (X^{(2)}_{i})_i, X^{(2)}_{>T'}$, let 
\[\Phi_{\textsc{stat}}(X^{(1)}, X^{(2)}) = (\Phi^{121}(X^{(1)}, X^{(2)}), \Phi^{212} (X^{(1)}, X^{(2)}))
\]
be attained as follows. If $(X^{(1)}, X^{(2)})\notin \Gamma_{\trivincr,L}^- \cap \Gamma_{\trivincr,L}^+$, let $\Phi_{\textsc {stat}}(X^{(1)}, X^{(2)}) = (X^{(1)}, X^{(2)})$; otherwise
\begin{enumerate}
\item Find the smallest indices $\tau_L^-$ and $\tau_L^+$ for which the events $\Gamma_{\trivincr,L}^-$ and $\Gamma_{\trivincr,L}^+$ are satisfied.  
\item Let $\Phi^{121}(X^{(1)}, X^{(2)})$ is the pair of increment sequences given by 
\[\Phi^{121}(X^{(1)}, X^{(2)})= \big(X_1^{(1)},\ldots, X^{(1)}_{j-\tau_L^-}, X^{(2)}_{j'- \tau_L^- +1},\ldots, X^{(2)}_{j'},\ldots, X^{(2)}_{j' + \tau_L^+}, X^{(1)}_{j+\tau_L^+ +1},\ldots, X^{(1)}_{T}, X^{(1)}_{>T}\big)\,. 
\]
\item Let $\Phi^{212}(X^{(1)}, X^{(2)})$  is the pair of increment sequences given by 
\[\Phi^{212} (X^{(1)}, X^{(2)})  = (X^{(2)}_1,\ldots, X^{(2)}_{j'-\tau_L^-}, X_{j-\tau_L^-+1},\ldots, X^{(1)}_{j},\ldots, X^{(1)}_{j + \tau_L^+}, X^{(2)}_{j'+\tau_L^++1},\ldots, X^{(2)}_{T'}, X^{(2)}_{>T'}\big)\,.
\]
\end{enumerate}
\end{definition}
Abusing notation, we can define $\Phi_{\textsc{stat}}$ on $\bar{\mathbf I}_{x,T}\times \bar{\mathbf I}_{x,T}$ that uses the same truncations of the pair $(\cI, \cI')$ but applies the map $\Phi_{\textsc{stat}}$ to their increment sequences in the pillars $(\cP_x, \cP_x')$. If the two interfaces are both tame and additionally satisfy $\Tsp \leq K\log T$ and $\Tsp' \leq K\log T'$, the resulting pair would be in $\mathbf I_{x,T} \times \mathbf I_{x,T}$.

\subsubsection{Strategy of the map \texorpdfstring{$\Phi_{\textsc{stat}}$}{Phi\_stat}}\label{subsec:strategy-stat}
Similarly to the mixing map, if one were to take a naive approach of constructing a map that sends a single interface to a single interface, a possible choice would be a map that e.g., inserts an increment $X_0$ at the bottom of the increment sequence, shifting the remainder of the increment sequence and showing that the weights of interfaces with $\mathscr X_i = X$ are close to those with $\mathscr X_{i+1} = X$. (Notice that any map we construct must increase the number of increments as we wish to show not only that the law is close to stationary in shifts for fixed $T$, but that it remains stationary as $T\to\infty$.) Similar to the explanation in Section~\ref{subsec:strategy-mix}, however, the addition of an increment means that, the best one could hope for is a ratio of weights that is $1\pm \epsilon_\beta$, rather than $1+o_T(1)$.  

Instead, we use the two-to-two map which shifts an increment $X_j$ in a spine of $T$ increments, to a position $j'$ in a spine of $T'$ increments. As with $\Phi_{\textsc{mix}}$, we use that $j,j'$ are far from $1, T\wedge T'$ to find paired stretches of trivial increments equal distances above and below $X_j$ and $X_j'$. We then splice in the middle of these trivial increment sequences, and use them to decorrelate $X_j, X_j'$ from the rest of their respective interfaces, showing that the relative weight of the pair of interfaces is almost unchanged by the map $\Phi_{\textsc{stat}}$. Refer to Figure~\ref{fig:stationarity-map} for a visualization of this map.

\subsubsection{Analysis of the map \texorpdfstring{$\Phi_{\textsc{stat}}$}{Phi\_stat}}
We now define, analogously to the proof of mixing, a good set of pairs of increment sequences, denoted $\Gamma_L$ on which we can control the ratio of probabilities under the map $\Phi$. Let $\Gamma_L$ be the set of $(\cI, \cI')\in \bar{\mathbf I}_{x,T} \times \bar{\mathbf I}_{x,T}$ such that its pair of increment sequences are in $\Gamma_{\varnothing,L}^- \cap \Gamma_{\varnothing,L}^{+}$, and have  
\begin{enumerate}
\item Their source point indices satisfy $\Tsp\leq K\log T$ and $\Tsp'\leq K\log T'$; denote this event $\Gamma_\Tsp$. 
\item The pair of interfaces $(\cI, \cI')$ are such that $\Phi_{\textsc{stat}}(\cI, \cI')$ are both tame; denote this event $\bar{\Gamma}$. 
\item The increment sequence $(X_i)$ satisfies the events (denoted $\Gamma_{i>\tau_L^-}$ and $\Gamma_{i>\tau_L^+}$)
\begin{align}
|\cF(X_{>T})|e^{- \bar c(T+1- j+ \tau_L^-+L)} + \sum_{i\geq 1} (|\cF(X_{j-\tau_L^-+i})|) e^{- \bar c (L+i)} & \leq e^{ - \bar c L/2}\,, \qquad \mbox{as well as}   \label{eq:F_int-contribution} \\
|\cF(X_{>T})|e^{- \bar c(T+1- j- \tau_L^++L)} + \sum_{i\geq 1} (|\cF(X_{j+\tau_L^+ +i})| e^{ - \bar c (L+i)}  & \leq e^{ - \bar c l/2}\,. \label{eq:F_+-contribution}
\end{align}
and $(X_i')$ satisfies the analogous events with respect to $T'$(denoted $\Gamma'_{i>\tau_L^-}$ and $\Gamma'_{i>\tau_L^+}$). 
\end{enumerate}
As in the proof of Proposition~\ref{prop:increment-mixing}, we can bound the contribution from pairs of interfaces not in $\Gamma_L$ by 
\begin{align*}
\pi_T^{\otimes 2}(\Gamma_L^c) & \leq  2 \pi_T(\Gamma_{\Tsp}^c) + \pi_T^{\otimes 2}( \bar \Gamma^c \mid \Gamma_\Tsp) + \pi_T^{\otimes 2} \big((\Gamma_{\varnothing,L}^-)^c \mid \Gamma_\Tsp\big)+\pi_T^{\otimes 2} \big((\Gamma_{\varnothing,L}^+)^c \mid \Gamma_\Tsp\big) +\pi_{T} (\Gamma^c_{i>\tau_L^-} \mid \Gamma_{\varnothing,L},  \Gamma_{\Tsp})\\
& \quad + \pi_{T} (\Gamma^c_{i>\tau_L^+} \mid \Gamma_{\varnothing,L},  \Gamma_{\Tsp})+ \pi_{T}((\Gamma_{i>\tau_L^-}')^c\mid \Gamma_{\varnothing,L},  \Gamma_{\Tsp})+\pi_{T}((\Gamma'_{i>\tau_L^+})^c \mid \Gamma_{\varnothing,L},  \Gamma_{\Tsp})\,.
\end{align*}
The bounds on the first two terms above are identical to those in the proof of Lemma~\ref{lem:good-pairs-of-interfaces}, so that their contribution is at most $T^{-2\gamma}$. The bounds on the third and fourth terms are as in the proof of Lemma~\ref{lem:good-pairs-of-interfaces}, noticing that on $\Gamma_{\Tsp}$, the sequence of indicator functions $(\one_{\{\sX_{j-i} = X_\trivincr\}}\one_{\{\sX'_{j'-i}= X_{\trivincr}\}})_{i\leq D}$ stochastically dominate  i.i.d.\ $\ber((1-\epsilon_\beta)^2)$ random variables; therefore, their contribution is at most $\exp[-D^{3/4}]$ once $\beta$ is sufficiently large (independently of $j,j'$). The sixth and eight terms above are also bounded as in Lemma~\ref{lem:good-pairs-of-interfaces} by $2D^{-2\gamma}$ using the conditional version of Corollary~\ref{cor:increment-interaction-bound}. 

A crucial difference arises in the bounds on the fifth and seventh terms, since knowledge of $\tau_L^-$ gives information regarding the increment  sequence above index $\tau_L^- +L$ (namely that there is no possible smaller choice of $\tau_L^-$), so Corollary~\ref{cor:increment-interaction-bound} does not immediately bound  $\pi_T(\Gamma^c_{i>\tau_L^-} \mid \Gamma_{\varnothing,L}, \Gamma_\Tsp)$. Instead, we union bound over the $D$ possible choices of $\tau_L^-$ and sustaining this union bound, see that
\begin{align*}
\pi_T(\Gamma^c_{i>\tau_L^-} \mid \Gamma_{\varnothing,L}, \Gamma_{\Tsp}) \leq D e^{ - c\beta \bar c L/2} \leq D^{-2\gamma+1}
\end{align*}
as long as $\beta c>1$, and likewise for $\pi_T((\Gamma'_{i>\tau_L^-})^c \mid \Gamma_{\varnothing,L}, \Gamma_{\Tsp})$. Combining all these estimates yields the desired bound of $\pi_T^{\otimes 2} \leq C D^{-\gamma}$ for $\beta$ sufficiently large (depending on $\gamma$ and $\bar c$).

Now that we've restricted to nice pairs of increment sequences, we can naturally view $\Phi_{\textsc{stat}}$ as a map on $(\bar{\mathbf I}_{x,T} \times \bar{\mathbf I}_{x,T'}) \cap \Gamma_L$: as in Claim~\ref{claim:mixing-bijection}, we arrive at the following claim. 

\begin{claim}
The restriction of $\Phi_{\textsc{stat}}$ to $(\bar{\mathbf I}_{x,T} \times \bar{\mathbf I}_{x,T'}) \cap \Gamma_L$ is a bijection from $(\bar{\mathbf I}_{x,T} \times \bar{\mathbf I}_{x,T'}) \cap \Gamma_L$ to itself. 
\end{claim}
We are therefore left to bound 
\begin{align*}
\sum_{Z_j\in E} \sum_{\tilde Z_{j'} , A_T, A_{T'}} \Big | \pi_T ( Z_j, A_T) \pi_T( \tilde Z_{j'}, A_{T'}) - \pi_T \big(\Phi^{121}((Z_j,A_T), (\tilde Z_{j'},A_{T'}))\big) \pi_T\big(\Phi^{212}((Z_j,A_T), (\tilde Z_{j'},A_{T'}))\big)\Big|\,.
\end{align*} 
In order to bound the summands above, as before, let us focus on the ratio of the probabilities under application of $\Phi_{\textsc{stat}}$, and use the machinery of Theorem~\ref{thm:cluster-expansion}. 

We will use the short-hand $\cI^{ZA} = (Z_j, A_T)$ for the interface in $\bar{\mathbf I}_{x,T}$ with $\sZ_j = Z_j$ and $A_T$ elsewhere, and $\cI^{Z'A'} = (\tilde Z_{j'}, A_{T'})$ for the interface in $\bar{\mathbf I}_{x,T'}$ that has $\sZ_{j'} = \tilde Z_{j'}$ and $A_{T'}$ elsewhere.  Moreover, let 
\begin{align*}
\cI^{ZA'} = \Phi^{212}(\cI^{ZA}, \cI^{Z'A'})\,, \qquad \mbox{and}\qquad \cI^{Z'A} = \Phi^{121}(\cI^{ZA}, \cI^{Z'A'})\,.
\end{align*}
(In particular, $A'$ (resp., $A$) are interfaces of $T'$ (resp., $T$) increments and the $Z$ or $Z'$ in the superscript indicates whether the increments in indices $j'-\tau_L^-, \ldots, j'+\tau_L^+$ (resp., $j- \tau_L^-, \ldots, j+\tau_L^+$) are those coming from $j-\tau_L^-,.., j+ \tau_L^+$ in $(Z,A)$ or $j'-\tau_L^-,\ldots, j'+\tau_L^+$.) Then, for any such interfaces in $(\bar{\mathbf I}_{x,T} \times \bar{\mathbf I}_{x,T'}) \cap \Gamma_L$, 
\begin{align*}
& \frac{\pi_T (\cI^{ZA}) \pi_T ( \cI^{Z'A'})}{\pi_T( \cI^{ZA'}) \pi_T(\cI^{Z'A})}  = \frac{\mu_n(\cI^{ZA}) \mu_n ( \cI^{Z'A'})}{\mu_n( \cI^{ZA'}) \mu_n(\cI^{Z'A})} \\
& \qquad \quad= \exp\Big( \sum_{f\in \cI^{ZA}} \g(f , \cI^{ZA}) + \sum_{f\in \cI^{Z'A'}} \g(f, \cI^{Z'A'})-\sum_{\tilde f\in \cI^{Z'A}} \g(f , \cI^{Z'A}) + \sum_{f\in \cI^{ZA'}} \g(f , \cI^{ZA'})\Big)\,.
\end{align*}

Proposition~\ref{prop:increment-stationarity} then follows from the following lemma, just as in the proof of Proposition~\ref{prop:increment-mixing}. 

\begin{lem}\label{lem:stationarity-part-functions}
There is a universal constant $\bar C$ such that for any pair $(\cI,\cI')\in \Gamma_l$, we have
\begin{align*}
\Big|\sum_{f\in \cI^{ZA}} \g(f,\cI_{ZA})  +\sum_{f\in \cI^{Z'A'}}\g(f,\cI^{Z'A'}) -\sum_{f\in \cI^{Z' A}}\g(f,\cI^{Z'A})  - \sum_{f\in \cI^{ZA'}}\g(f,\cI^{ZA'})\Big| \leq \bar C \exp\big(-  \bar c L/2\big)\,.
\end{align*}
\end{lem} 

We wish to bound the absolute value of the quantity in the exponential by pairing various subsets of the different interfaces together in a manner that they look locally alike. We denote by $F_{\textsc {int}}^{ZA}$ the face set of the increments in $\cI^{ZA}$ between index $j- \tau_L^- $  and $j+\tau_L^+ $ and denote the two connected components of $\cI^{ZA} \setminus F_{\textsc{int}}^{ZA}$ by $F_-^{ZA}$ and $F_+^{ZA}$ respectively. Likewise define the $F_{\textsc{int}}$, $F_-$ and $F_+$ for $\cI^{Z'A'}, \cI^{ZA'}, \cI^{Z'A}$, where if the superscript is $A'$, the interior will have indices $j'- \tau_L^- $ and $j'+\tau_L^+$ (instead of $j- \tau_L^-$ and $j + \tau_L^+$).

Notice that $F_-^{ZA} = F_-^{Z'A}$ and  $F_-^{Z'A'} = F_-^{ZA'}$. We can then define the shift maps $\theta^{(1)}_{A,A'}$ which is the shift by the vector $-v_{j - \tau_L^- +1} + -v'_{j' - \tau_L^-+1}$ (where $v_{j - \tau_L^-+1}$ is the cut-point in $\cI^{ZA}$ and $v'_{j'-\tau_L^-+1}$ is the cut-point in $\cI^{Z'A'}$ and $\theta^{(2)}_{A,A'}$ which is the shift by the vector $-v_{j+ \tau_L^+} + v'_{j'+\tau_L^+}$. With these definitions, we see that 
\begin{align*}
\theta_{A,A'}^{(1)} F_{\textsc {int}}^{ZA} & = F_{\textsc {int}}^{ZA'}\,, & \qquad \theta_{A',A}^{(1)} F_{\textsc {int}}^{Z'A'} & = F_{\textsc {int}}^{Z'A}\,, \qquad \mbox{and} \\ 
 \qquad \theta_{A,A'}^{(2)}\theta^{(1)}_{A'A} F_{+}^{ZA} & = F_{+}^{Z'A}\,, & \qquad  \theta_{A',A}^{(2)}\theta_{A'A}^{(1)} F_+^{Z'A'} & = F_+^{ZA'}\,.
\end{align*}
With this decomposition, we see that  
\begin{align}\label{eq:part-function-splitting-stationarity}
\Big|\sum_{f\in \cI^{ZA}} \g(f,\cI_{ZA}) &  +\sum_{f\in \cI^{Z'A'}}\g(f,\cI^{Z'A'}) -\sum_{f\in \cI^{Z' A}}\g(f,\cI^{Z'A})  - \sum_{f\in \cI^{ZA'}}\g(f,\cI^{ZA'})\Big| \nonumber \\ 
\leq &  \sum_{f\in F_{+}^{ZA}} \Big|\g(f,\cI^{ZA}) - \g(\theta^{(2)}_{A, A'} \theta^{(1)}_{A',A} f, \cI^{Z'A})\Big| + \sum_{f\in F_+^{Z'A'}} \Big|\g(f,\cI^{Z'A'}) - \g(\theta^{(2)}_{A', A} \theta^{(1)}_{A,A'} f, \cI^{ZA'})\Big| \nonumber\\
& + \sum_{f\in F_{\textsc{int}}^{ZA}} \Big| \g(f,\cI^{ZA}) - \g(\theta^{(1)}_{A,A'}f, \cI^{ZA'})\Big| + \sum_{f\in F_{\textsc {int}}^{Z'A'}} \Big|\g(f,\cI^{Z'A'}) - \g(\theta^{(1)}_{A',A} f, \cI^{ZA'})\Big|  \nonumber\\ 
& + \sum_{f\in F_-^{ZA}} \Big| \g(f, \cI^{ZA}) - \g(f,\cI^{Z'A})\Big| + \sum_{f\in F_-^{Z'A'}} \Big|\g(f,\cI^{Z'A'}) - \g(f, \cI^{ZA'})\Big|\,. 
\end{align}
The first two terms are bounded above by $O(e^{ - \bar c L/2})$ analogously to the contribution of faces in $F_+^{jk}$ in~\eqref{eq:part-function-splitting-mixing}; by construction for a face in the $(j+\tau_L^+ +i)$-th increment, the radius $\br (f,\cI^{ZA}; \theta^{(2)}_{A,A'} \theta^{(1)}_{A',A} f, \cI^{Z'A})$ is at least $L + i$; the first $L$ such increments have exactly four faces, and their contribution is thus at most $4\bar C e^{ - \bar c L}$, while the contributions of increments above $j+ \tau_L^+ + L$ is bounded by $\bar K e^{ - \bar c L/2}$ by~\eqref{eq:F_+-contribution}.  

The last two terms in~\eqref{eq:part-function-splitting-stationarity} are bounded in the same manner as the term $F_-^{jk}$ in~\eqref{eq:part-function-splitting-mixing}; for these faces, the radius of congruence is attained by some face in $F_{\textsc {int}}^{ZA} \cup \theta^{(1)}_{A,A'} F_{\textsc {int}}^{ZA} \cup F_+^{ZA} + \theta^{(2)}_{A,A'} \theta^{(1)}_{A',A} F_+^{ZA}$. Then the set of faces can be split into those faces that are between increment $j-\tau_L^- $ and $j- \tau_L^- -L$, whose contribution is easily seen to be at most $4\bar C e^{ - \bar c L}$, and those that are below increment $j-\tau_L^- - L$ along with the truncation $\cI_{\textsc {tr}}$. The contribution of these latter faces is bounded as in the bound of~\eqref{eq:part-function-splitting-mixing}, by additionally summing over the possible faces that attain the radius of convergence, and using integrability of exponential tails to reduce this to a multiple of the quantities~\eqref{eq:F_int-contribution}--\eqref{eq:F_+-contribution}. 

It remains to bound the contribution of the middle two terms, say that of faces in $F_{\textsc{int}}^{ZA}$. These terms can be bounded by decomposing into the event that the radius of congruence is attained by a face in $F_{-}^{ZA} \cup F_-^{Z'A'}$ and the event that it is attained by a face in $F_+^{ZA} \cup \theta^{(2)}_{A',A} \theta^{(1)}_{A,A'} F_+^{Z'A'}$. In the former case, these terms are treated analogously to the first two terms, and therefore their contribution is at most  $\bar C e^{ -\bar c L/2}$ by~\eqref{eq:F_int-contribution}. In the latter case, they are treated analogously to the last two terms, swapping the summation into one over faces in $F_+^{ZA} \cup \theta^{(2)}_{A',A} \theta^{(1)}_{A,A'} F_+^{Z'A'}$, and the contribution is at most $O(e^{ - \bar cL/2})$, by~\eqref{eq:F_+-contribution}. 
\qed

\subsection{Proof of Corollary~\ref{cor:limiting-distribution}: existence of a limiting measure}\label{subsec:limiting-distribution}
We first claim that for each $k$, the subsequence of measures 
\[
\Big(\pi_T((\sX_{\frac T2 - k}\ldots, \sX_{\frac T2}, \ldots\sX_{\frac T2 + k})\in \cdot)\Big)_T
\] 
is a Cauchy sequence in the total-variation metric: indeed for every $T' \geq T$, we have by Proposition~\ref{prop:increment-stationarity} that
\begin{align*}
\|\pi_T((\sX_{\frac T2 - k}\ldots, \sX_{\frac T2}, \ldots\sX_{\frac T2 + k})\in \cdot)- \pi_{T'}((\sX_{\frac {T'}2 - k}\ldots, \sX_{\frac {T'}2}, \ldots\sX_{\frac {T'}2 + k})\in \cdot) \|_\tv \leq C \left(\tfrac T2 - k\right)^{- \gamma}\,.
\end{align*}
 By completeness of the space of probability measures on $\fX^{2k}$ with respect to the total-variation distance, this implies that for each $k$, there exists a limiting measure $\nu_k$ on $\fX^{2k}$ such that the marginals above converge to~$\nu_k$. If the family $(\nu_k)$ is viewed as marginals on $\sX_{-k},\ldots, \sX_0, \ldots, \sX_{k}$ of a limiting law $\nu$ on $\fX^{\Z}$, the Kolmogorov consistency criterion is trivially satisfied as these finite-dimensional distributions are arising as limits of marginals of a single consistent distribution (the law of $\sX_1,\ldots, \sX_T$ under $\pi_T$ viewed about $\sX_{\frac T2}$).

To see that any other sequence $a_T$ satisfying $(a_T\vee T- a_T) \gg \log T$  has the same limit, take any such $a_T$ (without loss of generality $a_T\leq T/2$) as well as any $k$, and bound 
\begin{align*}
\|\pi_{T} ((\sX_{a_T - k} , \ldots, & \sX_{a_T},  \ldots, \sX_{a_T + k} ) \in \cdot) - \nu ( (\sX_{- k} ,\ldots, \sX_{0 } ,\ldots, \sX_{k})\in \cdot)\|_{\tv}  \\
\leq  &   \| \pi_{T} ((\sX_{\frac T2 - k} , \ldots, \sX_{\frac T2}, \ldots, \sX_{\frac T2 + k} ) \in \cdot)- \nu ( (\sX_{- k} ,\ldots, \sX_{0 } ,\ldots, \sX_{k})\in \cdot)\|_{\tv} \\ 
&  + \|\pi_{T} ((\sX_{a_T - k} , \ldots, \sX_{a_T}, \ldots, \sX_{a_T + k} ) \in \cdot) - \pi_{T} ((\sX_{\frac T2 - k} , \ldots, \sX_{\frac T2}, \ldots, \sX_{\frac T2 + k} ) \in \cdot)\|_{\tv}\,.
\end{align*}
The first term on the right-hand side above is $o(1)$ as $T\to\infty$ by the convergence of $\pi_T$ to $\nu$ in total-variation. The second term on the right-hand side above is at most $C ({a_T - K\log T})^{ - \gamma}$ for $K$ satisfying $c_\sB \beta K>2\gamma$ for some $\gamma>2$ by Proposition~\ref{prop:increment-stationarity};  this is also $o(1)$ as $T\to\infty$. 
The two consequences of this follow immediately from the definition of weak convergence and Proposition~\ref{prop:exp-tails-increments} and Proposition~\ref{prop:increment-mixing}.
\qed

\section{Mean and variance of observables of the increment sequence}\label{sec:mean-variance}

In this section, we prove estimates for the mean and variance of running sums of increment observables $f:\fX \to \R$  (these appear in e.g., Theorem~\ref{mainthm:clt}). In Section~\ref{subsec:mean-zero} we prove that any function $f$ with rotational or reflective symmetries, has mean zero under $\nu$. In Section~\ref{subsec:linear-variance}, we prove that non-constant functions will have a variance that diverges linearly in $T$ variance in the central limit theorem. These will, in particular, imply the choices of the mean and covariances in items (1)--(2) of Corollary~\ref{maincor:displacement-surface-area}. 

\subsection{Anti-symmetric observables have mean zero}\label{subsec:mean-zero}
In this section, we prove that for any observable $f$ that is anti-symmetric in reflections or rotations in the $xy$-plane, its mean under $\nu$ is zero as long as $\beta$ is sufficiently large. In particular, its central limit theorem, holds without any recentering. The proof follows by applying a reflection map above an atypically long stretch of trivial increments, and seeing that this map essentially leaves the probability distribution over the increment sequence invariant. 

We say a map $\varphi: \fX \to \fX$ is a \emph{reflection} map if it is a reflection about one of the two planes with normal vector $e_1$ or $e_2$ going through the point $(\frac 12, \frac 12, \frac 12)$. We say it is a \emph{rotation map} if it is a $\frac{\pi n}4$ rotation about the $e_3$ axis through $(\frac 12, \frac 12, \frac 12)$. (Notice that the trivial increment $X_{\trivincr}$ is fixed by any of these maps.) The same $\varphi$ can naturally also be viewed as a map on remainder increments. (Note that $\varphi(X_\trivincr) = X_\trivincr$.)

\begin{proposition}\label{prop:mean-zero-observables}
 There exists $\beta_0>0$ such that the following holds for every $\beta>\beta_0$. If $f: \fX \to \R$ satisfies $|f(X)| \leq M$ and $f(X) = - f(\varphi(X))$ for all $X\in \fX$, for some $M\in\R$ and reflection or rotation map $\varphi$, then
\begin{enumerate}
\item $\E_{\nu} [ f(\sX_0 ) ] = 0$\;; 
\item if $g:  \fX \to \R$ has $|g(X)|\leq M$ and $g(\varphi(X)) = g(X)$ for every $X$, then 
\[\E_{\nu} \Big[\sum_{i\in \Z} f(\sX_0 ) g(\sX_i)\Big] = \sum_{i\in \Z} \cov_\nu (f(\sX_0), g(\sX_i)) = 0\,.
\]
\end{enumerate}
\end{proposition}

The proof of Proposition~\ref{prop:mean-zero-observables} relies crucially on bounding the effect of a map that reflects or rotates the pillar above some stretch of $O(\log T)$ consecutive trivial increments. To that end, let $\varphi$ be a reflection or rotation map and define the map $\Phi_\varphi:\bar{\mathbf I}_{x,T} \to \bar{\mathbf I}_{x,T}$ as follows.  
\begin{definition}\label{def:reflection-map}
For a given $L$, we can denote $\tau_{L}=\tau_L(\cI)$ as the smallest index greater than $\Tsp+L$ such that all of $X_{\tau_L -L},\ldots, X_{\tau_L} = X_\trivincr$. Then, for an interface $\cI$, let $\Phi_\varphi$ agree with $\cI$ on its truncation and its increments up to the $(\tau_{L} - \frac L2)$-th increment, then apply the map $\varphi$ to all increments with index ranging from $\tau_L - \frac L2$ to $T$, as well as the remainder increment. Notice that this is the same as applying the map $\varphi$ to the entire subset of the pillar above the $(\tau_L - \frac{L}{2})$-th increment, by correspondingly reflecting/rotating it about the $e_3$ axis going through $v_{\tau_L - \frac L2+1}$. (If $\tau_L$ does not exist, then let $\Phi_{\varphi}$ be the identity.) 
\end{definition}

\begin{claim} \label{claim:ratio-of-probabilities-reflection}
For every reflection or rotation map $\varphi:\fX \to\fX$, every $\cI\in  \bar{\mathbf I}_{x,T}$, and every $\gamma>1$, there exists some $\beta_0>0$ such that, for all $\beta>\beta_0$, 
\begin{align*}
\Big|\frac {\pi_{T}(\cI)}{\pi_T(\Phi_\varphi(\cI))}- 1\Big| \leq O(T^{-\gamma+1})\,.
\end{align*}
\end{claim}
\begin{proof}
Fix any $\gamma$ and let $L = \lceil \frac{2\gamma}{\bar c} \log T \rceil$. By Theorem~\ref{thm:cluster-expansion} and Definition~\ref{def:reflection-map}, for $\cI \in \bar{\mathbf I}_{x,T}$, 
\begin{align*}
\frac{\pi_T(\cI)}{\pi_T(\Phi_\varphi (\cI))} = \frac {\mu_n(\cI)}{\mu_n(\Phi_\varphi (\cI))} = \exp\Big( \sum_{f\in \cI} \g(f,\cI ) - \sum_{f'\in \Phi_\varphi (\cI)} \g(f', \Phi_\varphi (\cI))\Big)\,.
\end{align*}  
For every interface $\cI$, let us split its faces up as $\cI^-$ denoting the union of $\cI_{\textsc{tr}}$ and the increment sequence up to $X_{\tau_L- \frac L2}$, and $\cI^+$ denoting the union of the increments above $X_{\tau_L - \frac L2}$ along with the remainder $X_{>T}$. Moreover, for a face $f \in \cS_x$, let $\varphi(f)$ be the image of that face $f$ under the reflection/rotation map $\varphi$, viewed as a face in $\Phi_\varphi (\cI)$. Then,  
\begin{align*}
\Big| \sum_{f\in \cI} \g(f, \cI ) - \sum_{f'\in \Phi_\varphi (\cI)} \g(f',\Phi_\varphi (\cI))\Big|\leq  \sum_{f\in \cI^-} |\g(f, \cI) - \g(f,\Phi_\varphi (\cI))| + \sum_{f\in \cI^+} |\g(f,\cI) - \g(\varphi(f), \Phi_\varphi(\cI))|\,.
\end{align*}
It is clear that if $f\in X_i$ for $i>\tau_L - \frac L2$, the radius $\br(f,\cI; \varphi(f), \Phi_\varphi(\cI))$ is attained by a face a distance at least $i - (\tau_L- L)$, because the spine is tamed. We used crucially that in Theorem~\ref{thm:cluster-expansion}, the radius of congruence is congruence up to rotation and reflection in the $xy$-plane, and the increments between $\tau_L - L$ and $\tau_L$ are fixed by such reflection and rotations. Consequently, 
\begin{align*}
\sum_{f\in \cI^+} e^{ - \bar c \br(f,\cI; \varphi(f), \Phi_\varphi(\cI))} \leq 4 \bar C e^{ - \bar c L/2} +  \sum_{i \geq 1} \bar K |\cF(X_{\tau_L+i})| e^{ - \bar c (L + i)}
\end{align*}
which is at most $O(T^{ - \gamma+1})$ by the tameness of $\cI$ and the choice of $L$. 

At the same time, for each $f\in \cI^-$, the radius $\br ( f, \cI; f, \Phi_\varphi(\cI))$ is attained by a face in $\cI^+ \cup \Phi_\varphi(\cI^+)$, so that proceeding as usual with these terms, their contribution is bounded by $\bar K$ times
\begin{align*}
\sum_{f\in \cI^-} e^{ - \bar c \br (f,\cI; f,\Phi_\varphi (\cI))} &  \leq \sum_{i\leq L/2} 4 e^{- \bar c (i+\frac L2)}+ \sum_{g\in \cI^+ \cup \Phi_\varphi (\cI^+)} \Big[\sum_{f\in \cI_{\textsc {tr}}} e^{ - \bar c d(f, g)} + \sum_{\Tsp \leq i\leq \tau_L-L} \sum_{f\in \cF(X_i)} e^{ - \bar c d(f, g)}\Big]  \\
&  \leq 4\bar C e^{- \bar c L/2} + \sum_{i'\geq 1}  \sum_{g\in \cF(X_{\tau_L+i'})\cup \cF(\Phi_\varphi (X_{\tau_L+i'}))} \sum_{ f\in \cF(\Z^3) :\, d(f,g)\geq L+i' } e^{- \bar c d(g,f)} \\ 
& \leq 4\bar Ce^{- \bar c L/2} + \sum_{i' \geq 1} 2|\cF(X_{\tau_L+i'})| \bar C e^{ - \bar c (L+i')}\,. 
\end{align*}
which is at most $O(T^{-\gamma +1})$ since $\cI$ is tame. 
 Putting these together implies that for every tame interface $\cI$ (otherwise $\pi_T(\cI) = 0$), the log of the ratio of probabilities is $O(T^{-\gamma+1})$ as desired.
\end{proof}

\begin{proof}[\textbf{\emph{Proof of Proposition~\ref{prop:mean-zero-observables}}}]
By Corollary~\ref{cor:limiting-distribution}, and boundedness of $f$, 
\begin{align*}
\E_{\pi_T}\Big[\sum_{i\leq T}  f(\sX_i) \Big] 
 = O(\log T) + T (1-o(1))\E_{\nu}[f(\sX_i)]\,.
\end{align*} 
Consequently, if we prove that the left-hand side is $o(T)$, it will imply that $\E_{\nu} [f(\sX_0)] =0$. 
We can split up  
\begin{align*}
\E_{\pi_T} \Big[ \sum_{i \leq T} f(\sX_i)\Big] =  \E_{\pi_T} \Big[ \sum_{i=1}^{\tau_L} f(\sX_i) + \sum_{i=\tau_L+1}^{T} f(\sX_i)\Big]\,,
\end{align*} 
for $\tau_L$ as in Definition~\ref{def:reflection-map}, 
and begin by bounding the first of these sums. Recall that by Corollary~\ref{cor:increment-interaction-bound}, the sequence $(\one_{\{\sX_i= X_\trivincr\}})_{i\geq \Tsp}$ stochastically dominates a sequence of i.i.d. $\ber(1-\epsilon_\beta)$ for some $\epsilon_\beta>0$ satisfying $\epsilon_\beta\to 0$  as $\beta\to\infty$. Using this, we can estimate
\begin{align*}
\E_{\pi_T} \Big[ \sum_{i \leq \tau_L} \big|f(\sX_i)\big| \Big]  \leq MT \pi_T(\tau_L \geq T^{1/4})+ \E_{\pi_T} \Big[ \sum_{i\leq T^{1/4}}|f(\sX_i)|\Big]\,.  
\end{align*} 
In order for $\tau_L \geq T^{1/4}$, either $\Tsp \geq K\log T = O(T^{-\gamma})$ (for $K$ large enough), which by Proposition~\ref{prop:base-exp-tail} has probability $e^{ - c\beta K\log T}$, or there is no stretch of $L$ consecutive $X_\trivincr$ increments in the first $T^{1/4}- K\log T$ spine increments; as argued in the proof of Lemma~\ref{lem:good-pairs-of-interfaces}, for large enough $\beta$ (depending on $\gamma, \bar c$) this latter probability is at most $\exp[-T^{3/16}]$. The second term above is at most $MT^{1/4}$ by the bound on $f$. 

Let us now turn to $|\E_{\pi_T} [ \sum_{i= \tau_L+1 , \ldots, T} f(\sX_i)]| \leq \sup_{\tau_L} \E_{\pi_T} [ \sum_{i= \tau_L+1, \ldots, T} f(\sX_i) \mid \tau_L]$. For each instantiation of $\tau_L$, we can expand,   
\begin{align*}
\E_{\pi_T} \Big[ \sum_{i=\tau_L+1, \ldots, T} f(\sX_i)  \mid \tau_L\Big] &  \leq  MT  \E_{\pi_T} \Big[\Big| \frac{\pi_T(\cI)}{\pi_T(\Phi_\varphi(\cI))} -1\Big| \mid \tau_L\Big]\,.
\end{align*}
By Claim~\ref{claim:ratio-of-probabilities-reflection} the quantity in the expectation is $O(T^{-\gamma +1})$ for every tame interface, and therefore also in expectation under $\E_{\pi_T} [ \cdot \mid \tau_L]$ for every $\tau_L$. All in all, we have 
\begin{align*}
\E_{\pi_T} \Big[ \sum_{i \leq T} f(\sX_i) \Big] = O(T^{1/4}) + O(T^{-\gamma+1}) + O(T^{-\gamma+2})\,,
\end{align*}
which is $o(T)$ as long as $\gamma>1$ implying item (1). 

Let us now turn to the proof of item (2). The proof is analogous and we therefore do not include all details. Suppose by way of contradiction that $\E_{\nu} [\sum_{i} f(\sX_0) g(\sX_i)] \neq 0$. We claim that it suffices, in order to obtain a contradiction, to show that for sufficiently large $K$, the following is $o(T)$:
\begin{align}\label{eq:covariance-splitting}
\E_{\pi_T} \Big[ \sum_{i \leq T} \sum_{j \leq T} f(\sX_i) g(\sX_j)\Big] & = \E_{\pi_T} \Big[ \sum_{i ,j \in \llb 1, T^{1/4}\rrb \cup \llb T- T^{1/4}, T\rrb} |f(\sX_i)| |g(\sX_j)|  \Big] \nonumber\\
& + \E_{\pi_T} \Big[\sum_{ i,j \in \llb T^{1/4}, T- T^{1/4}\rrb:\, d(i,j)\geq K\log T} f(\sX_i) g(\sX_j) \Big] \nonumber\\ 
&  + \sum_{i\in \llb T^{1/4}, T- T^{1/4}\rrb} \E_{\pi_T} \Big[ \sum_{j:\, d(i,j)\leq K\log T} f(\sX_i) g(\sX_j)\Big]\,.  
\end{align} 
To see that this is sufficient, notice that the first term of~\eqref{eq:covariance-splitting} is at most $ 16 M^2 K^2 \log^2 T$. Arguing as in item (1) above, by Claim~\ref{claim:ratio-of-probabilities-reflection}, $\E_{\pi_T}[f(\sX_i)]=O(T^{-\gamma})$ if $i\geq T^{1/4}$.  Using Proposition~\ref{prop:increment-mixing}, for each $i,j$ at least $T^{1/4}$ away from $1$ and $T$, with $d(i,j)\geq K\log T$, if $K$ is sufficiently large  we have  
\begin{align*}
\E_{\pi_T} \Big[ f(\sX_i) g(\sX_i) \Big] & \leq \E_{\pi_T}[f(\sX_i)]\E_{\pi_T} [g(\sX_j)]  + M^2 \|\pi_T(\sX_i \in \cdot, \sX_j\in \cdot)- \pi_T(\sX_i\in \cdot) \pi_T(\sX_j\in \cdot)\|_\tv\,,
\end{align*}
which is at most $MT^{-\gamma}+ M^2T^{-\gamma}$ for $\gamma>2$, so that the second term in~\eqref{eq:covariance-splitting} is $o(1)$. Finally, by Proposition~\ref{prop:increment-stationarity}, specifically Corollary~\ref{cor:limiting-distribution}, together with Proposition~\ref{prop:increment-mixing}, we deduce that  
\begin{align*}
\E_{\pi_T} \Big[ \sum_{i \leq T} \sum_{j \leq T} f(\sX_i) g(\sX_j)\Big] = O(\log^2 T)+ O(T^{-\gamma+2}) + T(1-o(1)) \E_{\nu} \Big[ \sum_{j\in \Z} f(\sX_0) g(\sX_j)\Big]\,.
\end{align*}
Thus, if we showed that the left-hand side of~\eqref{eq:covariance-splitting} is $o(T)$, we would deduce that $\E_{\nu} [ \sum_{j\in \Z} f(\sX_0)g(\sX_j)] =0$. 
Proceeding as in item (1), it suffices to show that the following is $o(T)$ as $T\to\infty$:  
\begin{align*}
\E_{\pi_T} \Big[ \sum_{i,j\geq T^{1/4}} f(\sX_i) g(\sX_j)\Big] = M^2 T^2 \pi_T(\tau_L \geq T^{1/4}) + \sup_{\tau_L\leq T^{1/4}} \E_{\pi_T} \Big[\sum_{\tau_L+1 \leq i,j \leq T} f(\sX_i) g(\sX_j) \mid \tau_L \Big]\,.
\end{align*}
As before, the first term is $O(\exp[-T^{3/16}])$ for large enough $\beta$ (depending on $\gamma, \bar c$). The second term is bounded by 
\begin{align*}
M^2 T^2 \sup_{\tau_L \leq T^{1/4}} \E_{\pi_T} \Big[ \Big|\frac{\pi_T(\cI)}{\pi_{T}(\Phi_{\varphi}(\cI))} - \Big| \given \tau_L \Big] \leq O(M^2 T^2 T^{- \gamma+1})\,,
\end{align*}
for some $\gamma>2$ as long as $\beta$ is large enough, which is in turn $o(T)$ in $T$.  
\end{proof}

\subsection{Linearity of variances}\label{subsec:linear-variance}
In this section, we prove that the running sum of $T$ increment observables $f$, for every $f$ that is non-constant on the set of possible increments $\fX$, will have a variance of order $T$. This will in particular imply such a scaling for the variance of the total surface area of a pillar, the excess area of a pillar, and its $xyz$-displacements, conditional on having $T$ increments.  

\begin{proposition}\label{prop:lin-var}
There exists $\beta_0$ such that for every $\beta>\beta_0$ the following holds. If $f:\fX \to\R$ is bounded, $|f(X)|\leq M$ for all $X \in\fX$, and moreover, there exist distinct $X_A, X_B\in \fX$ such that $f(X_A) \neq f(X_B)$, then 
\begin{align*}
\sum_{j\in \Z} \cov_\nu \big(f(\sX_0), f(\sX_j)\big) = \upsigma_f^2 >0\,.
\end{align*}
\end{proposition}
\begin{proof}
It will suffice for us to prove that the following variance simply diverges as a function of $t$:
\begin{align*}
\var_\nu \Big(\mbox{$\sum_{i\in \llb -t,t\rrb}$} f(\sX_i)\Big) \to \infty \quad \mbox{ as $t\to\infty$}\,.
\end{align*}
Indeed, this will follow from the next well-known claim; we include its short proof for completeness.
\begin{claim}\label{clm:nu-sigma-cov}
Let $(Z_k)_{k\in\Z}$ be stationary with $\sum_{k\in\Z} |\cov(Z_0,Z_k)|<\infty$,
and let $V_t = \var(\sum_{k=-t}^t Z_k)$.
Then  
\[\exists \lim_{t\to\infty}\frac{V_t}{2t}=:\upsigma^2\geq 0\quad\mbox{and}\quad \upsigma^2 = \sum_{k\in\Z}\cov(Z_0,Z_k)\,.\] Furthermore, if $\sum_{k\in\Z} |k||\cov(Z_0,Z_k)|<\infty$ then $\upsigma = 0$ iff $\sup_t V_t < \infty$.
\end{claim}
\begin{proof}
Let $S_n=\sum_{k=1}^n Z_k $, so $\var(S_n) = \sum_{t=1}^n \fC_t$ for $\fC_t=\var(Z_t) + 2\sum_{k=1}^{t-1} \cov(Z_k,Z_t)$. By the stationarity, $\fC_t = \sum_{|k|\leq t-1} \cov(Z_0,Z_k)$, and since $\lim_{t\to\infty} \fC_t$ exists (by the absolute convergence hypothesis for this sum), C\'esaro's lemma shows  this limit is equal to  $\lim_{n\to\infty}\frac1n \var(S_n) = \upsigma^2 \geq 0$. For the last statement, if $\upsigma=0$ then  
$\fC_t = -\sum_{|k|\geq t}\cov(Z_0,Z_k)$, whence $|\var(S_n)|= |\sum_{t=1}^n \fC_t|\leq \sum_{t=1}^\infty |\fC_t|\leq \sum_{k\in\Z}|k||\cov(Z_0,Z_k)|$.
\end{proof}

In fact, it suffices for us to show that the following diverges as $T\to\infty$:  
\begin{align*}
\var_{\pi_T} \Big(\sum_{i\leq T} f(\sX_i)\Big)= \sum_{i,j:\,d(i,j)>K\log T} \cov_{\pi_T} \big(f(\sX_i),f(\sX_j)\big)+ \sum_{i,j:\, d(i,j)\leq K\log T} \mbox{Cov}_{\pi_T} \big(f(\sX_i),f(\sX_j)\big)\,.
\end{align*}
This is because by Proposition~\ref{prop:increment-mixing}, the first sum on the right-hand side is $O(T^{-\gamma+1})$ for $\gamma>1$ as long as $\beta$ is large enough, and by Corollary~\ref{cor:limiting-distribution}, the second sum is $(1+o(1))\var_{\nu}(\sum_{i\in \llb-\frac{T}{2}, \frac T2\rrb} f(\sX_i))$. 

The strategy to show this will be to find long stretches of trivial increments, which serve to decorrelate increments, and inject variance coming from either an $X_A$ or $X_B$ increment, into their centers. These injections will behave essentially independently, and therefore, will add some amount of variance proportional to the number of long stretches of trivial increments found. 
Fix $L= \lceil \tfrac{5}{\bar c} \log T\rceil$.
For any interface $\cI\in \bar{\mathbf I}_{x,T}$, mark the first $T^{3/4}$ indices $j_k$ in increasing order, that satisfy $j_k \geq \Tsp$ and have $j_k \geq \Tsp$, $\sX_{j_k-L},\ldots,\sX_{j_k-1}, \sX_{j_k+1},\ldots,\sX_{j_k+L} = X_\trivincr$ along with $\sX_j \in \{X_A, X_B\}$. 
Let $\cG= \cG(L)$ be the $\sigma$-algebra generated by the truncated interface $\cI_{\textsc{tr}}$,   the sequence $(j_k)_{k\leq T^{3/4}}$ and all increments $(\sX_i)_{i\notin \{j_k\}}$. By the law of total variance, we can express 
\begin{align*}
\var_{\pi_T}\Big(\mbox{$\sum_{i\in \llb 1,T\rrb}$} f(\sX_i)\Big)  \geq \E_{\pi_T} \Big [\var_{\pi_T} \Big(\mbox{$\sum_{i\in \llb 1,T\rrb}$} f(\sX_i)\given \cG\Big)\Big]\,.
\end{align*}
However, conditionally on $\cG$, the only contributions to the variance come from the increments $\sX_{j_k}$, so that this quantity is the same as 
\begin{align*}
\E_{\pi_T} \Big[ \var_{\pi_T} \Big(\sum_{k\leq T^{3/4}} f(\sX_{j_k}) \given (\sX_{j_k + \ell})_{\ell \in \llb -L,L\rrb \setminus \{0\}} = X_{\trivincr}, (\sX_{j_k}) \in \{X_A,X_B\}, (\sX_{j})_{d(j, \bigcup_k \{j_k\})>L},\cI_{\textsc{tr}}\Big)\Big]\,.
\end{align*} 
Now fix any set of indices $(j_k)_{k\leq T^{3/4}}$ which identify the trivial increments surrounding them, as well as the fact that $\sX_{j_k}$ is either $X_A$ or $X_B$, and also fix all the other increments $(\sX_j)_{d(j, \bigcup _k \{j_k\})>L}$. We will show that for most such choices, the sum $\sum_k f(\sX_{j_k})$ has a variance that diverges in $T$. 

Let us define a good set $\Gamma_L$ in $\cG$ on which we can prove the variance above is at least $T^{3/4}$, say, as follows: an element of $\cG$, given by $\cI_{\textsc{tr}}, \{j_k\}_k, (\sX_j)_{d(j,\bigcup_k j_k) >L}$ is in $\Gamma_L$ if there are indeed $T^{3/4}$ many $\{j_k\}$ and for every assignment of $X_A, X_B$ to $\{j_k\}_k$, the resulting interface is tame. We will prove that ${\pi_T} (\Gamma_L) \geq 1-o(1)$, and then that for any element of $\Gamma_L$, the variance of $\sum_{k\leq T^{3/4}} f(\sX_{j_k})$ goes to infinity with $T$. 

\begin{claim}\label{claim:increment-lower-bound}
For every $i\leq T$, every $\cI_{\textsc{tr}}$ with $\Tsp \leq i$, and every sequence of increments $(X_j)_{j<i}$, for any fixed increment $X_\star \in \fX$, we have 
\begin{align*}
\mu_n(\sX_i = X_\star \mid \cI_{\textsc {tr}}, (X_j)_{j<i}, \bar{\mathbf I}_{x,T}) \geq \exp\big[- (\beta + C)\fm(X_\star)\big]\,.
\end{align*}
The same estimate holds if we condition, e.g., on the increments above $X_i$ as long as the first $L$ are trivial: in particular, for every $k\leq T^{3/4}$ and  $X_\star \in \{X_A, X_B\}$, 
\[
\inf_{G\in \Gamma_L} \pi_T (\sX_{j_k} = X_\star \mid G ) \geq \exp\big[ - (\beta + C)\fm(X_\star)\big]\,.
\]
\end{claim}

\begin{proof}[\emph{\textbf{Proof of Claim~\ref{claim:increment-lower-bound}}}] In the interest of brevity we do not include a full proof. The first bound can be shown via a similar (simplified) version to the proof of Proposition~\ref{prop:exp-tails-increments}, with the following modifications. Define a map $\Psi_{\star}:\bar{\mathbf I}_{x,T}\to \bar{\mathbf I}_{x,T}$ which replaces the $i$-th increment of a pillar by $X_\star$; one can readily see that for every $\cI$ with $T$-admissible truncation $\cI_{\textsc{tr}}$ with $\Tsp <i$ and increment sequence $(X_i)_{i\leq T}$, we have 
\[\Big|\log \frac{\pi_T( \cI)}{\pi_T(\cI_\star)}+ \beta \fm(X_{i}; X_\star)\Big| \leq \bar K (|\cF(X_{i})|+ |\cF(X_{\star})|)+ \sum_{j>i} \bar C \bar K |\cF(X_j)|e^{ - \bar c (j-i)}\,.
\]
We can bound the latter term on the right-hand side above by Corollary~\ref{cor:increment-interaction-bound}, and we can bound the multiplicity of the map for interfaces with $\fm(\cI;\Phi_\star \cI) = k$ by $s^{k}$ via Observation~\ref{obs:counting-connected}. Together these would imply the desired estimate, as the bound of Corollary~\ref{cor:increment-interaction-bound} holds uniformly over all increment sequences below the $i$-th one, and the map $\Psi_\star$ leaves those increments fixed.

The second part (where we may condition also on the increments sequence above $i$ whilst  in $\Gamma_L$) is similar: for an interface in $\Gamma_L$, as the first $L$ increments above $\sX_{j_k}$ are trivial and the increment sequence is tame, 
\[
\sum_{j>i} |\cF(X_j)|e^{ - \bar c(j-i)} \leq 4\bar C + \bar C Te^{- \bar c L}\,, 
\]
and this is $O(1)$ by our choice of $L$. Therefore, applying the map $\Psi_\star$ for the $j_k$-th increment, we see that the probabilities of having $X_A, X_B$ at marked indices $\{j_k\}$ are comparable.
\end{proof}

First, by Proposition~\ref{prop:base-exp-tail}, with high probability the truncated interface $\cI_{\textsc{tr}}$ is such that $\Tsp \leq T^{1/4}$, so let us work only with truncated interfaces that satisfy that bound. 
By Proposition~\ref{prop:exp-tails-increments} and Claim~\ref{claim:increment-lower-bound}, for any stretch of $2L+1$ increments, the probability of the first and last $L$ being trivial increments, and the middle element being in $\{X_A, X_B\}$ is at least $e^{-(\beta + C)\fm(X_A)} (1-\epsilon_\beta)^{2L}$ for some $\epsilon_\beta$ going to zero as $\beta \to\infty$. There are $\frac{T- T^{1/4}}{2L+1}$ stretches for which this lower bound holds independently of the others; as before, a simple calculation yields that for $\beta$ sufficiently large, the probability of having $T^{3/4}$ such increment stretches of $L$ trivial increments, an element of $\{X_A,X_B\}$ then another $L$ trivial increments, is $1-o(1)$.  
Now, let us lower bound the following quantity by something diverging as $T\to\infty$: 
\begin{align*}
\inf_{G\in \Gamma_L} \var_{\pi_T} \Big(\sum_{k\leq T^{3/4}} f(\sX_{j_k}) \mid G\Big)= \inf_{G\in \Gamma_L} \sum_{k,k'\leq T^{3/4}} \cov_{\pi_T}\Big(f(\sX_{j_k}),f(\sX_{j_{k'}}) \given G\Big) \,.
\end{align*}
By the second item in Claim~\ref{claim:increment-lower-bound}, for every $G\in \Gamma_L$, the contribution of the diagonals, satisfies
\begin{align*}
\sum_{k\leq T^{3/4}} \var_{\pi_T} (f(\sX_{j_k}) \mid G) \geq c_{f,\beta} T^{3/4}\,,
\end{align*}
for some $c_{f,\beta}>0$, as $f$ takes on different values on $X_A$ and $X_B$, and both have strictly positive probability under $\pi_T (\cdot \mid G)$. On the other hand, we claim that the contribution from any off-diagonal term, $\cov_{\pi_T} (f(\sX_{j_k}), f(\sX_{j_{k'}}) \mid G)$ is at most $\bar C T e^{ -  \bar c L/2}$ for every pair $k,k'$. This can be shown via a straightforward modification of the map $\Phi_{\textsc{mix}}$ of Proposition~\ref{prop:increment-mixing}; namely, the map would use $j_k + \frac L2$ as the index above which it swaps the increment sequences. Following this through would imply that 
\[
\|\pi_T(\sX_{j_k} \in \cdot, \sX_{j_{k'}} \in \cdot) - \pi_T(\sX_{j_k}\in \cdot) \pi_{T}(\sX_{j_{k'}}\in \cdot)\|_\tv \leq \bar C T e^{ - \bar c L/2}\,,
\]
thus $\cov_{\pi_T}(f(\sX_{j_k}), f(\sX_{j_{k'}}))\leq \bar C M^2 T e^{- \bar c L/2}$ which is $o(T^{-2})$ by our choice of $L$. Therefore, we see that 
\begin{align*}
\var_{\pi_T} \Big(\sum_{i\in \llb 1,T \rrb} f(\sX_i)\Big) \geq  \inf_{G \in \Gamma_L} \var_{\pi_T} \Big(\sum_{k\leq T^{3/4}} f(\sX_{j_k}) \mid G \Big) + o(1) \geq c_{f,\beta} T^{3/4}+ o(1)\,,
\end{align*}
which diverges as $T\to\infty$, yielding the desired. 
\end{proof}

\section{Central limit theorem for observables of the increment sequence}\label{sec:clt}
In the section we prove the following proposition, which, combined with Remark~\ref{rem:clt-conditioning} yields the CLT from Theorem~\ref{mainthm:clt} as well as the expressions for the mean and variance of the limiting distribution in terms of the measure $\nu$ on bi-infinite sequences of increments that was derived in~\S\ref{sec:mixing-stationarity}, Corollary~\ref{cor:limiting-distribution}.

\begin{proposition}\label{prop:clt}
There exist $\beta_0,\kappa_0>0$ such that the following holds. For every $\beta>\beta_0$, every non-constant function of the increments $f: \fX \to \R$ such that
\begin{equation}\label{eq:f-growth} f(X) \leq \exp(\kappa_0 |\cF(X)|) \quad\mbox{for all $X\in\fX$}\,,\end{equation}
every sequence $1 \ll T_n \ll n$, and every $x_n\in\llb -n+\Delta_n , n- \Delta_n \rrb^2\times\{0\}$ for $\Delta_n\gg T_n$, the increment sequence $\{\sX_i\}$ of $\cP_{x_n}$ under $\pi_T$, the Ising measure conditioned on  $\bar{\mathbf I}_{x_n,T_n}$, satisfies
\begin{align*}
\frac1{\sqrt T_n}\sum_{t=1}^{T_n} (f(\sX_t)-\lambda)\implies  \cN(0, \upsigma_f^2) \quad\mbox{ as } n\to\infty\,,
\end{align*}
where \[ \lambda = \E_\nu [f(\sX_0)]\,,\qquad \upsigma_f^2 = \sum_{j=-\infty}^{\infty} \cov_\nu(f(\sX_0),f(\sX_j)) > 0\]  
for the measure $\nu$ on bi-infinite sequences of increments $(\sX_i)_{-\infty}^{\infty}$ given by Corollary~\ref{cor:limiting-distribution}.
\end{proposition}
Modulo this result, the CLT readily extends to $f:\fX\to\R^d$ for any $d$:
\begin{corollary}\label{cor:clt:Rd}
In the setting of Proposition~\ref{prop:clt}, if $f : \fX \to \R^d$ for some fixed $d\geq 2$, where $f=(f_1,\ldots,f_d)$ is such that each $f_i$ is non-constant and satisfies~\eqref{eq:f-growth}, then under $\pi_T$
\[ 
\frac1{\sqrt T_n}\sum_{t=1}^{T_n} (f(\sX_t)-(\lambda_1,\ldots,\lambda_d))\implies \cN(0, \Sigma)\,,\]
where $\lambda_i = \E_\nu [f_i(\sX_0)]$ and $\Sigma_{ij} = \Sigma_{ji} = \sum_{k=-\infty}^\infty \cov_\nu(f_i(\sX_0),f_j(\sX_k))$ for $1 \leq i,j\leq d$.
\end{corollary}
\begin{proof}[\textbf{\emph{Proof of Corollary~\ref{cor:clt:Rd}}}]
First note that the fact that the matrix $\Sigma$ is symmetric follows form the stationarity of the sequence $(\sX_i)$ under $ \nu$. From the expression for $\upsigma^2$ given by Proposition~\ref{prop:clt}, we see that for  every linear combination $f=\sum_i a_i f_i$ for $a\in\R^d$ of functions centered w.r.t.\ $\E_\nu$ and satisfying~\eqref{eq:f-growth}, one has  $\sum_{i=1}^T f(\sX_i)/\sqrt{T} \implies \cN(0,a^{\textsc t} \Sigma a)$ as $n\to\infty$.
The proof is concluded via the Cram\'er--Wold device.
\end{proof}

\begin{remark}\label{rem:clt-conditioning}
Both Proposition~\ref{prop:clt} and Corollary~\ref{cor:clt:Rd} hold identically under the measure $\mu_n(\cdot \mid \mathbf I_{x,T})$ (as opposed to $\pi_T = \mu_n(\cdot \mid \bar{\mathbf I}_{x,T})$). To see this, let $S_n=\frac{1}{\sqrt {T_n}} \sum_{t=1}^{T_n}(f(\sX_t)- \lambda)$; for every Borel set $B\subset \R$,
\begin{align*}
    \mu_n(S_n\in B \mid \mathbf I_{x,T})&\leq \frac{\mu_n(S_n \in B, \bar{\mathbf I}_{x,T})}{\mu_n(\mathbf I_{x,T})} + \mu_n(\bar{\mathbf I}_{x,T}^c \mid {\mathbf I}_{x,T}) = \pi_T(S_n \in B) \mu_n(\bar{\mathbf I}_{x,T}\mid \mathbf I_{x,T})+ \mu_n(\bar{\mathbf I}_{x,T}^c \mid {\mathbf I}_{x,T})\,,
\end{align*}
which is at most $ \pi_T(S_n\in B) + O(\exp(-(\beta-C)r_0 T))$ by Lemma~\ref{lem:tame}. Hence, $\limsup_{n\to\infty} \mu_n(S_n \in B \mid \mathbf I_{x,T})$ is at most $ \lim_{n\to\infty} \pi_T(S_n \in B)=\P(\cN(0,\upsigma_f^2)\in B)$, so (by Portmanteau)
 $S_n \implies \cN(0,\upsigma_f^2)$ under $\mu_n(\cdot\mid\mathbf I_{x,T})$.
\end{remark}

\subsection{Strategy of proof of the CLT}\label{subsec:strategy-clt}
We prove Proposition~\ref{prop:clt} by adapting a useful Stein's method type argument by Bolthausen~\cite{Bolthausen82} for treating stationary, mixing sequences of random variables. Our setting has several complications compared to~\cite{Bolthausen82}:
\begin{enumerate}
\item Our sequence of random variables, rather than being infinite and stationary, is a triangular array, where the individual laws $\pi_T$ change due to the conditioning on $\{\hgt(\cP_x)\geq T\}$.
\item Our $\alpha$-mixing estimates are invalid for base increments, and instead hold (see Proposition~\ref{prop:increment-mixing}) only beyond a prefix of $K\log T$ increments.
\item The increments are not stationary, and only become asymptotically stationary (see Proposition~\ref{prop:increment-stationarity}) away from the base and from the tip.  
\end{enumerate} 
The asymptotic stationarity obstacle was handled by slight modifications of Bolthausen's argument in~\cite{GLM15}; our proof follows a similar route, yet becomes somewhat simpler thanks to the nature of our $\alpha$-mixing estimates and control over higher moments of functions of the increment sequence.

\subsection{Proof of Proposition~\ref{prop:clt}}
The first step in establishing the CLT is a standard truncation argument, using our control on $\alpha$-mixing and on moments of the increment sequence. Take $b>0$ to be a large enough constant, in particular larger than the constant $K$ as given by Proposition~\ref{prop:increment-mixing} w.r.t.\ $\gamma=20$. Our first step is to truncate the prefix and suffix of the increment sequence, as well as individual increment contributions.
In what follows, recall $\lambda = \E_\nu[f(\sX_0)]$, and let $\upsigma_f^2 = \sum_{j=-\infty}^{\infty} \cov_\nu(f(\sX_0),f(\sX_j))$ for any function $f:\fX\to\R$.
\begin{claim}
In the setting of Proposition~\ref{prop:clt}, let 
\[ \ell = \lceil T^{1/5}\rceil \qquad\mbox{and}\qquad f_M(X):=f(X)\one_{\{|f(X)|\leq M\}}\,.\]
If $\frac1{\sqrt{T}}\sum_{j=\ell}^{T-\ell}(f_M(\sX_j)-\E_{\pi_T}[f_M(\sX_j)])\implies \cN(0,\upsigma_{f_M}^2)$ for every $M$ then $\frac1{\sqrt{T}}\sum_{j=1}^T (f(\sX_j)-\lambda) \implies \cN(0,\upsigma_f^2)$.
\end{claim}
\begin{proof}
Let us first look at the effect of omitting the $\ell_0$-prefix and $\ell_0$-suffix of the summation over
\[ Y_j = f(\sX_j) - \E_{\pi_T} f(\sX_j)\,,\]
where 
\[ \ell_0 := \lceil b\log T\rceil \] 
for some large $b>0$, taken to be at least $K$ from Proposition~\ref{prop:increment-mixing} for a choice of $\gamma=20$. Following this step, we will be able to truncate the $Y_j$'s, and thereafter omit the $\ell$-prefix and $\ell$-suffix of the sum.

Proposition~\ref{prop:base-exp-tail} (specifically, the exponential tail in~\eqref{eq:base-total-increment-decay}) implies that, for a sufficiently small $\kappa$, we have 
\[ \E_{\pi_T}\left[ e^{\kappa \sum_{i<\Tsp} |\cF(\sX_i)|}\right] \leq T^{1/3}\,.\]
For the spine increments, Proposition~\ref{prop:exp-tails-increments} (together with Lemma~\ref{lem:tame} on the tameness of the spine) shows that,
conditioned on $\{\sX_{\Tsp+j} : j < i\}$, the variable $\fm(\sX_{\Tsp+i})$ is dominated by an exponential variable with parameter $c_0\beta$ (for $c_0>0$ from that proposition). In particular, $\sum_{i=\Tsp}^{\ell_0}\fm(\sX_i)+\sum_{i> T-\ell_0} \fm(\sX_i)$ is stochastically dominated by a gamma-distributed random variable with parameters $(2\ell_0, c_0 \beta)$, which again satisfies
\[ \E_{\pi_T}\left[ e^{\kappa \big(\sum_{i=\Tsp}^{\ell_0} |\cF(\sX_i)|+\sum_{i>T-\ell_0} |\cF(\sX_i)|\big)
}\right] \leq T^{1/3}\]
(e.g., take $\kappa=(1-e^{-1/(8b)})(c_0 \beta)^{-1}$). 
Overall, the hypothesis~\eqref{eq:f-growth} implies, for a small enough $\kappa_0$, that 
\begin{align}\label{eq:ell0-prefix-L1}
\frac1{\sqrt{T}}\sum_{j\leq \ell_0} \Big(|f(\sX_j)|+|f(\sX_{T+1-j})|\Big) \xrightarrow{L^1}  0 \quad\mbox{under $\pi_T$}\,,\end{align}
so $\frac1{\sqrt{T}}\sum_{j\leq\ell_0}(Y_j + Y_{T+1-j})\to 0$ in probability, and hence does not affect the limiting law of $\frac1{\sqrt{T}}\sum_j Y_j$. 

Again recalling Proposition~\ref{prop:base-exp-tail}, each variable $Y_j$ is a function of a spine increment except with probability $\exp(-c \log T) = O(T^{-5})$ for a large enough choice of $b$. Consequently, as per the exponential tail on spine increments established by Proposition~\ref{prop:exp-tails-increments} and the hypothesis $|f(X)|\leq\exp[\kappa_0 |\cF(X)|]$ for all $X\in\fX$,
\begin{equation}\label{eq:Yj-exp-tail} 
\pi_T(|Y_j| \geq a) \leq a^{-c_0 \beta \kappa_0} + O(T^{-5}) \,.
\end{equation}
Moreover, on the event that the index $j$ is not a spine index,  $\fm(\sX_j)$ has an exponential tail beyond $K\log T$ by~\eqref{eq:base-total-increment-decay}. Combining these two implies that $Y_j$ has uniformly bounded moments of $k$-th order for small enough $\kappa_0(k)$. Namely, on the event $j>\Tsp$,  Proposition~\ref{prop:exp-tails-increments} implies that its $k$-th moment is finite as long as $\kappa <\kappa_0(k)$;  the event $j\leq \Tsp$, has probability $O(T^{-5})$, and in that case, we can bound $\E[|Y_j|^k]\leq O(e^{\kappa K\log T})$, so that an application of Cauchy--Schwarz implies that for each $k$, there exists $\kappa_0(k)$ such that for $\kappa<\kappa_0$,
\[\max_j\E_{\pi_T}[ |Y_j|^k] < C(\kappa,\beta, k)\,.\]
For random variables $Z_1, Z_2$, let $\sigma(Z_i)$ be the $\sigma$-algebra generated by $Z_i$ and define the  $\alpha$-mixing coefficient 
\[\alpha(Z_1,Z_2) := \max_{A_1 \in \sigma(Z_1), A_2 \in \sigma(Z_2)} |\P (Z_1 \in A_1, Z_2 \in A_2) - \P (Z_1 \in A_1) \P (Z_2\in A_2)|\,.
\]
Write $Y_j = Y'_j + Y''_j$ where $Y''_j = Y_j \one_{\{|Y_j|> M\}}$;  noting that
$ \E_{\pi_T} (Y_j'')^3 \leq (\E_{\pi_T} Y_j^4)^{3/4} \P(|Y_j|> M)^{1/4}$ by H\"older's inequality, and that every two random variables $Z_1, Z_2$ satisfy
\begin{equation}\label{eq:davydov}
\cov(Z_1, Z_2) \leq 2  \left(\alpha(Z_1, Z_2) \E [Z_1^{3}]\E [Z_2^{3}]\right)^{1/3}
\end{equation}
(see, e.g.,~\cite[\S1]{Rio17} for this inequality, originally by Davydov~\cite{Davydov68} with a larger constant pre-factor), one has
\begin{align*} \E_{\pi_T}\bigg( \frac1{\sqrt T} \sum_{j=\ell_0}^{T-\ell_0} (Y''_j - \E_{\pi_T} Y''_j)\bigg)^2 &\leq \frac{C}T \max_{ j\geq \ell_0} \E_{\pi_T} [(Y''_j)^3]^{\frac23} \sum_{j,k\geq\ell_0}\alpha(\sX_{ j},\sX_{ k})^{1/3} \\ 
 &\leq C' \max_{j\geq\ell_0} \sqrt{\P(|Y_j|>M)}\,,\end{align*}
using $\alpha(\sX_{j},\sX_{k})\leq C|k-j|^{-\gamma}$ for $k> j\geq \ell_0$ and
$\gamma=20$ by the above application of Proposition~\ref{prop:increment-mixing}.
As the expression on the right can be made arbitrarily small as a function of $M$, uniformly over $n$ (it is at most $C M^{-c_0 \kappa_0/2}+o(1)$  by~\eqref{eq:Yj-exp-tail}), we see that showing $\frac1{\sqrt{T}}\sum_{j=\ell_0}^{T-\ell_0} (Y'_j - \E_{\pi_T} Y'_j) \implies \cN(0,\upsigma^2_{f_M})$ as $n\to\infty$ for every fixed $M$, as well as $\upsigma_{f_M}^2\to\upsigma_f^2$, will imply that $\frac1{\sqrt{T}}\sum_{j=1}^{T} (Y_j-\E_{\pi_T} Y_j)\implies \cN(0,\upsigma_f^2)$. 

To verify that $\upsigma_{f_M}^2\to\upsigma_f^2$ as $M\to\infty$, recall from Corollary~\ref{cor:limiting-distribution} that $\nu(\fm(\sX_0)\geq r)\leq \exp(-c_0 \beta r)$, so for some $C,c>0$ we get
\[ \E_\nu \left|f(\sX_0) - f_M(\sX_0)\right|^3\leq \sqrt{\E_\nu\left[f(\sX_0)^6\right]\nu(|f(\sX_0)|\geq M)} \leq C M^{-c}
\]
using Corollary~\ref{cor:limiting-distribution} and~\eqref{eq:f-growth} to uniformly bound $\E_\nu [f(\sX_0)^6]$ and the $\nu$-probability of $|f(\sX_0)\geq M|$. Writing
\begin{align*} \upsigma_{f_M}^2 = \upsigma_{f}^2 &+ \sum_{k=-\infty}^\infty \cov\left(f_M(\sX_0)-f(\sX_0),f(\sX_k)\right) + \sum_{k=-\infty}^\infty \cov\left(f_M(\sX_0),f_M(\sX_k)-f(\sX_k)\right)\,,
\end{align*}
we can infer from the fact $\alpha(\sX_0,\sX_k)\leq C k^{-\gamma}$ under $\nu$, and another application of~\eqref{eq:davydov}, that
\[ \sum_{k=-\infty}^{\infty} \left|\cov\left(f_M(\sX_0)-f(\sX_0),f(\sX_k)\right)\right|\leq C' M^{-c'} \sum_{k} k^{-\gamma} \leq C'' M^{-c'}\,,\]
and the same holds for $\sum \cov(f_M(\sX_0),f_M(\sX_k)-f(\sX_k))$ in the same manner, implying $\upsigma_{f_M}^2\to\upsigma_f^2$.

Thus far we established that it suffices to show $\frac1{\sqrt{T}}\sum_{j=\ell_0}^{T-\ell_0}Y'_j\implies\cN(0,\upsigma_{f_M}^2)$ for every $M>0$. Note that
\[ 
\frac1T \E_{\pi_T}\Big[\Big(\sum_{j=\ell_0}^\ell \big(Y'_j +Y'_{T+1-j}\big) \Big)^2\Big] \leq \frac{(2\ell)^2}{T} \max_{\ell_0\leq j,k\leq T-\ell_0} |\E_{\pi_T}[Y'_j Y'_k]| = O\left(T^{-3/5} \right)
\]
since $|Y'_j|\leq M$ for all $j$. Hence, we may indeed replace $\ell_0$ by $\ell$ when considering $\sum Y'_j$, as the contributions to the limiting law by the $\ell$-prefix and $\ell$-suffix in $\frac1{\sqrt{T}}\sum_{j=\ell_0}^{T-\ell_0} Y'_j$ are negligible.

Finally, we wish to replace centering term $\E_{\pi_T} f(\sX_j)$ by $\lambda = \E_\nu f(\sX_0)$ for each $j=1,\ldots,T$. Recall from~\eqref{eq:ell0-prefix-L1} that $\frac1{\sqrt{T}}\sum_{j\leq\ell_0}\big(\E_{\pi_T}|f(\sX_j)|+\E_{\pi_T}|f(\sX_{T+1-j})|\big)\to 0$.
For each $j\geq \ell_0$, we have $\E_{\pi_T}|f(\sX_j)| \leq C$, as we had established above (following~\eqref{eq:Yj-exp-tail}). Therefore, we can neglect the $\ell$-prefix and $\ell$-suffix of the sequence of expectations, as $\frac1{\sqrt{T}}\sum_{j\leq\ell}\big(\E_{\pi_T}|f(\sX_j)|+\E_{\pi_T}|f(\sX_{T+1-j})|\big)=O(\ell/\sqrt{T})=o(1)$.
For each of the remaining indices $\ell\leq j\leq T-\ell$, by  Proposition~\ref{prop:increment-stationarity} (as used in the proof of Corollary~\ref{cor:limiting-distribution}), we have that
 $\|\pi_T(\sX_j\in\cdot)-\nu(\sX_j\in\cdot)\|_{\tv}\leq C \ell^{-\gamma}$.
 Hence, looking at the truncated function $f_{\ell}$, we have $|\E_{\pi_T} f_{\ell}(\sX_j) - \E_\nu f_{\ell}(\sX_j)| = O(\ell^{1-\gamma}) = O(T^{-2})$ (so the sum of these over all $j$ is $o(1)$), whereas $\E_{\pi_T}|(f-f_\ell)(\sX_j)|$ and $\E_\nu|(f-f_\ell)(\sX_j)|$ are each $O(\ell^{-c \beta \kappa_0})=o(1/T)$ by Cauchy--Schwarz, the $O(1)$ bounds on the means of $f(\sX_j)$ under $\pi_T$ and $\nu$, and the exponential tails of $\fm(\sX_j)$ together with~\eqref{eq:f-growth} and~\eqref{eq:base-total-increment-decay}.\end{proof}

Through the remainder of the proof, let $M>0$ and $Y_j = f_M(\sX_j)-\E_{\pi_T}f_M(\sX_j)$ (so that $|Y_j|\leq M$) for each $j=\ell,\ldots,T-\ell$, with the goal of showing that $\frac1{\sqrt{T}}\sum_{j=\ell}^{T-\ell}Y_j \implies \cN(0,\upsigma_{f_M}^2)$. We further assume w.l.o.g.\ that $f_M$ is non-constant (as this holds for all $M>M_0$ for some $M_0=M_0(f)>0$), whence
\[ \upsigma_{f_M}^2 = \sum_{j=-\infty}^\infty \cov_\nu (f_M(\sX_0), f_M(\sX_j)) = \lim_{t\to\infty} \frac1{2t}\var_{\nu}\bigg(\sum_{j=-t}^t f_M(\sX_j)\bigg) > 0\,,\]
by  Proposition~\ref{prop:lin-var}, applied to the bounded non-constant function $f_M$.
Defining 
\[ \fS_T := \sum_{j=\ell}^{T-\ell} \E_{\pi_T}\bigg[ Y_j \sum_{k=\ell}^{T-\ell}\one_{\{|k-j|\leq \ell\}} Y_k\bigg] \,,\]
 recall from Propositions~\ref{prop:increment-mixing}--\ref{prop:increment-stationarity} and Corollary~\ref{cor:limiting-distribution} that $ \sum_{|j|\leq \ell}\cov_\nu(f_M(\sX_0),f_M(\sX_j)) = \upsigma_{f_M}^2 + o(1)$ and
 \begin{align*} \sum_{k=\ell}^{T-\ell} &\left|\cov_{\pi_T}(f_M(\sX_j),f_M(\sX_k))\right|\one_{\{|k-j|>\ell\}} = O(T \ell^{-\gamma}) = o(1)\,,\\
  \sum_{k=\ell}^{T-\ell} &\left|\cov_{\pi_T}(f_M(\sX_j),f_M(\sX_k))-\cov_\nu(f_M(\sX_0),f_M(\sX_k))\right|\one_{\{|k-j|\leq\ell\}} =  O( \ell^{1-\gamma})=o(1)\end{align*}
for all $2\ell\leq j \leq T-2\ell$ (whereas both terms are $O(1)$ for $\ell\leq j \leq 2\ell$ and $T-2\ell\leq j \leq T-\ell$), we get
\[ \fS_T = (1+o(1))\var_{\pi_T}\Big(\sum_{k=\ell}^{T-\ell} Y_j\Big)\quad\mbox{ as well as }\quad \fS_T = (1+o(1))\upsigma^2_{f_M} T\,.\]

The following simple argument of Bolthausen~\cite{Bolthausen82} gives a convenient approach for establishing CLTs for mixing random fields, even in the situation where (unlike the original setting of~\cite{Bolthausen82}) the sequence of increments is only asymptotically  stationary. At the heart of the argument is the following observation:
\begin{lemma}[{\cite[Lemma~2]{Bolthausen82}}]
If $(Z_n)$ is a sequence of real-valued random variables with $\sup_n \E Z_n^2 < \infty$ and
\begin{equation}\label{eq:bolthausen-cond}
\lim_{n\to\infty} \E\left[(i\lambda -Z_n)e^{i\lambda Z_n}\right]	= 0\quad\mbox{ for every $\lambda\in\R$}\,,
\end{equation}
then $Z_n$ converges weakly to the standard Gaussian $\cN(0,1)$.
\end{lemma}
(Indeed, tightness is implied by the uniform bound on the $\E Z_n^2$, and verifying that every subsequential limit point is standard Gaussian can be derived from~\eqref{eq:bolthausen-cond}, as a variable $Z$ having the law of such a limit point has $\E[ f'(Z) - Zf(Z) ] = 0$ for every $f\in C^1(\R)$, hence must be standard Gaussian by Stein's characterization.) Define the random variables $Z_T$ and for $r\in \llb \ell, T- \ell\rrb$, $Z_{r,T}$ by 
\[ Z_T = \frac{1}{\sqrt{\fS_T}}  \sum_{j=\ell}^{T-\ell} Y_j\,, \qquad \mbox{and} \qquad Z_{r,T} = \frac1{\sqrt{\fS_T}} \sum_{j=\ell}^{T-\ell}\one_{\{|j-r|\leq \ell\}} Y_j\,,
\]
We aim to verify~\eqref{eq:bolthausen-cond} for the random variables $(Z_T)$ 
 via the following useful decomposition of the term $(i\lambda - Z_n)e^{i\lambda Z_n}$ given in~\cite{Bolthausen82}: 
for every $\lambda\in\R$,
\[ (i\lambda - Z_T)e^{i \lambda Z_T}  = \Xi_1 - \Xi_2 - \Xi_3\,,\]
for
\begin{align}
\Xi_1 &= i\lambda e^{i\lambda Z_T} \bigg(1- \frac{1}{\sqrt{\fS_T}} \sum_{r=\ell}^{T-\ell} Y_r  Z_{r,T}\bigg)\,, \\
\Xi_2 &= \frac{1}{\sqrt{\fS_T}} e^{i\lambda Z_T} \sum_{r=\ell}^{T-\ell} Y_r\Big(1- e^{-i\lambda Z_{r,T}}  -i \lambda Z_{r,T}\Big)\,, \\ 
\Xi_3 &= \frac{1}{\sqrt {\fS_T}}  \sum_{r=\ell}^{T-\ell} Y_r  e^{i \lambda (Z_T - Z_{r,T})}\,,
\end{align}
(where the equality used only that $Z_T = \frac1{\sqrt{\fS_T}}\sum_{j=\ell}^{T-\ell} Y_j$, irrespective of the definitions of $Z_{r,T}$ and $\ell$).
Thus, it will suffice to show that $\E_{\pi_T}[\Xi_i]\to 0$ as $n\to\infty$ for each $i=1,2,3$ in order to verify~\eqref{eq:bolthausen-cond} for $(Z_T)$.

For the first of these terms, recall that by the definition of $\fS_T$ and $Z_{r,T}$ one has that
\[  \E_{\pi_T} \bigg[ \frac{1}{\sqrt{\fS_T}}\sum_{r=\ell}^{T-\ell} Y_r Z_{r,T} \bigg] = \frac{1}{\fS_T}
 \sum_{r=\ell}^{T-\ell} \E_{\pi_T}\bigg[ Y_r \sum_{j=\ell}^{T-\ell}\one_{\{|j-r|\leq \ell\}} Y_j\bigg] = 1\,,\]
whence
\begin{align*} \E_{\pi_T}|\Xi_1|^2 &= \lambda^2 \var_{\pi_T} \bigg(\frac1{\sqrt{\fS_T}}\sum_{r=\ell}^{T-\ell}Y_r Z_{r,T}\bigg)
=\frac{\lambda^2}{\fS_T^2} \var_{\pi_T}\bigg(\sum_{r=\ell}^{T-\ell}Y_r \sum_{\substack{j\in\llb\ell,T-\ell\rrb \\ |j-r|\leq \ell}} Y_j\bigg) \\
&\leq \frac{\lambda^2}{\fS_T^2} \sum_{r,r'\in\llb\ell,T-\ell\rrb} \sum_{\substack{j,j'\in\llb\ell,T-\ell\rrb\\ |j-r|\leq \ell|, |j'-r'|\leq\ell}} \left|\cov_{\pi_T}(Y_j Y_r, Y_{j'} Y_{r'})\right|\,.
\end{align*}
Splitting the sum over $r,r'$ according to $|r-r'|$, we see that if $|r'-r|\geq 3\ell$ then each of the $O(T^2\ell^2)$ summands satisfies
\[ \big|\cov_{\pi_T}(Y_j Y_r, Y_{j'} Y_{r'})\big| \leq M^4 c \ell^{-\gamma} = O(T^{-3})\]
for some $c>0$, by Proposition~\ref{prop:increment-mixing}; on the other hand, there are $O(T\ell^3)$ summands $r,r',j,j'$ with $|r-r'|\leq 3\ell$, each of which is uniformly bounded by $M^4$. Altogether, recalling that $\fS_T= (1+o(1))\upsigma_{f_M}^2 T$, we deduce
\[ \E_{\pi_T}|\Xi_1|^2 = O(\ell^3 / T) = O\big(T^{-2/5}\big)\,.\]

For the second term, observe that
\[ |\Xi_2| \leq M \frac{1}{\sqrt{\fS_T}} \sum_{r=\ell}^{T-\ell}\left| 1-e^{-i\lambda Z_{r,T}}-i\lambda Z_{r,T}\right| \leq M\lambda^2 \frac1{\sqrt{\fS_T}}  \sum_{r=\ell}^{T-\ell}\left| Z_{r,T} \right|^2 \,,\]
where the last inequality was obtained by Taylor expanding $1-\cos(\lambda Z_{r,T})$ and $\sin(\lambda Z_{r,T})-\lambda Z_{r,T}$ (the real and imaginary parts of of each summand, respectively). Since $|Z_{r,T}|\leq (2\ell+1)M/\sqrt{\fS_T}$ by its definition (and the truncation bound on the $Y_j$'s), it follows that
\[ |\Xi_2| = O\bigg(\frac{T\ell^2}{\fS_T^{3/2}}\bigg) = O\left(T^{-1/10}\right)\,.\]

Finally, when treating $\Xi_3$, we can use Proposition~\ref{prop:increment-mixing} to decompose $\E_{\pi_T} \left[Y_r   e^{i\lambda(Z_T-Z_{r,T})}\right]$ as follows:
\begin{align*} \bigg|\E_{\pi_T} \left[Y_r e^{i\lambda(Z_T-Z_{r,T})}\right]\bigg| &= \bigg|\E_{\pi_T} \bigg[Y_r  \prod_{j=\ell}^{r-\ell} e^{i\lambda \fS_T^{-1/2} Y_j} 
\prod_{j=r+\ell}^{T-\ell} e^{i\lambda \fS_T^{-1/2} Y_j}\bigg]  \bigg|\\
& \leq  \bigg|\E_{\pi_T} \bigg[Y_r  \prod_{j=\ell}^{r-\ell} e^{i\lambda \fS_T^{-1/2} Y_j}\bigg] 
\E_{\pi_T}\bigg[\prod_{j=r+\ell}^{T-\ell} e^{i\lambda \fS_T^{-1/2} Y_j}\bigg] \bigg|+ M c \ell^{-\gamma}
\,,
\end{align*}
(using that the variables in the two expectations in the last line are at most $M$ and $1$ in absolute value, respectively). 
Since $|\E_{\pi_t}[\prod_j e^{i\lambda\fS_T^{-1/2}Y_j}]|\leq 1$, we can apply Proposition~\ref{prop:increment-mixing} again to obtain that
the last expression is, in turn, at most
\begin{align*} \bigg|\E_{\pi_T} \bigg[Y_r  \prod_{j=\ell}^{r-\ell} e^{i\lambda \fS_T^{-1/2} Y_j}\bigg]\bigg|+M c \ell^{-\gamma}  &\leq 
 \bigg(\bigg|\E_{\pi_T}[Y_r] \E_{\pi_T}\bigg[\prod_{j=\ell}^{r-\ell} e^{i\lambda \fS_T^{-1/2} Y_j}\bigg]\bigg| + M c\ell^{-\gamma}\bigg)
+ M c\ell^{-\gamma}	= 2M c \ell^{-\gamma}\,,
\end{align*}
using the fact that $\E_{\pi_T} Y_r = 0$. This concludes the proof.
\qed

\subsection{Proof of Corollary~\ref{maincor:displacement-surface-area}}
We wish to apply Theorem~\ref{mainthm:clt} with specific choices of observables, that contain the information about the distribution of the tip and volume and surface area of the pillar. We begin with item (1), regarding the distribution of the tip, $(Y_1, Y_2, Y_3)$. Define observables $f_i: \fX\to \R$  for $i=1,2,3$, by
\[
f_i(X) = (v_{+}(X) - v_-(X))\cdot e_i,
\]
where $v_+(X)$ is the midpoint of the highest cell of a rooted increment $X\in \fX$, and $v_-(X)$ is the midpoint of its lowest cell, i.e., $(\frac 12 , \frac 12, \frac 12)$. 
Then, we can express, for $(Y_1, Y_2,Y_3)$ the tip of $\cP_x$,
\begin{align*}
    \Big| \frac{Y_i - x_i}{\sqrt T} - \frac{1}{\sqrt T}\sum_{j=1}^{T}f_i (\sX_j)\Big|\leq \frac{1}{\sqrt T}\big(\diam (\sB_x)+ \hgt(v_{\Tsp})+\frac 12  + |\cF(\sX_{>T})|\big)\,.
\end{align*}
By Proposition~\ref{prop:base-exp-tail}, both $\E_{\pi_T} [\diam (\sB_x)^2]$ and $\E[\hgt(v_{\Tsp})^2]$ are $ O(\log^2 T)$, and by Lemma~\ref{lem:tip-highest-increment}, $\E[|\cF(\sX_{>T})|^2]$ is $O(1)$. Hence, the right-hand side goes to 0 in probability as $T\to\infty$, and a CLT for $\frac{1}{\sqrt T}\sum_{j=1}^{T}f_i(\sX_j)$ yields the same CLT for $\frac{Y_i- x_i}{\sqrt T}$. Since $|f_i(X)|\leq \fm(X)$ for every $X\in \fX$, by Corollary~\ref{cor:clt:Rd}, 
\begin{align*}
    \frac{(Y_1, Y_2, Y_3)- (x_1, x_2,0) - (\lambda_1,\lambda_2, \lambda_3)T}{\sqrt T}\implies \cN\big(0,\Sigma\big)\,,
\end{align*}
where $\lambda_{i}= \E_{\nu}[f_i(\sX_0)]$ and $\Sigma_{i,j} = \sum_{k\in \Z} \cov(f_i(\sX_0)f_j(\sX_k))$. 

The observables $f_1,f_2$ are anti-symmetric with respect to reflections about the plane with outward normal $e_1,e_2$, so by item (1) Proposition~\ref{prop:mean-zero-observables}, they have $\E_{\nu} [ f_i(\sX_0)]=\lambda_i = 0$ (though $f_i$ are not bounded, this follows by truncating $f_i$, and using Corollary~\ref{cor:limiting-distribution} to deduce that the truncated means converge to the true means). The heights $f_3$ are  at least one, and thus $\lambda_3 \geq 1$. 
Since $f_i$ are  invariant under reflection about the plane with outward normal $e_j$ (for $j \neq i$, $j\in \{1,2\}$), by item (2) of Proposition~\ref{prop:mean-zero-observables}, the off-diagonals of $\Sigma$ are $0$ (again truncating the observables and noticing that the truncated covariances converge to the true covariances). By Proposition~\ref{subsec:linear-variance} and the observation that $f_i$ are non-constant on $\fX$, the diagonals of $\Sigma$ are positive, say $\upsigma_{f_i}^2>0$. It remains to verify that $\upsigma^2_{f_1} = \upsigma^2_{f_2}$. Note that for every $\Tsp\leq j,k\leq T$ the observable  $g(\sX_j,\sX_k)=f_1(\sX_j)f_1(\sX_k)-f_2(\sX_j) f_2(\sX_k)$ is anti-symmetric in application of the map that rotates the increments above some $\tau_L$ by $\frac{\pi}{2}$, as long as $j>\tau_L$. Analogously to  Proposition~\ref{prop:mean-zero-observables}, we would then see that $\E_\nu[g(\sX_j, \sX_k)] = 0$ for every $T^{\frac 14}\leq j<k\leq T$ as long as $|k-j|\leq T^{\frac 14}$. By linearity and the decay estimate of item (2) of Corollary~\ref{cor:limiting-distribution}, we see that $\E_\nu[\sum_{j\in \Z} f_1 (\sX_0)f_1(\sX_j)] = \E_{\nu} [\sum_{j\in \Z} f_2(\sX_0)f_2(\sX_j)]$.  

Item (2) follows in a similar fashion. Let $f_V(X) = |\cC(X)| -1$ and $f_A(X)= |\cF(X)| -4$. We can bound 
\begin{align*}
    \Big||\cC(\cP_x)| -\sum_{i\leq T}f_V(\sX_i)\Big| & \leq (\diam  (\sB_x))^2(\hgt(v_{\Tsp})+\tfrac 12)+|\cF(\sX_{>T})|^3\,, \qquad \mbox{and} \\ \Big||\cF(\cP_x)|-\sum_{i\leq T} f_A(\sX_i)\Big| & \leq 4 (\diam  (\sB_x))^2 (\hgt(v_{\Tsp})+\tfrac 12) + |\cF(\sX_{>T})|\,.
\end{align*}
Thus, as before, by Proposition~\ref{prop:base-exp-tail} and Lemma~\ref{lem:tip-highest-increment}, a CLT for $\frac{1}{\sqrt T}\sum_i f_{V}(\sX_i)$ implies the same CLT for $\frac 1{\sqrt T} |\cC(\cP_x)|$ and a CLT for $\frac{1}{\sqrt T} \sum_i f_A(\sX_i)$ implies the same CLT for $\frac{1}{\sqrt T}|\cF(\cP_x)|$. Since both $f_V(X)$ and $f_A(X)$ are positive, are at most $4+\fm(X)^2$, and are non-constant, they satisfy central limit theorems with positive means and variances, implying the same for the volume and surface area of $\cP_x$. \qedhere 

\subsection*{Acknowledgment} We thank the referees for valuable suggestions. E.L.~was supported in part by NSF grant DMS-1812095.

\bibliographystyle{abbrv}
\bibliography{ising}

\end{document}